%% file: main.tex
\documentclass[12pt,twoside]{report}
\usepackage[utf8]{inputenc}
\usepackage{graphicx}
\graphicspath{{images/}}
\usepackage[a4paper,width=150mm,top=25mm,bottom=25mm,bindingoffset=6mm]{geometry}

\usepackage{fancyhdr}
\usepackage{amsmath}
\usepackage{amssymb}
\usepackage{amsthm}
\usepackage{stmaryrd}
\usepackage{mathrsfs}
\usepackage{tikz-cd}
\usepackage{comment}
\usepackage{hyperref}

\usepackage{tikz}
\usetikzlibrary{positioning}

\newtheorem{proposition}{Proposition}[section]
\newtheorem{definition}[proposition]{Definition}
\newtheorem{lemma}[proposition]{Lemma}
\newtheorem{theorem}[proposition]{Theorem}
\newtheorem{remark}[proposition]{Remark}
\newtheorem{example}[proposition]{Example}
\newtheorem{corollary}[proposition]{Corollary}

\newtheorem*{theorem*}{Theorem}
\newtheorem*{example*}{Example}

\newtheorem{propositionApp}{Proposition}[chapter]
\newtheorem{definitionApp}[propositionApp]{Definition}
\newtheorem{lemmaApp}[propositionApp]{Lemma}

\title{
        {Braided Commutative Geometry and Drinfel'd Twist Deformations}\\
        {~}\\
        \includegraphics[width=0.30\textwidth]{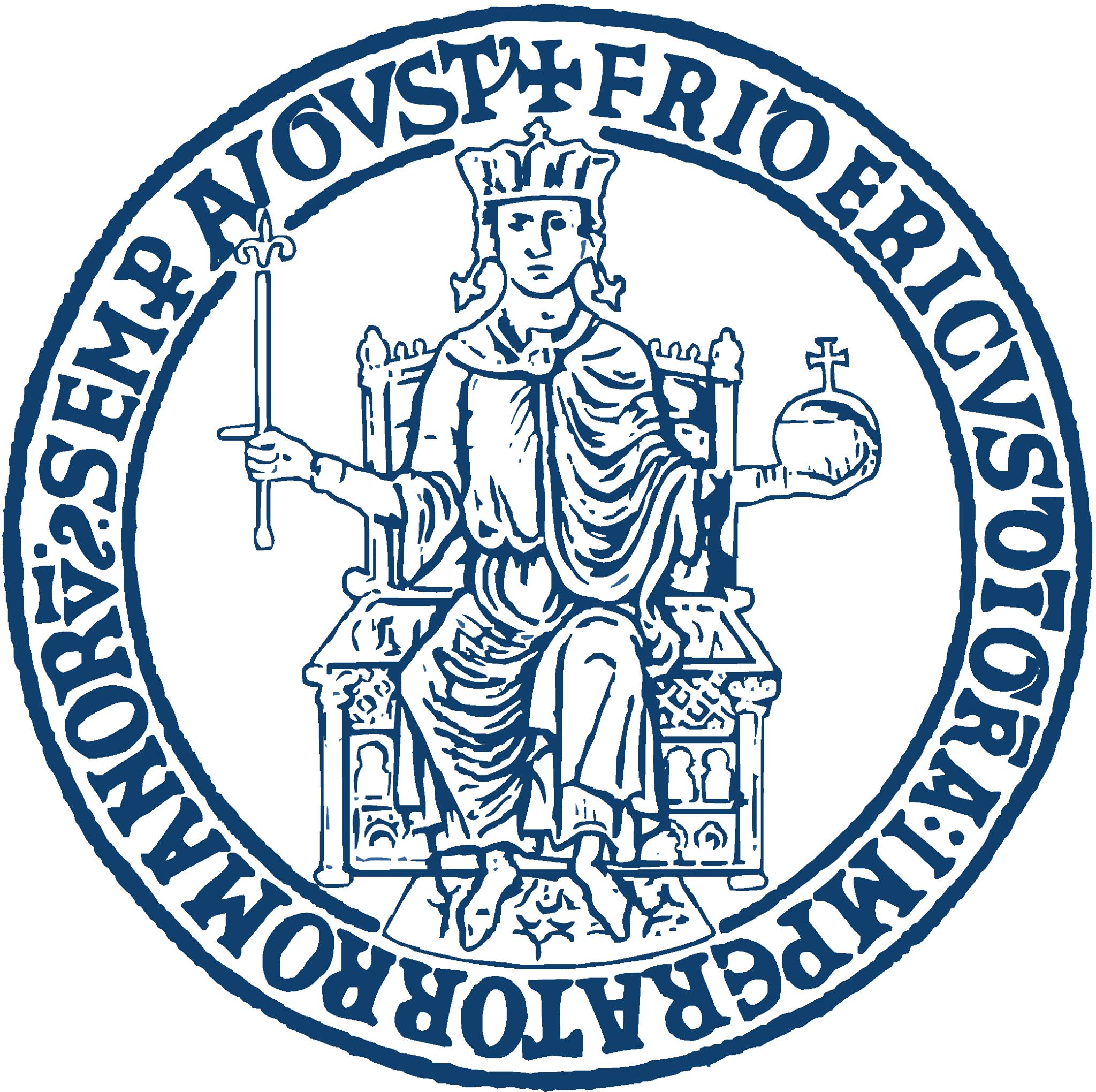}\\
        \vspace{0.5cm}
        \large{Università degli Studi di Napoli Federico II\\
        Dottorato di ricerca in Scienze Matematiche e Informatiche\\
        XXXII Ciclo}\\
        \vspace{0.3cm}
        \large{Tesi di Dottorato di Ricerca}\\
        \vspace{1.8cm}
        \textbf{\large{Thomas Weber}}
\vspace{1.5cm}
\begin{center}
\begin{tabular}{lcl}
\textbf{Advisors:}  &  \hspace{4.5cm} & \textbf{Referees:} \\
Prof. Francesco D'Andrea & & Prof. Paolo Aschieri \\
Prof. Gaetano Fiore & & Prof. Andrzej Borowiec
\end{tabular}
\end{center} 
}
\author{~}
\date{\normalsize Dicembre 2019}

\begin{document}

\maketitle

\newpage
\thispagestyle{plain} 
\mbox{}

\chapter*{Abstract}
This thesis revolves around the notion of twist star product, which is a certain
type of deformation quantization induced by quantizations of a symmetry of the system.
On one hand we discuss obstructions of twist star products, while on the other hand
we provide a recipe to obtain new examples as projections from known ones.
Furthermore, we construct a noncommutative Cartan calculus on braided commutative algebras,
generalizing the calculus on twist star product algebras. The starting point is the
observation that Drinfel'd twists not only deform the algebraic structure of
quasi-triangular Hopf algebras and their representations but also induce star products on
Poisson manifolds with symmetry. We further investigate the
correspondence of Drinfel'd twists and classical $r$-matrices as well as twist
deformation of Morita equivalence bimodules. It turns out that connected compact
symplectic manifolds are homogeneous spaces if they admit a twist star product and that
we can assume the corresponding classical $r$-matrix to be non-degenerate. Furthermore,
invariant line bundles with non-trivial Chern class and twists star products cannot
coexist if they are based on the same symmetry. In particular, the symplectic $2$-sphere
and the connected orientable symplectic Riemann surfaces of genus $g>1$ do not admit 
a star product induced by a Drinfel'd twist, while the complex projective spaces cannot
be endowed with a twist star product based on a matrix Lie algebra.
Another goal of the thesis is
to study braided commutative algebras and provide a noncommutative Cartan calculus on them,
in complete analogy to differential geometry. Notably, this recovers the calculus
on twist star product algebras. We further discuss equivariant covariant
derivatives and metrics, resulting in the existence and uniqueness of an
equivariant Levi-Civita covariant derivative for any non-degenerate equivariant metric.
We also verify that the constructions are compatible with Drinfel'd twist gauge equivalences.
Under certain conditions, the braided geometry projects to submanifold algebras and 
twist deformation commutes with the projection. As a consequence, twisted products
can be projected to submanifold algebras if the latter are respected by the symmetry.
The thesis is partially based on the following two papers and two preprints.
\begin{enumerate}
\item[i.)] \textsc{{D'Andrea}, F. and {Weber}, T.:}
\newblock\emph{Twist star products and Morita equivalence}.
\newblock C. R. Acad. Sci. Paris, 355(11):1178-1184, 2017.

\item[ii.)] \textsc{{Bieliavsky}, P. and {Esposito}, C. and {Waldmann}, S. and {Weber}, T.:}
\newblock\emph{Obstructions for twist star products}.
\newblock Lett. Math. Phys., 108(5):1341–1350, 2018.

\item[iii.)] \textsc{{Weber}, T.:}
\newblock\emph{Braided Cartan Calculi and Submanifold Algebras}.
\newblock Preprint arXiv:1907.13609, 2019.

\item[iv.)] \textsc{{Fiore}, G. and {Weber}, T.:}
\newblock In preparation, 2019.
\end{enumerate}

\chapter*{}

\begin{flushright}
Für meine Großeltern Hedi und Toni
\end{flushright}


\chapter*{Declaration}
I declare that this thesis was composed by myself, that the work contained herein is
my own except where explicitly stated otherwise in the text, and that this work has
not been submitted for any other degree or professional qualification.
\vspace{2cm}
\begin{flushright}
Thomas Weber
\end{flushright}

\chapter*{Acknowledgements}

First of all I have to thank my advisors Francesco and Gaetano for their patience and
constant support. Their guidance had a great influence on the outcome of the thesis.
In the same way I have to mention Chiara Esposito, Luca Vitagliano and Stefan Waldmann,
who acted as if they were my co-supervisors and supported me throughout many years.
Thank you! I am also grateful to the PhD committee for giving me the opportunity to
spend my PhD in such a beautiful and interesting city. Studying alongside Jonas Schnitzer,
a fellow student of mine from Würzburg, made it a lot easier for me to gain ground in Naples.
In his own words, it is not clear if this instance was "blessing or curse". The same applies to
my flatmates Valentino, Luca (alias Lupo di Mare) and Nicola, who introduced me to the city,
the Southern Italian cuisine and lifestyle.
I want to thank all my colleagues for many interesting conversations and discussions at conferences,
schools and workshops. I am happy that I was able to visit so many fascinating places.
Last but not least I want to thank my family and friends for keeping in contact and caring
about "the emigrant". It was amazing that so many of you were visiting me
(some even more than once or for a whole semester) despite the distance. 
You make me feel home wherever I go.

\tableofcontents

\chapter{Introduction}
\input{chapters/introduction.tex}

\chapter{Quasi-Triangular Hopf Algebras and their Representations}\label{chap02}
\input{chapters/chapter02.tex}

\chapter{Obstructions of Twist Star Products}\label{chap03}
\input{chapters/chapter25.tex}
\chapter{Braided Cartan Calculi}\label{chap04}
\input{chapters/chapter03.tex}

\chapter{Submanifold Algebras}\label{chap05}
\input{chapters/chapter04.tex}

\appendix

\renewcommand\chaptername{Appendix}

\chapter{Monoidal Categories}\label{App01}
\input{chapters/appendixA.tex}

\chapter{Braided Graßmann and Gerstenhaber Algebras}\label{App02}
\input{chapters/appendixB.tex}

\newpage

\fancyhead[RE]{BIBLIOGRAPHY}
\fancyhead[LO]{\leftmark}
\fancyhead[LE,RO]{\thepage}

\end{document}

%% file: chapters/introduction.tex
In this thesis we employ techniques from quantum group theory to
give obstructions for twist deformation quantization
on several classes of symplectic manifolds, while
new examples of twist star products are obtained via submanifold
algebra projection. Motivated from this quantization procedure,
we further construct a noncommutative
Cartan calculus on any braided commutative algebra, as well as an 
equivariant Levi-Civita covariant derivative. This generalizes
and unifies the classical Cartan calculus of differential geometry and
the calculus on twist star product algebras (see e.g. \cite{Aschieri2006}).
This first chapter will
serve as a motivation, while we also lay down the agenda of the thesis and notation.

\section*{Hamiltonian Mechanics}

Hamilton's equations of motion provide a good description of macroscopic objects.
We sketch how this axiomatic framework can be
implemented in terms of Poisson geometry, following \cite{B2001}~Chap.~0 
and \cite{WaldmannBuch}~Chap.~1.
The dynamics of a \textit{Hamiltonian system} are controlled by a Hamiltonian function
$H(q,p)\in\mathbb{R}$ via \textit{Hamilton's equations of motion}
\begin{equation}\label{Ham}
    \begin{pmatrix}
    \dot{q}(t)\\
    \dot{p}(t)
    \end{pmatrix}
    =\underbrace{\begin{pmatrix}
    0 & 1_n\\
    -1_n & 0
    \end{pmatrix}
    \begin{pmatrix}
    \frac{\partial H}{\partial q}\\
    \frac{\partial H}{\partial p}
    \end{pmatrix}}_{=X_H}
    (q(t),p(t)),
\end{equation}
where $(q,p)\in\mathbb{R}^{2n}$ are phase-space coordinates. 
A solution $x(t)=(q(t),p(t))$ of (\ref{Ham}) is given by the flow of the
Hamiltonian vector field $X_H\colon\mathbb{R}^{2n}\rightarrow
\mathbb{R}^{2n}$, with respect to some initial conditions
$x(0)=(q_0,p_0)\in\mathbb{R}^{2n}$.
For example, we may consider a particle with mass $m$ influenced by a conservative
force $F\colon\mathbb{R}^3\rightarrow\mathbb{R}^3$. The movement of the particle 
is represented by a smooth curve $q\colon\mathbb{R}\rightarrow\mathbb{R}^3$, where
$q(t)\in\mathbb{R}^3$ is the position of the particle at time $t\in\mathbb{R}$.
The velocity $\dot{q}$ and acceleration $\ddot{q}$ of the particle are defined 
by the first and second derivative of $q$, respectively. Its momentum is given by
the product $p=m\dot{q}$. Since $F$ is conservative, there is a potential
$V\colon\mathbb{R}^3\rightarrow\mathbb{R}$ such that $F(q(t))=-(\nabla V)(q(t))$, 
where $\nabla$ denotes the gradient. Then, \textit{Newton's second} law
\begin{equation}\label{New}
    F(q(t))=m\ddot{q}(t)
\end{equation}
determines the dynamics of the particle. Employing the
Hamiltonian function $H(q,p)=\frac{p^2}{2m}+V(q)$, the second order
partial differential equation
(\ref{New}) is equivalent to the system (\ref{Ham}) of differential equations
of order one. In other words, in this case Hamilton's equations of motion constitute
Newtonian mechanics.
There is a notion of symplectic manifolds and more general of Poisson manifolds
(see Section~\ref{SecObs1}) that allow the formulation of a Hamiltonian formalism.
The crucial ingredient is the \textit{Poisson bracket} $\{\cdot,\cdot\}$, which is 
a Lie bracket on the algebra of smooth functions, satisfying a Leibniz
rule in each entry. For every function $f$ on a Poisson manifold $(M,\{\cdot,\cdot\})$
with Hamiltonian function $H$ and corresponding Hamiltonian flow
$\Phi_t^*\colon\mathscr{C}^\infty(M)\rightarrow\mathscr{C}^\infty(M)$, the time evolution $f(t)=\Phi^*_t(f)$ of $f$ is determined by
\begin{equation}
    \frac{\mathrm{d}}{\mathrm{d}t}\Phi_t^*(f)
    =\{\Phi^*_t(f),H\}.
\end{equation}
We regain (\ref{Ham}) if $M=\mathbb{R}^{2n}$ and $f=(q,p)$.
In the next section we examine the passage from
classical mechanics to quantum mechanics, in the lines of
\cite{WaldmannBuch}~Chap.~5.

\section*{Deformation Quantization}

At the beginning of the $20$th century \textit{quantum mechanics} 
revolutionized the conceptional and philosophical understanding of physics.
It recovers classical mechanics via a limit procedure.
In \cite{Dirac} Dirac discussed the \textit{canonical commutation relations}
\begin{equation}\label{Dirac}
    [q^j,p_k]=\mathrm{i}\hbar\delta^j_k,
\end{equation}
stating that position and momentum do not commute as quantum observables and
their noncommutativity depends on the \textit{Planck constant} $\hbar$.
In this picture, states are square integrable \textit{wave functions} in
$L^2(\mathbb{R}^3,\mathrm{d}^3q)$ and observables are (unbounded) self-adjoint
operators of wave functions. 
Interpreting position and momentum of a particle as
operators $\hat{q}^j$ and $\hat{p}_k$, which act on a wave function $\psi$
as $(\hat{q}^j\psi)(q)=q^j\psi(q)$ and $(\hat{p}_k\psi)(q)
=-\mathrm{i}\hbar\frac{\partial\psi}{\partial q^k}(q)$, we obtain,
\begin{equation}\label{CanComRel}
    [\hat{q}^j,\hat{p}_k]\psi(q)
    =-\mathrm{i}\hbar q^j\frac{\partial\psi}{\partial q^k}(q)
    +\mathrm{i}\hbar\bigg(\delta^j_k\psi(q)+q^j\frac{\partial\psi}{\partial q^k}(q)\bigg)
    =\mathrm{i}\hbar\delta^j_k\psi(q),
\end{equation}
which is a representation of (\ref{Dirac}).
In the \textit{classical
limit}, where $\hbar$ is considered to be arbitrarily small,
the operators $\hat{q}^j$ and $\hat{p}_k$ become commutative and we recover
the situation known from classical mechanics. This procedure of 
\textit{canonical quantization} can be formalized in the following way:
a \textit{quantization map}
\begin{equation}\label{QuantMap}
    Q\colon\mathscr{C}^\infty(M)\rightarrow\mathrm{End}_\mathbb{C}(\mathcal{H})
\end{equation}
assigns to any smooth complex-valued function $f$ on a Poisson manifold
$(M,\{\cdot,\cdot\})$ a $\mathbb{C}$-linear operator $Q(f)=\hat{f}
\colon\mathcal{H}\rightarrow\mathcal{H}$ on a complex Hilbert space $\mathcal{H}$.
The map (\ref{QuantMap}) should be $\mathbb{C}$-linear, satisfy
$Q(1)=\mathrm{id}_\mathcal{H}$ and
\begin{equation}\label{QLieAlg}
    Q(\{f,g\})=\frac{1}{\mathrm{i}\hbar}[Q(f),Q(g)]
\end{equation}
on generators $f,g\in\mathscr{C}^\infty(M)$.
However, it was pointed out in \cite{Dirac}, and later by \cite{Groenewold,vHove},
that (\ref{QLieAlg}) does not hold on the whole algebra of functions but
instead only up to higher orders of $\hbar$, i.e.
\begin{equation}\label{HighOrd}
    Q(\{f,g\})=\frac{1}{\mathrm{i}\hbar}[Q(f),Q(g)]+\mathcal{O}(\hbar).
\end{equation}
A particular quantization in accordance with (\ref{HighOrd}) is given by the notion of
deformation quantization: in \cite{Bayen1978} the authors suggested to 
\textit{deform} the algebra
structure of $\mathscr{C}^\infty(M)$ rather than the observables itself.
Namely, one considers $\hbar$ as a formal parameter and calls an associative
product
\begin{equation}
    f\star g
    =\sum_{n=0}^\infty C_n(f,g)
    \in\mathscr{C}^\infty(M)[[\hbar]]
\end{equation}
on the formal power series $\mathscr{C}^\infty(M)[[\hbar]]$
a quantization or \textit{star product}
on $(M,\{\cdot,\cdot\})$ if $C_n$ are bidifferential operators on $M$ and
$\star$ deforms the pointwise product of
functions, i.e. $f\star g=fg+\mathcal{O}(\hbar)\in\mathscr{C}^\infty(M)[[\hbar]]$,
such that $1\star f=f=f\star 1$ and
\begin{equation}
    \{f,g\}=\frac{1}{\mathrm{i}\hbar}[f,g]_\star+\mathcal{O}(\hbar)
\end{equation}
hold for all $f,g\in\mathscr{C}^\infty(M)$, where
$[f,g]_\star=f\star g-g\star f$. 
In other words, this amounts to consider
a noncommutative algebra $(\mathscr{C}^\infty(M)[[\hbar]],\star)$ rather
than endomorphisms on a Hilbert space and to set $Q=\mathrm{id}$. This agenda
is know as \textit{deformation quantization} and has proven its profundity by
many publications and a great interest in the community of mathematical physics.
We discuss this in more detail in Section~\ref{SecObs1}.

\section*{Drinfel'd Twists and Quantization}

The term \textit{quantum group} was proposed by Drinfel'd to refer to
Hopf algebras in the context of quantum integrable systems. 
In \cite{Dr83} he related solutions $r\in\Lambda^2\mathfrak{g}$ of the 
\textit{classical Yang-Baxter equation} $\llbracket r,r\rrbracket=0$
(see Section~\ref{Sec-r-matrix})
on a Lie algebra $\mathfrak{g}$ to normalized $2$-cocycles 
\begin{equation}
    \mathcal{F}=1\otimes 1+\hbar r+\mathcal{O}(\hbar^2)
    \in(\mathscr{U}\mathfrak{g}\otimes\mathscr{U}\mathfrak{g})[[\hbar]]
\end{equation}
on formal power series of the universal enveloping algebra
$\mathscr{U}\mathfrak{g}$. Note that $\mathscr{U}\mathfrak{g}$ is a Hopf
algebra with coproduct, counit and antipode defined on primitive elements
$x\in\mathfrak{g}$ by $\Delta(x)=x\otimes 1+1\otimes x$, $\epsilon(x)=0$
and $S(x)=-x$ (see Section~\ref{Sec2.1}).
Such a classical $r$-matrix is equivalent to
a $G$-invariant Poisson bivector $\pi_r$ on the Lie group $G$ corresponding to
$\mathfrak{g}$, while $\mathcal{F}$ corresponds to a $G$-invariant star
product $\star$ on $G$ quantizing $\pi_r$, where $G$-invariance means that
the pullback of the left multiplication
$\ell_g\colon G\ni h\mapsto g\cdot h\in G$ satisfies
$\ell^*_g\pi_r=\pi_r$ and $\ell^*_g(f_1\star f_2)
=(\ell^*_gf_1)\star(\ell^*_gf_2)$ for all $f_1,f_2\in\mathscr{C}^\infty(G)$ and
$g\in G$. Accordingly, the \textit{$2$-cocycle condition}
\begin{equation}\label{twist1}
    (\Delta\otimes\mathrm{id})(\mathcal{F})\cdot(\mathcal{F}\otimes 1)
    =(\mathrm{id}\otimes\Delta)(\mathcal{F})\cdot(1\otimes\mathcal{F})
\end{equation}
of $\mathcal{F}$ reflects associativity of $\star$, while
the \textit{normalization property}
\begin{equation}\label{twist2}
    (\epsilon\otimes\mathrm{id})(\mathcal{F})
    =1
    =(\mathrm{id}\otimes\epsilon)(\mathcal{F})
\end{equation}
is synonymous to the star product being unital with respect to the unit
function on $G$. Furthermore, $\mathcal{F}$ equals $1\otimes 1$ in zero order
of $\hbar$ if and only if $\star$ deforms the pointwise product of functions.
In a next step we want to induce star products on a Poisson manifold
$(M,\{\cdot,\cdot\})$ via a Drinfel'd twist $\mathcal{F}$
on $\mathscr{U}\mathfrak{g}$. For this we only need a Lie algebra action
$\phi\colon\mathfrak{g}\rightarrow\Gamma^\infty(TM)$ by derivations.
Such a Lie algebra action extends to a $\mathscr{U}\mathfrak{g}$-module algebra action
$\rhd\colon\mathscr{U}\mathfrak{g}\otimes\mathscr{C}^\infty(M)
\rightarrow\mathscr{C}^\infty(M)$ and
\begin{equation}\label{TwistStarProduct}
    f\star_\mathcal{F}g
    =\mu_0(\mathcal{F}^{-1}\rhd(f\otimes g)),
\end{equation}
is a star product on $M$, where $f,g\in\mathscr{C}^\infty(M)$ and $\mu_0
\colon\mathscr{C}^\infty(M)^{\otimes 2}\rightarrow\mathscr{C}^\infty(M)$
denotes the pointwise multiplication. A star product which can be expressed
as (\ref{TwistStarProduct}) is said to be a \textit{twist star product}.
The corresponding Poisson bracket is induced by the $r$-matrix of $\mathcal{F}$,
namely,
\begin{equation}
    \{f,g\}
    =\mu_0(r\rhd(f\otimes g))
\end{equation}
for all $f,g\in\mathscr{C}^\infty(M)$. Chapter~\ref{chap03} is devoted to
Drinfel'd twists, $r$-matrices and twist star products and to
obstructions of the latter. The main theorems provide a class of
symplectic manifolds that do not admit a twist star product deformation.
\begin{theorem*}[\cite{Thomas2016,dAWe17}]
There are no twist star products on
\begin{enumerate}
\item[i.)] the symplectic connected orientable Riemann surfaces of genus $g>1$;

\item[ii.)] the symplectic $2$-sphere $\mathbb{S}^2$;

\item[iii.)] the symplectic projective spaces $\mathbb{CP}^n$ for a Drinfel'd twist
based on $\mathscr{U}\mathfrak{gl}_{n+1}(\mathbb{C})$;
\end{enumerate} 
\end{theorem*}
Note that in the third class of examples only twists on universal enveloping algebras
of matrix Lie algebras are obstructed, while the first two classes prohibit
twist deformation based on any universal enveloping algebra. In fact, we are
providing more general results which imply the above examples.
For i.) and ii.) we prove
that connected compact symplectic manifolds endowed with a twist star product are
in fact homogeneous spaces and that the classical $r$-matrix corresponding to the
twist is non-degenerate (see Section~\ref{Sec-r-matrix}).
The third obstruction utilizes twist deformations of Morita equivalence
bimodules, which are studied in Section~\ref{SecObs4}. In a nutshell, we prove
that a symplectic manifold cannot inherit both a complex line bundle which is
invariant under a Lie group action and has non-trivial Chern class and a twist
star product with Drinfel'd twist based on the universal enveloping algebra of
the corresponding Lie algebra.

\section*{Braided Symmetries and Cartan Calculi}

Motivated by the example of universal enveloping algebras, one defines
Drinfel'd twists on an arbitrary Hopf algebra $H$ to be
invertible elements $\mathcal{F}\in H\otimes H$
such that (\ref{twist1}) and (\ref{twist2}) are satisfied. They lead to a deformed
Hopf algebra $H_\mathcal{F}$ with structure
\begin{equation}
    \Delta_\mathcal{F}(\cdot)
    =\mathcal{F}\Delta(\cdot)\mathcal{F}^{-1},~
    S_\mathcal{F}(\cdot)
    =\beta S(\cdot)\beta^{-1},
\end{equation}
where $\beta=\mu((\mathrm{id}\otimes S)(\mathcal{F}))\in H$ and 
$\mu\colon H^{\otimes 2}\rightarrow H$ is
the multiplication. Furthermore,
any left $H$-module algebra $(\mathcal{A},\cdot)$,
i.e. any associative unital algebra together with a Hopf algebra action which
respects the algebra structure,
can be deformed into a
left $H_\mathcal{F}$-module algebra $\mathcal{A}_\mathcal{F}
=(\mathcal{A},\cdot_\mathcal{F})$, where
\begin{equation}
    a\cdot_\mathcal{F}b
    =\mu_0(\mathcal{F}^{-1}\rhd(a\otimes b))
\end{equation}
for all $a,b\in\mathcal{A}$ and $\mu_0$ denotes the undeformed multiplication
on $\mathcal{A}$. More general, an $H$-equivariant $\mathcal{A}$-bimodule
$\mathcal{M}$ (see Section~\ref{Sec2.4})
is deformed into an $H_\mathcal{F}$-equivariant 
$\mathcal{A}_\mathcal{F}$-bimodule $\mathcal{M}_\mathcal{F}$. 
The category ${}_\mathcal{A}^H\mathcal{M}_\mathcal{A}$ of equivariant bimodules
is particularly interesting in the case of \textit{triangular Hopf algebras}.
The latter are Hopf algebras $H$ equipped with a 
\textit{universal $\mathcal{R}$-matrix} $\mathcal{R}\in H\otimes H$, which is
an invertible element controlling the noncocommutativity of $H$, namely
\begin{equation}\label{QuasiCoComm}
    \mathcal{R}^{-1}
    =\mathcal{R}_{21}
    \text{ and }
    \Delta_{21}(\cdot)
    =\mathcal{R}\Delta(\cdot)\mathcal{R}^{-1},
\end{equation}
such that the hexagon relations (see Section~\ref{Sec2.3}) are satisfied.
In (\ref{QuasiCoComm}), $\mathcal{R}_{21}$ denotes the tensor flip of $\mathcal{R}$ and 
$\Delta_{21}$ the flipped coproduct of $H$. The universal $\mathcal{R}$-matrix
encodes a braiding on the categorical level.
Any cocommutative Hopf algebra is triangular with universal $\mathcal{R}$-matrix
$1\otimes 1$. 
Let $(H,\mathcal{R})$ be triangular,
$\mathcal{A}$ \textit{braided commutative} and
$\mathcal{M}$ \textit{braided symmetric} in addition, i.e.
\begin{equation}
    b\cdot a
    =(\mathcal{R}_1^{-1}\rhd a)\cdot(\mathcal{R}_2^{-1}\rhd b),~
    m\cdot a
    =(\mathcal{R}_1^{-1}\rhd a)\cdot(\mathcal{R}_2^{-1}\rhd m)
\end{equation}
for all $a,b\in\mathcal{A}$ and $m\in\mathcal{M}$ in leg notation. Then, 
the twist deformed Hopf algebra $(H_\mathcal{F},\mathcal{R}_\mathcal{F})$ is triangular,
$\mathcal{A}_\mathcal{F}$ is braided commutative and the twisted module
$\mathcal{M}_\mathcal{F}$ braided symmetric with respect to the twisted
universal $\mathcal{R}$-matrix $\mathcal{R}_\mathcal{F}
=\mathcal{F}_{21}\mathcal{R}\mathcal{F}^{-1}$.
\begin{example*}
The smooth functions $\mathcal{A}=(\mathscr{C}^\infty(M),\cdot)$
on a Poisson manifold $(M,\{\cdot,\cdot\})$ with the pointwise product
are commutative, in other words braided commutative with respect to
a triangular Hopf algebra $H=(\mathscr{U}\mathfrak{g},1\otimes 1)$.
Then also twist star products (\ref{TwistStarProduct}) on a Poisson manifold $M$
determine a braided commutative algebra 
$\mathcal{A}_\mathcal{F}
=(\mathscr{C}^\infty(M)[[\hbar]],\star_\mathcal{F})$ with symmetry
$H_\mathcal{F}=(\mathscr{U}\mathfrak{g}_\mathcal{F},
\mathcal{F}_{21}\mathcal{F}^{-1})$.
\end{example*}
In categorical language, the Drinfel'd twist defines a braided monoidal functor
\begin{equation}
    \mathrm{Drin}_\mathcal{F}\colon
    {}_\mathcal{A}^H\mathcal{M}_\mathcal{A}^\mathcal{R}\rightarrow
    {}_{\mathcal{A}_\mathcal{F}}^{H_\mathcal{F}}
    \mathcal{M}_{\mathcal{A}_\mathcal{F}}^{\mathcal{R}_\mathcal{F}}
\end{equation}
between the representation theories of $H$ and $H_\mathcal{F}$. Such a functor is a 
braided monoidal equivalence of categories,
which implies that the two algebras in the example are equivalent on
categorical level (where one extends the pointwise product $\hbar$-bilinearly).
Braided multivector fields $\mathfrak{X}^\bullet_\mathcal{R}(\mathcal{A})$
and braided differential forms $\Omega^\bullet_\mathcal{R}(\mathcal{A})$ are
canonical examples of objects in ${}_\mathcal{A}^H\mathcal{M}_\mathcal{A}^\mathcal{R}$
(c.f. Chapter~\ref{chap04}) and we construct four $H$-equivariant maps
(i.e. maps commuting with the Hopf algebra action)
$\mathscr{L}^\mathcal{R},\mathrm{i}^\mathcal{R},\mathrm{d},
\llbracket\cdot,\cdot\rrbracket_\mathcal{R}$
%
%
in analogy to the Lie derivative, insertion of multivector fields, de Rham differential
and Schouten-Nijenhuis bracket, for any braided commutative algebra $\mathcal{A}$.
Their relations are clarified in the following theorem, providing a noncommutative
Cartan calculus.
\begin{theorem*}[\cite{Weber2019} Braided Cartan Calculus]
Let $H$ be a triangular Hopf algebra with universal $\mathcal{R}$-matrix
$\mathcal{R}$. For every braided commutative left $H$-module algebra
$\mathcal{A}$ the graded maps
$\mathscr{L}^\mathcal{R}_X\colon
\Omega^\bullet_\mathcal{R}(\mathcal{A})
\rightarrow\Omega^{\bullet-(k-1)}_\mathcal{R}(\mathcal{A})$,
$\mathrm{i}^\mathcal{R}_X\colon
\Omega^\bullet_\mathcal{R}(\mathcal{A})
\rightarrow\Omega^{\bullet-k}_\mathcal{R}(\mathcal{A})$,
where $X\in\mathfrak{X}^k_\mathcal{R}(\mathcal{A})$ and
$\mathrm{d}\colon\Omega^\bullet_\mathcal{R}(\mathcal{A})
\rightarrow\Omega^{\bullet+1}_\mathcal{R}(\mathcal{A})$,
satisfy
\begin{equation}
\begin{split}
        [\mathscr{L}^\mathcal{R}_X,\mathscr{L}^\mathcal{R}_Y]_\mathcal{R}
        =&\mathscr{L}^\mathcal{R}_{\llbracket X,Y\rrbracket_\mathcal{R}},\\
        [\mathscr{L}^\mathcal{R}_X,\mathrm{i}^\mathcal{R}_Y]_\mathcal{R}
        =&\mathrm{i}^\mathcal{R}_{\llbracket X,Y\rrbracket_\mathcal{R}},\\
        [\mathscr{L}^\mathcal{R}_X,\mathrm{d}]_\mathcal{R}=&0,
\end{split}
\hspace{1cm}
\begin{split}
        [\mathrm{i}^\mathcal{R}_X,\mathrm{i}^\mathcal{R}_Y]_\mathcal{R}=&0,\\
        [\mathrm{i}^\mathcal{R}_X,\mathrm{d}]_\mathcal{R}
        =&\mathscr{L}^\mathcal{R}_X,\\
        [\mathrm{d},\mathrm{d}]_\mathcal{R}=&0,
\end{split}
\end{equation}
for all $X,Y\in\mathfrak{X}^\bullet_\mathcal{R}(\mathcal{A})$,
where $[\cdot,\cdot]_\mathcal{R}$ denotes the graded braided commutator.
The twist deformation of this braided Cartan calculus
(induced by the Drinfel'd functor $\mathrm{Drin}_\mathcal{F}$) is
isomorphic to the braided Cartan calculus on $\mathcal{A}_\mathcal{F}$
with respect to $\mathcal{R}_\mathcal{F}$.
\end{theorem*}
Applying the theorem to the two algebras from our examples we recover the classical
Cartan calculus of differential geometry and the \textit{twisted Cartan calculus}
(c.f. \cite{Aschieri2006}, see also Section~\ref{Sec3.6}).
This explains the braided symmetries appearing in the
latter, e.g. that vector fields $X\in\mathfrak{X}^1(M)$ act
rather as braided derivations
\begin{equation}
    \mathscr{L}^\mathcal{F}_X(fg)
    =(\mathscr{L}_Xf)g
    +(\mathcal{R}_1^{-1}\rhd f)\mathscr{L}_{\mathcal{R}_2^{-1}\rhd X}g
\end{equation}
on functions $f,g\in\mathscr{C}^\infty(M)$, than as derivations.
We show the utility of the braided Cartan calculus and its similarity to the classical 
Cartan calculus by further discussing
equivariant covariant derivatives and submanifold algebras. An
\textit{equivariant covariant derivative} on $\mathcal{A}$ is an $H$-equivariant map
$\nabla^\mathcal{R}\colon\mathfrak{X}^1_\mathcal{R}(\mathcal{A})\otimes
\mathfrak{X}^1_\mathcal{R}(\mathcal{A})\rightarrow
\mathfrak{X}^1_\mathcal{R}(\mathcal{A})$,
which is left $\mathcal{A}$-linear in the first argument and satisfies
a braided Leibniz rule
\begin{equation}
    \nabla^\mathcal{R}_X(a\cdot Y)
    =(\mathscr{L}^\mathcal{R}_Xa)\cdot Y
    +(\mathcal{R}_1^{-1}\rhd a)\cdot(\nabla_{\mathcal{R}_2^{-1}\rhd X}Y)
\end{equation}
in the second argument, where $X,Y\in\mathfrak{X}^1_\mathcal{R}(\mathcal{A})$
and $a\in\mathcal{A}$. After defining their curvature and torsion we prove
that these objects behave similarly to their
counterparts from differential geometry.
\begin{theorem*}[\cite{Weber2019}]
Any equivariant covariant derivative on $\mathcal{A}$ extends to an equivariant
covariant derivative on $\mathfrak{X}^\bullet_\mathcal{R}(\mathcal{A})$ and
$\Omega^\bullet_\mathcal{R}(\mathcal{A})$. For every
non-degenerate equivariant metric ${\bf g}$ there exists a unique torsion-free
equivariant covariant derivative $\nabla^{\mathrm{LC}}$ on $\mathcal{A}$, such that
$\nabla^{\mathrm{LC}}{\bf g}=0$.
\end{theorem*}
We call $\nabla^{\mathrm{LC}}$ the \textit{equivariant Levi-Civita covariant derivative}
corresponding to ${\bf g}$. If the twist deformation ${\bf g}_\mathcal{F}$ of ${\bf g}$ is non-degenerate,
it follows that the twist deformation of $\nabla^{\mathrm{LC}}$ is the equivariant
Levi-Civita covariant derivative corresponding to ${\bf g}_\mathcal{F}$.
Phrasing the notion of submanifold in algebraic
terms, we prove that the braided Cartan calculus on a submanifold algebra coincides
with the projection of the calculus on the ambient algebra. Under some
hypotheses we are able to project equivariant covariant derivatives and
metrics to the submanifold algebra. An important observation is that these
projections commute with twist deformation. 

\section*{Organization of the Thesis}

In Chapter~\ref{chap02} we recall some notions concerning Hopf algebras. The strategy is
to develop the algebraic data parallel to the categorical data on representations.
A quasi-triangular Hopf algebra corresponds to the rigid braided monoidal
category of its finitely generated projective representations. Accordingly,
Drinfel'd twists can be understood as algebraic deformation tool and braided monoidal
functor in the category of equivariant braided symmetric bimodules.
We further give examples of Drinfel'd twists and include ${}^*$-involutions in the
picture. The case of Drinfel'd twist deformation of Poisson manifolds is studied
in Chapter~\ref{chap03}. We define star products as formal deformations and investigate
consequences if they are induced by Drinfel'd twists. In particular, the relation
of twists on formal power series of universal enveloping algebras and classical
$r$-matrices is pointed out. Another important result is that the deformation
theory of a commutative algebra is a Morita invariant, i.e. Morita equivalent
algebras share the same deformation theories. In the case of twist star products
this sets severe conditions on equivariant line bundles over the manifold.
Both approaches lead to obstructions of twist star products on several classes
of symplectic manifolds. Chapter~\ref{chap04} covers some topics in braided
geometry. We construct the braided Cartan calculus on any braided
commutative algebra and prove that it respects gauge equivalence classes of
the Drinfel'd functor. Furthermore, equivariant covariant derivatives are defined on
equivariant braided symmetric bimodules. In complete analogy to differential geometry
we study their properties and prove e.g. the existence and uniqueness of an
equivariant Levi-Civita covariant derivative for a fixed non-degenerate equivariant
metric. The Drinfel'd functor intertwines all constructions.
Finally, in Chapter~\ref{chap05}, we show that the braided Cartan calculus on
a submanifold algebra is in fact the projection of the braided Cartan calculus of
the ambient space. Employing certain assumptions we are able to project
equivariant metrics and covariant derivatives. As a central observation we
point out that the projection and twist deformation commute. In addition, an
explicit example of twist deformation quantization on a quadric surface is given.
There are two appendices, Appendix~\ref{App01} briefly covering the material
on category theory which is necessary to understand the thesis and
Appendix~\ref{App02} providing some additional material on braided exterior algebras
and braided Gerstenhaber algebras.

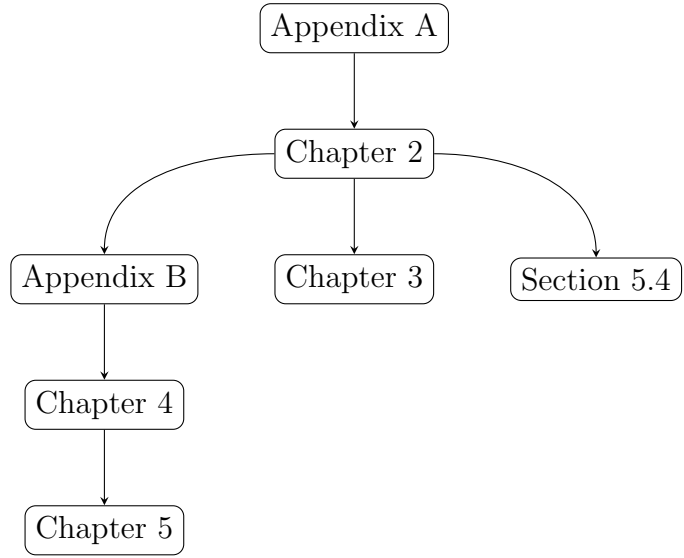
\begin{figure}
\centering
\begin{tikzpicture}[>=stealth,every node/.style={shape=rectangle,draw,rounded corners},]
    \node (c0) {Appendix \ref{App01}};
    \node (c2) [below =of c0]{Chapter \ref{chap02}};
    \node (c3) [below =of c2]{Chapter \ref{chap03}};
    \node (c1) [left =of c3]{Appendix \ref{App02}};
    \node (c4) [below =of c1]{Chapter \ref{chap04}};
    \node (c5) [below =of c4]{Chapter \ref{chap05}};
    \node (c6) [right =of c3]{Section~\ref{Sec4.4}};
    \draw[->] (c0) -- (c2);
    \draw[->] (c2) -- (c3);
    \draw[->] (c1) -- (c4);
    \draw[->] (c4) -- (c5);
    \draw[->] (c2.east) to[out=0,in=90] (c6);
    \draw[->] (c2.west) to[out=180,in=90] (c1);
\end{tikzpicture}
\caption{How to read the thesis} \label{HowToRead}
\end{figure}

Depending on the personal background and intentions we recommend to read the thesis
in the following ways (see also Figure~\ref{HowToRead}). In Appendix~\ref{App01}
the basic categorical language, which is used throughout the thesis, is provided.
If the reader is familiar with these concepts the appendix can be omitted.
Chapter~\ref{chap02} contains the algebraic
concepts of Hopf algebra and Drinfel'd twist which are central to this work.
The advanced reader might also skip this chapter. Section~\ref{Sec2.6} is mainly
relevant for the twist deformation of quadric surfaces in Section~\ref{Sec4.4},
where the latter can be understood independently of 
Chapter~\ref{chap03}-Section~\ref{Sec5.3}.
Furthermore, the braided Cartan calculus (Chapter~\ref{chap04}) and its compatibility
with submanifold algebras (Chapter~\ref{chap05}) can be considered independently
of Chapter~\ref{chap03}, which treats obstructions of twist star products and
deformation quantization. In other words, if the reader is not interested in
deformation quantization there is a shortcut from Chapter~\ref{chap02} to
Chapter~\ref{chap04}, where one might consider Appendix~\ref{App02} before,
to learn about braided Graßmann algebras and braided Gerstenhaber algebras.

\section*{Notation}

Throughout the thesis $\Bbbk$ denotes a commutative ring with unit $1$. If the
situation requires to work over a field we write $\mathbb{K}$ instead.
A \textit{$\Bbbk$-module} is an Abelian group $\mathcal{M}$ together with a
distributive left $\Bbbk$-action. Typically we write $\lambda\cdot m$ or simply
$\lambda m$ for the action of an element $\lambda\in\Bbbk$ on an element
$m\in\mathcal{M}$ of a $\Bbbk$-module.
A map between the Abelian groups of two
$\Bbbk$-modules is said to be \textit{$\Bbbk$-linear} or \textit{$\Bbbk$-module
homomorphism} if it intertwines the $\Bbbk$-actions. If such a map is invertible
in addition we call it \textit{$\Bbbk$-module isomorphism}.
The tensor product of $\Bbbk$-modules is denoted by $\otimes$. It is a monoidal
structure on the category ${}_\Bbbk\mathcal{M}$ of $\Bbbk$-modules.
The \textit{flip isomorphism} of two $\Bbbk$-modules $\mathcal{M}$ 
and $\mathcal{N}$ is defined by
\begin{equation}
    \tau_{\mathcal{M},\mathcal{N}}\colon\mathcal{M}\otimes\mathcal{N}\ni(m\otimes n)
    \mapsto(n\otimes m)\in\mathcal{N}\otimes\mathcal{M}.
\end{equation}
Furthermore we write $\mathcal{M}^{\otimes k}
=\mathcal{M}\otimes\ldots\otimes\mathcal{M}$ for the tensor product of
$k>1$ copies of a $\Bbbk$-module $\mathcal{M}$.
Let $\mathcal{M}$ be a $\Bbbk$-module. Then any element
$\mathcal{F}\in\mathcal{M}^{\otimes 2}$ can be expressed as a
finite sum $\sum_{i=1}^n\mathcal{F}^i_1\otimes\mathcal{F}^i_2$
of factorizing elements $\mathcal{F}^i_1\otimes\mathcal{F}^i_2$.
Omitting this sum and the summation indices we end up with
$\mathcal{F}=\mathcal{F}_1\otimes\mathcal{F}_2$, which is known
as \textit{leg notation}. Dealing with several copies of $\mathcal{F}$
we write $\mathcal{F}=\mathcal{F}'_1\otimes\mathcal{F}'_2$, etc.
to distinguish the summations. If $\mathcal{M}$ is an algebra with unit $1$
we further write
$\mathcal{F}_{21}=\tau_{\mathcal{M},\mathcal{M}}(\mathcal{F})$,
$\mathcal{F}_{12}=\mathcal{F}\otimes 1$,
$\mathcal{F}_{13}
=(\mathrm{id}_\mathcal{M}\otimes\tau_{\mathcal{M},\mathcal{M}})
(\mathcal{F}\otimes 1)$ and similarly for other permutations of the
\textit{legs} of $\mathcal{F}$.
If not stated otherwise every $\Bbbk$-algebra $\mathcal{A}$ is
assumed to be associative and unital. This means there are $\Bbbk$-linear
maps $\mu\colon\mathcal{A}\otimes\mathcal{A}\rightarrow\mathcal{A}$ and
$\eta\colon\Bbbk\rightarrow\mathcal{A}$, called \textit{product} and
\textit{unit}, respectively, such that
\begin{equation}
    \mu\circ(\mu\otimes\mathrm{id})=\mu\circ(\mathrm{id}\otimes\mu)
    \colon\mathcal{A}^{\otimes 3}\rightarrow\mathcal{A}
\end{equation}
and
\begin{equation}\label{eq28}
    \mu\circ(\eta\otimes\mathrm{id})
    =\mathrm{id}
    =\mu\circ(\mathrm{id}\otimes\eta)
    \colon\mathcal{A}\rightarrow\mathcal{A}
\end{equation}
hold. The first property is said to be the \textit{associativity} of $\mu$.
In equation (\ref{eq28}) we employed the $\Bbbk$-module isomorphisms
$\Bbbk\otimes\mathcal{A}\cong\mathcal{A}\cong\mathcal{A}\otimes\Bbbk$.
The product of two elements $a$ and $b$ of an algebra 
$(\mathcal{A},\mu,\eta)$ is sometimes denoted by $a\cdot b$ or $ab$ if the
reference to $\mu$ is not essential. Since $\eta$ is determined by its value at
the unit of $\Bbbk$ we often write $1=\eta(1)\in\mathcal{A}$, calling this element
the unit of $\mathcal{A}$ as well. The algebra is said to be
\textit{commutative} if $ab=ba$ for all $a,b\in\mathcal{A}$. An
\textit{algebra homomorphism} is a $\Bbbk$-linear map
$\Phi\colon\mathcal{A}\rightarrow\mathcal{B}$ between algebras such that
$\Phi(ab)=\Phi(a)\Phi(b)$ for all $a,b\in\mathcal{A}$ and $\Phi(1)=1$. 
If $\Phi$ is invertible
in addition it is said to be an \textit{algebra isomorphism}. The category
of $\Bbbk$-algebras is denoted by ${}_\Bbbk\mathcal{A}$. It is monoidal with
respect to $\otimes$, since $\Bbbk$ is an algebra and
the tensor product $\mathcal{A}\otimes\mathcal{B}$
can be structured as a $\Bbbk$-algebra with unit $1\otimes 1$ and product
\begin{equation}\label{AlgTimesAlg}
    (a_1\otimes b_1)\cdot(a_2\otimes b_2)
    =(a_1a_2)\otimes(b_1b_2),
\end{equation}
where $a_1,a_2\in\mathcal{A}$ and $b_1,b_2\in\mathcal{B}$.

%% file: chapters/chapter02.tex
In this preliminary chapter we recall the notion of Hopf algebra
together with its category of representations. By adding more and more
algebraic structure we successively discuss coalgebras, bialgebras and
finally Hopf algebras, together with some fundamental examples. Depicting
those algebraic properties in terms of commutative diagrams is a
comfortable way of compressing the relations, furthermore revealing
their duality. This, together with proofs of some fundamental properties
of the Hopf algebra structure, is the agenda of Section~\ref{Sec2.1}.
Afterwards, in Section~\ref{SectionHopfAlgebraModules}, we focus on
modules of Hopf algebras. It turns out that, unlike for general algebras,
the representation theory of Hopf algebras has many additional features.
In fact, it is exactly the bialgebra structure which
shapes the corresponding modules as a monoidal category, while the
antipode gives rise to an additional rigidity property of the
monoidal subcategory of finitely generated projective modules.
Consequently, we introduce quasi-triangular Hopf algebras in Section~\ref{Sec2.3}
as those Hopf algebras whose monoidal category of representations is braided
and describe the corresponding algebraic structure on the algebraic level.
It is interesting that this so-called universal $\mathcal{R}$-matrix 
satisfies the quantum Yang-Baxter equation, which connects our considerations
to the field of integrable systems. There is a notion of gauge equivalence
respecting both pictures in equal measure: on the algebraic side,
Drinfel'd twists deform the quasi-triangular Hopf algebra structure such
that the result is still a quasi-triangular Hopf algebra, while on the
categorical side the Drinfel'd functor leads to a braided monoidal
equivalence of the representations of the deformed and undeformed Hopf algebra.
This is what we discuss in Section~\ref{Sec2.4}. In Section~\ref{Sec2.5},
we refine the mentioned braided monoidal equivalence
to braided symmetric equivariant bimodules of braided commutative
module algebras of a triangular Hopf algebra.
At the end of this chapter, in Section~\ref{Sec2.6}, it is discussed how
${}^*$-involutions fit into the picture and how they can be deformed using
unitary Drinfel'd twists. The main references for the first
four sections are \cite{ChPr94,Ka95,Ma95,Mo93}, while the last two sections are
inspired by \cite{AsSh14,Schenkel2015,GiZh98} and \cite{Fiore2010},
respectively.

\section{Hopf Algebras}\label{Sec2.1}

The first formal definition
of Hopf algebra goes to back to early works of Cartier (see 
\cite{andruskiewitsch09} and references therein for a discussion
on the origin of Hopf algebras). Hopf algebras are essential objects in
a broad spectrum of disciplines in mathematics and likewise in theoretical
physics. In this thesis Hopf algebras incarnate symmetries of the
algebra of observables of our interest. This point of view was first promoted
in early articles \cite{Dr83,Dr86,Dr89,ReTaFa1990} about quantum groups. 
Since then, numerous
works on algebraic deformation proved the fruitfulness of this approach.
From an algebraic point of view, Hopf algebras are natural specifications
of algebras: a coalgebra is the dual object to an algebra and the notion
of bialgebra describes objects which inherit both structures in a
compatible way. The antipode, completing the Hopf algebra structure, is an
inverse of the corresponding convolution algebra of the bialgebra.
We add more substance to this line of thought in this section.
Namely, we give the definition of Hopf algebra building on intermediate
steps and concrete examples. Several fundamental properties of the Hopf
algebra structure are discussed. Since they are used throughout the thesis
we give full proofs, also to get used to the developed notation.
We refer to the textbooks \cite{ChPr94}~Sec.~4.1, \cite{Ka95}~Chap.~III,
\cite{Ma95}~Chap.~1 and \cite{Mo93}~Chap.~1 for excellent
introductions to Hopf algebras.

Let us start this section by discussing two fundamental examples of Hopf algebras:
the universal enveloping algebra of a Lie algebra and the group algebra.
Having concrete examples in mind the abstract definitions of the following
sections are easier to digest.
Consider a Lie algebra $\mathfrak{g}$ over a field $\mathbb{K}$, with Lie bracket
denoted by $[\cdot,\cdot]$. Its \textit{universal enveloping algebra}
$\mathscr{U}\mathfrak{g}$ is defined to be the tensor algebra 
$$
\mathrm{T}\mathfrak{g}
=\bigoplus_{k\geq 0}\mathfrak{g}^{\otimes k}
=\mathbb{K}\oplus\mathfrak{g}\oplus(\mathfrak{g}\otimes\mathfrak{g})\oplus\cdots
$$
of $\mathfrak{g}$ modulo the ideal generated by
$x\otimes y-y\otimes x-[x,y]$ for all $x,y\in\mathfrak{g}$. It inherits the
structure of an associative unital algebra from $\mathrm{T}\mathfrak{g}$,
with product induced by the tensor product and unit induced by $1\in\mathbb{K}$.
The universal enveloping algebra is, up to isomorphism, uniquely determined by
the following universal property: for any associative unital algebra
$\mathcal{A}$, seen as a Lie algebra with the commutator, and any Lie algebra
homomorphism $\phi\colon\mathfrak{g}\rightarrow\mathcal{A}$, there is a unique
homomorphism $\Phi\colon\mathscr{U}\mathfrak{g}\rightarrow\mathcal{A}$ of
associative unital algebras such that $\phi=\Phi\circ\iota$, where
$\iota\colon\mathfrak{g}\rightarrow\mathscr{U}\mathfrak{g}$ denotes the canonical
embedding of $\mathfrak{g}$ in $\mathscr{U}\mathfrak{g}$. Making use of this
universal property we can define three $\mathbb{K}$-linear maps 
$\Delta\colon\mathscr{U}\mathfrak{g}\rightarrow\mathscr{U}\mathfrak{g}^{\otimes 2}$,
$\epsilon\colon\mathscr{U}\mathfrak{g}\rightarrow\mathbb{K}$ and
$S\colon\mathscr{U}\mathfrak{g}\rightarrow\mathscr{U}\mathfrak{g}$ by declaring
them to satisfy
$
\Delta(x)=x\otimes 1+1\otimes x,~
\epsilon(x)=0
\text{ and }
S(x)=-x
$
on elements $x\in\mathfrak{g}$ and extending $\Delta$ and $\epsilon$ as
algebra homomorphisms and $S$ as algebra anti-homomorphism. A calculation shows
that the equations 
\begin{equation}\label{eq01}
    (\Delta\otimes\mathrm{id})\circ\Delta
    =(\mathrm{id}\otimes\Delta)\circ\Delta\text{ and }
    (\epsilon\otimes\mathrm{id})\circ\Delta
    =\mathrm{id}
    =(\mathrm{id}\otimes\epsilon)\circ\Delta    
\end{equation}
hold in addition as maps $\mathscr{U}\mathfrak{g}\rightarrow
\mathscr{U}\mathfrak{g}^{\otimes 3}$ and $\mathscr{U}\mathfrak{g}\rightarrow
\mathscr{U}\mathfrak{g}$, respectively. By the universal
property of $\mathscr{U}\mathfrak{g}$ it is sufficient to prove this on
elements of $\mathfrak{g}$. If one denotes the product and unit of 
$\mathscr{U}\mathfrak{g}$ by 
$\mu\colon\mathscr{U}\mathfrak{g}^{\otimes 2}\rightarrow\mathscr{U}\mathfrak{g}$
and $\eta\colon\mathbb{K}\rightarrow\mathscr{U}\mathfrak{g}$, respectively, one
easily verifies that
\begin{equation}\label{eq02}
    \mu\circ(S\otimes\mathrm{id})\circ\Delta
    =\eta\circ\epsilon
    =\mu\circ(\mathrm{id}\otimes S)\circ\Delta
    \colon\mathscr{U}\mathfrak{g}\rightarrow\mathscr{U}\mathfrak{g}
\end{equation}
holds. Before commenting further on
these maps and their properties we introduce another algebra with additional maps
$\Delta$, $\epsilon$ and $S$ obedient to the same relations. It is the
\textit{group algebra} $\mathbb{K}[G]$ of a finite group $G$, which is defined as 
the free $\mathbb{K}$-module generated by the elements of $G$. Its associative product
is given by the $\mathbb{K}$-linearly extended group multiplication and the unit
is the neutral element of $G$. On elements $g\in G$ we define
$
\Delta(g)=g\otimes g,~
\epsilon(g)=1
\text{ and }
S(g)=g^{-1}
$
and extend those maps as algebra (anti)-homomorphisms to 
$\Delta\colon\mathbb{K}[G]\rightarrow\mathbb{K}[G]^{\otimes 2}$,
$\epsilon\colon\mathbb{K}[G]\rightarrow\mathbb{K}$ and
$S\colon\mathbb{K}[G]\rightarrow\mathbb{K}[G]$. It is easy to verify that they
satisfy (\ref{eq01}) and (\ref{eq02}).
The natural question arises if there are more examples of associative unital
algebras allowing for such additional structure, maybe even revealing a greater
concept. Besides, the equations~(\ref{eq01}) seem to mimic the axioms of an
associative unital algebra in a dual fashion and the map $S$ reminds of
some kind of inverse. In fact it is the notion of Hopf algebra which unites 
and generalizes those two examples. To state a rigorous definition we need some
preparation. First of all we operate slightly more general by considering
commutative rings $\Bbbk$ instead of fields $\mathbb{K}$ 
and $\Bbbk$-modules rather than $\mathbb{K}$-vector spaces. The reason is that
we obviously enrich the number of examples in this way, in particular allowing
for formal power series $V[[\hbar]]$ with coefficients in a $\mathbb{K}$-vector
space $V$ in this way: $V[[\hbar]]$ is a $\Bbbk=\mathbb{K}[[\hbar]]$-module.
In the next definition we axiomatize equations~(\ref{eq01}).
\begin{definition}[Coalgebra]
Let $\mathcal{C}$ be a $\Bbbk$-module for a commutative ring $\Bbbk$.
It is said to be a $\Bbbk$-\textit{coalgebra} if there are
$\Bbbk$-linear maps $\Delta\colon\mathcal{C}
\rightarrow\mathcal{C}\otimes\mathcal{C}$ and
$\epsilon\colon\mathcal{C}\rightarrow\Bbbk$ such that (\ref{eq01}) hold.
In this case $\Delta$ is called \textit{coproduct} and $\epsilon$
\textit{counit} of $\mathcal{C}$,
while to axioms (\ref{eq01}) are called coassociativity and counitality,
respectively. A $\Bbbk$-linear map
$\phi\colon\mathcal{C}\rightarrow\mathcal{C}'$ between
$\Bbbk$-coalgebras $(\mathcal{C},\Delta,\epsilon)$ and
$(\mathcal{C}',\Delta',\epsilon')$ is said
to be a coalgebra homomorphism if $(\phi\otimes\phi)\circ\Delta
=\Delta'\circ\phi$ and $\epsilon'\circ\phi=\epsilon$. The category of
$\Bbbk$-coalgebras is denoted by ${}_\Bbbk\mathcal{C}$. 
\end{definition}
We introduce \textit{Sweedler's notation} $\Delta(c)
=c_{(1)}\otimes c_{(2)}$ to denote the coproduct of an element $c$ of a
coalgebra $\mathcal{C}$. Namely, we omit a possibly finite sum of 
factorizing elements
in the tensor product, similar to the leg notation which we introduced in the
introduction. Using the coassociativity of $\Delta$ we define
$$
c_{(1)}\otimes c_{(2)}\otimes c_{(3)}
=c_{(1)(1)}\otimes c_{(1)(2)}\otimes c_{(2)}
=c_{(1)}\otimes c_{(2)(1)}\otimes c_{(2)(2)}
$$
and similarly for higher coproducts of $c$. Combining Sweedler's notation
with leg notation we further write 
$\mathcal{F}_{12,3}=(\Delta\otimes\mathrm{id})(\mathcal{F})$,
$\mathcal{F}_{1,23}=(\mathrm{id}\otimes\Delta)(\mathcal{F})$
for any element $\mathcal{F}=\mathcal{F}_1\otimes\mathcal{F}_2\in
\mathcal{C}^{\otimes 2}$ and
$\mathcal{F}_{3,12}=(\mathrm{id}\otimes\Delta)(\tau_{\mathcal{C},\mathcal{C}}
(\mathcal{F}))$,
$\mathcal{F}_{21,3}=(\tau_{\mathcal{C},\mathcal{C}}\otimes\mathrm{id})(\Delta
(\mathcal{F}))$, etc.
if we consider permutations of the legs of $\mathcal{F}$.
A coalgebra $\mathcal{C}$ is said to be
\textit{cocommutative} if $c_{(2)}\otimes c_{(1)}=c_{(1)}\otimes c_{(2)}$
for all $c\in\mathcal{C}$. This is the case for our previous examples of the
universal enveloping algebra $\mathscr{U}\mathfrak{g}$ and the group algebra
$\mathbb{K}[G]$. Furthermore, motivated from exactly these two examples we
call an element $c\in\mathcal{C}$ \textit{$\xi$-$\chi$-primitive} for two
elements $\xi,\chi\in\mathcal{C}$, if
$\Delta(c)=c\otimes\xi+\chi\otimes c$.
We call $c\in\mathcal{C}$ \textit{group-like} if $\Delta(c)=c\otimes c$.
The $1$-$1$-primitive elements of $\mathscr{U}\mathfrak{g}$ are exactly the
elements of $\mathfrak{g}$, usually they are said to be \textit{primitive}
for short in this case, while the group-like elements of $\mathbb{K}[G]$ are
exactly the elements of $G$. 
Depicting the axioms of a coalgebra $(\mathcal{C},\Delta,\epsilon)$ and a
coalgebra homomorphism $\phi\colon\mathcal{C}\rightarrow\mathcal{C}'$ via
commutative diagrams
\begin{equation*}
\begin{tikzcd}
\mathcal{C}  \arrow{r}{\Delta} \arrow{d}[swap]{\Delta}
& \mathcal{C}\otimes\mathcal{C} \arrow{d}{\mathrm{id\otimes\Delta}} \\
\mathcal{C}\otimes\mathcal{C} \arrow{r}{\Delta\otimes\mathrm{id}}
& \mathcal{C}\otimes\mathcal{C}\otimes\mathcal{C}
\end{tikzcd},
\begin{tikzcd}
\Bbbk\otimes\mathcal{C} \arrow{rd}[swap]{\cong} 
& \mathcal{C}\otimes\mathcal{C}  \arrow{l}[swap]{\epsilon\otimes\mathrm{id}}
\arrow{r}{\mathrm{id}\otimes\epsilon} 
& \mathcal{C}\otimes\Bbbk \arrow{ld}{\cong}\\
& \mathcal{C} \arrow{u}[swap]{\Delta} &
\end{tikzcd}
\end{equation*}
and
\begin{equation*}
\begin{tikzcd}
\mathcal{C}\arrow{r}{\Delta} \arrow{d}[swap]{\phi} 
& \mathcal{C}\otimes\mathcal{C} \arrow{d}{\phi\otimes\phi} \\
\mathcal{C}' \arrow{r}{\Delta'}
& \mathcal{C}'\otimes\mathcal{C}'
\end{tikzcd},
\begin{tikzcd}
\mathcal{C} \arrow{r}{\phi} \arrow{rd}[swap]{\epsilon} 
& \mathcal{C}' \arrow{d}{\epsilon'} \\
& \Bbbk
\end{tikzcd},
\end{equation*}
respectively, the duality with algebras and algebra homomorphisms becomes visible:
the corresponding diagrams in the category ${}_\Bbbk\mathcal{A}$
of algebras are obtained by reversing the arrows, replacing
$\Delta$ by the product and $\epsilon$ by the unit in the above diagrams.
Another observation affirming this duality is given by the following lemma
(c.f. \cite{Mo93}~Lem.~1.2.2 and the subsequent discussion).
\begin{lemma}\label{lemma02}
Let $(\mathcal{C},\Delta,\epsilon)$ be a $\Bbbk$-coalgebra.
Its dual $\Bbbk$-module
$\mathcal{C}^*=\mathrm{Hom}_\Bbbk(\mathcal{C},\Bbbk)$ is a $\Bbbk$-algebra
with product and unit given by the dual maps 
$\Delta^*\colon\mathcal{C}^*\otimes\mathcal{C}^*\rightarrow\mathcal{C}^*$
and $\epsilon^*\colon\Bbbk\rightarrow\mathcal{C}^*$, respectively.
If $(\mathcal{A},\mu,\eta)$ is a finite-dimensional algebra over a field
$\mathbb{K}$, its dual
$\mathcal{A}^*=\mathrm{Hom}_\mathbb{K}(\mathcal{A},\mathbb{K})$ is
a coalgebra with coproduct and counit given by the dual maps 
$\mu^*\colon\mathcal{A}^*\rightarrow\mathcal{A}^*\otimes\mathcal{A}^*$ and 
$\eta^*\colon\mathcal{A}^*\rightarrow\mathbb{K}$, respectively.
\end{lemma}
\begin{proof}
We only prove the second statement. For finite-dimensional vector spaces
there is an isomorphism $(\mathcal{A}\otimes\mathcal{A})^*
\cong\mathcal{A}^*\otimes\mathcal{A}^*$. This implies that the dual of the
multiplication
$$
(\mu^*(\alpha))(a\otimes b)
=\alpha(\mu(a\otimes b)),
$$
where $\alpha\in\mathcal{A}^*$ and $a,b\in\mathcal{A}$, is a
$\mathbb{K}$-linear map $\mu^*\colon\mathcal{A}^*\rightarrow
\mathcal{A}^*\otimes\mathcal{A}^*$. It is coassociative, since
\begin{align*}
    (\alpha_{(1)(1)}\otimes\alpha_{(1)(2)}\otimes\alpha_{(2)})
    (a\otimes b\otimes c)
    =&\mu(\mu(a\otimes b)\otimes c)\\
    =&\mu(a\otimes\mu(b\otimes c))\\
    =&(\alpha_{(1)}\otimes\alpha_{(2)(1)}\otimes\alpha_{(2)(2)})
    (a\otimes b\otimes c)
\end{align*}
for all $\alpha\in\mathcal{A}^*$ and $a,b,c\in\mathcal{A}$ by the
associativity of $\mu$, where we denoted $\alpha_{(1)}\otimes\alpha_{(2)}
=\mu^*(\alpha)$. Furthermore, the dual $\eta^*(\alpha)=\alpha(1)$
of the unit satisfies the counit axiom
\begin{align*}
    ((\eta^*\otimes\mathrm{id})\mu^*(\alpha))(a)
    =&\alpha_{(1)}(1)\otimes\alpha_{(2)}(a)\\
    =&\alpha(\mu(1\otimes a))\\
    =&\alpha(a)\\
    =&((\mathrm{id}\otimes\eta^*)\mu^*(\alpha))(a)
\end{align*}
for all $a\in\mathcal{A}$.
\end{proof}
However, the second part of the lemma indicates that the duality of algebras and
coalgebras should be understood with a grain of salt. For a infinite-dimensional
algebra $\mathcal{A}$ it is known that $\mathcal{A}^*\otimes\mathcal{A}^*$ is a
proper subspace of $(\mathcal{A}\otimes\mathcal{A})^*$ and we can not expect
$\mathcal{A}^*$ to be a coalgebra in general. A way out of this problem is given
by a stronger notion of duality which is presented later in this section.
Besides the dual $\Bbbk$-module, there is another fundamental construction given
by the tensor product $\mathcal{C}\otimes\mathcal{D}$ of coalgebras. It is
a coalgebra with coproduct 
$$
\Delta_{\mathcal{C}\otimes\mathcal{D}}(c\otimes d)
=(\mathrm{id}_\mathcal{C}\otimes\tau_{\mathcal{C},\mathcal{D}}
\otimes\mathrm{id}_\mathcal{D})
(\Delta_\mathcal{C}\otimes\Delta_\mathcal{D})(c\otimes d)
=(c_{(1)}\otimes d_{(1)})\otimes(c_{(2)}\otimes d_{(2)})
$$
and counit $\epsilon_{\mathcal{C}\otimes\mathcal{D}}
=\epsilon_\mathcal{C}\otimes\epsilon_\mathcal{D}$. Remark the duality of
this construction to the tensor product of two algebras (see eq.(\ref{AlgTimesAlg})).
Furthermore, any commutative ring $\Bbbk$ is a $\Bbbk$-coalgebra
with coproduct $\Delta_\Bbbk(\lambda)=\lambda\cdot(1\otimes 1)$
and counit $\epsilon_\Bbbk=\mathrm{id}_\Bbbk$. In fact, coassociativity
is trivial and counitality follows from the isomorphism
$\Bbbk\otimes\Bbbk\cong\Bbbk$.
We introduce an algebra which is useful in the theory of quantum groups. In fact,
it is an essential tool in the subsequent proofs of this section.
\begin{lemma}[Convolution Algebra]\label{lemma04}
Consider a $\Bbbk$-algebra $(\mathcal{A},\mu,\eta)$ and a $\Bbbk$-coalgebra
$(\mathcal{C},\Delta,\epsilon)$.
Then, the $\Bbbk$-linear maps $\mathrm{Hom}_\Bbbk(\mathcal{C},\mathcal{A})$
from $\mathcal{C}$ to $\mathcal{A}$ form an algebra with associative
product given by
\begin{equation}
    f\star g=\mu\circ(f\otimes g)\circ\Delta
\end{equation}
for all $f,g\in\mathrm{Hom}_\Bbbk(\mathcal{C},\mathcal{A})$
and with unit given by
$
\eta\circ\epsilon.
$
\end{lemma}
The algebra $(\mathrm{Hom}_\Bbbk(\mathcal{C},\mathcal{A}),\star,
\eta\circ\epsilon)$ is said
to be the \textit{convolution algebra} and $\star$ the \textit{convolution
product}.
Using Sweedler's notation and omitting the product in $\mathcal{A}$ the
convolution product reads $(f\star g)(c)=f(c_{(1)})g(c_{(2)})$ for all
$f,g\in\mathrm{Hom}_\Bbbk(\mathcal{C},\mathcal{A})$ and $c\in\mathcal{C}$.
As a reference consider \cite{Ka95}~Prop.~III.3.1.(a).
Setting $\mathcal{A}=\Bbbk$ we recover the situation of the first
part of Lemma~\ref{lemma02}.
\begin{proof}
Let $f,g,h\in\mathrm{Hom}_\Bbbk(\mathcal{C},\mathcal{A})$ and
$c\in\mathcal{C}$ be arbitrary.
As concatenation of $\Bbbk$-(bi)linear maps, $\star$ and $\eta\circ\epsilon$
are $\Bbbk$-(bi)linear and 
$f\star g,\eta\circ\epsilon\in\mathrm{Hom}_\Bbbk(\mathcal{C},\mathcal{A})$.
Then
\begin{align*}
    ((f\star g)\star h)(c)
    =&(f(c_{(1)(1)})g(c_{(1)(2)}))h(c_{(2)})\\
    =&f(c_{(1)})(g(c_{(2)(1)})h(c_{(2)(2)}))\\
    =&(f\star(g\star h))(c)
\end{align*}
proves that $\star$ is associative and
\begin{align*}
    ((\eta\circ\epsilon)\star f)(c)
    =&\epsilon(c_{(1)})f(c_{(2)})\\
    =&f(\epsilon(c_{(1)})c_{(2)})\\
    =&f(c)\\
    =&(f\star(\eta\circ\epsilon))(c)
\end{align*}
shows that $\eta\circ\epsilon$ is a unit, concluding the proof of the lemma.
\end{proof}
Focusing again on our motivating examples we
realize that their coalgebra structures are not independent from their algebra
structures. In fact, $\Delta$ and $\epsilon$ are algebra homomorphisms and
$\mu$ and $\eta$ are coalgebra homomorphisms, where we endow the tensor
product and $\Bbbk$ with the corresponding (co)algebra structure.
This indicates that we are not finished in formalizing our examples.
\begin{definition}[Bialgebra]\label{def03}
A $\Bbbk$-module $\mathcal{B}$ is said to be a \textit{bialgebra} if it is an
algebra and a coalgebra such that the coproduct and the counit are
algebra homomorphisms and the product and unit are coalgebra homomorphisms.
A homomorphism of bialgebras is an algebra
homomorphism, which is a coalgebra homomorphism in addition. The category
of bialgebras is denoted by ${}_\Bbbk\mathcal{B}$.
\end{definition}
In Section~\ref{SectionHopfAlgebraModules} we connect the notion
of bialgebra with properties of its representation theory.
The conditions on an algebra $(\mathcal{B},\mu,\eta)$ with coalgebra
structures $\Delta$ and $\epsilon$ to be a bialgebra is depicted
in the commutativity of
\begin{equation}\label{eq22}
\begin{tikzcd}
\mathcal{B}\otimes\mathcal{B} 
\arrow{r}{\mu} 
\arrow{d}[swap]{\Delta\otimes\Delta}
& \mathcal{B} \arrow{r}{\Delta}
& \mathcal{B}\otimes\mathcal{B} \\
\mathcal{B}\otimes\mathcal{B}\otimes\mathcal{B}\otimes\mathcal{B} 
\arrow{rr}{\mathrm{id}_\mathcal{B}\otimes
\tau_{\mathcal{B}\otimes\mathcal{B}}\otimes\mathrm{id}_\mathcal{B}}
&
& \mathcal{B}\otimes\mathcal{B}\otimes\mathcal{B}\otimes\mathcal{B}
\arrow{u}[swap]{\mu\otimes\mu}
\end{tikzcd},
\begin{tikzcd}
\Bbbk \arrow{r}{\eta} \arrow{d}[swap]{\cong}
& \mathcal{B}\arrow{d}{\Delta} \\
\Bbbk\otimes\Bbbk \arrow{r}{\eta\otimes\eta}
& \mathcal{B}\otimes\mathcal{B}
\end{tikzcd},
\end{equation}
\begin{equation}\label{eq23}
\begin{tikzcd}
\mathcal{B}\otimes\mathcal{B} \arrow{r}{\mu}
\arrow{d}[swap]{\epsilon\otimes\epsilon}
& \mathcal{B} \arrow{d}{\epsilon} \\
\Bbbk\otimes\Bbbk \arrow{r}{\cong}
& \Bbbk
\end{tikzcd}
\text{ and }
\begin{tikzcd}
\Bbbk \arrow{r}{\eta} \arrow{rd}[swap]{\mathrm{id}_\Bbbk}
& \mathcal{B} \arrow{d}{\epsilon} \\
 & \Bbbk
\end{tikzcd}.
\end{equation}
Remark that it is sufficient to demand
$\mu$ and $\eta$ to be coalgebra homomorphisms or $\Delta$ and $\epsilon$
to be algebra homomorphisms in Definition~\ref{def03},
which is clear from the symmetry of the diagrams (\ref{eq22}) and (\ref{eq23}),
see also \cite{ChPr94}~Sec.~4.1 Rem.~1.
\begin{lemma}
Let $(\mathcal{B},\mu,\eta,\Delta,\epsilon)$ be a $\Bbbk$-algebra and
a $\Bbbk$-coalgebra. Then $\mu$ and $\eta$ are coalgebra homomorphisms
if and only if $\Delta$ and $\epsilon$ are algebra homomorphisms.
\end{lemma}
We did not abstract the map $S$ yet, which becomes the main character
in the following definition, finally completing the notion of Hopf algebra.
\begin{definition}[Hopf Algebra]
A $\Bbbk$-bialgebra $(H,\mu,\eta,\Delta,\epsilon)$ is said to be a
$\Bbbk$-\textit{Hopf algebra} if there is a
$\Bbbk$-linear bijection $S\colon H\rightarrow H$ such that (\ref{eq02}) holds
and we call $S$ an antipode of $H$ in that case.
A bialgebra homomorphism between Hopf algebras is said to be a Hopf algebra
homomorphism if it intertwines the antipodes in addition.
We denote the category of Hopf algebras by ${}_\Bbbk\mathcal{H}$.
\end{definition}
In the following we often drop the reference to the commutative ring $\Bbbk$
and simply refer to Hopf algebras, etc.
Remark that there are slightly weaker definitions of Hopf algebra, not
assuming the antipode to have an inverse (see \cite{Ka95,Ma95,Mo93}).
We follow the convention of \cite{ChPr94}, arguing that in all
examples which are relevant for us the antipode is invertible
and we do not want to state this as an additional condition throughout
the thesis. 
The \textit{antipode axioms} (\ref{eq02}) can be reformulated in
pictorial language by the commutativity of
\begin{equation*}
\begin{tikzcd}
& H\otimes H \arrow{rr}{S\otimes\mathrm{id}_H} & 
& H\otimes H \arrow{rd}{\mu} & \\
H \arrow{ru}{\Delta} \arrow{rd}[swap]{\Delta} \arrow{rr}{\epsilon} & 
& \Bbbk \arrow{rr}{\eta} & & H \\
& H\otimes H \arrow{rr}{\mathrm{id}_H\otimes S} & 
& H\otimes H \arrow{ru}[swap]{\mu} &
\end{tikzcd}.
\end{equation*}
Besides $\mathscr{U}\mathfrak{g}$ and $\mathbb{K}[G]$
there are further interesting examples of Hopf algebras taken from
\cite{ChPr94}~Sec.~4.1~B and \cite{Mo93}~Ex.~1.5.6, respectively.
\begin{example}\label{exa01}
\begin{enumerate}
\item[i.)] Let $G$ be a finite group with neutral element $e\in G$ and
consider the $\Bbbk$-module $\mathcal{F}(G)$ of functions on
$G$ with values in $\Bbbk$. It is a commutative $\Bbbk$-algebra,
where the product $F_1\cdot F_2$ of two functions $F_1,F_2\in\mathcal{F}(G)$ is
defined by
$$
(F_1\cdot F_2)(g)=F_1(g)F_2(g)
$$
for all $g\in G$ and with unit function defined by
$G\ni g\mapsto 1\in\Bbbk$. We further define two $\Bbbk$-linear maps
$\Delta\colon\mathcal{F}(G)\rightarrow\mathcal{F}(G\times G)$ and
$\epsilon\colon\mathcal{F}(G)\rightarrow\Bbbk$ as
$\Delta(F)(g,h)=F(gh)$ and $\epsilon(F)=F(e)$ for all
$F\in\mathcal{F}(G)$ and $g,h\in G$.
Note that there is an isomorphism
$\mathcal{F}(G)\otimes\mathcal{F}(G)\cong\mathcal{F}(G\times G)$
of $\Bbbk$-modules given by
$$
F_1\otimes F_2\mapsto\widehat{F_1\otimes F_2}
\colon G\times G\ni(g,h)\mapsto F_1(g)F_2(h)\in\Bbbk.
$$
Using this identification it follows that
$$
((\Delta\otimes\mathrm{id})\Delta(F))(g,h,\ell)
=F((gh)\ell)=F(g(h\ell))
=((\mathrm{id}\otimes\Delta)\Delta(F))(g,h,\ell)
$$
and
$
((\epsilon\otimes\mathrm{id})\Delta(F))(g)
=F(eg)=F(g)=F(ge)
=((\mathrm{id}\otimes\epsilon)\Delta(F))(g)
$
for all $F\in\mathcal{F}(G)$ and $g,h,\ell\in G$. Furthermore
$$
\Delta(F_1F_2)(g,h)=(F_1F_2)(gh)=F_1(gh)F_2(gh)=(\Delta(F_1)\Delta(F_2))(g,h),
$$
$\Delta(1)(g,h)=1$,
$\epsilon(F_1F_2)=(F_1F_2)(e)=F_1(e)F_2(e)=\epsilon(F_1)\epsilon(F_2)$
and $\epsilon(1)=1$
hold for all $F_1,F_2\in\mathcal{F}(G)$ and $g,h\in G$, proving that
$(\mathcal{F}(G),\Delta,\epsilon)$ is a bialgebra. Finally, one defines
an antipode on $\mathcal{F}(G)$ as the $\Bbbk$-linear map
$S\colon\mathcal{F}(G)\rightarrow\mathcal{F}(G)$ such that
$S(F)(g)=F(g^{-1})$ for all $F\in\mathcal{F}(G)$ and $g\in G$. In fact
$$
(S(F_{(1)})F_{(2)})(g)=F(g^{-1}g)=F(e)
=\epsilon(F)1
=(F_{(1)}S(F_{(2)}))(g)
$$
for all $F\in\mathcal{F}(G)$ and $g\in G$, where we used Sweedler's notation
for the coproduct. This describes the Hopf algebra $\mathcal{F}(G)$ of
$\Bbbk$-valued functions on $G$. It is cocommutative if and only if
$G$ is Abelian. The same computations hold if $G$ is an
affine algebraic group over a field. If $G$ is a compact topological group,
the finite-dimensional real representations $\mathrm{Rep}(G)$ of $G$ are not
only a dense subalgebra of $\mathcal{F}(G)$ but even a Hopf algebra, since
$\mathrm{Rep}(G\times G)\cong\mathrm{Rep}(G)\otimes\mathrm{Rep}(G)$.
This is called the Hopf algebra of representative functions on $G$.
Note that in general $\mathcal{F}(G)$ is not a Hopf algebra for a
compact topological group $G$;

\item[ii.)] As an algebra, Sweedler's Hopf algebra $H_4$ is generated by
a unit element $1$ and three elements $g,x$ and $gx$ such that the relations
$$
g^2=1,~
x^2=0
\text{ and }
xg=-gx
$$
hold. The coproduct and counit are defined on generators by
$$
\Delta(g)=g\otimes g,~
\Delta(x)=x\otimes 1+g\otimes x,~
\epsilon(g)=1
\text{ and }
\epsilon(x)=0,
$$
respectively, while the antipode reads $S(g)=g$ and $S(x)=-gx$
on generators. $g$ is group-like, while $x$ is $(1,g)$-primitive.
In fact, all relations are easily verified on generators.
Obviously, $H_4$ is neither commutative nor cocommutative. It is the smallest
Hopf algebra with this property;
\end{enumerate}
\end{example}
Let $G$ be a finite group and consider the two Hopf algebras
$\mathcal{F}(G)$ and $\mathbb{K}[G]$ arising from it. There is a
non-degenerate dual pairing 
$$
\langle\cdot,\cdot\rangle
\colon\mathcal{F}(G)\otimes\mathbb{K}[G]\rightarrow\mathbb{K}
$$
between them and in a remarkable way it mirrors the algebra structure
of $\mathcal{F}(G)$ with the coalgebra structure of $\mathbb{K}[G]$
and vice versa. Namely,
\begin{align*}
    \langle F_1F_2,g\rangle
    =&F_1(g)F_2(g)=\langle F_1\otimes F_2,g_{(1)}\otimes g_{(2)}\rangle,\\
    \langle 1,g\rangle
    =&1=\epsilon(g),\\
    \langle F,gh\rangle
    =&F(gh)=\langle F_{(1)}\otimes F_{(2)},g\otimes h\rangle \text{ and} \\
    \langle F,e\rangle
    =&F(e)=\epsilon(F)
\end{align*}
for all $F,F_1,F_2\in\mathcal{F}(G)$ and $g,h\in G$. Furthermore
the pairing mirrors the antipodes, i.e.
$\langle S(F),g\rangle=F(g^{-1})=\langle F,S(g)\rangle$. Recalling 
Lemma~\ref{lemma02} for a finite-dimensional Hopf algebra $H$
over a field $\mathbb{K}$ encourages this
duality: the bialgebra structure of $H$ is the transpose of the
bialgebra structure of $H^*$ via the dual pairing of (finite-dimensional)
vector spaces. It is not hard to prove that the transpose of an antipode
on $H$ leads to an antipode on $H^*$ and vice versa. Let us formalize
these observations.
\begin{definition}[Dual Pair of Hopf Algebras]
Consider two Hopf algebras $H$ and $H'$ over a commutative ring $\Bbbk$.
They are called \textit{dual pair of Hopf algebras} if there is a
$\Bbbk$-bilinear map $\langle\cdot,\cdot\rangle\colon H'\otimes H
\rightarrow\Bbbk$ such that
\begin{align*}
    \langle ab,\xi\rangle
    =&\langle a\otimes b,\xi_{(1)}\otimes\xi_{(2)}\rangle,\\
    \langle 1,\xi\rangle
    =&\epsilon(\xi),\\
    \langle a,\xi\chi\rangle
    =&\langle a_{(1)}\otimes a_{(2)},\xi\otimes\chi\rangle,\\
    \langle a,1\rangle
    =&\epsilon(a)
\end{align*}
and
$$
\langle S(a),\xi\rangle
=\langle a,S(\xi)\rangle
$$
for all $\xi,\chi\in H$ and $a,b\in H'$. They are called
\textit{strict dual pair of Hopf algebras} if the pairing 
$\langle\cdot,\cdot\rangle$ is \textit{non-degenerate} in addition, i.e.
if $\langle a,\xi\rangle=0$ for all $\xi\in H$ implies $a=0$ and
$\langle a,\xi\rangle=0$ for all $a\in H'$ implies $\xi=0$.
\end{definition}
We already observed that $(\mathbb{K}[G],\mathcal{F}(G))$, where
$G$ is a finite group, and $(H,H^*)$ for a finite-dimensional
$\mathbb{K}$-Hopf algebra $H$, are examples of strict dual pairs of
Hopf algebras. However, in the infinite-dimensional setting or more
general for commutative rings $\Bbbk$ there are issues as already
indicated by Lemma~\ref{lemma02}. A solution helping to avoid these
problems is given by not considering the whole dual $\Bbbk$-module
$\mathrm{Hom}_\Bbbk(H,\Bbbk)$ of an arbitrary $\Bbbk$-Hopf algebra
$H$, but rather its \textit{finite-dual}
$$
H^\circ=\{\alpha\in H^*~|~\exists\text{ algebra ideal }
I\subseteq H\text{ such that }\dim(H/I)<\infty\text{ and }\alpha(I)=0\}.
$$
It is a Hopf algebra with respect to the transposed Hopf algebra structure
(c.f. \cite{Mo93}~Thm.~9.1.3). In particular, $(H,H^\circ)$ is a dual pair
of Hopf algebras.
After this short excursus on duality we return to discuss general properties
of Hopf algebras. In particular we intend to prove additional features of
antipodes, beginning with uniqueness, c.f. \cite{ChPr94}~Sec.~4.1 Rem.~4.
\begin{lemma}\label{lemma03}
Let $(\mathcal{B},\mu,\eta,\Delta,\epsilon)$ be a $\Bbbk$-bialgebra.
A $\Bbbk$-linear bijection $S\colon H\rightarrow H$
is an antipode for $H$ if and only if $S$ is the convolution inverse
of the identity $\eta\circ\epsilon$ in the convolution algebra
$\mathrm{Hom}_\Bbbk(H,H)$.
\end{lemma}
\begin{proof}
By Lemma~\ref{lemma04} the $\Bbbk$-linear maps 
$\mathrm{Hom}_\Bbbk(H,H)$ from the coalgebra $(H,\Delta,\epsilon)$
to the algebra $(H,\mu,\eta)$ form an associative algebra with
respect to the convolution product $\star$ and with unit $\eta\circ\epsilon$.
Let $S\colon H\rightarrow H$ be a $\Bbbk$-linear bijection. Then
$S\star\mathrm{id}=\eta\circ\epsilon$ if and only if
$S(\xi_{(1)})\xi_{(2)}=\epsilon(\xi)1$ for all $\xi\in H$ and 
$\mathrm{id}\star S=\eta\circ\epsilon$ if and only 
$\xi_{(1)}S(\xi_{(2)})=\epsilon(\xi)1$ for all $\xi\in H$.
\end{proof}
Since the inverse element of an algebra element is unique,
Lemma~\ref{lemma03} implies the following statement.
\begin{corollary}
Let $(H,\mu,\eta,\Delta,\epsilon,S)$ be a $\Bbbk$-Hopf algebra.
Then the underlying bialgebra structure of $H$ admits a unique
antipode $S$.
\end{corollary}
Moreover, \textit{the} antipode of a Hopf algebra respects the
underlying bialgebra structure in the sense that it is an
\textit{anti-bialgebra homomorphism}, i.e. a bialgebra homomorphism in a
contravariant way. In detail, $S\colon H\rightarrow H$ is an 
anti-algebra homomorphism if
\begin{equation}\label{eq16}
    S(\xi\chi)=S(\chi)S(\xi)
    \text{ and }S(1)=1
\end{equation}
hold for all $\xi,\chi\in H$ and it is an anti-coalgebra homomorphism if
\begin{equation}\label{eq17}
    S(\xi)_{(1)}\otimes S(\xi)_{(2)}
    =S(\xi_{(2)})\otimes S(\xi_{(1)})
    \text{ and }
    \epsilon(S(\xi))=\epsilon(\xi)
\end{equation}
hold for all $\xi\in H$.
Furthermore, the examples we mentioned suggest that the
inverse of the antipode is given by the antipode itself. Even if this
does not hold in general, it is the case for a huge class of Hopf algebras,
namely for commutative or cocommutative ones, including our examples
besides Example~\ref{exa01}~ii.).
Both statements are discussed in the following proposition
(c.f. \cite{Ka95}~Thm.~III.3.4. and \cite{Ma95}~Prop.~1.3.1).
\begin{proposition}\label{prop07}
Let $H$ be a $\Bbbk$-Hopf algebra. Then its antipode $S$ is an anti-bialgebra
homomorphism. If $H$ is either commutative or cocommutative
$S^2=\mathrm{id}$ follows.
\end{proposition}
\begin{proof}
Consider the coalgebra $(H\otimes H,\Delta_{H\otimes H},
\epsilon_{H\otimes H})$ and the convolution algebra
$\mathrm{Hom}_\Bbbk(H\otimes H,H)$ with convolution product $\star$
and unit $\eta\circ\epsilon_{H\otimes H}$. It is an algebra according to
Lemma~\ref{lemma04}. We are going to prove the first equation of (\ref{eq16})
by defining the left- and right-hand side to be
$\Bbbk$-linear maps $f,g\colon H\otimes H\rightarrow H$, respectively,
and showing that $f\star h=\eta\circ\epsilon_{H\otimes H}
=h\star g$ for a $\Bbbk$-linear map $h\colon H\otimes H\rightarrow H$.
In fact this is sufficient, since the left and right inverse of an algebra
element coincide if both exist. For the first equation of (\ref{eq17}) we
use a similar strategy, considering the convolution algebra of
$\Bbbk$-linear maps from the coalgebra $(H,\Delta,\epsilon)$ to the algebra
$(H\otimes H,\mu_{H\otimes H},\eta_{H\otimes H})$.
We divide the proof in three parts.
\begin{enumerate}
\item[i.)] \textbf{$S$ is an anti-algebra homomorphism: }
define $f(\xi\otimes\chi)=S(\xi\chi)$ and $g(\xi\otimes\chi)=S(\chi)S(\xi)$
for all $\xi,\chi\in H$. Then, using that $\Delta$ and $\epsilon$ are algebra 
homomorphisms and the antipode properties, we obtain
\begin{align*}
    (f\star\mu)(\xi\otimes\chi)
    =&f(\xi_{(1)}\otimes\chi_{(1)})\mu(\xi_{(2)}\otimes\chi_{(2)})\\
    =&S(\xi_{(1)}\chi_{(1)})\xi_{(2)}\chi_{(2)}\\
    =&S((\xi\chi)_{(1)})(\xi\chi)_{(2)}\\
    =&\eta(\epsilon(\xi\chi))\\
    =&\eta(\epsilon(\xi)\epsilon(\chi))\\
    =&\eta(\epsilon_{H\otimes H}(\xi\otimes\chi))
\end{align*}
and
\begin{allowdisplaybreaks}
\begin{align*}
    (\mu\star g)(\xi\otimes\chi)
    =&\mu(\xi_{(1)}\otimes\chi_{(1)})g(\xi_{(2)}\otimes\chi_{(2)})\\
    =&\xi_{(1)}\chi_{(1)}S(\chi_{(2)})S(\xi_{(2)})\\
    =&\xi_{(1)}\eta(\epsilon(\chi))S(\xi_{(2)})\\
    =&\xi_{(1)}S(\xi_{(2)})\eta(\epsilon(\chi))\\
    =&\eta(\epsilon(\xi))\eta(\epsilon(\chi))\\
    =&\eta(\epsilon(\xi)\epsilon(\chi))\\
    =&\eta(\epsilon_{H\otimes H}(\xi\otimes\chi))
\end{align*}
\end{allowdisplaybreaks}
for all $\xi,\chi\in H$, implying that $f=g$. Furthermore,
$$
S(1)=S(1_{(1)})1_{(2)}=\eta(\epsilon(1))=1
$$
implies (\ref{eq16}).

\item[ii.)]\textbf{ $S$ is an anti-coalgebra homomorphism: }
define $f(\xi)=S(\xi)_{(1)}\otimes S(\xi)_{(2)}$ and
$g(\xi)=S(\xi_{(2)})\otimes S(\xi_{(1)})$ for all $\xi\in H$. Then
\begin{align*}
    (f\star\Delta)(\xi)
    =&\mu_{H\otimes H}(f(\xi_{(1)})\otimes\Delta(\xi_{(2)}))\\
    =&S(\xi_{(1)})_{(1)}\xi_{(2)(1)}\otimes S(\xi_{(1)})_{(2)}\xi_{(2)(2)}\\
    =&(S(\xi_{(1)})\xi_{(2)})_{(1)}\otimes(S(\xi_{(1)})\xi_{(2)})_{(2)}\\
    =&(\eta(\epsilon(\xi)))_{(1)}\otimes(\eta(\epsilon(\xi)))_{(2)}\\
    =&\epsilon(\xi)\eta(1)\otimes\eta(1)\\
    =&\eta_{H\otimes H}(\epsilon(\xi))
\end{align*}
and
\begin{align*}
    (\Delta\otimes g)(\xi)
    =&\xi_{(1)(1)}S(\xi_{(2)(2)})\otimes\xi_{(1)(2)}S(\xi_{(2)(1)})\\
    =&\xi_{(1)}S(\xi_{(2)(2)(2)})\otimes\xi_{(2)(1)}S(\xi_{(2)(2)(1)})\\
    =&\xi_{(1)}S(\xi_{(2)(2)})\otimes\xi_{(2)(1)(1)}S(\xi_{(2)(1)(2)})\\
    =&\xi_{(1)}S(\xi_{(2)(2)})\otimes\eta(\epsilon(\xi_{(2)(1)}))\\
    =&\xi_{(1)}S(\xi_{(2)})\otimes\eta(1)\\
    =&\eta(\epsilon(\xi))\otimes\eta(1)\\
    =&\eta_{H\otimes H}(\epsilon(\xi))
\end{align*}
imply $f=g$. Furthermore,
$$
\epsilon(S(\xi))
=\epsilon(S(\epsilon(\xi_{(1)})\xi_{(2)}))
=\epsilon(\xi_{(1)}S(\xi_{(2)}))
=\epsilon(\epsilon(\xi)1)
=\epsilon(\xi)
$$
for all $\xi\in H$, implying (\ref{eq17}).

\item[iii.)]\textbf{For a commutative or cocommutative $H$ the antipode is
an involution: } for all $\xi\in H$ we prove
\begin{align*}
    (S\star S^2)(\xi)
    =S(\xi_{(1)})S^2(\xi_{(2)})
    =S(S(\xi_{(2)})\xi_{(1)})
    =S(\eta(\epsilon(\xi)))
    =\eta(\epsilon(\xi)),
\end{align*}
using i.),
where the third equality holds if $H$ is commutative or cocommutative.
Since $S\star\mathrm{id}=\eta\circ\epsilon$ by the antipode
property we obtain $S^2=\mathrm{id}$, because the right inverse
of an algebra element is unique if it exists.
\end{enumerate}
This concludes the proof of the proposition.
\end{proof}
As a terminal observation we want to prove that
commutativity of a bialgebra homomorphism with the antipodes is in fact
a redundant condition in the definition of Hopf algebra homomorphism
(compare to \cite{ChPr94}~Sec.~4.1 Rem.~2).
\begin{lemma}
Let $\phi\colon H\rightarrow H'$ be a bialgebra homomorphism between
$\Bbbk$-Hopf algebras
$(H,\mu,\eta,\Delta,\epsilon,S)$ and $(H',\mu',\eta',\Delta',\epsilon',S')$.
Then $\phi_S=\phi\circ S$ and $\phi_{S'}=S'\circ\phi$ are $\Bbbk$-linear maps
which are convolution inverse to $\phi$. In particular
$\phi_S=\phi_{S'}$ and $\phi$ is a Hopf algebra homomorphism if and only
if it is a bialgebra homomorphism.
\end{lemma}
\begin{proof}
The $\Bbbk$-linearity of $\phi_S$ and $\phi_{S'}$ is clear since they
are defined as a concatenation of $\Bbbk$-linear maps. For all $\xi\in H$
one obtains
\begin{align*}
    (\phi_S\star\phi)(\xi)
    =\phi(S(\xi_{(1)}))\phi(\xi_{(2)})
    =\phi(S(\xi_{(1)})\xi_{(2)})
    =\phi(\eta(\epsilon(\xi)))
    =\eta'(\epsilon(\xi))
\end{align*}
using that $\phi$ is an algebra homomorphism and
\begin{align*}
    (\phi\star\phi_{S'})(\xi)
    =\phi(\xi_{(1)})S'(\phi(\xi_{(2)}))
    =\phi(\xi)_{(1)}S'(\phi(\xi)_{(2)})
    =\eta'(\epsilon'(\phi(\xi)))
    =\eta'(\epsilon(\xi))
\end{align*}
using that $\phi$ is a coalgebra homomorphism,
implying $\phi_S=\phi_{S'}$ by the uniqueness of the inverse element.
This shows that commuting with the antipode is a redundant condition
for a Hopf algebra homomorphism.
\end{proof}
In the next section we focus on algebra representations
and characterize bialgebras and Hopf algebras via additional properties of the
category of representations of the underlying algebras. These considerations
further lead to the definition of quasi-triangular structures in 
Section~\ref{Sec2.3}.

\section{Hopf Algebra Modules}\label{SectionHopfAlgebraModules}

We introduce Hopf algebra modules and prove that they form a monoidal
category with monoidal structure given by the tensor product of
$\Bbbk$-modules. In fact it turns out that bialgebras are those
algebras whose category of representations is monoidal with
respect to the usual associativity and unit constraints, leading
to an important characterization of bialgebras. The antipode
of a Hopf algebra gives rise to an additional duality property on
the categorical level, however only for finitely generated projective
modules. All definitions and statements can also be found in
\cite{ChPr94}~Sec.~5.1, \cite{Ka95}~Sec.~XI.3 and
\cite{Ma95}~Sec.~9.1.

Consider a $\Bbbk$-algebra $(\mathcal{A},\mu,\eta)$
for a commutative ring $\Bbbk$. A $\Bbbk$-module $\mathcal{M}$
is said to be a \textit{left $\mathcal{A}$-module} if there exists a
$\Bbbk$-linear map $\lambda\colon\mathcal{A}\otimes\mathcal{M}
\rightarrow\mathcal{M}$ such that the diagrams
\begin{equation*}
\begin{tikzcd}
\mathcal{A}\otimes\mathcal{A}\otimes\mathcal{M}
\arrow{r}{\mathrm{id}_\mathcal{A}\otimes\lambda}
\arrow{d}[swap]{\mu\otimes\mathrm{id}_\mathcal{M}}
& \mathcal{A}\otimes\mathcal{M} \arrow{d}{\lambda} \\
\mathcal{A}\otimes\mathcal{M} \arrow{r}{\lambda}
& \mathcal{M}
\end{tikzcd}
\text{ and }
\begin{tikzcd}
\Bbbk\otimes\mathcal{M} \arrow{r}{\eta\otimes\mathrm{id}_\mathcal{M}}
\arrow{rd}[swap]{\cong}
& \mathcal{A}\otimes\mathcal{M} \arrow{d}{\lambda} \\
& \mathcal{M}
\end{tikzcd}
\end{equation*}
commute. The tuple $(\mathcal{M},\lambda)$ is called a
\textit{left representation of $\mathcal{A}$ on $\mathcal{M}$} and
$\lambda$ is said to be a \textit{left $\mathcal{A}$-module action} or
\textit{left $\mathcal{A}$-module structure}.
The left $\mathcal{A}$-modules form a category ${}_\mathcal{A}\mathcal{M}$
with morphisms given by \textit{left $\mathcal{A}$-module homomorphism},
where a $\Bbbk$-linear map $\phi\colon\mathcal{M}\rightarrow\mathcal{M}'$
between two left representations $(\mathcal{M},\lambda_\mathcal{M})$
and $(\mathcal{M}',\lambda_{\mathcal{M}'})$ of $\mathcal{A}$ is said to be a
left $\mathcal{A}$-module homomorphism if the diagram
\begin{equation*}
\begin{tikzcd}
\mathcal{A}\otimes\mathcal{M} \arrow{r}{\lambda_\mathcal{M}}
\arrow{d}[swap]{\mathrm{id}_\mathcal{A}\otimes\phi}
& \mathcal{M} \arrow{d}{\phi} \\
\mathcal{A}\otimes\mathcal{M}' \arrow{r}{\lambda_{\mathcal{M}'}}
& \mathcal{M}'
\end{tikzcd}
\end{equation*}
commutes. This means that $\phi$ respects the left $\mathcal{A}$-module
structures of $\mathcal{M}$ and $\mathcal{M}'$ or that $\phi$ is
left $\mathcal{A}$-linear in other words. If $\mathcal{A}=H$ is a Hopf algebra
we often refer to left $H$-module homomorphisms as \textit{$H$-equivariant maps}.
The algebra $\mathcal{A}$ and the corresponding category
${}_\mathcal{A}\mathcal{M}$ can be seen as two sides of the same coin
via the so called \textit{Tannaka-Krein Duality}. By assigning
to any left $\mathcal{A}$-module its underlying $\Bbbk$-module and to any left
$\mathcal{A}$-module homomorphism itself, now seen as a
$\Bbbk$-linear map, we obtain a
functor $F\colon{}_\mathcal{A}\mathcal{M}\rightarrow{}_\Bbbk\mathcal{M}$.
There is a $1:1$-correspondence between $\mathcal{A}$ and the natural 
transformations $\mathrm{Nat}(F,F)$ of $F$. This allows us to
reconstruct $\mathcal{A}$ from its representation theory
${}_\mathcal{A}\mathcal{M}$ and the functor $F$.
On the other hand, we are able to structure
the natural transformations of $F$ as an algebra in this way.
The functor $F$ is called the \textit{forgetful functor} for obvious reasons.
\begin{proposition}[Tannaka Reconstruction of Algebras]\label{prop08}
Let $\mathcal{A}$ be an algebra and consider the forgetful functor $F$.
For every $a\in\mathcal{A}$ there is a natural transformation $\Theta^a$ of
$F$ given on objects $\mathcal{M}$ of ${}_\mathcal{A}\mathcal{M}$ by
$$
\Theta^a_\mathcal{M}\colon
F(\mathcal{M})\ni m\mapsto a\cdot m\in F(\mathcal{M})
$$
and which is the identity on morphisms. This leads to an algebra isomorphism
$\mathcal{A}\cong\mathrm{Nat}(F,F)$.
\end{proposition}
\begin{proof}
The proof is taken from \cite{Ma95}~Ex.~9.1.1.
Recall that $F(\mathcal{M})=\mathcal{M}$ since $F$ is the
forgetful functor. $\Theta^a$ is a natural transformation
since for any left $\mathcal{A}$-module homomorphism 
$\phi\colon\mathcal{M}\rightarrow\mathcal{M}'$
we obtain
\begin{align*}
    F(\phi)(\Theta^a_\mathcal{M}(m))=\phi(a\cdot m)
    =a\cdot\phi(m)
    =\Theta^a_{\mathcal{M}'}(F(\phi)(m))
\end{align*}
for all $m\in\mathcal{M}$. On the other hand, we can select an element
$a$ in $\mathcal{A}$
starting from a natural transformation $\Theta$ of $F$, by defining
$a=\Theta_\mathcal{A}(1)$, where $\mathcal{A}$ is a left $\mathcal{A}$-module
via the left multiplication and $1$ denotes the unit in $\mathcal{A}$. We prove
that these constructions are inverse to each other, for which the 
$1:1$-correspondence follows. Any $a\in\mathcal{A}$ can be recovered via
$\Theta^a_\mathcal{A}(1)=a\cdot 1=a$. On the other hand let $\Theta$
be a natural transformation of $F$ and consider the corresponding element
$a=\Theta_\mathcal{A}(1)\in\mathcal{A}$. Then
\begin{align*}
    \Theta^a_\mathcal{M}(m)
    =a\cdot m
    =F(\phi_m)(\Theta_\mathcal{A}(1))
    =\Theta_\mathcal{M}(F(\phi_m)(1))
    =\Theta_\mathcal{M}(1\cdot m)
    =\Theta_\mathcal{M}(m)
\end{align*}
for all $m\in\mathcal{M}$ and any left $\mathcal{A}$-module $\mathcal{M}$, where
$\phi_m\colon\mathcal{A}\ni b\mapsto b\cdot m\in\mathcal{M}$ is a left
$\mathcal{A}$-module
homomorphism. The associative unital algebra structure on $\mathrm{Nat}(F,F)$
is describe as follows: two natural transformations $\Theta^a$ and $\Theta^b$ 
of $F$ are specified by two elements $a,b\in\mathcal{A}$, while their product
$\Theta^a\cdot\Theta^b$ is defined as the natural transformation
$\Theta^{a\cdot b}$. Explicitly, for any left $\mathcal{A}$-module
$\mathcal{M}$ one has
$$
(\Theta^{a}\cdot\Theta^b)_\mathcal{M}(m)
=\Theta^{a\cdot b}_\mathcal{M}(m)
=(a\cdot b)\cdot m
=a\cdot(b\cdot m)
=\Theta^a_\mathcal{M}(\Theta^b_\mathcal{M}(m))
$$
for all $m\in\mathcal{M}$ and $\Theta^1_\mathcal{M}=\mathrm{id}_\mathcal{M}$.
Since the product of $\mathcal{A}$ is associative, so is the product of
$\mathrm{Nat}(F,F)$.
\end{proof}
This duality of algebras and their representation theory is 
a fundamental concept and will be applied throughout the whole thesis.
We are going to examine further algebraic properties of algebras
parallel to categorical properties of their modules. Shifting between
these two pictures turns out to be an extremely useful tool.

Analogously to left modules, one defines a 
\textit{right $\mathcal{A}$-module}
to be a $\Bbbk$-module $\mathcal{M}$ together with a $\Bbbk$-linear map
$\rho\colon\mathcal{M}\otimes\mathcal{A}\longrightarrow\mathcal{M}$,
making
\begin{equation*}
\begin{tikzcd}
\mathcal{M}\otimes\mathcal{A}\otimes\mathcal{A}
\arrow{r}{\rho\otimes\mathrm{id}_\mathcal{A}}
\arrow{d}[swap]{\mathrm{id}_\mathcal{M}\otimes\mu}
& \mathcal{M}\otimes\mathcal{A} \arrow{d}{\rho} \\
\mathcal{M}\otimes\mathcal{A} \arrow{r}{\rho}
& \mathcal{M}
\end{tikzcd}
\text{ and }
\begin{tikzcd}
\mathcal{M}\otimes\Bbbk \arrow{r}{\mathrm{id}_\mathcal{M}\otimes\eta}
\arrow{rd}[swap]{\cong}
& \mathcal{M}\otimes\mathcal{A} \arrow{d}{\rho} \\
& \mathcal{M}
\end{tikzcd}
\end{equation*}
commute. A \textit{right $\mathcal{A}$-module homomorphism} is a 
$\Bbbk$-linear map $\phi\colon\mathcal{M}\rightarrow\mathcal{M}'$,
where $(\mathcal{M},\rho_\mathcal{M})$ and $(\mathcal{M}',
\rho_{\mathcal{M}'})$ are right $\mathcal{A}$-modules, such that
\begin{equation*}
\begin{tikzcd}
\mathcal{M}\otimes\mathcal{A} \arrow{r}{\rho_\mathcal{M}}
\arrow{d}[swap]{\phi\otimes\mathrm{id}_\mathcal{A}}
& \mathcal{M} \arrow{d}{\phi} \\
\mathcal{M}'\otimes\mathcal{A} \arrow{r}{\rho_{\mathcal{M}'}}
& \mathcal{M}'
\end{tikzcd}
\end{equation*}
commutes.
The category of right $\mathcal{A}$-modules is denoted by 
$\mathcal{M}_\mathcal{A}$. If an object in ${}_\mathcal{A}\mathcal{M}$
is also a right $\mathcal{A}'$-module for another $\Bbbk$-algebra 
$(\mathcal{A}',\mu',\eta')$ it is natural to ask if both module actions
are compatible. This leads to the notion of bimodules.
An \textit{$\mathcal{A}$-$\mathcal{A}'$-bimodule} is an object 
$(\mathcal{M},\lambda,\rho)$ in
${}_\mathcal{A}\mathcal{M}\cap\mathcal{M}_{\mathcal{A}'}$
such that 
\begin{equation*}
\begin{tikzcd}
\mathcal{A}\otimes\mathcal{M}\otimes\mathcal{A}'
\arrow{r}{\lambda\otimes\mathrm{id}_{\mathcal{A}'}}
\arrow{d}[swap]{\mathrm{id}_\mathcal{A}\otimes\rho}
& \mathcal{M}\otimes\mathcal{A}' \arrow{d}{\rho} \\
\mathcal{A}\otimes\mathcal{M} \arrow{r}{\lambda}
& \mathcal{M}
\end{tikzcd}
\end{equation*}
commutes.
If $\mathcal{A}'=\mathcal{A}$ as algebras we are calling $\mathcal{M}$
an $\mathcal{A}$-bimodule. The category of 
$\mathcal{A}$-$\mathcal{A}'$-bimodules is denoted by 
${}_\mathcal{A}\mathcal{M}_{\mathcal{A}'}$. Its morphisms
are left $\mathcal{A}$-module homomorphisms which are also right 
$\mathcal{A}'$-module homomorphisms.
Since every right $\mathcal{A}$-module can be viewed as a left
$\mathcal{A}^{\mathrm{op}}$-module for the opposite algebra
$\mathcal{A}^{\mathrm{op}}=
(\mathcal{A},\mu\circ\tau_{\mathcal{A}\otimes\mathcal{A}},\eta)$, we
can focus on left $\mathcal{A}$-modules without loss of generality.
We come back to right and bimodules in the subsequent sections.
In the following we denote a left $\mathcal{A}$-module action by
$\cdot$ for short, if there is no danger of confusion with other operations.

The category ${}_\mathcal{A}\mathcal{M}$ of left $\mathcal{A}$-modules
is a subcategory of the category ${}_\Bbbk\mathcal{M}$ of all $\Bbbk$-modules
for any $\Bbbk$-algebra $\mathcal{A}$.
In the following lines we show how to use the additional data
of a bialgebra to structure ${}_\mathcal{A}\mathcal{M}$ even as a monoidal
subcategory of ${}_\Bbbk\mathcal{M}$. Assume that $(\mathcal{A},\mu,\eta,
\Delta,\epsilon)$ is a $\Bbbk$-bialgebra. Then, the tensor product
$\mathcal{M}\otimes\mathcal{M}'$ of two left $\mathcal{A}$-modules becomes
an object in ${}_\mathcal{A}\mathcal{M}$ via
$$
a\cdot(m\otimes m')
=(a_{(1)}\cdot m)\otimes(a_{(2)}\cdot m')
$$
for all $a\in\mathcal{A}$, $m\in\mathcal{M}$ and $m'\in\mathcal{M}'$, by
making use of the fact that $\Delta$ is an algebra homomorphism.
Furthermore, since $\epsilon$ is an algebra homomorphism, the commutative
ring $\Bbbk$ becomes a left $\mathcal{A}$-module via
$$
a\cdot\lambda=\epsilon(a)\lambda
$$
for all $a\in\mathcal{A}$ and $\lambda\in\Bbbk$.
This action respects the usual associativity constraint of
${}_\Bbbk\mathcal{M}$, since $\Delta$ is coassociative. Namely,
\begin{align*}
    a\cdot((m\otimes m')\otimes m'')
    =&((a_{(1)(1)}\cdot m)\otimes(a_{(1)(2)}\cdot m'))\otimes(a_{(2)}\cdot m'')\\
    =&(a_{(1)}\cdot m)\otimes((a_{(2)(1)}\cdot m')\otimes(a_{(2)(2)}\cdot m''))\\
    =&a\cdot(m\otimes(m'\otimes m''))
\end{align*}
for another left $\mathcal{A}$-module $\mathcal{M}''$ and $m''\in\mathcal{M}''$.
It further respects the usual unit constraints of
${}_\Bbbk\mathcal{M}$, i.e.
\begin{align*}
    a\cdot(\lambda\otimes m)
    =&(\epsilon(a_{(1)})\lambda)\otimes(a_{(2)}\cdot m)\\
    =&\lambda\otimes(a\cdot m)\\
    =&\lambda(a\cdot m)\\
    =&a\cdot(m\otimes\lambda)
\end{align*}
for all $a\in\mathcal{A}$, $m\in\mathcal{M}$ and $\lambda\in\Bbbk$, 
since $\epsilon$ satisfies the counit axiom.
Strictly speaking the usual associativity and unit constraints
$\alpha$, $\ell$ and $r$ are left $\mathcal{A}$-module homomorphisms
on objects and we should write
$$
\alpha_{\mathcal{M},\mathcal{M}',\mathcal{M}''}(a\cdot((m\otimes m')\otimes m''))
=a\cdot\alpha_{\mathcal{M},\mathcal{M}',\mathcal{M}''}((m\otimes m')\otimes m'')
$$
and
$$
\ell_\mathcal{M}(a\cdot(\lambda\otimes m))
=a\cdot\ell_\mathcal{M}(\lambda\otimes m),~
r_\mathcal{M}(a\cdot(m\otimes\lambda))
=a\cdot r_\mathcal{M}(m\otimes\lambda).
$$
Instead we treat those isomorphisms as equalities in the above computations.
In the next proposition (c.f. \cite{Ka95}~Prop.~XI.3.1 and
\cite{Ma95}~Ex.~9.1.3)
we show that coassociativity and the counit axiom are not
only sufficient but also necessary for ${}_\mathcal{A}\mathcal{M}$ to be a
monoidal subcategory of ${}_\Bbbk\mathcal{M}$ with usual associativity and
unit constraints.
\begin{proposition}\label{prop04}
Let $\mathcal{A}$ be an algebra over a commutative ring $\Bbbk$
and consider two algebra homomorphisms
\begin{equation}\label{eq58}
    \Delta\colon\mathcal{A}\rightarrow\mathcal{A}\otimes\mathcal{A}
    \text{ and }
    \epsilon\colon\mathcal{A}\rightarrow\Bbbk.
\end{equation}
Then $(\mathcal{A},\Delta,\epsilon)$ is a $\Bbbk$-bialgebra if and only
if ${}_\mathcal{A}\mathcal{M}$ is a monoidal subcategory of
${}_\Bbbk\mathcal{M}$ with respect to the usual tensor product and unit of
$\Bbbk$-modules and the usual associativity and unit constraints.
\end{proposition}
\begin{proof}
We already proved that ${}_\mathcal{A}\mathcal{M}$ is a monoidal subcategory
if $\mathcal{A}$ is a bialgebra. So assume that
$\mathcal{A}$ is an algebra endowed with two algebra homomorphisms
(\ref{eq58}) and assume further that
${}_\mathcal{A}\mathcal{M}$ is monoidal with respect to $\alpha,\ell,r$. Since
$\alpha_{\mathcal{M},\mathcal{M}',\mathcal{M}''}$ is a left $\mathcal{A}$-module
homomorphism for all objects $\mathcal{M},\mathcal{M}',\mathcal{M}''$ in
${}_\mathcal{A}\mathcal{M}$ by definition, we obtain
\begin{align*}
    \alpha_{\mathcal{A},\mathcal{A},\mathcal{A}}
    (((\Delta\otimes\mathrm{id})\circ\Delta)(a))
    =&\alpha_{\mathcal{A},\mathcal{A},\mathcal{A}}(a\cdot((1\otimes 1)\otimes 1))\\
    =&a\cdot\alpha_{\mathcal{A},\mathcal{A},\mathcal{A}}((1\otimes 1)\otimes 1)\\
    =&((\mathrm{id}\otimes\Delta)\circ\Delta)(a)
\end{align*}
for all $a\in\mathcal{A}$,
by considering the left $\mathcal{A}$-module $\mathcal{A}$ itself with
multiplication from the left as module action. Furthermore,
\begin{align*}
    \ell_\mathcal{A}(((\epsilon\otimes\mathrm{id})\circ\Delta)(a))
    =\ell_\mathcal{A}(a\cdot(1\otimes 1))
    =a\cdot\ell_\mathcal{A}(1\otimes 1)
    =a\cdot 1
    =a
\end{align*}
and
\begin{align*}
    r_\mathcal{A}(((\mathrm{id}\otimes\epsilon)\circ\Delta)(a))
    =r_\mathcal{A}(a\cdot(1\otimes 1))
    =a\cdot e_\mathcal{A}(1\otimes 1)
    =a\cdot 1
    =a
\end{align*}
follow for all $a\in\mathcal{A}$. Since we treat $\alpha,\ell,r$ as 
identities on objects this means that 
$\Delta$ is coassociative and $\epsilon$ is a
counit. This concludes the proof.
\end{proof}
This characterizes $\Bbbk$-bialgebras completely. Focusing on representation
theory one might for this reason reformulate the definition of a bialgebra
by declaring: an algebra $\mathcal{A}$ with algebra homomorphisms
$\Delta\colon\mathcal{A}\rightarrow\mathcal{A}\otimes\mathcal{A}$ and
$\epsilon\colon\mathcal{A}\rightarrow\Bbbk$ is a bialgebra if
${}_\mathcal{A}\mathcal{M}$ is a monoidal subcategory of ${}_\Bbbk\mathcal{M}$
with usual associativity and unit constraints. 

Incorporating the antipode of a Hopf algebra in the picture we receive
a duality property on the side of monoidal categories. However, for
this we have to restrict our class of modules from arbitrary
$\Bbbk$-modules ${}_\Bbbk\mathcal{M}$ for a commutative ring $\Bbbk$
to finitely generated projective $\Bbbk$-modules. A $\Bbbk$-module
$\mathcal{M}$ is said to be \textit{finitely generated projective}
if there is a finite set $\{m_i\}_{i\in I}$ of elements in $\mathcal{M}$,
called \textit{generators}, and a set of elements $\{\alpha^i\}_{i\in I}$
called \textit{dual generators} in the dual $\Bbbk$-module
$\mathcal{M}^*=\mathrm{Hom}_\Bbbk(\mathcal{M},\Bbbk)$, subscripted by the
same finite index set $I$, such that
$$
\alpha^i(m_j)=\delta^i_j
\text{ and }
m=\sum_{i\in I}\alpha^i(m)m_i
$$
for all $m\in\mathcal{M}$. It follows that $\mathcal{M}^*$ is finitely
generated projective with generators $\{\alpha^i\}_{i\in I}$ and dual
generators $\{\hat{m}_i\}_{i\in I}$ in $(\mathcal{M}^*)^*$, defined by
$\hat{m}_i(\alpha)=\alpha(m_i)$ for all $\alpha\in\mathcal{M}^*$.
Furthermore, the map $\hat{ }\colon\mathcal{M}\rightarrow(\mathcal{M}^*)^*$
is a $\Bbbk$-module isomorphism. Note that this identification
fails for general $\Bbbk$-modules, even if $\Bbbk$ is a field. We denote
the category of finitely generated projective $\Bbbk$-modules with
$\Bbbk$-module homomorphisms as morphisms by ${}_\Bbbk\mathcal{M}_f$.
It is a monoidal subcategory of ${}_\Bbbk\mathcal{M}$.
\begin{lemma}\label{lemma05}
The monoidal category ${}_\Bbbk\mathcal{M}_f$ of finitely generated
projective modules of a commutative ring $\Bbbk$ is a rigid category.
\end{lemma}
\begin{proof}
We follow \cite{ChPr94}~Ex.~5.1.3.
Consider a finitely generated projective $\Bbbk$-module
$\mathcal{M}$ with (dual) generators $m_i$ and $\alpha^i$ and define
\begin{align*}
    \mathrm{ev}_\mathcal{M}(\alpha\otimes m)
    =&\hat{m}_i(\alpha)\alpha^i(m),\\
    \pi_\mathcal{M}(1)
    =&m_i\otimes\alpha^i,\\
    \mathrm{ev}'_\mathcal{M}(m\otimes\alpha)
    =&\hat{m}_i(\alpha)\alpha^i(m)
    \text{ and} \\
    \pi'_\mathcal{M}(1)
    =&\alpha^i\otimes m_i
\end{align*}
where $m=\alpha^i(m)m_i\in\mathcal{M}$
and $\alpha=\hat{m}_i(\alpha)\alpha^i\in\mathcal{M}^*$ and we used Einstein
sum convention. Then
$$
((\mathrm{id}_\mathcal{M}\otimes\mathrm{ev}_\mathcal{M})
\circ(\pi_\mathcal{M}\otimes\mathrm{id}_\mathcal{M}))(m)=m
$$
and
$$
((\mathrm{ev}_{\mathcal{M}}\otimes\mathrm{id}_{\mathcal{M}^*})
\circ(\mathrm{id}_{\mathcal{M}^*}\otimes\pi_\mathcal{M}))(\alpha)=\alpha
$$
follow.
\end{proof}
Note that without a finite set of dual generators we are in general
not able to define the maps $\pi$ required for rigidity.
In particular, ${}_\Bbbk\mathcal{M}$ is \textit{not} rigid in general.
This motivates
the passage from general $\Bbbk$-modules to finitely generated projective
$\Bbbk$-modules. As a special case we recover the rigid monoidal
category of finite-dimensional $\mathbb{K}$-vector spaces 
${}_\mathbb{K}\mathrm{Vec}_f$ if $\Bbbk=\mathbb{K}$ is a field.
In Proposition~\ref{prop04} we proved that bialgebras correspond to
monoidal categories of representation. The analogue for Hopf algebras
in the setting of finitely generated projective modules is given
in the following statement (c.f. \cite{ChPr94}~Ex.~5.1.4).
\begin{proposition}\label{prop09}
Let $H$ be a $\Bbbk$-Hopf algebra and consider the monoidal category
${}_H\mathcal{M}$ of left $H$-modules, characterized by the bialgebra
structure of $H$. The monoidal subcategory ${}_H\mathcal{M}_f$ of
finitely generated projective left $H$-modules is rigid, where the
left and right dual $\mathcal{M}^*$ and ${}^*\mathcal{M}$ of an object
$\mathcal{M}$ in ${}_H\mathcal{M}_f$ are defined as the finitely generated
projective $\Bbbk$-module $\mathcal{M}^*$ with left $H$-module actions
given by
$$
\langle\xi\cdot\alpha,m\rangle
=\langle\alpha,S(\xi)\cdot m\rangle
$$
and
$$
\langle\xi\cdot\alpha,m\rangle
=\langle\alpha,S^{-1}(\xi)\cdot m\rangle
$$
for all $\xi\in H$, $m\in\mathcal{M}$ and $\alpha\in\mathcal{M}^*$,
respectively. The forgetful functor
$$
F\colon{}_H\mathcal{M}_f\rightarrow {}_\Bbbk\mathcal{M}_f
$$
is monoidal.
\end{proposition}
\begin{proof}
This follows from Lemma~\ref{lemma05}, using the same evaluation and
projection maps. It only remains to prove that they are morphisms
in the right category, i.e. that they are left $H$-module homomorphisms.
Consider a finitely generated projective left $H$-module $\mathcal{M}$ and
let $\xi\in H$, $m\in\mathcal{M}$ and $\alpha\in\mathcal{M}^*$. Then
\begin{align*}
    \mathrm{ev}_\mathcal{M}(\xi\cdot(\alpha\otimes m))
    =&\mathrm{ev}_\mathcal{M}((\xi_{(1)}\cdot\alpha)\otimes(\xi_{(2)}\cdot m))\\
    =&(\xi_{(1)}\cdot\alpha)(\xi_{(2)}\cdot m)\\
    =&\alpha((S(\xi_{(1)})\xi_{(2)})\cdot m)\\
    =&\epsilon(\xi)\cdot\alpha(m)\\
    =&\xi\cdot\mathrm{ev}_\mathcal{M}(\alpha\otimes m)
\end{align*}
for a dual set of generators $\{\alpha^i\}_{i\in I}$ and $\{m_i\}_{i\in I}$.
Similarly one proves that $\pi_\mathcal{M}$, $\mathrm{ev}_{\mathcal{M}^*}$
and
$\pi_{\mathcal{M}^*}$ are $H$-linear.
\end{proof}
The converse also holds true, giving a Tannaka-Krein duality for Hopf algebras
similar to Proposition~\ref{prop08}.

\begin{theorem}[Tannaka-Krein Reconstruction of Hopf Algebras]
For a commutative ring $\Bbbk$ consider an essentially small, $\Bbbk$-linear,
rigid, Abelian, monoidal category $\mathcal{C}$ together with a
$\Bbbk$-linear, exact, faithful, monoidal functor 
$F\colon\mathcal{C}\rightarrow {}_\Bbbk\mathcal{M}_f$. Then there is a
$\Bbbk$-Hopf algebra $H$ and an equivalence 
$G\colon\mathcal{C}\rightarrow {}_H\mathcal{M}_f$ such that
$F=G'\circ G$, where $G'\colon{}_H\mathcal{M}_f\rightarrow 
{}_\Bbbk\mathcal{M}_f$ is the forgetful functor.
\end{theorem}
For a proof we refer to \cite{ChPr94}~Thm.~5.1.11. This reveals the
true analogue of the Hopf algebra structure on the monoidal category of
its representations and gives us a deeper understanding of the notion
of Hopf algebra.

\section{Quasi-Triangular Structures}\label{Sec2.3}

Following the spirit of the last section, i.e.
characterizing algebraic structures
by properties of the corresponding representation theory, we introduce
quasi-triangular bialgebras as those bialgebras whose
corresponding monoidal category is braided. We prove that this leads to
universal $\mathcal{R}$-matrices, which satisfy the hexagon relations
and control the noncocommutativity of the coproduct. The corresponding
references are \cite{ChPr94}~Sec.~5.2, \cite{Ka95}~Sec.~XIII~1.3
and \cite{Ma95}~Sec.~9.2. 

\begin{definition}[Quasi-Triangular Bialgebra]\label{def02}
A bialgebra $\mathcal{A}$ is said to be \textit{quasi-triangular} if
its monoidal category ${}_\mathcal{A}\mathcal{M}$ of left $\mathcal{A}$-modules is
braided. We call $\mathcal{A}$ \textit{triangular} if
${}_\mathcal{A}\mathcal{M}$ is braided symmetric.
\end{definition}
We expect to find additional algebraic structure
underlying a quasi-triangular bialgebra. In fact, the following
proposition, taken from \cite{Ka95}~Prop.~XIII.1.4,
recovers the original Definition from \cite{Dr86}.
\begin{proposition}\label{prop06}
A bialgebra $\mathcal{A}$ is quasi-triangular if and only if 
there is an invertible element
$\mathcal{R}\in\mathcal{A}\otimes\mathcal{A}$ such that
\begin{equation}\label{eq08}
    \Delta_{21}(a)=\mathcal{R}\Delta(a)\mathcal{R}^{-1}
\end{equation}
holds for all $a\in\mathcal{A}$, and such that the equations
\begin{align}\label{eq09}
    (\Delta\otimes\mathrm{id})(\mathcal{R})
    =\mathcal{R}_{13}\mathcal{R}_{23}
    \text{ and }
    (\mathrm{id}\otimes\Delta)(\mathcal{R})
    =\mathcal{R}_{13}\mathcal{R}_{12}
\end{align}
are satisfied. The bialgebra $\mathcal{A}$ is triangular
if and only if the element $\mathcal{R}$ satisfies
\begin{equation}\label{eq24}
    \mathcal{R}^{-1}=\mathcal{R}_{21}
\end{equation}
in addition.
\end{proposition}
If the conditions (\ref{eq08}) and (\ref{eq09}) are satisfied, the
element $\mathcal{R}$ is said to be a \textit{universal
$\mathcal{R}$-matrix} or \textit{quasi-triangular structure}
for $\mathcal{A}$. If (\ref{eq24}) holds in addition, $\mathcal{R}$
is called \textit{triangular structure}. $\mathcal{A}$ is said
to be \textit{quasi-cocommutative} with respect to $\mathcal{R}$ if
(\ref{eq08}) holds. The equations (\ref{eq09}) are the so-called
\textit{hexagon relations}.
\begin{proof}
Let $\beta$ be a braiding on the monoidal category
${}_\mathcal{A}\mathcal{M}$ and define 
$$
\mathcal{R}
=\tau_{\mathcal{A},\mathcal{A}}(
\beta_{\mathcal{A},\mathcal{A}}(1\otimes 1)).
$$
Clearly $\mathcal{R}\in\mathcal{A}\otimes\mathcal{A}$ is invertible.
Before showing that it satisfies equations (\ref{eq08}) and (\ref{eq09})
we prove that
$
\beta_{\mathcal{M},\mathcal{M}'}(m\otimes m')
=\tau_{\mathcal{M},\mathcal{M}'}(\mathcal{R}\cdot(m\otimes m'))
$
for all $m\in\mathcal{M}$, $m'\in\mathcal{M}$ and left $\mathcal{A}$-modules
$\mathcal{M}$ and $\mathcal{M}'$. Since $\beta$ is a natural
isomorphism of the functors $\otimes$ and $\otimes\circ\tau$ we obtain
\begin{align*}
    \beta_{\mathcal{M},\mathcal{M}'}((\phi_m\otimes\phi_{m'})(a\otimes b))
    =(\phi_{m'}\otimes\phi_m)
    (\beta_{\mathcal{A}\otimes\mathcal{A}}(a\otimes b)),
\end{align*}
where $\phi_m(a)=a\cdot m$ for all $a\in\mathcal{A}$ and $m\in\mathcal{M}$
is a left $\mathcal{A}$-module homomorphism. Then
\begin{align*}
    \beta_{\mathcal{M},\mathcal{M}'}(m\otimes m')
    =&\beta_{\mathcal{M},\mathcal{M}'}((\phi_m\otimes\phi_{m'})(1\otimes 1))\\
    =&(\phi_{m'}\otimes\phi_m)
    (\beta_{\mathcal{A}\otimes\mathcal{A}}(1\otimes 1))\\
    =&(\phi_{m'}\otimes\phi_m)(\tau_{\mathcal{A},\mathcal{A}}(
    \mathcal{R}))\\
    =&\tau_{\mathcal{M},\mathcal{M}'}(\mathcal{R}\cdot(m\otimes m'))
\end{align*}
follows, since $\tau$ is a natural transformation. Using this we prove
\begin{align*}
    \Delta(a)\tau_{\mathcal{A},\mathcal{A}}(\mathcal{R})
    =&a\cdot(\beta_{\mathcal{A},\mathcal{A}}(1\otimes 1))\\
    =&\beta_{\mathcal{A},\mathcal{A}}(a\cdot(1\otimes 1))\\
    =&\tau_{\mathcal{A},\mathcal{A}}(\mathcal{R}\cdot\Delta(a)),
\end{align*}
since $\beta$ is a left $\mathcal{A}$-module homomorphism on objects.
Applying the flip isomorphism on both sides of the above equation leads to
equation (\ref{eq08}). By making use of the hexagon relations of
$\beta$ we obtain
\begin{align*}
    \mathcal{R}_{3,12}
    =&\alpha_{\mathcal{A},\mathcal{A},\mathcal{A}}(
    \beta_{\mathcal{A},\mathcal{A}\otimes\mathcal{A}}(
    \alpha_{\mathcal{A},\mathcal{A},\mathcal{A}}((1\otimes 1)\otimes 1)))\\
    =&(\mathrm{id}_\mathcal{A}\otimes\beta_{\mathcal{A},\mathcal{A}})(
    \alpha_{\mathcal{A},\mathcal{A},\mathcal{A}}(
    (\beta_{\mathcal{A},\mathcal{A}}\otimes\mathrm{id}_\mathcal{A})
    ((1\otimes 1)\otimes 1)))\\
    =&\mathcal{R}_{32}\mathcal{R}_{31}
\end{align*}
and
\begin{align*}
    \mathcal{R}_{23,1}
    =&\alpha^{-1}_{\mathcal{A},\mathcal{A},\mathcal{A}}(
    \beta_{\mathcal{A}\otimes\mathcal{A},\mathcal{A}}(
    \alpha^{-1}_{\mathcal{A},\mathcal{A},\mathcal{A}}(
    1\otimes(1\otimes 1))))\\
    =&(\beta_{\mathcal{A},\mathcal{A}}\otimes\mathrm{id}_\mathcal{A})(
    \alpha^{-1}_{\mathcal{A},\mathcal{A},\mathcal{A}}(
    (\mathrm{id}_\mathcal{A}\otimes\beta_{\mathcal{A},\mathcal{A}})(
    1\otimes(1\otimes 1))))\\
    =&\mathcal{R}_{21}\mathcal{R}_{31},
\end{align*}
concluding that $\mathcal{R}$ satisfies the equations (\ref{eq09}):
for the first we perform a tensor shift $(123)\rightarrow(231)$ to obtain
$\mathcal{R}_{1,23}=\mathcal{R}_{13}\mathcal{R}_{12}$ and a shift
$(123)\rightarrow(312)$ for the second equation 
$\mathcal{R}_{12,3}=\mathcal{R}_{13}\mathcal{R}_{23}$.
Let on the other hand $\mathcal{A}$ be a bialgebra and assume
the existence of an invertible element
$\mathcal{R}\in\mathcal{A}\otimes\mathcal{A}$ satisfying equations
(\ref{eq08}) and (\ref{eq09}). For two arbitrary left $\mathcal{A}$-modules
$\mathcal{M}$ and $\mathcal{M}'$ we define an isomorphism
$\beta_{\mathcal{M},\mathcal{M}'}\colon\mathcal{M}\otimes\mathcal{M}'
\rightarrow\mathcal{M}'\otimes\mathcal{M}$
of left $\mathcal{A}$-modules by
$$
\beta_{\mathcal{M},\mathcal{M}'}(m\otimes m')
=\tau_{\mathcal{M},\mathcal{M}'}(\mathcal{R}\cdot(m\otimes m'))
$$
for all $m\in\mathcal{M}$ and $m'\in\mathcal{M}'$. In fact, for all
$a\in\mathcal{A}$ one obtains
\begin{align*}
    \beta_{\mathcal{M},\mathcal{M}'}(a\cdot(m\otimes m'))
    =&\tau_{\mathcal{M},\mathcal{M}'}
    ((\mathcal{R}\Delta(a))\cdot(m\otimes m'))\\
    =&\tau_{\mathcal{M},\mathcal{M}'}
    ((\Delta_{21}(a)\mathcal{R})\cdot(m\otimes m'))\\
    =&\Delta(a)\cdot\tau_{\mathcal{M},\mathcal{M}'}(
    \mathcal{R}\cdot(m\otimes m'))\\
    =&a\cdot\beta_{\mathcal{M},\mathcal{M}'}(m\otimes m')
\end{align*}
by property (\ref{eq08}) and the inverse of 
$\beta_{\mathcal{M},\mathcal{M}'}$ is given by
$$
\mathcal{M}'\otimes\mathcal{M}\ni(m\otimes m')\mapsto
\mathcal{R}^{-1}\cdot(\tau_{\mathcal{M}',\mathcal{M}}(m\otimes m'))
\in\mathcal{M}\otimes\mathcal{M}'.
$$
This implies that $\beta\colon\otimes\rightarrow\otimes\circ\tau$
is a natural transformation. It remains to prove that hexagon
relations for $\beta$ and it is not surprising that they are
following by the hexagon relations of $\mathcal{R}$. In detail we obtain
\begin{align*}
    \alpha_{\mathcal{M}',\mathcal{M}'',\mathcal{M}}&(
    \beta_{\mathcal{M},\mathcal{M}'\otimes\mathcal{M}''}(
    \alpha_{\mathcal{M},\mathcal{M}',\mathcal{M}''}(
    (m\otimes m')\otimes m'')))\\
    =&\alpha_{\mathcal{M}',\mathcal{M}'',\mathcal{M}}(
    \tau_{\mathcal{M},\mathcal{M}'\otimes\mathcal{M}''}(
    (\mathcal{R}_1\cdot m)\otimes((\mathcal{R}_{2(1)}\cdot m')
    \otimes(\mathcal{R}_{2(2)}\cdot m''))))\\
    =&(\mathcal{R}_{2(1)}\cdot m')\otimes(
    (\mathcal{R}_{2(2)}\cdot m'')\otimes(\mathcal{R}_1\cdot m))\\
    =&(\mathcal{R}_2\cdot m')
    \otimes((\mathcal{R}'_2\cdot m'')
    \otimes((\mathcal{R}'_1\mathcal{R}_1)\cdot m))\\
    =&(\mathrm{id}_{\mathcal{M}'}\otimes\beta_{\mathcal{M},\mathcal{M}''})
    ((\mathcal{R}_2\cdot m')\otimes((\mathcal{R}_1\cdot m)\otimes m''))\\
    =&(\mathrm{id}_{\mathcal{M}'}\otimes\beta_{\mathcal{M},\mathcal{M}''})(
    \alpha_{\mathcal{M}',\mathcal{M},\mathcal{M}''}(
    (\beta_{\mathcal{M},\mathcal{M}'}\otimes\mathrm{id}_{\mathcal{M}''})
    ((m\otimes m')\otimes m'')))
\end{align*}
and
\begin{align*}
    \alpha^{-1}_{\mathcal{M}'',\mathcal{M},\mathcal{M}'}&(
    \beta_{\mathcal{M}\otimes\mathcal{M}',\mathcal{M}''}(
    \alpha^{-1}_{\mathcal{M},\mathcal{M}',\mathcal{M}''}(
    m\otimes(m'\otimes m''))))\\
    =&\alpha^{-1}_{\mathcal{M}'',\mathcal{M},\mathcal{M}'}(
    (\mathcal{R}_2\cdot m'')\otimes((\mathcal{R}_{1(1)}\cdot m)
    \otimes(\mathcal{R}_{1(2)}\cdot m')))\\
    =&((\mathcal{R}_2\cdot m'')\otimes(\mathcal{R}_{1(1)}\cdot m))
    \otimes(\mathcal{R}_{1(2)}\cdot m')\\
    =&(((\mathcal{R}'_2\mathcal{R}_2)\cdot m'')
    \otimes(\mathcal{R}'_1\cdot m))
    \otimes(\mathcal{R}_1\cdot m')\\
    =&(\beta_{\mathcal{M},\mathcal{M}''}\otimes\mathrm{id}_{\mathcal{M}'})
    ((m\otimes(\mathcal{R}_2\cdot m''))\otimes(\mathcal{R}_1\cdot m'))\\
    =&(\beta_{\mathcal{M},\mathcal{M}''}\otimes\mathrm{id}_{\mathcal{M}'})
    \alpha^{-1}_{\mathcal{M},\mathcal{M}'',\mathcal{M}'}(
    \mathrm{id}_\mathcal{M}\otimes\beta_{\mathcal{M}',\mathcal{M}''})(
    m\otimes(m'\otimes m''))
\end{align*}
for another left $\mathcal{A}$-module $\mathcal{M}''$, where we used
equations (\ref{eq09}) in leg notation, i.e.
$$
\mathcal{R}_1\otimes\mathcal{R}_{2(1)}\otimes\mathcal{R}_{2(2)}
=(\mathcal{R}'_1\mathcal{R}_1)\otimes\mathcal{R}_2\otimes\mathcal{R}'_2
$$
and
$$
\mathcal{R}_{1(1)}\otimes\mathcal{R}_{1(2)}\otimes\mathcal{R}_2
=\mathcal{R}'_1\otimes\mathcal{R}_1\otimes(\mathcal{R}'_2\mathcal{R}_2).
$$
This concludes the characterization of quasi-triangular bialgebras.
\end{proof}
Every universal $\mathcal{R}$-matrix satisfies an additional
equation, connecting it to the theory of integrable systems
(see \cite{Sem1995} for more information).
\begin{corollary}[QYBE]\label{cor01}
A universal $\mathcal{R}$-matrix $\mathcal{R}$
on a quasi-triangular bialgebra satisfies the
\textit{quantum Yang-Baxter equation}
$$
\mathcal{R}_{12}\mathcal{R}_{13}\mathcal{R}_{23}
=\mathcal{R}_{23}\mathcal{R}_{13}\mathcal{R}_{12}.
$$
\end{corollary}
\begin{proof}
Using (\ref{eq08}) and (\ref{eq09}) we obtain
\begin{align*}
    \mathcal{R}_{12}\mathcal{R}_{13}\mathcal{R}_{23}
    =\mathcal{R}_{12}\mathcal{R}_{12,3}
    =\mathcal{R}_{21,3}\mathcal{R}_{12}
    =\mathcal{R}_{23}\mathcal{R}_{13}\mathcal{R}_{12}.
\end{align*}
\end{proof}
\begin{definition}[Quasi-Triangular Hopf Algebra]
A Hopf algebra is said to be (quasi-)triangular if its underlying bialgebra is.
\end{definition}
We conclude this section with a specification of
Proposition~\ref{prop09} in the setting of quasi-triangular Hopf algebras
(compare also to \cite{ChPr94}~Ex.~5.1.4).
\begin{corollary}
Let $H$ be a quasi-triangular Hopf algebra. Then
${}_H\mathcal{M}_f$ is a rigid braided monoidal category.
\end{corollary}
\begin{proof}
From Proposition~\ref{prop09} it follows that ${}_H\mathcal{M}_f$ is a
rigid monoidal category and by definition ${}_H\mathcal{M}$ is braided.
Since ${}_H\mathcal{M}_f$ is a monoidal subcategory of ${}_H\mathcal{M}$
the latter also applies to ${}_H\mathcal{M}_f$.
\end{proof}

\section{The Drinfel'd Functor}\label{Sec2.4}

In this section we discuss gauge
transformations of quasi-triangular bialgebras. The corresponding
algebraic tool is given by a normalized $2$-cocycle, called Drinfel'd twist.
Following \cite{Ma95}~Sec.~2.3, we elaborate how to twist deform
the coproduct such that the deformed structure still corresponds
to a quasi-triangular bialgebra. Similarly, one twist deforms bialgebra
modules, which leads to the definition of the Drinfel'd functor. We prove that
this functor is braided monoidal and gives rise to a braided monoidal equivalence
of the representation theory of the deformed and undeformed bialgebra.
A twist deformation of antipodes is incorporated in the next section.
We also refer to \cite{Ka95}~Sec.~XV.3.
In the following, $\mathcal{B}$ denotes a bialgebra over a commutative ring $\Bbbk$
with coproduct $\Delta$ and counit $\epsilon$.
\begin{definition}[Drinfel'd Twist]
An invertible element $\mathcal{F}\in\mathcal{B}\otimes\mathcal{B}$ is said to
be a Drinfel'd twist, or twist for short, if it is
normalized, i.e. $(\epsilon\otimes\mathrm{id})(\mathcal{F})=1
=(\mathrm{id}\otimes\epsilon)(\mathcal{F})$
and satisfies the $2$-cocycle condition
$$
(\mathcal{F}\otimes 1)(\Delta\otimes\mathrm{id})(\mathcal{F})
=(1\otimes\mathcal{F})(\mathrm{id}\otimes\Delta)(\mathcal{F}).
$$
\end{definition}
The original definition goes back to Drinfel'd \cite{Dr83}, where
the author introduces twists as quantizations of solutions of the classical
Yang-Baxter equation. Those twists where considered as elements on 
(formal power series of) universal enveloping algebras of Lie
algebras (see also Section~\ref{Sec-r-matrix}).
The generalization to arbitrary bialgebras was undertaken by
Giaquinto and Zhang in \cite{GiZh98}. Drinfel'd twists on bialgebroids are
considered in e.g. \cite{Borowiec2017,Xu1998}.
We start by giving some examples of Drinfel'd twists 
to convince the reader of the richness of this concept.
\begin{example}\label{ex02}
\begin{enumerate}
\item[i.)] On every bialgebra $\mathcal{B}$ there is a twist given by the
unit element $1\otimes 1$. We refer to it as the \textit{trivial twist} in the
following;

\item[ii.)][c.f. \cite{GiZh98}~Thm.~2.1] Let $\Bbbk$ be a commutative ring
such that $\mathbb{Q}\subseteq\Bbbk$ and consider a commutative bialgebra
$\mathcal{B}$. Denote the primitive elements of $\mathcal{B}$ by $P$.
Then
$$
\exp(\hbar r)
=\sum_{n=0}^\infty\frac{\hbar^n}{n!}r^n
\in(\mathcal{B}\otimes\mathcal{B})[[\hbar]]
$$
is a twist on $\mathcal{B}[[\hbar]]$ for any $r\in P\otimes P$.
Here $\mathcal{B}[[\hbar]]$ is a topologically free module
and we consider the completed tensor product (c.f. \cite{ES2010}~Sec.~1.1);
\begin{proof}
Let $r=r_1\otimes r_2\in P\otimes P$. In particular 
$(\epsilon\otimes\mathrm{id})(r)=0=(\mathrm{id}\otimes\epsilon)(r)$
and since $\epsilon$ is an algebra homomorphism this implies that
$\exp(\hbar r)$ is normalized. Then 
$$
(\Delta\otimes\mathrm{id})\exp(\hbar r)
=\sum_{n=0}^\infty\frac{\hbar^n}{n!}\Delta(r_1^n)\otimes r_2^n
=\sum_{n=0}^\infty\frac{\hbar^n}{n!}\Delta(r_1)^n\otimes r_2^n
=\exp((\Delta\otimes\mathrm{id})(\hbar r))
$$
since $\Delta$ is an algebra homomorphism and similarly
$(\mathrm{id}\otimes\Delta)\exp(\hbar r)
=\exp((\mathrm{id}\otimes\Delta)(\hbar r))$ follows. Since
$\mathcal{B}$ is commutative, the Baker-Campbell-Hausdorff series of
$\mathbb{B}[[\hbar]]$ is trivial and
\begin{align*}
    (\Delta\otimes\mathrm{id})(\exp(\hbar r))
    (\exp(\hbar r)\otimes 1)
    =&\exp((\Delta\otimes\mathrm{id})(\hbar r))
    \exp(\hbar r\otimes 1)\\
    =&\exp((\Delta\otimes\mathrm{id})(\hbar r)+\hbar r\otimes 1)
\end{align*}
as well as $(\mathrm{id}\otimes\Delta)(\exp(\hbar r))
(1\otimes\exp(\hbar r))
=\exp((\mathrm{id}\otimes\Delta)(\hbar r)+1\otimes\hbar r)$ follow.
Now $r_1,r_2\in P$, which implies that
$(\Delta\otimes\mathrm{id})(\hbar r)+\hbar r\otimes 1
=(\mathrm{id}\otimes\Delta)(\hbar r)+1\otimes\hbar r$. This proves that
$\exp(\hbar r)$ is a twist on $\mathcal{B}[[\hbar]]$.
\end{proof}

\item[iii.)] A variation of ii.) is described in \cite{Reshetikhin1990}:
let $\Bbbk$ be a commutative ring such that $\mathbb{Q}\subseteq\Bbbk$
and consider a Lie algebra $\mathfrak{g}$ over $\Bbbk$ together with a set
of elements
$x_1,\ldots,x_n,y_1,\ldots,y_n\in\mathfrak{g}$, where $n\in\mathbb{N}$, such
that $[x_i,x_j]=[x_i,y_j]=[y_i,y_j]=0$ for all $1\leq i\leq n$. Then
$\exp(\hbar r)\in(\mathscr{U}\mathfrak{g}\otimes\mathscr{U}\mathfrak{g})[[\hbar]]$ 
is a twist on $\mathscr{U}\mathfrak{g}[[\hbar]]$, where $r
=\sum_{i=1}^nx_i\otimes y_i$. In this case we refer to $\exp(\hbar r)$
as an \textit{Abelian twist};

\item[iv.)][c.f. \cite{GiZh98}~Thm.~2.10]
Let $\Bbbk$ be a commutative ring such that $\mathbb{Q}\subseteq\Bbbk$
and consider the Lie algebra $\mathfrak{g}$ over $\Bbbk$ which is generated by
two elements $H,E\in\mathfrak{g}$ such that $[H,E]=2E$. Then
$$
\sum_{n=0}^\infty\sum_{m=0}^n\frac{\hbar^n}{n!}
(-1)^m\binom{n}{m}E^{n-m}H^{\langle m\rangle}
\otimes E^mH^{\langle n-m\rangle}
\in(\mathscr{U}\mathfrak{g}\otimes\mathscr{U}\mathfrak{g})[[\hbar]]
$$
is a twist on $\mathscr{U}\mathfrak{g}[[\hbar]]$, where
$H^{\langle m\rangle}$ is defined inductively by $H^{\langle 0\rangle}=1$
and $H^{\langle m+1\rangle}=H(H+1)\cdots(H+m)$;

\item[v.)][c.f. \cite{Ewen1992,Ohn1992}]
In the same setting as in iv.) there is a twist on
$\mathscr{U}\mathfrak{g}[[\hbar]]$ given by
$$
\exp\bigg(\frac{1}{2}H\otimes\mathrm{log}(1+\hbar E)\bigg)
\in(\mathscr{U}\mathfrak{g}\otimes\mathscr{U}\mathfrak{g})[[\hbar]],
$$
which is known as \textit{Jordanian twist}. We further refer to 
\cite{Aschieri2006}~Sec.~2.3.3 and references therein for a generalization 
of Jordanian twist and to \cite{Borowiec2019,Pachol2017} for more recent
discussions on Jordanian twists;
\end{enumerate}
\end{example}
Besides the trivial twist, all Drinfel'd twists we discussed above are
formal power series with entries in a bialgebra or Hopf algebra.
Further remark that all of them are the identity in zero order of the
formal parameter. In fact, many examples of Drinfel'd twists occur in
the context of deformation quantization, which naturally utilizes
formal power series. Furthermore
in the primordial article \cite{Dr83} Drinfel'd introduced twists
$\mathcal{F}
=\sum_{n=0}^\infty\hbar^nF_n\in\mathscr{U}\mathfrak{g}^{\otimes 2}[[\hbar]]$
on formal power series of universal enveloping algebras of a real or complex
Lie algebra $\mathfrak{g}$, such that $F_0=1\otimes 1$. The first order term
$F_1$ satisfies the so-called \textit{classical Yang-Baxter equation}
and Drinfel'd proved in the mentioned article that conversely any
solution of the classical Yang-Baxter equation on $\mathfrak{g}$ can be
realized as the first order term of a twist on
$\mathscr{U}\mathfrak{g}[[\hbar]]$ starting with the identity.
We revive this thought in Section~\ref{Sec-r-matrix}, discussing
the correspondence of (formal) Drinfel'd twists and their first order in
detail. However for the rest of this section we
come back to discuss general properties of Drinfel'd twists and how
they deform the underlying quasi-triangular bialgebra structure.
\begin{lemma}[Inverse Twist]
Let $\mathcal{F}$ be a twist on $\mathcal{B}$. Its inverse $\mathcal{F}^{-1}$
is normalized, i.e. $(\epsilon\otimes\mathrm{id})(\mathcal{F}^{-1})=1
=(\mathrm{id}\otimes\epsilon)(\mathcal{F}^{-1})$ and satisfies the
so called inverse $2$-cocycle condition
\begin{equation}
    (\Delta\otimes\mathrm{id})(\mathcal{F}^{-1})(\mathcal{F}^{-1}\otimes 1)
    =(\mathrm{id}\otimes\Delta)(\mathcal{F}^{-1})(1\otimes\mathcal{F}^{-1}).
\end{equation}
\end{lemma}
As already indicated, twist
are deformation tools, leading to compatibility on a categorical level.
We start by explaining how the twist deforms the underlying bialgebra
$\mathcal{B}$ before
passing to module algebras and more general to equivariant bialgebra
module algebra modules in the next section. We define the
\textit{twisted coproduct}
\begin{equation}
    \Delta_\mathcal{F}(\xi)=\mathcal{F}\Delta(\xi)\mathcal{F}^{-1}
\end{equation}
for all $\xi\in\mathcal{B}$.
\begin{proposition}\label{prop01}
Let $\mathcal{F}$ be a twist on $\mathcal{B}$. Then
$\mathcal{B}_\mathcal{F}
=(\mathcal{B},\mu,\eta,\Delta_\mathcal{F},\epsilon)$ is a bialgebra.
If $\mathcal{B}$ is quasi-triangular with universal $\mathcal{R}$-matrix
$\mathcal{R}$, so is $\mathcal{B}_\mathcal{F}$ with quasi-triangular structure
\begin{equation}
    \mathcal{R}_\mathcal{F}=\mathcal{F}_{21}\mathcal{R}\mathcal{F}^{-1}.
\end{equation}
If $\mathcal{R}$ is triangular, so is $\mathcal{R}_\mathcal{F}$.
\end{proposition}
\begin{proof}
We split the proof into four parts.
\begin{enumerate}
\item[i.)] \textbf{$(\Delta_\mathcal{F},\epsilon)$ is a coalgebra structure on
$\mathcal{B}$: }let $\xi\in\mathcal{B}$. Then
\begin{align*}
    (\epsilon\otimes\mathrm{id})\Delta_\mathcal{F}(\xi)
    =&\epsilon(\mathcal{F}_1)\epsilon(\xi_{(1)})\epsilon(\mathcal{F}_1^{'-1})
    \mathcal{F}_2\xi_{(2)}\mathcal{F}_2^{'-1}\\
    =&\epsilon(\mathcal{F}_1)\mathcal{F}_2
    \epsilon(\xi_{(1)})\xi_{(2)}
    \epsilon(\mathcal{F}_1^{'-1})\mathcal{F}_2^{'-1}\\
    =&\xi
\end{align*}
and similarly one proves
$(\mathrm{id}\otimes\epsilon)\Delta_\mathcal{F}(\xi)=\xi$.
\item[ii.)] \textbf{$\Delta_\mathcal{F}$ is an algebra homomorphism: }let
$\xi,\chi\in\mathcal{B}$. Then
\begin{align*}
    \Delta_\mathcal{F}(\xi\chi)
    =\mathcal{F}\Delta(\xi\chi)\mathcal{F}^{-1}
    =\mathcal{F}\Delta(\xi)\mathcal{F}^{-1}\mathcal{F}\Delta(\chi)\mathcal{F}^{-1}
    =\Delta_\mathcal{F}(\xi)\Delta_\mathcal{F}(\chi)
\end{align*}
and $\Delta_\mathcal{F}(1)=\mathcal{F}(1\otimes 1)\mathcal{F}^{-1}=1\otimes 1$,
i.e. $\Delta_\mathcal{F}$ respects the algebra structure of $\mathcal{B}$ and
$\mathcal{B}\otimes\mathcal{B}$.
\item[iii.)] \textbf{$\mathcal{R}_\mathcal{F}$ is a quasi-triangular structure: }
let $\mathcal{R}$ be a universal $\mathcal{R}$-matrix on $\mathcal{B}$ and
$\xi\in\mathcal{B}$. Then
\begin{align*}
    \Delta_\mathcal{F}(\xi)_{21}
    =\mathcal{F}_{21}\Delta_{21}(\xi)\mathcal{F}_{21}^{-1}
    =\mathcal{F}_{21}\mathcal{R}\Delta(\xi)\mathcal{R}^{-1}
    \mathcal{F}_{21}^{-1}
    =\mathcal{R}_\mathcal{F}\Delta_\mathcal{F}(\xi)\mathcal{R}_\mathcal{F}^{-1}
\end{align*}
proves that $\Delta_\mathcal{F}$ is quasi-cocommutative with respect to
$\mathcal{R}_\mathcal{F}$. Moreover
\begin{align*}
    (\Delta_\mathcal{F}\otimes\mathrm{id})(\mathcal{R}_\mathcal{F})
    =&\mathcal{F}_{12}(\Delta\otimes\mathrm{id})(\mathcal{F}_{21})
    (\Delta\otimes\mathrm{id})(\mathcal{R})
    (\Delta\otimes\mathrm{id})(\mathcal{F}^{-1})
    \mathcal{F}_{12}^{-1}\\
    =&\mathcal{F}_{12}\mathcal{F}_{3,12}
    \mathcal{R}_{13}\mathcal{R}_{23}
    \mathcal{F}_{12,3}^{-1}\mathcal{F}_{12}^{-1}\\
    =&\mathcal{F}_{31}\mathcal{F}_{31,2}
    \mathcal{R}_{13}\mathcal{R}_{23}
    \mathcal{F}_{1,23}^{-1}\mathcal{F}_{23}^{-1}\\
    =&\mathcal{F}_{31}
    \mathcal{R}_{13}\mathcal{F}_{13,2}\mathcal{F}_{1,32}^{-1}\mathcal{R}_{23}
    \mathcal{F}_{23}^{-1}\\
    =&\mathcal{R}_{\mathcal{F}13}\mathcal{F}_{13}
    \mathcal{F}_{13,2}\mathcal{F}_{1,32}^{-1}
    \mathcal{F}_{32}^{-1}\mathcal{R}_{\mathcal{F}23}\\
    =&\mathcal{R}_{\mathcal{F}13}\mathcal{R}_{\mathcal{F}23}
\end{align*}
and
\begin{align*}
    (\mathrm{id}\otimes\Delta_\mathcal{F})(\mathcal{R}_\mathcal{F})
    =&\mathcal{F}_{23}(\mathrm{id}\otimes\Delta)(\mathcal{F}_{21})
    (\mathrm{id}\otimes\Delta)(\mathcal{R})
    (\mathrm{id}\otimes\Delta)(\mathcal{F}^{-1})
    \mathcal{F}_{23}^{-1}\\
    =&\mathcal{F}_{23}\mathcal{F}_{23,1}
    \mathcal{R}_{13}\mathcal{R}_{12}
    \mathcal{F}_{1,23}^{-1}
    \mathcal{F}_{23}^{-1}\\
    =&\mathcal{R}_{\mathcal{F}13}\mathcal{F}_{13}\mathcal{F}_{2,13}
    \mathcal{F}_{21,3}^{-1}\mathcal{F}_{21}^{-1}\mathcal{R}_{\mathcal{F}12}\\
    =&\mathcal{R}_{\mathcal{F}13}\mathcal{R}_{\mathcal{F}12}
\end{align*}
are the hexagon relations of $\mathcal{R}_\mathcal{F}$ with respect to
$\mathcal{B}_\mathcal{F}$.

\item[iv.)] \textbf{If $\mathcal{R}$ is triangular so is 
$\mathcal{R}_\mathcal{F}$: }assume that $\mathcal{R}$ is triangular. Then
\begin{align*}
    \mathcal{R}_{\mathcal{F}21}
    =\mathcal{F}\mathcal{R}_{21}\mathcal{F}_{21}^{-1}
    =\mathcal{F}\mathcal{R}^{-1}\mathcal{F}_{21}^{-1}
    =\mathcal{R}_{\mathcal{F}}^{-1},
\end{align*}
i.e. $\mathcal{R}_\mathcal{F}$ is triangular, too.
\end{enumerate}
This is all we need to conclude the proposition.
\end{proof}
The bialgebra constructed in Proposition~\ref{prop01} is called
\textit{twisted bialgebra} and we denote it by $\mathcal{B}_\mathcal{F}$. In the
following we use Sweedler's notation $\xi_{\widehat{(1)}}\otimes\xi_{\widehat{(2)}}
=\Delta_\mathcal{F}(\xi)$ to denote the twisted coproduct of $\xi\in\mathcal{B}$.
One might ask if the process of twisting is reversible
and what happens if one deforms repetitively. Both issues are
discussed in the following lemma.
\begin{lemma}\label{lemma11}
If $\mathcal{F}$ is a Drinfel'd twist on $\mathcal{B}$ and $\mathcal{F}'$ a
Drinfel'd twist on $\mathcal{B}_\mathcal{F}$
then $\mathcal{F}'\mathcal{F}$ is a Drinfel'd twist on $\mathcal{B}$
such that 
$$
\mathcal{B}_{\mathcal{F}'\mathcal{F}}
=(\mathcal{B}_{\mathcal{F}})_{\mathcal{F}'}.
$$
Furthermore,
$\mathcal{F}^{-1}$ is a Drinfel'd twist on $\mathcal{B}_\mathcal{F}$ and
$$
(\mathcal{B}_\mathcal{F})_{\mathcal{F}^{-1}}
=\mathcal{B}
$$
is an equality of bialgebras.
\end{lemma}
\begin{proof}
We prove that $\mathcal{F}'\mathcal{F}$ is a Drinfel'd twist on
$\mathcal{B}$. Since $\epsilon$ is an algebra homomorphism we have
$(\epsilon\otimes\mathrm{id})(\mathcal{F}'\mathcal{F})=1
=(\mathrm{id}\otimes\epsilon)(\mathcal{F}'\mathcal{F})$, which means that
$\mathcal{F}'\mathcal{F}$ is normalized. Furthermore,
\begin{align*}
    (\mathcal{F}'\mathcal{F}\otimes 1)
    (\Delta\otimes\mathrm{id})(\mathcal{F}'\mathcal{F})
    =&(\mathcal{F}'\mathcal{F}\otimes 1)
    (\Delta\otimes\mathrm{id})(\mathcal{F}')
    (\mathcal{F}^{-1}\otimes 1)(\mathcal{F}\otimes 1)
    (\Delta\otimes\mathrm{id})(\mathcal{F})\\
    =&(1\otimes\mathcal{F}'\mathcal{F})
    (\mathrm{id}\otimes\Delta)(\mathcal{F}')
    (1\otimes\mathcal{F}^{-1})(\mathcal{F}\otimes 1)
    (\Delta\otimes\mathrm{id})(\mathcal{F})\\
    =&(1\otimes\mathcal{F}'\mathcal{F})
    (\mathrm{id}\otimes\Delta)(\mathcal{F}')
    (\mathrm{id}\otimes\Delta)(\mathcal{F})\\
    =&(1\otimes\mathcal{F}'\mathcal{F})
    (\mathrm{id}\otimes\Delta)(\mathcal{F}'\mathcal{F})
\end{align*}
proves that $\mathcal{F}'\mathcal{F}$ satisfies the $2$-cocycle condition
with respect to $\mathcal{B}$. Moreover, $\mathcal{F}^{-1}$ is
a Drinfel'd twist on $\mathcal{B}_\mathcal{F}$, since $\mathcal{F}^{-1}$
is normalized and
\begin{align*}
    (\mathcal{F}^{-1}\otimes 1)(\Delta_\mathcal{F}\otimes\mathrm{id})
    (\mathcal{F}^{-1})
    =&(\Delta\otimes\mathrm{id})(\mathcal{F}^{-1})
    (\mathcal{F}^{-1}\otimes 1)\\
    =&(\mathrm{id}\otimes\Delta)(\mathcal{F}^{-1})
    (1\otimes\mathcal{F}^{-1})\\
    =&(1\otimes\mathcal{F}^{-1})(\mathrm{id}\otimes\Delta_\mathcal{F})
    (\mathcal{F}^{-1})
\end{align*}
by the inverse $2$-cocycle property. Since $\mathcal{B}_{1\otimes 1}
=\mathcal{B}$ we conclude the proof of the lemma.
\end{proof}
Fix a Drinfel'd twist $\mathcal{F}$ on $\mathcal{B}$. In
Proposition~\ref{prop01} we proved that $\mathcal{B}_\mathcal{F}$
is a bialgebra, so according to Proposition~\ref{prop04}
${}_\mathcal{B}\mathcal{M}$ and ${}_{\mathcal{B}_\mathcal{F}}\mathcal{M}$
are both monoidal subcategories of ${}_\Bbbk\mathcal{M}$. Ignoring
the monoidal structures for a while we can define an assignment
\begin{equation}
    \mathrm{Drin}_\mathcal{F}
    \colon{}_\mathcal{B}\mathcal{M}
    \rightarrow{}_{\mathcal{B}_\mathcal{F}}\mathcal{M}
\end{equation}
to be the identity on objects and morphisms. It assigns to any object
$\mathcal{M}$ in ${}_\mathcal{B}\mathcal{M}$ itself, however seen as
a left $\mathcal{B}_\mathcal{F}$-module and to any left $\mathcal{B}$-module
homomorphism $\phi\colon\mathcal{M}\rightarrow\mathcal{M}'$ itself,
however viewed as a left $\mathcal{B}_\mathcal{F}$-module
homomorphism. This assignment is well-defined since
$\mathcal{B}$ and $\mathcal{B}_\mathcal{F}$ coincide as algebras.
To distinguish those two pictures we write 
$$
\mathrm{Drin}_\mathcal{F}(\mathcal{M})=\mathcal{M}_\mathcal{F}
\text{ and }
\mathrm{Drin}_\mathcal{F}(\phi)=\phi^\mathcal{F}.
$$
Since
$\mathrm{Drin}_\mathcal{F}$ is the identity on objects and morphisms
it follows that $\mathrm{Drin}_\mathcal{F}$ is a functor. We prove that
it is even a (braided) monoidal functor, leading to a (braided) monoidal
equivalence. Recall that for two left $\mathcal{B}$-modules
$\mathcal{M}$ and $\mathcal{M}'$, the tensor product
$\mathcal{M}\otimes\mathcal{M}'$ is again a left $\mathcal{B}$-module, or
equivalently a left $\mathcal{B}_\mathcal{F}$-module.
As a $\Bbbk$-module it coincides with the left
$\mathcal{B}_\mathcal{F}$-module
$\mathcal{M}_\mathcal{F}\otimes_\mathcal{F}\mathcal{M}'_\mathcal{F}$,
where we denote the monoidal structure on
${}_{\mathcal{B}_\mathcal{F}}\mathcal{M}$ by $\otimes_\mathcal{F}$.
However the left 
$\mathcal{B}_\mathcal{F}$-actions differ, since for all
$\xi\in\mathcal{B}$, $m\in\mathcal{M}$ and $m'\in\mathcal{M}'$ one obtains
$$
\xi\cdot(m\otimes m')
=(\xi_{(1)}\cdot m)\otimes(\xi_{(2)}\cdot m')
$$
while
$$
\xi\cdot(m\otimes_\mathcal{F}m')
=(\xi_{\widehat{(1)}}\cdot m)\otimes_\mathcal{F}(\xi_{\widehat{(2)}}\cdot m').
$$
To compare these two pictures we define a left
$\mathcal{B}_\mathcal{F}$-module isomorphism
$$
\varphi_{\mathcal{M},\mathcal{M}'}
\colon\mathcal{M}_\mathcal{F}\otimes_\mathcal{F}\mathcal{M}'_\mathcal{F}
\ni(m\otimes_\mathcal{F}m')\mapsto
(\mathcal{F}_1^{-1}\cdot m)\otimes(\mathcal{F}_2^{-1}\cdot m')
\in(\mathcal{M}\otimes\mathcal{M}')_\mathcal{F}.
$$
It functions as the natural transformation.
\begin{theorem}[c.f. \cite{Ka95}~Lem.~XV.3.7]\label{thm01}
Let $\mathcal{F}$ be a Drinfel'd twist on $\mathcal{B}$.
The functor
$$
\mathrm{Drin}_\mathcal{F}\colon
({}_\mathcal{B}\mathcal{M},\otimes)
\rightarrow
({}_{\mathcal{B}_\mathcal{F}}\mathcal{M},\otimes_\mathcal{F})
$$
is monoidal and the monoidal categories
$({}_{\mathcal{B}}\mathcal{M},\otimes)$
and $({}_{\mathcal{B}_\mathcal{F}}\mathcal{M},\otimes_\mathcal{F})$
are monoidally equivalent. If $\mathcal{B}$ is quasi-triangular,
$\mathrm{Drin}_\mathcal{F}$ is a braided monoidal functor and the two
categories are braided monoidally equivalent.
\end{theorem}
\begin{proof}
Let $\mathcal{M},\mathcal{M}'$ and $\mathcal{M}''$ be objects in
${}_\mathcal{B}\mathcal{M}$.
By the inverse $2$-cocycle property of $\mathcal{F}^{-1}$ we deduce that
the diagram
\begin{equation*}
\begin{tikzcd}
\mathcal{M}_\mathcal{F}\otimes_\mathcal{F}
\mathcal{M}'_\mathcal{F}\otimes_\mathcal{F}
\mathcal{M}''_\mathcal{F}
\arrow{rr}{\mathrm{id}\otimes_\mathcal{F}\varphi_{\mathcal{M}',\mathcal{M}''}} \arrow{d}[swap]{\varphi_{\mathcal{M},\mathcal{M}'}
\otimes_\mathcal{F}\mathrm{id}}
& & \mathcal{M}_\mathcal{F}\otimes_\mathcal{F}
(\mathcal{M}'\otimes\mathcal{M}'')_\mathcal{F}
\arrow{d}{\varphi_{\mathcal{M},\mathcal{M}'\otimes\mathcal{M}''}} \\
(\mathcal{M}\otimes\mathcal{M}')_\mathcal{F}\otimes_\mathcal{F}
\mathcal{M}''_\mathcal{F}
\arrow{rr}{\varphi_{\mathcal{M}\otimes\mathcal{M}',\mathcal{M}''}}
& & (\mathcal{M}\otimes\mathcal{M}'\otimes\mathcal{M}'')_\mathcal{F}
\end{tikzcd}
\end{equation*}
commutes, while the normalization property of $\mathcal{F}^{-1}$
is sufficient to make
\begin{equation*}
\begin{tikzcd}
\Bbbk_\mathcal{F}\otimes_\mathcal{F}
\mathcal{M}_\mathcal{F}
\arrow{r}{\varphi_{\Bbbk,\mathcal{M}}} \arrow{dr}[swap]{\cong}
& (\Bbbk\otimes\mathcal{M})_\mathcal{F}
\arrow{d}{\cong} \\
& \mathcal{M}_\mathcal{F}
\end{tikzcd},
\begin{tikzcd}
\mathcal{M}_\mathcal{F}\otimes_\mathcal{F}
\Bbbk_\mathcal{F}
\arrow{r}{\varphi_{\mathcal{M},\Bbbk}} \arrow{d}[swap]{\cong}
& (\mathcal{M}\otimes\Bbbk)_\mathcal{F}
\arrow{ld}{\cong} \\
\mathcal{M}_\mathcal{F}
&
\end{tikzcd}
\end{equation*}
commute. This proves that $\mathrm{Drin}_\mathcal{F}$ is a monoidal
functor. Assume now that $\mathcal{B}$ is quasi-triangular with
universal $\mathcal{R}$-matrix $\mathcal{R}\in\mathcal{B}\otimes\mathcal{B}$.
Equivalently, $({}_\mathcal{B}\mathcal{M},\otimes)$ is braided monoidal
with braiding 
$$
\beta_{\mathcal{M},\mathcal{M}'}
\colon\mathcal{M}\otimes\mathcal{M}'\ni(m\otimes m')
\mapsto\mathcal{R}^{-1}\cdot(m'\otimes m)\in
\mathcal{M}'\otimes\mathcal{M}.
$$
This was proven in Proposition~\ref{prop06}. Furthermore
$(\mathcal{B}_\mathcal{F},\otimes_\mathcal{F})$ is braided monoidal
with braiding
$$
\beta^\mathcal{F}_{\mathcal{M}_\mathcal{F},\mathcal{M}'_\mathcal{F}}
\colon\mathcal{M}_\mathcal{F}\otimes_\mathcal{F}
\mathcal{M}'_\mathcal{F}\ni(m\otimes_\mathcal{F}m')
\mapsto\mathcal{R}_\mathcal{F}^{-1}\cdot(m'\otimes_\mathcal{F} m)\in
\mathcal{M}'_\mathcal{F}\otimes_\mathcal{F}\mathcal{M}_\mathcal{F}
$$
according to Proposition~\ref{prop01}, where $\mathcal{R}_\mathcal{F}
=\mathcal{F}_{21}\mathcal{R}\mathcal{F}^{-1}$ is the quasi-triangular
structure on $\mathcal{B}_\mathcal{F}$. Then
\begin{equation*}
\begin{tikzcd}
\mathcal{M}_\mathcal{F}\otimes_\mathcal{F}\mathcal{M}'_\mathcal{F}
\arrow{r}{\varphi_{\mathcal{M},\mathcal{M}'}}
\arrow{d}[swap]{\beta^\mathcal{F}_{\mathcal{M}_\mathcal{F},
\mathcal{M}'_\mathcal{F}}}
& (\mathcal{M}\otimes\mathcal{M}')_\mathcal{F}
\arrow{d}{\mathrm{Drin}_\mathcal{F}(\beta_{\mathcal{M},\mathcal{M}'})} \\
\mathcal{M}'_\mathcal{F}\otimes_\mathcal{F}\mathcal{M}_\mathcal{F}
\arrow{r}{\varphi_{\mathcal{M}',\mathcal{M}}}
& (\mathcal{M}'\otimes\mathcal{M})_\mathcal{F}
\end{tikzcd}
\end{equation*}
commutes by definition, proving that $\mathrm{Drin}_\mathcal{F}$ is
a braided monoidal functor. Recalling Lemma~\ref{lemma11},
the inverse of $\mathrm{Drin}_\mathcal{F}$ is given by
$$
\mathrm{Drin}_{\mathcal{F}^{-1}}
\colon({}_{\mathcal{B}_\mathcal{F}}\mathcal{M},\otimes_\mathcal{F})
\rightarrow({}_{\mathcal{B}}\mathcal{M},\otimes).
$$
This concludes the proof of the theorem.
\end{proof}
%
%

\section{Module Algebras and their Equivariant Bimodules}\label{Sec2.5}

In this section we discuss bialgebra module algebras, i.e.
bialgebra modules with an additional algebra structure that is
respected by the bialgebra action. 
We prove that the Drinfel'd functor restricts to an isomorphism of the
equivariant algebra modules \cite{GiZh98}. Passing to equivariant
algebra bimodules this even becomes a monoidal equivalence if one
considers the tensor product over the algebra. In the end we are
interested in braided commutative algebras and  equivariant algebra bimodules which
are braided symmetric. As concrete examples one may think of twist star product
algebras together with twisted multivector fields and differential forms
(see Chapter~\ref{chap04}).
We obtain a braided monoidal equivalence via the Drinfel'd functor 
in this case (c.f. \cite{AsSh14,Schenkel2015}).
At the end of this section we also include twist deformations of antipodes
and rigid categories.
Since there are two different module
actions involved in the following, one from the bialgebra and one from them module
algebra, we denote the first by $\rhd$ and the latter by $\cdot$ for convenience.
In a first step we prove that a twist on a bialgebra $\mathcal{B}$ deforms the category of left
$\mathcal{B}$-module algebras. 
\begin{definition}
A $\Bbbk$-algebra $\mathcal{A}$ is said
to be a \textit{left $\mathcal{B}$-module algebra} of there is a $\Bbbk$-linear map
$\rhd\colon\mathcal{B}\otimes\mathcal{A}\rightarrow\mathcal{A}$, structuring
$\mathcal{A}$ as a left $\mathcal{B}$-module such that
$$
\xi\rhd(a\cdot b)
=(\xi_{(1)}\rhd a)\cdot(\xi_{(2)}\rhd b)
\text{ and }
\xi\rhd 1=\epsilon(\xi)1
$$
hold for all $\xi\in\mathcal{B}$ and $a,b\in\mathcal{A}$.
A \textit{left $\mathcal{B}$-module algebra homomorphism} is an algebra homomorphism
between left $\mathcal{B}$-module algebras which is a left $\mathcal{B}$-module 
homomorphism in addition. This constitutes the category ${}_\mathcal{B}\mathcal{A}$
of left $\mathcal{B}$-module algebras.
\end{definition}
In other words, the left $\mathcal{B}$-module action is respecting
the algebra structure of $\mathcal{A}$. Fix a Drinfel'd twist
$\mathcal{F}$ on $\mathcal{B}$ and a left $\mathcal{B}$-module algebra
$\mathcal{A}$ for now. Since $\mathcal{B}$ and $\mathcal{B}_\mathcal{F}$ are
representatives of the same gauge equivalence class it is natural to ask 
if also $\mathcal{B}_\mathcal{F}$ respects the algebra structure of
$\mathcal{A}$. However, since $\Delta$ differs from $\Delta_\mathcal{F}$
in general, this can not be expected. On the other hand there is a way
to gauge transform $\mathcal{A}$ using the twist $\mathcal{F}$, such
that $\mathcal{B}_\mathcal{F}$ respects the gauge transformed algebra.
\begin{proposition}[c.f. \cite{AsSh14}~Thm.~3.4]\label{prop02}
Let $\mathcal{F}$ be a twist on $\mathcal{B}$ and consider an object $\mathcal{A}
\in{}_\mathcal{B}\mathcal{A}$ with product $\cdot$ and unit
$1$. Then $\mathcal{A}_\mathcal{F}=(\mathcal{A},\cdot_\mathcal{F},1)$
is an (associative unital) algebra, where
\begin{equation}
    a\star_\mathcal{F}b
    =(\mathcal{F}_1^{-1}\rhd a)\cdot(\mathcal{F}_2^{-1}\rhd b)
\end{equation}
for all $a,b\in\mathcal{A}$.
Moreover, $\mathcal{A}_\mathcal{F}$ is an object in
${}_{\mathcal{B}_\mathcal{F}}\mathcal{A}$ with respect to the same Hopf algebra
action, i.e. it is a left $\mathcal{B}_\mathcal{F}$-module algebra.
\end{proposition}
\begin{proof}
We split the proof into two parts. Denote $\mu_\mathcal{F}=\cdot_\mathcal{F}$
and the product and unit of $\mathcal{A}$ by $\mu_\mathcal{A}$ and
$\eta_\mathcal{A}$, respectively.
\begin{enumerate}
\item[i.)] \textbf{$\mathcal{A}_\mathcal{F}$ is an associative unital algebra: }
let $a,b,c\in\mathcal{A}$. Then
\begin{align*}
    \mu_\mathcal{F}(\mu_\mathcal{F}(a\otimes b)\otimes c)
    =&\mu_\mathcal{A}(\mu_\mathcal{A}
    (((\mathcal{F}_{1(1)}^{-1}\mathcal{F}_1^{'-1})\rhd a)
    \otimes((\mathcal{F}_{1(2)}^{-1}\mathcal{F}_2^{'-1})\rhd b))
    \otimes(\mathcal{F}_2^{-1}\rhd c))\\
    =&\mu_\mathcal{A}((\mathcal{F}_1^{-1}\rhd a)\otimes\mu_\mathcal{A}(
    ((\mathcal{F}_{2(1)}^{-1}\mathcal{F}_1^{'-1})\rhd b)
    \otimes((\mathcal{F}_{2(2)}^{-1}\mathcal{F}_2^{'-1})\rhd c)))\\
    =&\mu_\mathcal{F}(a\otimes\mu_\mathcal{F}(b\otimes c))
\end{align*}
shows that $\mu_\mathcal{F}$ is associative, where we made use of the associativity
of $\mu_\mathcal{A}$ and the $2$-cocycle property of $\mathcal{F}^{-1}$.
Furthermore, $\mu_\mathcal{F}$ is unital with respect to $\eta_\mathcal{A}$,
since for all $\lambda\in\Bbbk$
\begin{align*}
    (\mu_\mathcal{F}\circ(\mathrm{id}\otimes\eta_\mathcal{A}))(a\otimes\lambda)
    =&\lambda
    \mu_\mathcal{A}((\mathcal{F}_1^{-1}\rhd a)\otimes(\mathcal{F}_2^{-1}\rhd 1))\\
    =&\lambda
    \mu_\mathcal{A}(((\mathcal{F}_1^{-1}\epsilon(\mathcal{F}_2^{-1}))\rhd a)
    \otimes 1)\\
    =&\lambda\mu_\mathcal{A}((1\rhd a)\otimes 1)\\
    =&\lambda a
\end{align*}
and similarly 
$(\mu_\mathcal{F}\circ(\eta_\mathcal{A}\otimes\mathrm{id}))(\lambda\otimes a)
=\lambda a$ follows. Here we used that $\mathcal{F}^{-1}$ is normalized and that
$\mu_\mathcal{A}$ is unital with respect to $\eta_\mathcal{A}$.

\item[ii.)] \textbf{$\rhd\colon\mathcal{B}_\mathcal{F}\otimes\mathcal{A}_\mathcal{F}
\rightarrow\mathcal{A}_\mathcal{F}$ respects the algebra structure of
$\mathcal{A}_\mathcal{F}$: }let $\xi\in\mathcal{B}$ and $a,b\in\mathcal{A}$. Then
\begin{align*}
    \xi\rhd\mu_\mathcal{F}(a\otimes b)
    =&\mu_\mathcal{A}(((\mathcal{F}_1^{''-1}\mathcal{F}_1^{'}\xi_{(1)}
    \mathcal{F}_1^{-1})\rhd a)
    \otimes((\mathcal{F}_2^{''-1}\mathcal{F}_2^{'}
    \xi_{(2)}\mathcal{F}_2^{-1})\rhd b))\\
    =&\mu_\mathcal{F}((\xi_{\widehat{(1)}}\rhd a)\otimes
    (\xi_{\widehat{(2)}}\rhd b))
\end{align*}
and $\xi\rhd 1=\epsilon(\xi) 1$ are satisfied.
\end{enumerate}
Since $\mathcal{A}_\mathcal{F}$ equals $\mathcal{A}$ as a $\Bbbk$-module
and $\mathcal{B}_\mathcal{F}$ equals $\mathcal{B}$ as an algebra this is all
we have to prove.
\end{proof}
More generally, we are able to deform the category
of $\mathcal{B}$-equivariant left $\mathcal{A}$-modules. 
\begin{definition}
Let $\mathcal{A}$ be a left $\mathcal{B}$-module algebra and consider the
left $\mathcal{B}$-modules $\mathcal{M}$
which are left $\mathcal{A}$-modules in addition, such that
\begin{equation}
    \xi\rhd(a\cdot m)=(\xi_{(1)}\rhd a)\cdot(\xi_{(2)}\rhd m)
\end{equation}
holds for all $\xi\in\mathcal{B}$, $a\in\mathcal{A}$ and $m\in\mathcal{M}$.
They are said to be $\mathcal{B}$-equivariant left $\mathcal{A}$-modules,
forming a category ${}^\mathcal{B}_\mathcal{A}\mathcal{M}$ with
morphisms being left $\mathcal{B}$-linear and left $\mathcal{A}$-linear maps
between $\mathcal{B}$-equivariant left $\mathcal{A}$-modules.
\end{definition}
In complete analogy to Proposition~\ref{prop02} one proves the
following statement.
\begin{proposition}[c.f. \cite{AsSh14}~Thm.~3.5]\label{prop03}
Let $\mathcal{F}$ be a twist on $\mathcal{B}$ and consider a left
$\mathcal{B}$-module algebra 
$\mathcal{A}$. For every object $\mathcal{M}$ in
${}^\mathcal{B}_\mathcal{A}\mathcal{M}$ the twisted left
$\mathcal{A}_\mathcal{F}$-module action
\begin{equation}
    a\bullet_\mathcal{F}m
    =(\mathcal{F}_1^{-1}\rhd a)\cdot(\mathcal{F}_2^{-1}\rhd m),
\end{equation}
where $a\in\mathcal{A}$ and $m\in\mathcal{M}$,
structures $\mathcal{M}$ as an $\mathcal{B}_\mathcal{F}$-equivariant left 
$\mathcal{A}_\mathcal{F}$-module.
\end{proposition}
The $\mathcal{B}_\mathcal{F}$-equivariant left $\mathcal{A}_\mathcal{F}$-module
$\mathcal{M}$ with module action $\bullet_\mathcal{F}$ is denoted by
$\mathcal{M}_\mathcal{F}$. If the left $\mathcal{A}$-module action is trivial
the assignment of Proposition~\ref{prop03} reduces to the
Drinfel'd functor $\mathrm{Drin}_\mathcal{F}\colon
{}_\mathcal{B}\mathcal{M}\rightarrow{}_{\mathcal{B}_\mathcal{F}}\mathcal{M}$.
On the other hand we can extend the Drinfel'd functor to
$\mathcal{B}$-equivariant left $\mathcal{A}$-modules.
\begin{proposition}[c.f. \cite{Schenkel2015}~Prop.~3.9]\label{prop11}
Let $\mathcal{A}\in{}_\mathcal{B}\mathcal{A}$ and $\mathcal{F}$ be a twist on
$\mathcal{B}$. There is a functor
\begin{equation}
    \mathrm{Drin}_\mathcal{F}\colon{}^\mathcal{B}_\mathcal{A}\mathcal{M}
    \rightarrow{}^{\mathcal{B}_\mathcal{F}}_{\mathcal{A}_\mathcal{F}}\mathcal{M}
\end{equation}
from the category of $\mathcal{B}$-equivariant left $\mathcal{A}$-modules to the
category of $\mathcal{B}_\mathcal{F}$-equivariant left
$\mathcal{A}_\mathcal{F}$-modules.
It is the identity on morphisms and assigns to any $\mathcal{B}$-equivariant left
$\mathcal{A}$-module $\mathcal{M}$ the $\mathcal{B}_\mathcal{F}$-equivariant left
$\mathcal{A}_\mathcal{F}$-module $\mathcal{M}_\mathcal{F}$ defined in
Proposition~\ref{prop03}. Furthermore, $\mathrm{Drin}_\mathcal{F}$ is an
isomorphism of categories.
\end{proposition}
\begin{proof}
By Proposition~\ref{prop03} the assignment $\mathrm{Drin}_\mathcal{F}$ is
well-defined on
objects. To prove that it is well-defined on morphisms consider a left
$\mathcal{A}$-module homomorphism $\phi\colon\mathcal{M}\rightarrow\mathcal{M}'$
between $\mathcal{B}$-equivariant left $\mathcal{A}$-module homomorphisms, which is
also a left $\mathcal{B}$-module homomorphism. Let $\xi\in\mathcal{B}$,
$a\in\mathcal{A}$ and $m\in\mathcal{M}$. Then
\begin{align*}
    \phi(a\bullet_\mathcal{F}m)
    =&\phi((\mathcal{F}_1^{-1}\rhd a)\cdot(\mathcal{F}_2^{-1}\rhd m))\\
    =&(\mathcal{F}_1^{-1}\rhd a)\cdot\phi(\mathcal{F}_2^{-1}\rhd m)\\
    =&(\mathcal{F}_1^{-1}\rhd a)\cdot(\mathcal{F}_2^{-1}\rhd\phi(m))\\
    =&a\bullet_\mathcal{F}\phi(m),
\end{align*}
while $\phi(\xi\rhd m)=\xi\rhd\phi(m)$ holds by definition. Consequently,
$\mathrm{Drin}_\mathcal{F}$ is also well-defined on morphisms. The functorial
properties are clear since the concatenation of morphisms is well-defined.
There is an inverse functor 
$$
\mathrm{Drin}_{\mathcal{F}^{-1}}
\colon{}^{\mathcal{B}_\mathcal{F}}_{\mathcal{A}_\mathcal{F}}\mathcal{M}
\rightarrow{}^\mathcal{B}_\mathcal{A}\mathcal{M},
$$
being the identity on morphisms and assigning to any
$\mathcal{B}_\mathcal{F}$-equivariant
left $\mathcal{A}_\mathcal{F}$-module $(\boldsymbol{\mathcal{M}},\bullet)$
the same left $\mathcal{B}$-module but with left $\mathcal{A}$-module structure
$$
a\cdot m=(\mathcal{F}_1\rhd a)\bullet(\mathcal{F}_2\rhd m)
$$
for all $a\in\mathcal{A}$ and $m\in\boldsymbol{\mathcal{M}}$. In fact,
this defines a left $\mathcal{A}$-module action, since for all
$a,b\in\mathcal{A}$ and $m\in\mathcal{M}$ one obtains
\begin{allowdisplaybreaks}
\begin{align*}
    a\cdot(b\cdot m)
    =&a\cdot((\mathcal{F}_1\rhd b)\bullet(\mathcal{F}_2\rhd m))\\
    =&(\mathcal{F}_1^{'}\rhd a)\bullet
    (\mathcal{F}_2^{'}\rhd((\mathcal{F}_1\rhd b)\bullet(\mathcal{F}_2\rhd m)))\\
    =&(\mathcal{F}_1^{'}\rhd a)
    \bullet(((\mathcal{F}_{2\widehat{(1)}}^{'}\mathcal{F}_1)\rhd b)
    \bullet((\mathcal{F}_{2\widehat{(2)}}^{'}\mathcal{F}_2)\rhd m))\\
    =&(\mathcal{F}_1^{'}\rhd a)
    \bullet(((\mathcal{F}_1\mathcal{F}_{2(1)}^{'})\rhd b)
    \bullet((\mathcal{F}_2\mathcal{F}_{2(2)}^{'})\rhd m))\\
    =&((\mathcal{F}_1\mathcal{F}_{1(1)}^{'})\rhd a)
    \bullet(((\mathcal{F}_2\mathcal{F}_{1(2)}^{'})\rhd b)
    \bullet(\mathcal{F}_{2}^{'}\rhd m))\\
    =&(((\mathcal{F}_1\mathcal{F}_{1(1)}^{'})\rhd a)
    \star_\mathcal{F}((\mathcal{F}_2\mathcal{F}_{1(2)}^{'})\rhd b))
    \bullet(\mathcal{F}_{2}^{'}\rhd m)\\
    =&((\mathcal{F}_{1(1)}^{'}\rhd a)
    \cdot(\mathcal{F}_{1(2)}^{'}\rhd b))
    \bullet(\mathcal{F}_{2}^{'}\rhd m)\\
    =&(\mathcal{F}_1^{'}\rhd(a\cdot b))\bullet(\mathcal{F}_2^{'}\rhd m)\\
    =&(a\cdot b)\cdot m
\end{align*}
\end{allowdisplaybreaks}
and $1\cdot m=(\mathcal{F}_1\rhd 1)\bullet(\mathcal{F}_2\rhd m)=m$.
Moreover, $\boldsymbol{\mathcal{M}}$ is $\mathcal{B}$-equivariant, since
\begin{align*}
    \xi\rhd(a\cdot m)
    =&\xi\rhd((\mathcal{F}_1\rhd a)\bullet(\mathcal{F}_2\rhd m))\\
    =&((\xi_{\widehat{(1)}}\mathcal{F}_1)\rhd a)\bullet
    ((\xi_{\widehat{(2)}}\mathcal{F}_2)\rhd m)\\
    =&((\mathcal{F}_1\xi_{(1)})\rhd a)\bullet
    ((\mathcal{F}_2\xi_{(2)})\rhd m)\\
    =&(\xi_{(1)}\rhd a)\cdot(\xi_{(2)}\rhd m)
\end{align*}
for all $\xi\in\mathcal{B}$, $a\in\mathcal{A}$ and $m\in\mathcal{M}$. We denote
$\boldsymbol{\mathcal{M}}$ with the module action $\cdot$ by
$\boldsymbol{\mathcal{M}}_{\mathcal{F}^{-1}}$.
Any morphisms $\boldsymbol{\phi}\colon\boldsymbol{\mathcal{M}}\rightarrow
\boldsymbol{\mathcal{M}'}$ in
${}^{\mathcal{B}_\mathcal{F}}_{\mathcal{A}_\mathcal{F}}\mathcal{M}$ can be viewed as
a morphisms
$\boldsymbol{\phi}\colon\boldsymbol{\mathcal{M}}_{\mathcal{F}^{-1}}\rightarrow
\boldsymbol{\mathcal{M}'}_{\mathcal{F}^{-1}}$
in ${}^\mathcal{B}_\mathcal{A}\mathcal{M}$. While the left $\mathcal{B}$-linearity 
is clear we check that
\begin{align*}
    \boldsymbol{\phi}(a\cdot m)
    =&\boldsymbol{\phi}((\mathcal{F}_1\rhd a)\bullet(\mathcal{F}_2\rhd m))\\
    =&(\mathcal{F}_1\rhd a)\bullet\boldsymbol{\phi}(\mathcal{F}_2\rhd m)\\
    =&(\mathcal{F}_1\rhd a)\bullet(\mathcal{F}_2\rhd\boldsymbol{\phi}(m))\\
    =&a\cdot(\boldsymbol{\phi}(m))
\end{align*}
holds in addition for all $a\in\mathcal{A}$ and
$m\in\boldsymbol{\mathcal{M}}_{\mathcal{F}^{-1}}$.
The functors $\mathrm{Drin}_\mathcal{F}$
and $\mathrm{Drin}_{\mathcal{F}^{-1}}$ are inverse to each other:
let $\mathcal{M}$ be on object in ${}^\mathcal{B}_\mathcal{A}\mathcal{M}$. Then
$(\mathcal{M}_\mathcal{F})_{\mathcal{F}^{-1}}=\mathcal{M}$ as objects in
${}^\mathcal{B}_\mathcal{A}\mathcal{M}$, since
\begin{align*}
    (\mathcal{F}_1\rhd a)\bullet_\mathcal{F}(\mathcal{F}_2\rhd m)
    =a\cdot m
\end{align*}
for all $a\in\mathcal{A}$ and $m\in\mathcal{M}$.
On the other hand, let $\boldsymbol{\mathcal{M}}$ be an object in
${}^{\mathcal{B}_\mathcal{F}}_{\mathcal{A}_\mathcal{F}}\mathcal{M}$. Then
$(\boldsymbol{\mathcal{M}}_{\mathcal{F}^{-1}})_\mathcal{F}
=\boldsymbol{\mathcal{M}}$ as objects in
${}^{\mathcal{B}_\mathcal{F}}_{\mathcal{A}_\mathcal{F}}\mathcal{M}$. To see this
let $a\in\mathcal{A}$ and $m\in\boldsymbol{\mathcal{M}}$ be arbitrary and
consider
\begin{align*}
    (\mathcal{F}_1^{-1}\rhd a)\cdot(\mathcal{F}_2^{-1}\rhd m)
    =a\bullet m.
\end{align*}
On the level of morphisms there is nothing to prove. This concludes
the proof of the proposition.
\end{proof}
The natural question arises if the functor from Proposition~\ref{prop11}
is still (braided) monoidal. This has to be negated, since
${}_\mathcal{A}^\mathcal{B}\mathcal{M}$ is not monoidal in general.
We need two further specifications, the first being to consider
$\mathcal{B}$-equivariant $\mathcal{A}$-bimodules. Namely, we consider
the subcategory ${}_\mathcal{A}^\mathcal{B}\mathcal{M}_\mathcal{A}$
of ${}_\mathcal{A}^\mathcal{B}\mathcal{M}$, consisting of those objects
$\mathcal{M}$ which inherit an additional right $\mathcal{A}$-module action,
which commutes with the left $\mathcal{A}$-module action such that
$$
\xi\rhd(m\cdot a)
=(\xi_{(1)}\rhd m)\cdot(\xi_{(2)}\rhd a)
$$
for all $\xi\in\mathcal{B}$, $m\in\mathcal{M}$ and $a\in\mathcal{A}$.
Morphisms in ${}_\mathcal{A}^\mathcal{B}\mathcal{M}_\mathcal{A}$ are
left $\mathcal{B}$-linear and left and right $\mathcal{A}$-linear maps.
In a second step we replace the tensor product $\otimes$ of $\Bbbk$-modules
with the tensor product $\otimes_\mathcal{A}$ over $\mathcal{A}$. Namely,
for two objects $\mathcal{M}$ and $\mathcal{M}'$ in
${}_\mathcal{A}^\mathcal{B}\mathcal{M}_\mathcal{A}$, the product
$\mathcal{M}\otimes_\mathcal{A}\mathcal{M}'$ is defined by the quotient
$$
\mathcal{M}\otimes\mathcal{M}'/\mathcal{N}_{\mathcal{M},\mathcal{M}'},
$$
where $\mathcal{N}_{\mathcal{M},\mathcal{M}'}$ is the ideal
in $\mathcal{M}\otimes\mathcal{M}'$, defined by the image of
$$
\rho_\mathcal{M}\otimes\mathrm{id}_{\mathcal{M}'}
-\mathrm{id}_\mathcal{M}\otimes\lambda_{\mathcal{M}'},
$$
where $\rho_\mathcal{M}$ and $\lambda_{\mathcal{M}'}$ denote the right
and left $\mathcal{A}$-module action on $\mathcal{M}$ and
$\mathcal{M}'$, respectively. In particular this implies
$$
(m\cdot a)\otimes_\mathcal{A}m'
=m\otimes_\mathcal{A}(a\cdot m')
$$
for all $a\in\mathcal{A}$, $m\in\mathcal{M}$ and $m'\in\mathcal{M}'$.
Furthermore, $\mathcal{M}\otimes_\mathcal{A}\mathcal{M}'$ is a
$\mathcal{B}$-equivariant $\mathcal{A}$-bimodule with
induced left $\mathcal{B}$-action and left and right $\mathcal{A}$-actions
given for all $a\in\mathcal{A}$, $m\in\mathcal{M}$ and $m'\in\mathcal{M}'$ by
$$
a\cdot(m\otimes_\mathcal{A}m')
=(a\cdot m)\otimes_\mathcal{A}m'
\text{ and }
(m\otimes_\mathcal{A}m')\cdot a
=m\otimes_\mathcal{A}(m'\cdot a),
$$
respectively. On morphisms $\phi\colon\mathcal{M}\rightarrow\mathcal{N}$
and $\psi\colon\mathcal{M}'\rightarrow\mathcal{N}'$ of
${}_\mathcal{A}^\mathcal{B}\mathcal{M}_\mathcal{A}$ one defines
$(\phi\otimes_\mathcal{A}\psi)(m\otimes_\mathcal{A}m')
=\phi(m)\otimes_\mathcal{A}\psi(m')$ for all $m\in\mathcal{M}$ and
$m'\in\mathcal{M}'$. This implies the following statement.
\begin{lemma}[c.f. \cite{Schenkel2015}~Prop.~3.11]
The tuple $({}_\mathcal{A}^\mathcal{B}\mathcal{M}_\mathcal{A},
\otimes_\mathcal{A})$ is a monoidal category and the Drinfel'd functor
\begin{equation}\label{eq29}
    \mathrm{Drin}_\mathcal{F}
    \colon({}_\mathcal{A}^\mathcal{B}\mathcal{M}_\mathcal{A},
    \otimes_\mathcal{A})
    \rightarrow
    ({}_{\mathcal{A}_\mathcal{F}}^{\mathcal{B}_\mathcal{F}}
    \mathcal{M}_{\mathcal{A}_\mathcal{F}},
    \otimes_{\mathcal{A}_\mathcal{F}})
\end{equation}
is monoidal, leading to a monoidal equivalence.
\end{lemma}
However, the monoidal functor (\ref{eq29}) still fails to be
braided monoidal in general. In fact 
$({}_\mathcal{A}^\mathcal{B}\mathcal{M}_\mathcal{A},
\otimes_\mathcal{A})$ is not even braided in general if $\mathcal{B}$ is
quasi-triangular. The $\mathcal{B}$-equivariant $\mathcal{A}$-bimodules
are still too arbitrary. We have to demand even more symmetry before.
We do so by considering a \textit{braided commutative} (also called
\textit{quasi-commutative}) left $\mathcal{B}$-module algebra
$\mathcal{A}$ for a triangular bialgebra $(\mathcal{B},\mathcal{R})$
instead of a general left $\mathcal{B}$-module algebra. This means that
$b\cdot a=(\mathcal{R}_1^{-1}\rhd a)\cdot(\mathcal{R}_2^{-1}\rhd b)$
holds for all elements $a,b$ of a left $\mathcal{B}$-module algebra
$(\mathcal{A},\cdot)$. On the level of $\mathcal{A}$-bimodules we want to
keep this symmetry.
\begin{definition}
Let $(\mathcal{B},\mathcal{R})$ be a triangular bialgebra
and $(\mathcal{A},\cdot)$ be a braided commutative left $\mathcal{B}$-module
algebra. A $\mathcal{B}$-equivariant $\mathcal{A}$-bimodule
$\mathcal{M}$ is said to be braided symmetric if
$$
a\cdot m
=(\mathcal{R}_1^{-1}\rhd m)\cdot(\mathcal{R}_2^{-1}\rhd a)
$$
for all $a\in\mathcal{A}$ and $m\in\mathcal{M}$. Their morphisms 
are left $\mathcal{B}$-linear maps which are left and right
$\mathcal{A}$-linear in addition. We denote the category
of \textit{$\mathcal{B}$-equivariant braided symmetric
$\mathcal{A}$-bimodules} by
${}_\mathcal{A}^\mathcal{B}\mathcal{M}_\mathcal{A}^\mathcal{R}$.
\end{definition}
In other words, the left and right $\mathcal{A}$-module actions are related
via the universal $\mathcal{R}$-matrix $\mathcal{R}$, mirroring
the braided commutativity of $\mathcal{A}$.
We proceed by proving the main theorem of this section
(c.f. \cite{Schenkel2015}~Thm.~3.13).
It states that the Drinfel'd functor is braided monoidal
on equivariant braided symmetric bimodules.
\begin{theorem}\label{thm02}
Let $(\mathcal{B},\mathcal{R})$ be a triangular bialgebra and $\mathcal{F}$
a Drinfel'd twist on $\mathcal{B}$. Then,
the triple $({}_\mathcal{A}^\mathcal{B}\mathcal{M}^\mathcal{R}_\mathcal{A},
\otimes_\mathcal{A},\beta^\mathcal{R})$ is a braided monoidal category and
the Drinfel'd functor
\begin{equation}\label{eq30}
    \mathrm{Drin}_\mathcal{F}
    \colon({}_\mathcal{A}^\mathcal{B}\mathcal{M}_\mathcal{A}^\mathcal{R},
    \otimes_\mathcal{A},\beta^\mathcal{R})
    \rightarrow
    ({}_{\mathcal{A}_\mathcal{F}}^{\mathcal{B}_\mathcal{F}}
    \mathcal{M}_{\mathcal{A}_\mathcal{F}}^{\mathcal{R}_\mathcal{F}},
    \otimes_{\mathcal{A}_\mathcal{F}},\beta^{\mathcal{F}})
\end{equation}
is braided monoidal, leading to a braided monoidal equivalence.
\end{theorem}
\begin{proof}
It is sufficient to prove that $\otimes_\mathcal{A}$ and
$\mathrm{Drin}_\mathcal{F}$ are closed in the category of
equivariant braided symmetric bimodules. So let
$\mathcal{M}$ and $\mathcal{M}'$ be objects in
${}_\mathcal{A}^\mathcal{B}\mathcal{M}_\mathcal{A}$. Then
\begin{align*}
    a\cdot(m\otimes_\mathcal{A}m')
    =&(a\cdot m)\otimes_\mathcal{A}m'\\
    =&((\mathcal{R}_1^{-1}\rhd m)\cdot(\mathcal{R}_2^{-1}\rhd a))
    \otimes_\mathcal{A}m'\\
    =&(\mathcal{R}_1^{-1}\rhd m)
    \otimes_\mathcal{A}((\mathcal{R}_2^{-1}\rhd a)\cdot m')\\
    =&(\mathcal{R}_1^{-1}\rhd m)
    \otimes_\mathcal{A}((\mathcal{R}_1^{'-1}\rhd m')
    \cdot((\mathcal{R}_2^{'-1}\mathcal{R}_2^{-1})\rhd a))\\
    =&(\mathcal{R}_1^{-1}\rhd(m\otimes_\mathcal{A}m'))
    \cdot(\mathcal{R}_2^{-1}\rhd a)
\end{align*}
for all $a\in\mathcal{A}$, $m\in\mathcal{M}$ and
$m'\in\mathcal{M}'$. Furthermore $\Bbbk$ is a
$\mathcal{B}$-equivariant braided symmetric $\mathcal{A}$-bimodule,
since
$$
a\cdot\lambda
=\lambda\cdot a
=(\mathcal{R}_1^{-1}\rhd\lambda)\cdot(\mathcal{R}_2^{-1}\rhd a)
$$
for all $a\in\mathcal{A}$ and $\lambda\in\Bbbk$. We further observe that
\begin{align*}
    a\cdot_\mathcal{F}m
    =&(\mathcal{F}_1^{-1}\rhd a)\cdot(\mathcal{F}_2^{-1}\rhd m)\\
    =&((\mathcal{R}_1^{-1}\mathcal{F}_2^{-1})\rhd m)
    \cdot((\mathcal{R}_2^{-1}\mathcal{F}_1^{-1})\rhd a)\\
    =&(\mathcal{R}_{\mathcal{F}1}^{-1}\rhd m)
    \cdot_\mathcal{F}(\mathcal{R}_{\mathcal{F}2}^{-1}\rhd a)
\end{align*}
for all $a\in\mathcal{A}$ and $m\in\mathcal{M}$, proving that
$\mathcal{M}_\mathcal{F}$ is an object in
${}_{\mathcal{A}_\mathcal{F}}^{\mathcal{B}_\mathcal{F}}
\mathcal{M}_{\mathcal{A}_\mathcal{F}}^{\mathcal{R}_\mathcal{F}}$.
\end{proof}
The rest of this section is devoted to include twist deformations of
antipodes in the picture. 
If there is an antipode $S$ on $\mathcal{B}$ we define
$$
\beta=\mathcal{F}_1S(\mathcal{F}_2)\in\mathcal{B}\otimes\mathcal{B}.
$$
Since $\mathcal{B}$ is a Hopf algebra
in this case it is more convenient to write $H$ instead of $\mathcal{B}$.
\begin{lemma}\label{lemma01}
The element $\beta$ is invertible with inverse given by
$\beta^{-1}=S(\mathcal{F}_1^{-1})\mathcal{F}_2^{-1}$.
\end{lemma}
\begin{proof}
This is just a matter of computation using the axioms of Hopf algebra and Drinfel'd
twist. Explicitly we obtain
\begin{allowdisplaybreaks}
\begin{align*}
    \beta\beta^{-1}
    =&\mathcal{F}_1S(\mathcal{F}_2)S(\mathcal{F}_1^{'-1})\mathcal{F}_2^{'-1}\\
    =&\mathcal{F}_1^{''-1}\epsilon(\mathcal{F}_2^{''-1})\mathcal{F}_1
    S(\mathcal{F}_2)S(\mathcal{F}_1^{'-1})\mathcal{F}_2^{'-1}\\
    =&\mathcal{F}_1^{''-1}\mathcal{F}_1
    S(\mathcal{F}_2)S(\mathcal{F}_1^{'-1})
    S(\mathcal{F}_{2(1)}^{''-1})\mathcal{F}_{2(2)}^{''-1}\mathcal{F}_2^{'-1}\\
    =&\mathcal{F}_1^{''-1}\mathcal{F}_1
    S(\mathcal{F}_{2(1)}^{''-1}\mathcal{F}_1^{'-1}\mathcal{F}_2)
    \mathcal{F}_{2(2)}^{''-1}\mathcal{F}_2^{'-1}\\
    =&\mathcal{F}_{1(1)}^{''-1}\mathcal{F}_1^{'-1}\mathcal{F}_1
    S(\mathcal{F}_{1(2)}^{''-1}\mathcal{F}_2^{'-1}\mathcal{F}_2)
    \mathcal{F}_{2}^{''-1}\\
    =&\mathcal{F}_{1(1)}^{''-1}S(\mathcal{F}_{1(2)}^{''-1})
    \mathcal{F}_{2}^{''-1}\\
    =&\epsilon(\mathcal{F}_1^{''-1})\mathcal{F}_2^{''-1}\\
    =&1
\end{align*}
\end{allowdisplaybreaks}
and similarly $\beta^{-1}\beta=1$.
\end{proof}
Let us define a $\Bbbk$-linear map
$S_\mathcal{F}\colon H\rightarrow H$ by
\begin{equation}
    S_\mathcal{F}(\xi)=\beta S(\xi)\beta^{-1}
\end{equation}
for all $\xi\in H$. It is said to be the \textit{twisted antipode}.
\begin{proposition}
Let $(H,\mu,\eta,\Delta,\epsilon,S)$
be a Hopf algebra and $\mathcal{F}$ a Drinfel'd twist on $H$. Then
$H_\mathcal{F}=(H,\mu,\eta,\Delta_\mathcal{F},\epsilon,S_\mathcal{F})$
is a Hopf algebra. If $\mathcal{R}$ is a quasi-triangular structure on
$H$ then $\mathcal{R}_\mathcal{F}$ is a quasi-triangular structure on
$H_\mathcal{F}$. If $\mathcal{R}$ is triangular, so is
$\mathcal{R}_\mathcal{F}$.
\end{proposition}
\begin{proof}
$S_\mathcal{F}$ is an antipode on $H$ with respect to
$\Delta_\mathcal{F}$ and $\epsilon$ since for all $\xi\in H$ we obtain
\begin{align*}
    (\mu\circ(S_\mathcal{F}\otimes\mathrm{id})\circ\Delta_\mathcal{F})(\xi)
    =&S_\mathcal{F}(\mathcal{F}_1\xi_{(1)}\mathcal{F}_1^{'-1})
    \mathcal{F}_2\xi_{(2)}\mathcal{F}_2^{'-1}\\
    =&\mathcal{F}_1^{''}S(\mathcal{F}_2^{''})
    S(\mathcal{F}_1\xi_{(1)}\mathcal{F}_1^{'-1})
    S(\mathcal{F}_1^{'''-1})\mathcal{F}_2^{'''-1}
    \mathcal{F}_2\xi_{(2)}\mathcal{F}_2^{'-1}\\
    =&\mathcal{F}_1^{''}
    S(\mathcal{F}_1^{'''-1}\mathcal{F}_1\xi_{(1)}
    \mathcal{F}_1^{'-1}\mathcal{F}_2^{''})
    \mathcal{F}_2^{'''-1}\mathcal{F}_2\xi_{(2)}\mathcal{F}_2^{'-1}\\
    =&\mathcal{F}_1^{''}
    S(\mathcal{F}_1^{'-1}\mathcal{F}_2^{''})S(\xi_{(1)})
    \xi_{(2)}\mathcal{F}_2^{'-1}\\
    =&\epsilon(\xi)\mathcal{F}_1^{''}S(\mathcal{F}_2^{''})
    S(\mathcal{F}_1^{'-1})\mathcal{F}_2^{'-1}\\
    =&\epsilon(\xi)\beta\beta^{-1}\\
    =&\epsilon(\xi)1,
\end{align*}
where we used Lemma~\ref{lemma01}. In complete analogy one proves 
$\mu\circ(\mathrm{id}\otimes S_\mathcal{F})\circ\Delta_\mathcal{F}
=\eta\circ\epsilon$. The other statements only involve the underlying
bialgebra structure and have been proven in Proposition~\ref{prop01}.
\end{proof}
As a consequence of Theorem~\ref{thm02} we obtain the following result.
\begin{corollary}
Let $H$ be quasi-triangular and let furthermore $\mathcal{F}$ be a
Drinfel'd twist on $H$. Then,
the triple $({}_\mathcal{A}^H\mathcal{M}^\mathcal{R}_\mathcal{A},
\otimes_\mathcal{A},\beta^\mathcal{R})$ is a braided monoidal category and
the Drinfel'd functor
\begin{equation}\label{eq46}
    \mathrm{Drin}_\mathcal{F}
    \colon({}_\mathcal{A}^H\mathcal{M}_\mathcal{A}^\mathcal{R},
    \otimes_\mathcal{A},\beta^\mathcal{R})
    \rightarrow
    ({}_{\mathcal{A}_\mathcal{F}}^{H_\mathcal{F}}
    \mathcal{M}_{\mathcal{A}_\mathcal{F}}^{\mathcal{R}_\mathcal{F}},
    \otimes_{\mathcal{A}_\mathcal{F}},\beta^{\mathcal{F}})
\end{equation}
is braided monoidal, leading to a braided monoidal equivalence.
Restricting to finitely generated projective modules, similar
results hold on the rigid subcategory.
\end{corollary}
During the chapter we build up a setup to understand the braided monoidal
category $({}_\mathcal{A}^H\mathcal{M}^\mathcal{R}_\mathcal{A},
\otimes_\mathcal{A},\beta^\mathcal{R})$ and that (\ref{eq46}) acts as
a gauge equivalence on braided symmetric modules. In Chapter~\ref{chap04}
we resume by building a noncommutative Cartan calculus for any braided
commutative algebra in this category in a way that gauge equivalence classes
are respected. Before, in Chapter~\ref{chap03}, we recall
some concepts of \textit{deformation quantization} and how they relate to
Drinfel'd twist deformation. In particular we point out several cases
which do not allow for a twist deformation quantization.

\section{Star-Involutions and Unitary Twists}\label{Sec2.6}

In this short section we recall the notion of Hopf ${}^*$-algebras and
their representations. In other words we discuss how to incorporate a
${}^*$-involution in the previous picture. The most remarkable innovation
is that also the ${}^*$-involution admits a Drinfel'd twist deformation if
the twist is \textit{unitary}. While we omit this additional structure in
the general framework we are utilizing it in the explicit example of
Section~\ref{Sec4.4}. As reference we name \cite{Fiore2010}~Chap.~2
and \cite{Ma95}~Sec.~1.7.

Let $\Bbbk$ be a commutative unital ring endowed with a \textit{${}^*$-involution},
i.e. there is an (anti)automorphism $\Bbbk\ni\lambda\mapsto\overline{\lambda}\in\Bbbk$,
which is an involution in addition. On elements $\lambda,\mu\in\Bbbk$ this reads
\begin{align*}
    \overline{\lambda+\mu}
    =&\overline{\lambda}+\overline{\mu},\\
    \overline{\lambda\mu}
    =&\overline{\mu}\overline{\lambda}
    =\overline{\lambda}\overline{\mu},\\
    \overline{\overline{\lambda}}
    =&\lambda,\\
    \overline{1}
    =&1,
\end{align*}
where we used the commutativity of $\Bbbk$ in the second equation. Furthermore,
a \textit{${}^*$-algebra} over $\Bbbk$ is an associative unital algebra
$(\mathcal{A},\cdot,1_\mathcal{A})$ together with a ${}^*$-involution, i.e. a map
${}^*\colon\mathcal{A}\rightarrow\mathcal{A}$ such that
\begin{align*}
    (\lambda a+\mu b)^*
    =&\overline{\lambda}a^*+\overline{\mu}b^*,\\
    (a\cdot b)^*
    =&b^*\cdot a^*,\\
    (a^*)^*
    =&a,\\
    (1_\mathcal{A})^*
    =&1_\mathcal{A}
\end{align*}
for all $\lambda,\mu\in\Bbbk$ and $a,b\in\mathcal{A}$. Note that in general
$(a\cdot b)^*\neq a^*\cdot b^*$ since $\mathcal{A}$ might be noncommutative.
Coalgebras are defined as dual objects to algebras, while bialgebras unite both
concepts in a compatible way and Hopf algebras inherit an additional antipode
which is an anti-bialgebra homomorphism. In this light it is clear how to
define \textit{Hopf ${}^*$-algebras}.
\begin{definition}
A $\Bbbk$-Hopf algebra $(H,\cdot,1_H,\Delta,\epsilon,S)$ is said to be a
Hopf ${}^*$-algebra if $(H,\cdot,1_H)$ is a ${}^*$-algebra with
${}^*$-involution ${}^*\colon H\rightarrow H$ such that
\begin{equation}\label{eq55}
    \Delta(\xi^*)
    =(\xi_{(1)})^*\otimes(\xi_{(2)})^*,~
    \epsilon(\xi^*)
    =\overline{\epsilon(\xi)},
    \text{ and }
    S((S(\xi^*))^*)
    =\xi
\end{equation}
hold for all $\xi\in H$.
\end{definition}
If $S^2=\mathrm{id}$ the last equation of (\ref{eq55}) becomes
$S(\xi^*)=S(\xi)^*$, saying that $S$ respects the ${}^*$-involution. This happens
for example if $H$ is cocommutative or commutative.
We discuss an example of Hopf ${}^*$-algebra which is commonly used in
deformation quantization.
\begin{example}
Consider a Lie ${}^*$-algebra $(\mathfrak{g},[\cdot,\cdot],{}^*)$, which
is defined as a $\Bbbk$-Lie algebra $(\mathfrak{g},[\cdot,\cdot])$ together
with a map ${}^*\colon\mathfrak{g}\rightarrow\mathfrak{g}$ satisfying
\begin{equation*}
    (\lambda x+\mu y)^*
    =\overline{\lambda}x^*+\overline{\mu}y^*,~
    [x,y]^*
    =[y^*,x^*]
    \text{ and }
    (x^*)^*
    =x
\end{equation*}
for all $\lambda,\mu\in\Bbbk$ and $x,y\in\mathfrak{g}$. The universal
enveloping algebra $\mathscr{U}\mathfrak{g}$ of $\mathfrak{g}$ is not only
a Hopf algebra but can even be structured as a Hopf ${}^*$-algebra in this case:
the extension of ${}^*\colon\mathfrak{g}\rightarrow\mathfrak{g}$ as an
algebra anti-homomorphism to ${}^*\colon\mathscr{U}\mathfrak{g}\rightarrow
\mathscr{U}\mathfrak{g}$ is well-defined since $1^*=\overline{1}=1$ and
$$
\bigg(x\otimes y-y\otimes x-[x,y]\bigg)^*
=y^*\otimes x^*-x^*\otimes y^*-[y^*,x^*]
$$
for all $x,y\in\mathfrak{g}$,
where the latter implies that ${}^*$-respects the relation which is used to construct
$\mathscr{U}\mathfrak{g}$ as a quotient form the tensor algebra. It is
sufficient to prove the compatibility of the coalgebra structure and the antipode
with the ${}^*$-involution on primitive elements. In fact
$$
\Delta(x^*)
=x^*\otimes 1+1\otimes x^*
=x^*\otimes 1^*+1^*\otimes x^*
=(x_{(1)})^*\otimes(x_{(2)})^*,
$$
$\epsilon(x^*)=0=\overline{\epsilon(x)}$ and
$S((S(x^*))^*)=S((-x^*)^*)=(x^*)^*=x$
hold for all $x\in\mathfrak{g}$ since $x^*\in\mathfrak{g}$ by definition.
Note that every $\mathbb{R}$-Lie algebra can be
structured as a Lie ${}^*$-algebra with ${}^*$-involution given by
$x^*=-x$ for all $x\in\mathfrak{g}$.
\end{example}
Fix a Hopf ${}^*$-algebra $H$. A \textit{left $H$-module ${}^*$-algebra}
is a ${}^*$-algebra $(\mathcal{A},\cdot,1_\mathcal{A},{}^*)$ which is a
left $H$-module algebra, such that
\begin{equation}
    (\xi\rhd a)^*
    =S(\xi)^*\rhd a^*
\end{equation}
holds for all $\xi\in H$ and $a\in\mathcal{A}$. Let
$\mathcal{A}$ be a left $H$-module ${}^*$-algebra.
An \textit{$H$-equivariant $\mathcal{A}$-${}^*$-bimodule} is an
$H$-equivariant $\mathcal{A}$-bimodule $\mathcal{M}$ endowed with a map
${}^*\colon\mathcal{M}\rightarrow\mathcal{M}$, such that
\begin{equation}
    (\lambda m+\mu n)^*
    =\overline{\lambda}m^*+\overline{\mu}n^*,~
    (a\cdot m\cdot b)^*
    =b^*\cdot m^*\cdot a^*
    \text{ and }
    (m^*)^*
    =m
\end{equation}
for all $\lambda,\mu\in\Bbbk$, $m,n\in\mathcal{M}$ and $a,b\in\mathcal{A}$.
Consider a Drinfel'd twist $\mathcal{F}$ on $H$. It deforms the Hopf algebra
structure of $H$, the algebra structure of $\mathcal{A}$ and the
$\mathcal{A}$-bimodule structure of $\mathcal{M}$ such that
$\mathcal{A}_\mathcal{F}$ is a left $H_\mathcal{F}$-module algebra
and $\mathcal{M}$ is an $H_\mathcal{F}$-equivariant 
$\mathcal{A}_\mathcal{F}$-bimodule. The natural question arises if there is
a way to twist deform the ${}^*$-involution ${}^*\mapsto{}^{*_\mathcal{F}}$
such that $(\mathcal{A}_\mathcal{F},{}^{*_\mathcal{F}})$
is a left $H_\mathcal{F}$-module ${}^*$-algebra and
$\mathcal{M}_\mathcal{F}$ an $H_\mathcal{F}$-equivariant
$\mathcal{A}_\mathcal{F}$-${}^*$-bimodule. It turns out that this is the case
if we impose another condition on the twist $\mathcal{F}$.
\begin{definition}[Unitary Twist]
A twist $\mathcal{F}=\mathcal{F}_1\otimes\mathcal{F}_2$
on a Hopf ${}^*$-algebra $H$ is said to be unitary if
$
    \mathcal{F}_1^*\otimes\mathcal{F}_2^*
    =\mathcal{F}^{-1}
$.
\end{definition}
We slightly modify Example~\ref{ex02}~iii.) and v.) to obtain two classes of
examples of unitary twists.
\begin{example}\label{example05}
Let $\mathfrak{g}$ be a Lie ${}^*$-algebra over $\mathbb{C}$.
\begin{enumerate}
\item[i.)] For $n\in\mathbb{N}$ consider a set
$x_1,\ldots,x_n,y_1,\ldots,y_n\in\mathfrak{g}$ of elements in $\mathfrak{g}$
which are all Hermitian, i.e. $x_i^*=x_i$, $y_i^*=y_i$, or anti-Hermitian, i.e.
$x_i^*=-x_i$, $y_i^*=-y_i$. Further assume that
$[x_i,y_j]=[x_i,x_j]=[y_i,y_j]=0$. Then
\begin{equation}
    \mathcal{F}=\exp(\mathrm{i}\hbar r)\in
    (\mathscr{U}\mathfrak{g}\otimes\mathscr{U}\mathfrak{g})[[\hbar]]
\end{equation}
is a unitary twist, where $r=\sum_{i=1}^nx_i\otimes y_i$ and $\mathrm{i}$ denotes
the imaginary unit in $\mathbb{C}$. It remains to prove that $\mathcal{F}$ is
unitary, which follows since
$$
\mathcal{F}_1^*\otimes\mathcal{F}_2^*
=\exp\bigg(\overline{\mathrm{i}\hbar}\sum_{i=1}^nx_i^*\otimes y_i^*\bigg)
=\exp(-\mathrm{i}\hbar r)
=\mathcal{F}^{-1}.
$$

\item[ii.)] Let $H,E\in\mathfrak{g}$ be anti-Hermitian elements such that
$[H,E]=2E$. Then
\begin{equation}
    \mathcal{F}=\exp\bigg(\frac{H}{2}\otimes\log(1+\mathrm{i}\hbar E)\bigg)
    \in(\mathscr{U}\mathfrak{g}\otimes\mathscr{U}\mathfrak{g})[[\hbar]]
\end{equation}
is a unitary twist. Using the formal power series expansion in the $\hbar$-adic
topology gives $(\log(1+\mathrm{i}\hbar E))^*=\log(\overline{1+\mathrm{i}\hbar E})
=\log(1+\mathrm{i}\hbar E)$. Then clearly $\mathcal{F}_1^*\otimes\mathcal{F}_2^*
=\exp\big(-\frac{H}{2}\otimes\log(1+\mathrm{i}\hbar E)\big)
=\mathcal{F}^{-1}$, which implies that $\mathcal{F}$ is unitary.
\end{enumerate}
\end{example}
After discussing two examples we return to the general theory. The claim
was that unitary twists allow for deformations of the ${}^*$-involution
such that $\mathcal{A}_\mathcal{F}$ and $\mathcal{M}_\mathcal{F}$ are
deformed as ${}^*$-algebras and ${}^*$-bimodules.
\begin{proposition}[\cite{Fiore2010}]\label{prop17}
Let $H$ be a cocommutative Hopf ${}^*$-algebra and
consider a left $H$-module ${}^*$-algebra $\mathcal{A}$. Then
$\mathcal{A}_\mathcal{F}$ is a left $H_\mathcal{F}$-module ${}^*$-algebra
with ${}^*$-involution ${}^{*_\mathcal{F}}\colon\mathcal{A}_\mathcal{F}\rightarrow
\mathcal{A}_\mathcal{F}$ defined by
\begin{equation}
    a^{*_\mathcal{F}}
    =S(\beta)\rhd a^*
\end{equation}
for all $a\in\mathcal{A}$, where $\beta=\mathcal{F}_1S(\mathcal{F}_2)$.
If furthermore, $\mathcal{M}$ is an $H$-equivariant $\mathcal{A}$-${}^*$-bimodule,
$\mathcal{M}_\mathcal{F}$ is an $H_\mathcal{F}$-equivariant 
$\mathcal{A}_\mathcal{F}$-${}^*$-bimodule with ${}^*$-involution
\begin{equation}
    m^{*_\mathcal{F}}
    =S(\beta)\rhd m^*
\end{equation}
for all $m\in\mathcal{M}$.
\end{proposition}
In particular
\begin{equation}
    (a\cdot_\mathcal{F}b)^{*_\mathcal{F}}
    =b^{*_\mathcal{F}}\cdot_\mathcal{F}a^{*_\mathcal{F}}
    \text{ and }
    (a\cdot_\mathcal{F}m\cdot_\mathcal{F}b)^{*_\mathcal{F}}
    =b^{*_\mathcal{F}}\cdot m^{*_\mathcal{F}}\cdot a^{*_\mathcal{F}}
\end{equation}
hold for all $a,b\in\mathcal{A}$ and $m\in\mathcal{M}$.
Note that the ${}^*$-involution on $H_\mathcal{F}$ is undeformed.
\begin{proof}
Let us first prove that
\begin{equation}\label{eq48}
    \mathcal{F}\cdot\Delta(\beta)\cdot(S(\mathcal{F}_2)\otimes S(\mathcal{F}_1))
    =\beta\otimes\beta.
\end{equation}
In fact, the $2$-cocycle property and the normalization property of $\mathcal{F}$
and the anti-bialgebra homomorphism property of $S$ as well as the antipode
property imply
\begin{align*}
    \mathcal{F}\Delta(\beta)(S(\mathcal{F}_2)\otimes S(\mathcal{F}_1))
    =&\mathcal{F}''_1\mathcal{F}'_{1(1)}S(\mathcal{F}'_{2})_{(1)}S(\mathcal{F}_2)
    \otimes\mathcal{F}''_2\mathcal{F}'_{1(2)}
    S(\mathcal{F}'_{2})_{(2)}S(\mathcal{F}_1)\\
    =&\mathcal{F}''_1\mathcal{F}'_{1(1)}S(\mathcal{F}_2\mathcal{F}'_{2(2)})
    \otimes\mathcal{F}''_2\mathcal{F}'_{1(2)}
    S(\mathcal{F}_1\mathcal{F}'_{2(1)})\\
    =&\mathcal{F}''_1(\mathcal{F}_{1}\mathcal{F}'_{1(1)})_{(1)}S(\mathcal{F}'_{2})
    \otimes\mathcal{F}''_2(\mathcal{F}_{1}\mathcal{F}'_{1(1)})_{(2)}
    S(\mathcal{F}_2\mathcal{F}'_{1(2)})\\
    =&\mathcal{F}''_1\mathcal{F}_{1(1)}\mathcal{F}'_{1(1)}S(\mathcal{F}'_{2})
    \otimes\mathcal{F}''_2\mathcal{F}_{1(2)}\mathcal{F}'_{1(2)}
    S(\mathcal{F}'_{1(3)})S(\mathcal{F}_2)\\
    =&\mathcal{F}''_1\mathcal{F}_{1(1)}\mathcal{F}'_{1}S(\mathcal{F}'_{2})
    \otimes\mathcal{F}''_2\mathcal{F}_{1(2)}S(\mathcal{F}_2)\\
    =&\mathcal{F}_{1}\mathcal{F}'_{1}S(\mathcal{F}'_{2})
    \otimes\mathcal{F}''_1\mathcal{F}_{2(1)}S(\mathcal{F}''_2\mathcal{F}_{2(2)})\\
    =&\mathcal{F}_{1}\mathcal{F}'_{1}S(\mathcal{F}'_{2})
    \otimes\mathcal{F}''_1\mathcal{F}_{2(1)}S(\mathcal{F}_{2(2)})S(\mathcal{F}''_2)\\
    =&\mathcal{F}'_{1}S(\mathcal{F}'_{2})\otimes\mathcal{F}''_1S(\mathcal{F}''_2)\\
    =&\beta\otimes\beta.
\end{align*}
Then (\ref{eq48}) leads to
\begin{equation}\label{eq49}
    \Delta(\beta)
    =\mathcal{F}^{-1}\cdot(\beta\otimes\beta)\cdot
    ((S\otimes S)(\mathcal{F}_{21}^{-1}))
\end{equation}
and
\begin{align*}
    (a\cdot_\mathcal{F}b)^{*_\mathcal{F}}
    =&S(\beta)\rhd(a\cdot_\mathcal{F}b)^*\\
    =&S(\beta)\rhd((\mathcal{F}_1^{-1}\rhd a)\cdot(\mathcal{F}_2^{-1}\rhd b))^*\\
    =&S(\beta)\rhd((\mathcal{F}_2^{-1}\rhd b)^*\cdot(\mathcal{F}_1^{-1}\rhd a)^*)\\
    =&S(\beta)\rhd((S(\mathcal{F}_2^{-1})^*\rhd b^*)
    \cdot(S(\mathcal{F}_1^{-1})^*\rhd a^*))\\
    =&S(\beta)\rhd((S(\mathcal{F}_2)\rhd b^*)\cdot(S(\mathcal{F}_1)\rhd a^*))\\
    =&((S(\beta)_{(1)}S(\mathcal{F}_2)S(\beta^{-1}))\rhd b^{*_\mathcal{F}})
    \cdot((S(\beta)_{(2)}S(\mathcal{F}_1)S(\beta^{-1}))\rhd a^{*_\mathcal{F}})\\
    =&(S(\beta^{-1}F_2\beta_{(2)})\rhd b^{*_\mathcal{F}})
    \cdot(S(\beta^{-1}F_1\beta_{(1)})\rhd b^{*_\mathcal{F}})\\
    =&(\mathcal{F}_1^{-1}\rhd b^{*_\mathcal{F}})
    \cdot(\mathcal{F}_2^{-1}\rhd a^{*_\mathcal{F}})\\
    =&b^{*_\mathcal{F}}\cdot_\mathcal{F}a^{*_\mathcal{F}}
\end{align*}
holds for all $a,b\in\mathcal{A}$,
using that $\mathcal{F}$ is unitary, eq.(\ref{eq49}) and
$S^2=\mathrm{id}$, which holds since $H$ is cocommutative.
Furthermore, for $a\in\mathcal{A}$ we obtain
\begin{align*}
    (a^{*_\mathcal{F}})^{*_\mathcal{F}}
    =&(S(\beta)\rhd a^*)^{*_\mathcal{F}}\\
    =&S(\beta)\rhd(S(\beta)\rhd a^*)^*\\
    =&(S(\beta)S^2(\beta^*))\rhd(a^*)^*\\
    =&(S(\beta)\beta^*)\rhd a\\
    =&a
\end{align*}
where we employed that
\begin{align*}
    S(\beta)\beta^*
    =&\mathcal{F}_2S(\mathcal{F}_1)S(\mathcal{F}_2^{'*})\mathcal{F}_1^{'*}\\
    =&\mathcal{F}_2S(\mathcal{F}_1)S(\mathcal{F}_2^{'-1})\mathcal{F}_1^{'-1}\\
    =&\mathcal{F}_2^{''-1}\mathcal{F}_2S(\mathcal{F}_1)
    S(\mathcal{F}_2^{'-1})S(\mathcal{F}_{1(2)}^{''-1})
    \mathcal{F}_{1(1)}^{''-1}\mathcal{F}_1^{'-1}\\
    =&\mathcal{F}_{2(2)}^{''-1}\mathcal{F}_2^{'-1}\mathcal{F}_2S(\mathcal{F}_1)
    S(\mathcal{F}_{2(1)}^{''-1}\mathcal{F}_1^{'-1})\mathcal{F}_{1}^{''-1}\\
    =&\mathcal{F}_{2(1)}^{''-1}S(\mathcal{F}_{2(2)}^{''-1})\mathcal{F}_1^{''-1}\\
    =&1.
\end{align*}
The $\Bbbk$-linearity of ${}^{*_\mathcal{F}}$ is clear since $\rhd$ is
$\Bbbk$-linear and we also obtain
$$
1^{*_\mathcal{F}}
=S(\beta)\rhd 1^*
=\epsilon(\beta)1
=1.
$$
This proves that $(\mathcal{A}_\mathcal{F},{}^{*_\mathcal{F}})$ is a
${}^*$-algebra. The left $H_\mathcal{F}$-action is compatible with the
${}^*$-involution since
\begin{align*}
    S_\mathcal{F}(\xi)^*\rhd a^{*_\mathcal{F}}
    =&(\beta S(\xi)\beta^{-1})^*\rhd
    (S(\beta)\rhd a^*)\\
    =&((\beta^{-1})^*S(\xi^*)\beta^*S(\beta))\rhd a^*\\
    =&((\beta^{-1})^*S(\xi^*))\rhd a^*\\
    =&S(\beta)\rhd(S(\xi^*)\rhd a^*)\\
    =&S(\beta)\rhd(\xi\rhd a)^*\\
    =&(\xi\rhd a)^{*_\mathcal{F}}.
\end{align*}
This proves that $(\mathcal{A}_\mathcal{F},{}^{*_\mathcal{F}})$ is a 
left $H_\mathcal{F}$-module ${}^*$-algebra. Replacing the product of the algebra
with the left and right $\mathcal{A}$-module action on $\mathcal{M}$ the
same calculation leads to the conclusion that
$(\mathcal{M}_\mathcal{F},{}^{*_\mathcal{F}})$ is an $H_\mathcal{F}$-equivariant
$\mathcal{A}$-${}^*$-bimodule. This concludes the proof of the proposition.
\end{proof}
Note that there is another notion of twist deformed ${}^*$-involution,
given by \textit{real} Drinfel'd twists, which is discussed in \cite{Ma95}~Prop.~2.3.7.
The corresponding condition on a twist $\mathcal{F}$ is
$S(\mathcal{F}_1^*)\otimes S(\mathcal{F}_2^*)=\mathcal{F}_{21}$. Note that 
following this convention one has to deform the ${}^*$-involution on the Hopf
algebra as well.
We end this section by discussing a classical example from differential
geometry that also appears in the later sections.
\begin{example}
Let $M$ be a smooth manifold and consider the algebra $\mathscr{C}^\infty(M)$
of smooth complex-valued functions on $M$. It is commutative with respect to
the pointwise product and there is a ${}^*$-involution
${}^*\colon\mathscr{C}^\infty(M)\rightarrow\mathscr{C}^\infty(M)$ defined by
$f^*(p)=\overline{f(p)}$ for all $f\in\mathscr{C}^\infty(M)$ and $p\in M$.
This structures $\mathscr{C}^\infty(M)$ as a ${}^*$-algebra over $\mathbb{C}$.
On smooth vector fields $X\in\mathfrak{X}^1(M)$ we define
\begin{equation}
    \mathscr{L}_{X^*}f
    =-(\mathscr{L}_X(f^*))^*
\end{equation}
for all $f\in\mathscr{C}^\infty(M)$, where
$\mathscr{L}\colon\mathfrak{X}^1(M)\rightarrow\mathrm{Der}(\mathscr{C}^\infty(M))$
denotes the Lie derivative. One easily proves that this structures
$\mathfrak{X}^1(M)$ as a Lie ${}^*$-algebra. In particular
$[X,Y]^*=[Y^*,X^*]$ for all $X,Y\in\mathfrak{X}^1(M)$. Moreover, with
the usual $\mathscr{C}^\infty(M)$-bimodule actions
$$
\mathscr{L}_{f\cdot X\cdot g}h
=f(\mathscr{L}_Xh)g,
\text{ where }f,g,h\in\mathscr{C}^\infty(M)
\text{ and }X\in\mathfrak{X}^1(M),
$$
$\mathfrak{X}^1(M)$ becomes an $\mathcal{A}$-${}^*$-bimodule. This can be
extended to multivector $\mathfrak{X}^\bullet(M)$ fields by defining
$(X\wedge Y)^*=Y^*\wedge X^*$ for $X,Y\in\mathfrak{X}^\bullet(M)$.
Remark that the Gerstenhaber bracket
satisfies $\llbracket X,Y\rrbracket^*=\llbracket Y^*,X^*\rrbracket$ for all
$X,Y\in\mathfrak{X}^\bullet(M)$. Using the dual pairing
$$
\langle\cdot,\cdot\rangle\colon\Omega^1(M)\times\mathfrak{X}^1(M)
\rightarrow\mathscr{C}^\infty(M)
$$
the ${}^*$-involution on $\mathfrak{X}^1(M)$ induces a ${}^*$-involution on
$\Omega^1(M)$. Namely, for $\omega\in\Omega^1(M)$ we define
$\omega^*\in\Omega^1(M)$ by
$$
\langle\omega^*,X\rangle=\langle\omega,X^*\rangle
$$
for all $X\in\mathfrak{X}^1(M)$. Inductively this extends to $\Omega^\bullet(M)$,
structuring it as a $\mathscr{C}^\infty(M)$-${}^*$-bimodule with respect to the
usual left and right $\mathscr{C}^\infty(M)$-module actions, defined by
$(f\cdot\omega\cdot g)(X)=f\cdot\omega(X)\cdot g$ for all
$f,g\in\mathscr{C}^\infty(M)$, $\omega\in\Omega^1(M)$ and $X\in\mathfrak{X}^1(M)$.

Assume now the existence of a left $H$-module ${}^*$-algebra action
$\rhd$ on $\mathscr{C}^\infty(M)$ for a cocommutative Hopf ${}^*$-algebra $H$ and
let $\mathcal{F}$ be a unitary Drinfel'd twist on $H$.
A natural candidate is given by the universal enveloping algebra of the
Lie ${}^*$-algebra $\mathfrak{X}^1(M)$.
The adjoint action
$$
\mathscr{L}_{\xi\rhd X}f
=\xi_{(1)}\rhd(\mathscr{L}_X(S(\xi_{(2)})\rhd f))
$$
structures $\mathfrak{X}^1(M)$ as an $H$-equivariant $\mathcal{A}$-${}^*$-bimodule.
Again, the dual pairing induces the same structure on $\Omega^1(M)$ via
$$
\langle\xi\rhd\omega,X\rangle
=\xi_{(1)}\rhd\langle\omega,S(\xi_{(2)})\rhd X\rangle.
$$
Extending the $H$-action via the coproduct to higher wedge products we end up with
$H$-equivariant $\mathcal{A}$-${}^*$-bimodules $\mathfrak{X}^\bullet(M)$ and
$\Omega^\bullet(M)$. By Proposition~\ref{prop17} the unitary twist $\mathcal{F}$
deforms them into $H_\mathcal{F}$-equivariant 
$\mathcal{A}_\mathcal{F}$-${}^*$-bimodules.
\end{example}
In Chapter~\ref{chap04} we come back to this example and further examine how
the structure of the Cartan calculus on $M$ can be twist deformed.
Motivated by this, we construct a noncommutative Cartan calculus on a
huge class of noncommutative algebras.

%% file: chapters/chapter25.tex
The aim of this chapter is to give a quick recap on symplectic geometry and
its quantization in the form of Drinfel'd twist deformation quantization,
before discussing several obstructions to this construction. We start
by reviewing basic definitions and properties of Poisson manifolds
in Section~\ref{SecObs1}. As examples, constant Poisson structures on
$\mathbb{R}^n$, the KKS Poisson structure on the dual of a Lie algebra
and the connected orientable symplectic Riemann surfaces
are depicted. Equivalently to the Poisson bracket we may consider a
bivector, squaring to zero under the Schouten-Nijenhuis bracket. If this
so called Poisson bivector is non-degenerate, we obtain a symplectic manifold.
Star products are introduced as formal deformations of the algebra of
smooth functions, deforming the underlying Poisson bracket in addition.
Then, examples of star products on some of the former Poisson
manifolds are given.
We end the first section with results about existence and classification of star
products in the symplectic case. In Section~\ref{Sec-r-matrix}, classical
$r$-matrices are defined as skew-symmetric solutions of the classical
Yang-Baxter equation. Equivalently, one can think of $G$-equivariant Poisson
bivectors on a Lie group $G$. Then it is not surprising that they appear as
skew-symmetrization of the first order of formal $\mathcal{R}$-matrices
and Drinfel'd twists on universal enveloping algebras. We discuss
quantization of $r$-matrices before turning to the notion of twist star 
products. The latter are star products on Poisson manifolds which are induced by
Drinfel'd twists. It follows that the underlying Poisson bracket is
induced by the corresponding $r$-matrix. In the symplectic case it is not
hard to see that this implies that the manifold is a homogeneous space
under certain conditions, leading to a first obstruction: the symplectic
Pretzel surfaces of genus $g>1$ admit formal deformations but no twist star product.
With some more effort we conclude the same statement for the symplectic
$2$-sphere. After recalling the definition and several characterizations
and instances of Morita equivalence in Section~\ref{SecObs3}, we study
formal deformations of algebra modules, relative to deformations of the
algebra, in Section~\ref{SecObs4}. In particular, deforming Morita
equivalence bimodules it follows that the deformation theory of an algebra
is a Morita invariant. In the case of twist star products one can even prove
that the action of the Picard group is trivial, i.e. any star product
which is Morita equivalent to a twist star product is in fact equivalent as
star product to the latter. This leads to another class of
star products on complex projective spaces that can not be induced by 
Drinfel'd twists based on matrix Lie algebra symmetries.

\section{Poisson Manifolds and Deformation Quantization}\label{SecObs1}

As indicated in the introduction, \textit{Poisson geometry} is a
decent mathematical framework to describe classical mechanics,
while \textit{deformation quantization} pictures the quantized system.
In this section
we briefly recall the notion of Poisson manifold. As corresponding literature
we mention \cite{B2001}~Chap.~2-3 and \cite{WaldmannBuch}~Chap.~3-4. 
Afterwards, following \cite{ChiaraBuch}~Sec.~2.3, \cite{Gutt2005}~Chap.~II~2
and \cite{WaldmannBuch}~Sec~6.1,
star products are introduced as formal deformations of the algebra of functions
on a Poisson manifold and examples are given. We discuss the notion of equivalence
of star products and their classification on symplectic manifolds. A more
detailed overview of the field of deformation quantization is given in
\cite{Waldmann2016}.

A \textit{Poisson manifold} is a smooth manifold $M$ \cite{Lee2003},
together with a \textit{Poisson bracket} on its algebra $\mathscr{C}^\infty(M)$
of smooth functions. The latter is a Lie bracket
$\{\cdot,\cdot\}\colon\mathscr{C}^\infty(M)\times\mathscr{C}^\infty(M)
\rightarrow\mathscr{C}^\infty(M)$ satisfying a \textit{Leibniz rule}
\begin{equation}
    \{f,gh\}
    =\{f,g\}h+g\{f,h\}
\end{equation}
in addition, where $f,g,h\in\mathscr{C}^\infty(M)$. The Leibniz rule implies that
$$
X_f:=-\{f,\cdot\}\colon\mathscr{C}^\infty(M)
\rightarrow\mathscr{C}^\infty(M)
$$
is a smooth vector field on $M$ for all $f\in\mathscr{C}^\infty(M)$,
called \textit{Hamiltonian vector field}. On a local chart $(U,x)$
around a point $p\in M$ the Poisson bracket reads
$$
\{f,g\}|_U
=\sum_{ij}\pi^{ij}\frac{\partial f}{\partial x^i}\frac{\partial g}{\partial x^j}
$$
for all $f,g\in\mathscr{C}^\infty(M)$, with $\pi^{ij}=\{x^i,x^j\}=-\pi^{ji}\in
\mathscr{C}^\infty(U)$. This determines a bivector
$$
\pi_U
=\frac{1}{2}\sum_{ij}\pi^{ij}
\frac{\partial}{\partial x^i}\wedge\frac{\partial}{\partial x^j}
\in\mathfrak{X}^2(U)
$$
on $U$, which corresponds to a global bivector $\pi\in\mathfrak{X}^2(M)$
such that
$$
\{f,g\}=\pi(\mathrm{d}f,\mathrm{d}g)
$$
for all $f,g\in\mathscr{C}^\infty(M)$. One can prove that the Jacobi identity
of $\{\cdot,\cdot\}$ is equivalent to the vanishing $\llbracket\pi,\pi\rrbracket=0$
of the Schouten-Nijenhuis bracket of $\pi$ with itself.
The latter is defined as the extension
\begin{align*}
    \llbracket X_1\wedge\ldots\wedge X_k,
    Y_1\wedge\ldots\wedge Y_\ell\rrbracket
    =\sum_{i=1}^k\sum_{j=1}^\ell&(-1)^{i+j}[X_i,Y_j]\wedge
    X_1\wedge\ldots\wedge\widehat{X_i}\wedge\ldots\wedge X_k\\
    &\wedge Y_1\wedge\ldots\wedge\widehat{Y_j}\wedge\ldots\wedge Y_\ell
\end{align*}
of the Lie bracket $[\cdot,\cdot]$ of vector fields,
where $X_1,\ldots,X_k,Y_1,\ldots,Y_\ell\in\mathfrak{X}^1(M)$
and $\widehat{X_i}$, $\widehat{Y_j}$ means that $X_i$ and $Y_j$ are omitted
in the above wedge product.
It is defined to vanish if one of its entries is a function. In other words,
the data $\pi\in\mathfrak{X}^2(M)$ and $\llbracket\pi,\pi\rrbracket=0$ is
equivalent to the properties of $\{\cdot,\cdot\}$. For this reason we
sometimes refer to $(M,\pi)$ as a Poisson manifold and call $\pi$ the
corresponding \textit{Poisson bivector}.
\begin{example}\label{example02}
\begin{enumerate}
\item[i.)] Consider $\mathbb{R}^n$ with global coordinates $(x^1,\ldots,x^n)$.
Any skew-symmetric matrix $(\pi^{ij})_{ij}\in M_n(\mathbb{R})$ leads to a
Poisson bivector
\begin{equation}
    \pi=\sum_{i<j}\pi^{ij}\frac{\partial}{\partial x^i}
    \wedge\frac{\partial}{\partial x^j}
\end{equation}
on $\mathbb{R}^n$, called constant Poisson bivector. In fact,
since $\llbracket\cdot,\cdot\rrbracket$ is the extension of the Lie bracket
of vector fields and coordinate vector fields commute,
$\llbracket\pi,\pi\rrbracket=0$ follows immediately.

\item[ii.)] Consider a $\mathbb{C}$-Lie algebra 
$\mathfrak{g}$ and denote its dual by $\mathfrak{g}^*$.
Remark that polynomials $\mathrm{Pol}^\bullet(\mathfrak{g}^*)$ on $\mathfrak{g}^*$
can be identified with the symmetric algebra $\mathrm{S}^\bullet\mathfrak{g}$
of $\mathfrak{g}$. For example, $x\in\mathfrak{g}$ corresponds to
$\hat{x}\in\mathrm{Pol}^1(\mathfrak{g}^*)$, which is defined by
$\hat{x}(\xi)=\xi(x)$ for all $\xi\in\mathfrak{g}^*$.
We assume that $\mathfrak{g}$ is finite-dimensional with basis
$e^1,\ldots,e^n\in\mathfrak{g}$. The corresponding dual basis on
$\mathfrak{g}^*$ is denoted by $e_1,\ldots,e_n$. We define the
Kirillov-Kostant-Souriau bracket
\begin{equation}\label{eq47}
    \{f,g\}
    =\sum_{i,j,k}c_{ij}^ke_k\frac{\partial f}{\partial e_i}
    \frac{\partial g}{\partial e_j}
\end{equation}
for all $f,g\in\mathscr{C}^\infty(\mathfrak{g}^*)$, where
$c_{ij}^k=e_k([e^i,e^j])\in\mathbb{C}$ are the structure constants of
$(\mathfrak{g},[\cdot,\cdot])$. The $\mathbb{C}$-bilinear map (\ref{eq47})
is a Lie bracket exactly because $[\cdot,\cdot]$ is and it satisfies a
Leibniz rule, since $\frac{\partial}{\partial e_i}$ is a derivation of
$(\mathscr{C}^\infty(\mathfrak{g}^*),\cdot)$. Consider
\cite{BBGW2002} and references therein for more information about this example.
\end{enumerate}
\end{example}
The Hamiltonian vector field corresponding to $f\in\mathscr{C}^\infty(M)$ reads
$X_f=\llbracket f,\pi\rrbracket$ and
$
    X_{fg}
    =fX_g+gX_f,~
    [X_f,X_g]
    =-X_{\{f,g\}}
$
hold for all $f,g\in\mathscr{C}^\infty(M)$.
We say that a vector field
$X\in\mathfrak{X}^1(M)$ is a \textit{Poisson vector field} if
$\mathscr{L}_X\pi=0$. In particular, every Hamiltonian vector field is Poisson.
Moreover, Poisson vector fields form a Lie subalgebra of $\mathfrak{X}^1(M)$,
and the identity
\begin{align*}
    [X,X_f]=X_{X(f)},
\end{align*}
which holds for all Poisson vector fields $X$ and $f\in\mathscr{C}^\infty(M)$,
proves that the Hamiltonian vector fields are a Lie ideal of the Poisson vector
fields. Furthermore, Poisson vector fields are characterized as those vector
fields, which are derivations of the Poisson bracket, i.e. $X\in\mathfrak{X}^1(M)$
such that
$$
X(\{f,g\})
=\{X(f),g\}+\{f,X(g)\}
$$
for all $f,g\in\mathscr{C}^\infty(M)$ (c.f. \cite{WaldmannBuch}~Satz~4.1.9).
A \textit{Poisson diffeomorphism} is a diffeomorphism $\phi\colon(M,\{\cdot,\cdot\})
\rightarrow(M',\{\cdot,\cdot\}')$ between Poisson manifolds, such that
$$
\phi^*\{f,g\}'=\{\phi^*(f),\phi^*(g)\}
$$
for all $f,g\in\mathscr{C}^\infty(M')$. 
The \textit{musical
homomorphism} ${}^\sharp\colon T^*M\rightarrow TM$ is defined on any $p\in M$ by 
$$
T_p^*M\ni\alpha_p\mapsto\alpha_p^\sharp
=\pi_p(\cdot,\alpha_p)\in T_pM.
$$
It extends to a $\mathscr{C}^\infty(M)$-linear map
${}^\sharp\colon\Omega^1(M)\rightarrow\mathfrak{X}^1(M)$.
A Poisson manifold $(M,\pi)$ is said to be \textit{symplectic}
if the bivector field $\pi\in\mathfrak{X}^2(M)$ is non-degenerate.
Note that for any $p\in M$ the Hamiltonian vector fields at $p$ span $T_pM$
in this case.
Then, we can define a non-degenerate $2$-form $\omega\in\Omega^2(M)$ by
\begin{equation}\label{eq43}
    \omega(X_f,X_g)
    =\pi(\mathrm{d}f,\mathrm{d}g)
\end{equation}
for all $f,g\in\mathscr{C}^\infty(M)$. It follows that $\omega$ is closed, i.e.
$\mathrm{d}\omega=0$. On the other hand, a non-degenerate $2$-form 
$\omega\in\Omega^2(M)$ determines a non-degenerate bivector field
$\pi\in\mathfrak{X}^2(M)$ by (\ref{eq43}). $\pi$ is a Poisson bivector if and
only if $\omega$ is closed. It is remarkable that the quadratic relation
$\llbracket\pi,\pi\rrbracket=0$ reduces to a linear relation $\mathrm{d}\omega=0$
in the symplectic case.
\begin{example}[c.f. \cite{Radko}]\label{example06}
The connected orientable Riemann surfaces are characterized by
their genus $g\geq 0$, which indicates the number of holes. Since
they are orientable, they inherit non-degenerate $2$-forms, which are
automatically closed, since those manifolds are $2$-dimensional. Thus, the
choice of a volume form structures any connected orientable Riemann
surfaces as a symplectic manifold.
\end{example}
A vector field $X\in\mathfrak{X}^1(M)$ is said
to be a \textit{symplectic vector field} if $\mathscr{L}_X\omega=0$. In
particular every Hamiltonian vector field is symplectic and the symplectic
vector fields form a Lie subalgebra of $\mathfrak{X}^1(M)$. A
\textit{symplectomorphism} is a diffeomorphism $\phi\colon(M,\omega)
\rightarrow(M',\omega')$ between symplectic manifolds such that
$\phi^*\omega'=\omega$. Recall that the de Rham differential
$\mathrm{d}\colon\Omega^\bullet(M)\rightarrow\Omega^{\bullet+1}(M)$
constitutes a cochain complex
$$
0\xrightarrow{\mathrm{d}}\Omega^1(M)
\xrightarrow{\mathrm{d}}\Omega^2(M)
\xrightarrow{\mathrm{d}}\Omega^3(M)
\xrightarrow{\mathrm{d}}\cdots
$$
and since it is a differential, the \textit{de Rham cohomology}
(see e.g. \cite{Lee2003}~Chap.~11)
$$
\mathrm{H}^n_{\mathrm{dR}}(M)
=\frac{\ker\mathrm{d}\colon\Omega^n(M)\rightarrow\Omega^{n+1}(M)}
{\mathrm{im}~\mathrm{d}\colon\Omega^{n-1}(M)\rightarrow\Omega^n(M)}
$$
of $M$ is well-defined for all $n\geq 0$, where we set $\Omega^{-1}(M)=0$.
In the presence of a Lie group action $\Phi\colon G\times M\rightarrow M$
the de Rham differential forms a cochain complex with respect to
$G$-invariant differential forms $\Omega^{\bullet,G}(M)$, i.e.
elements $\omega\in\Omega^\bullet(M)$ such that $\Phi^*_g\omega=\omega$
for all $g\in G$. It is well-defined since $\mathrm{d}$ commutes with
pullbacks. The corresponding \textit{$G$-invariant de Rham cohomology}
is denoted by $\mathrm{H}^{\bullet,G}_{\mathrm{dR}}(M)$. Note that there
is also a \textit{Poisson cohomology} on $(M,\pi)$ defined by the differential
$\mathrm{d}_\pi=\llbracket\pi,\cdot\rrbracket$ (c.f. \cite{Lich1977}).

In a next step we describe how to \textit{quantize} Poisson manifolds.
To be more precise, we want to perform formal deformations of the
algebra of smooth functions $\mathscr{C}^\infty(M)$. Since the definition is
given in algebraic terms we first formulate it for general associative algebras
(following \cite{GerstenhaberAlgDef}), before coming back to smooth functions.
Let $(\mathcal{A},\cdot,1)$ be an associative unital algebra. A
\textit{formal deformation} of $(\mathcal{A},\cdot,1)$ is an
associative product $\star$ on $\boldsymbol{\mathcal{A}}=\mathcal{A}[[\hbar]]$
such that $a\star 1=a=1\star a$ and
\begin{equation}
    a\star b
    =a\cdot b+\mathcal{O}(\hbar)
\end{equation}
for all $a,b\in\mathcal{A}$. The \textit{classical limit} assigns to a
formal deformation $\star$ its lowest order $\cdot$ and we sometimes write
$\lim_{\hbar\rightarrow 0}a\star b=a\cdot b$. Heuristically, this amounts to set
the value of $\hbar$ to zero.
If $\mathcal{A}$ is commutative and $\star$ a formal deformation of $\mathcal{A}$,
the first order
\begin{equation}\label{eq39}
    \{a,b\}=
    \frac{1}{\hbar}(a\star b-b\star a)+\mathcal{O}(\hbar)
\end{equation}
of the $\star$-commutator $[a,b]_\star=a\star b-b\star a$
defines a \textit{Poisson bracket} on $\mathcal{A}$,
i.e. a Lie bracket satisfying a Leibniz rule
$$
\{a,bc\}
=\{a,b\}c+b\{a,c\}
$$
for all $a,b,c\in\mathcal{A}$. Equation~(\ref{eq39}) is said to be the
\textit{correspondence principle}, which can be expressed as
\begin{equation}
    \lim_{\hbar\rightarrow 0}\frac{1}{\hbar}[a,b]_\star
    =\{a,b\},
\end{equation}
using the classical limit.
Two formal deformations $\star$ and $\star'$
are said to be \textit{equivalent} if there are $\Bbbk$-linear maps
$T_n\colon\mathcal{A}\rightarrow\mathcal{A}$ for all $n>0$ such that
$T=\mathrm{id}+\sum_{n>0}\hbar^nT_n$ satisfies
\begin{equation}
    T(a\star b)
    =T(a)\star'T(b)
\end{equation}
for all $a,b\in\mathcal{A}$. Since $T$ starts with the identity it is
invertible in the $\hbar$-adic topology and the $\hbar$-linear extension
$T\colon(\boldsymbol{\mathcal{A}},\star)
\rightarrow(\boldsymbol{\mathcal{A}},\star')$
is an algebra isomorphism. The equivalence class of formal deformations which
are equivalent to $\star$ is denoted by $[\star]$.
From any automorphism $\phi$ of $\mathcal{A}$ and any formal deformation $\star$
of $\mathcal{A}$ we construct another formal deformation $\star_\phi$ by
$$
a\star_\phi b
=\phi^{-1}(\phi(a)\star\phi(b))
$$
for all $a,b\in\mathcal{A}$.
\begin{lemma}[\cite{Bursztyn2002}~Prop.~2.14]
Two formal deformations $\star$ and $\star'$ of $\mathcal{A}$ are isomorphic
if and only if there exists an automorphism $\phi$ of $\mathcal{A}$ such that
$[\star_\phi]=[\star']$.
\end{lemma}
In the case $\mathcal{A}=\mathscr{C}^\infty(M)$ for a smooth manifold $M$,
formal deformations of the pointwise product should be considered in the
smooth category.
\begin{definition}[\cite{Bayen1978}]
A formal deformation 
$$
f\star g
=f\cdot g+\sum_{r>0}\hbar^rC_r(f,g)
$$
of $(\mathscr{C}^\infty(M),\cdot)$
is said to be a \textit{star product} on $M$ if $C_r$ are bidifferential operators.
The corresponding Poisson bracket $\{\cdot,\cdot\}$
structures $M$ as a \textit{Poisson manifold}. Accordingly, a formal deformation
of a Poisson manifold $(M,\{\cdot,\cdot\})$ is a star product $\star$, such that
\begin{equation}
    \lim_{\hbar\rightarrow 0}\frac{1}{\hbar}[f,g]_\star
    =\{f,g\}
\end{equation}
holds for all $f,g\in\mathscr{C}^\infty(M)$.
\end{definition}
Two formal deformations $\star$, $\star'$ of a Poisson manifold
$(M,\{\cdot,\cdot\})$ are equivalent as star products if there is a
formal equivalence via differential operators.
They are isomorphic if and only if there is a Poisson diffeomorphism
$\phi\colon M\rightarrow M$ such that $[\star_\phi]=[\star']$, where
$$
\phi^*(f\star_\phi g)
=\phi^*(f)\star\phi^*(g)
$$
for all $f,g\in\mathscr{C}^\infty(M)$.
The set of equivalence classes of star products on
$(M,\{\cdot,\cdot\})$ is denoted by $\mathrm{Def}(M,\{\cdot.\cdot\})$
and by $\mathrm{Def}(M,\omega)$ in the symplectic case.
We give two examples of star products to 
become familiar with the notion. Namely, we quantize the Poisson structures
discussed in Example~\ref{example02}.
\begin{example}\label{example03}
\begin{enumerate}
\item[i.)] Consider $\mathbb{R}^{n}$ with coordinates $(x^1,\ldots,x^n)$
and a constant Poisson bivector $\pi
=\sum_{i<j}\pi^{ij}\frac{\partial}{\partial x^i}\wedge\frac{\partial}{\partial x^j}$.
Its exponential
\begin{equation}
    (f\star g)(x)
    =\exp\bigg(\hbar\sum_{i<j}\pi^{ij}\frac{\partial}{\partial x^i}
    \frac{\partial}{\partial y^j}\bigg)f(x)g(y)|_{y=x},
\end{equation}
where $f,g\in\mathscr{C}^\infty(M)$, is a formal star product quantizing
$(\mathbb{R}^n,\pi)$, called Moyal-Weyl star product
(compare e.g. to \cite{ChiaraBuch}~Ex.~2.3.3). This leads to
$$
\lim_{\hbar\rightarrow 0}\frac{1}{\hbar}[f,g]_\star
=\{f,g\},
$$
where $\{f,g\}=\sum_{i<j}\pi^{ij}\frac{\partial f}{\partial x^i}
\frac{\partial g}{\partial x^j}$ is the Poisson bracket corresponding to
the constant Poisson bivector $\pi$.

\item[ii.)] Consider a finite-dimensional $\mathbb{C}$-Lie algebra
$\mathfrak{g}$ with dual frame $e_i(e^j)=\delta_i^j$ and structure
constants $c_{ij}^k$ as in Example~\ref{example02}~ii.).
Then, the
Poincaré-Birkhoff-Witt theorem gives an isomorphism
$\rho_\hbar\colon\mathrm{S}^\bullet(\mathfrak{g})[\hbar]
\rightarrow\mathscr{U}\mathfrak{g}[\hbar]$
via total symmetrization. On elements
$x_1,\ldots,x_n\in\mathfrak{g}$ the latter reads
$$
\rho_\hbar(\widehat{x_1\vee\ldots\vee x_n})
=\frac{\hbar^n}{n!}\sum_{\sigma\in S_n}x_{\sigma(1)}\bullet\ldots
\bullet x_{\sigma(n)},
$$
where $\bullet$ denotes the
multiplication in $\mathscr{U}\mathfrak{g}$. One proves that
\begin{equation}
    f\star g
    =\rho_\hbar^{-1}(\rho_\hbar(f)\bullet\rho_\hbar(g))
\end{equation}
defines an associative unital product on $\mathrm{Pol}^\bullet(\mathfrak{g}^*)$,
which extends to a formal star product on $\mathscr{C}^\infty(\mathfrak{g}^*)$.
This is said to be the Gutt star product.
The corresponding Poisson bracket is the
Kirillov-Kostant-Souriau bracket. Again, we refer to \cite{BBGW2002} for more
information.
\end{enumerate}
\end{example}
The question of existence and classification of star products on general
Poisson manifolds was completely settled in \cite{Kont2003}.
We state some of the most important results in the case of symplectic manifold.
\begin{proposition}[\cite{BCG1997,Deligne1995}]\label{prop19}
Let $(M,\omega)$ be a symplectic manifold. Then there is a bijection
\begin{equation}
    c\colon\mathrm{Def}(M,\omega)\rightarrow
    \frac{[\omega]}{\hbar}+\mathrm{H}^2_{\mathrm{dR}}(M)[[\hbar]],
\end{equation}
which assigns to any star product $\star$ on $(M,\omega)$ its 
characteristic class $c(\star)$.
\end{proposition}
In particular, the above theorem reveals that there is a star product on
every symplectic manifold $(M,\omega)$. If
$\mathrm{H}^2_{\mathrm{dR}}(M)=\{0\}$, all star products on $(M,\omega)$
are equivalent. There is a geometric construction \cite{Fedosov} of Fedosov,
leading to the result of Proposition~\ref{prop19}. Moreover, it provides
a recursive formula to successively assemble the star product.

If there is a Lie group action $\Phi\colon G\times M\rightarrow M$
on a Poisson manifold $(M,\pi)$, a star product $\star$ on $(M,\pi)$ is
said to be \textit{$G$-invariant} if
\begin{equation}
    \Phi^*_g(f\star h)
    =(\Phi^*_g(f))\star(\Phi^*_g(h))
\end{equation}
for all $g\in G$ and $f,h\in\mathscr{C}^\infty(M)$, where
$\Phi^*_g\colon\mathscr{C}^\infty(M)\rightarrow\mathscr{C}^\infty(M)$ denotes
the pullback of the diffeomorphism $\Phi_g\colon M\rightarrow M$. Two
$G$-invariant star products $\star$ and $\star'$ are said to be $G$-invariantly 
equivalent if there exists an equivalence $T
=\mathrm{id}+\sum_{r>0}\hbar^rT_r$ of star products consisting of
$G$-invariant operators, i.e. $\Phi^*_g(T_r(f))=T_r(\Phi^*_g(f))$ for all 
$g\in G$, $f\in\mathscr{C}^\infty(M)$ and $r>0$. The corresponding set
of $G$-invariant equivalence classes of $G$-invariant star products on
$(M,\pi)$ is denoted by $\mathrm{Def}^G(M,\pi)$.
\begin{proposition}[\cite{BBG1998}]
Let $(M,\omega)$ be a symplectic manifold and consider a Lie group action
$\Phi\colon G\times M\rightarrow M$ on $M$ such that $\Phi^*_g\omega=\omega$ for
all $g\in G$. Then there is a bijection
\begin{equation}
    c^G\colon\mathrm{Def}^G(M,\omega)\rightarrow
    \frac{[\omega]^G}{\hbar}+\mathrm{H}^{2,G}_{\mathrm{dR}}(M)[[\hbar]],
\end{equation}
which assigns to any $G$-invariant star product $\star$ on $(M,\omega)$ its 
$G$-invariant characteristic class $c^G(\star)$.
\end{proposition}
There is another existence and classification theorem \cite{Thorsten2016}
of symplectic manifold incorporating the notion of momentum map.

\section{Classical r-Matrices and Drinfel'd Twists}\label{Sec-r-matrix}

In this section we examine Drinfel'd twists in the setting of deformation quantization.
It turns out that invariant star products on a Lie group $G$ can be identified with
Drinfel'd twists $\mathcal{F}$ on $\mathscr{U}\mathfrak{g}[[\hbar]]$,
where $\mathfrak{g}$ denotes the Lie algebra corresponding to $G$. The 
skew-symmetrization of the first order term of a $G$-invariant star product
is a $G$-invariant Poisson bivector. On the side of the universal enveloping algebra,
this means that the skew-symmetrization of the first order term of a twist
is a \textit{classical $r$-matrix} \cite{Sem1985}. After internalizing the
relation of twists and classical $r$-matrices, we consider star products
which are induced by Drinfel'd twists and study first obstructions of
this situation. For an introduction to classical $r$-matrices we refer to
\cite{ES2010}~Chap.~3 and \cite{KosSch2004}~Sec.~2, while Drinfel'd twists
on universal enveloping algebras and twist star products are discussed in e.g.
\cite{Aschieri2008,AsSh14,Thomas2016,Bloh2003,Bloh2004,Borowiec2014,dAWe17,Fiore2010,Juric2015,Pachol2017}.

Let $\mathfrak{g}$ be a Lie algebra over a commutative unital ring 
$\Bbbk$ such that $\mathbb{Q}\subseteq\Bbbk$
and denote the corresponding exterior algebra 
by $\Lambda^\bullet\mathfrak{g}$. By extending the Lie bracket
$[\cdot,\cdot]$ of $\mathfrak{g}$, using the following expression
\begin{align*}
    \llbracket x_1\wedge\ldots\wedge x_k,
    y_1\wedge\ldots\wedge y_\ell\rrbracket
    =\sum_{i=1}^k\sum_{j=1}^\ell&(-1)^{i+j}[x_i,y_j]\wedge
    x_1\wedge\ldots\wedge\widehat{x_i}\wedge\ldots\wedge x_k\\
    &\wedge y_1\wedge\ldots\wedge\widehat{y_j}\wedge\ldots\wedge y_\ell
\end{align*}
for all $x_1,\ldots,x_k,y_1,\ldots,y_\ell\in\mathfrak{g}$,
as well as
$$
\llbracket x,\lambda\rrbracket
=\llbracket\lambda,x\rrbracket
=\llbracket\lambda,\mu\rrbracket
=0
$$
for all $x\in\mathfrak{g}$ and $\lambda,\nu\in\Bbbk$,
we obtain a Gerstenhaber bracket $\llbracket\cdot,\cdot\rrbracket$ 
on $\Lambda^\bullet\mathfrak{g}$.
An element $r\in\mathfrak{g}\wedge\mathfrak{g}$
is said to be a \textit{classical $r$-matrix} if
\begin{equation}\label{eq44}
    \llbracket r,r\rrbracket=0.
\end{equation}
Clearly every scalar multiple of a classical $r$-matrix is a classical
$r$-matrix as well. In particular, $0$ is a trivial solution.
The equation (\ref{eq44}) is called \textit{classical Yang-Baxter equation}
on $\Lambda^3\mathfrak{g}$.
It is the skew-symmetrization of the 
\textit{classical Yang-Baxter equation} $\mathrm{CYB}(r)=0$ on
$\mathfrak{g}^{\otimes 3}$, where
\begin{align*}
    \mathrm{CYB}\colon\mathfrak{g}^{\otimes 2}\ni r&\mapsto
    [r_{12},r_{13}]+[r_{12},r_{23}]+[r_{13},r_{23}]\\
    =&[r_1,r'_1]\otimes r_2\otimes r'_2
    +r_1\otimes[r_2,r'_1]\otimes r'_2
    +r_1\otimes r'_1\otimes[r_2,r'_2]
    \in\mathfrak{g}^{\otimes 3}
\end{align*}
is said to be the \textit{classical Yang-Baxter map}.
\begin{example}
Consider the $\Bbbk$-Lie algebra $\mathfrak{g}$ generated by two elements
$H,E\in\mathfrak{g}$, such that $[H,E]=2E$. Then
$r=H\wedge E$ is a classical $r$-matrix on $\mathfrak{g}$. This is obvious since
$\Lambda^3\mathfrak{g}=\{0\}$.
\end{example}
In the next
theorem we prove that classical $r$-matrices naturally appear as the
(skew-symmetrization of the) classical limit of universal
$\mathcal{R}$-matrices on $\mathscr{U}\mathfrak{g}[[\hbar]]$
(c.f. \cite{ES2010}~Prop.~9.5).
\begin{theorem}\label{thm06}
Let $\mathcal{R}=1\otimes 1+\hbar\tilde{r}+\mathcal{O}(\hbar^2)$
be a universal $\mathcal{R}$-matrix on $\mathscr{U}\mathfrak{g}[[\hbar]]$.
Then
\begin{equation}
    r=\frac{\tilde{r}-\tilde{r}_{21}}{2}\in\Lambda^2\mathfrak{g}
\end{equation}
is a classical $r$-matrix on $\mathfrak{g}$.
\end{theorem}
\begin{proof}
We first prove $\tilde{r}
=r_1\otimes r_2\in\mathfrak{g}\otimes\mathfrak{g}$ by showing that
$r_i$ are primitive elements of $\mathscr{U}\mathfrak{g}$. The hexagon
relations $(\Delta\otimes\mathrm{id})(\mathcal{R})
=\mathcal{R}_{13}\mathcal{R}_{23}$ of $\mathcal{R}$
and
$(\mathrm{id}\otimes\Delta)(\mathcal{R})
=\mathcal{R}_{13}\mathcal{R}_{12}$
read
\begin{align*}
    \sum_{n=0}^\infty\hbar^n\Delta(R_1^n)\otimes R_2^n
    =\sum_{n=0}^\infty\hbar^n\sum_{m=0}^n
    (R_1^m\otimes 1\otimes R_2^m)
    (1\otimes R_1^{'m-n}\otimes R_2^{'m-n})
\end{align*}
and
\begin{align*}
    \sum_{n=0}^\infty\hbar^nR_1^n\otimes\Delta(R_2^n)
    =\sum_{n=0}^\infty\hbar^n\sum_{m=0}^n
    (R_1^m\otimes 1\otimes R_2^m)
    (R_1^{'m-n}\otimes R_2^{'m-n}\otimes 1),
\end{align*}
where we denoted $\mathcal{R}=\sum_{n=0}^\infty\hbar^nR_1^n\otimes R_2^n$.
In order one of $\hbar$ this gives
\begin{align*}
    r_{1(1)}\otimes r_{1(2)}\otimes r_2
    =(r_1\otimes 1+1\otimes r_1)\otimes r_2
\end{align*}
and
\begin{align*}
    r_1\otimes r_{2(1)}\otimes r_{2(2)}
    =r_1\otimes(r_2\otimes 1+1\otimes r_2),
\end{align*}
since $R_1^1\otimes R_2^1=1\otimes 1$ and
$R_1^1\otimes R_2^1=r_1\otimes r_2$. We conclude that
$\tilde{r}=r_1\otimes r_2\in\mathfrak{g}\otimes\mathfrak{g}$.
In particular, its skew-symmetrization $r$ is an element of $\mathfrak{g}
\wedge\mathfrak{g}$. Recall that $\mathcal{R}$ satisfies the
quantum Yang-Baxter equation
$\mathcal{R}_{12}\mathcal{R}_{13}\mathcal{R}_{23}
=\mathcal{R}_{23}\mathcal{R}_{13}\mathcal{R}_{12}$. In order two
of $\hbar$ this equation reads
\begin{align*}
    (v_1\otimes v_2\otimes 1)
    &+(v_1\otimes 1\otimes v_2)
    +(1\otimes v_1\otimes v_2)\\
    &+(r_1\otimes r_2\otimes 1)(r'_1\otimes 1\otimes r'_2
    +1\otimes r'_1\otimes r'_2)
    +(r_1\otimes 1\otimes r_2)(1\otimes r'_1\otimes r'_2)\\
    =&(v_1\otimes v_2\otimes 1)
    +(v_1\otimes 1\otimes v_2)
    +(1\otimes v_1\otimes v_2)\\
    &+(1\otimes r_1\otimes r_2)(r'_1\otimes 1\otimes r'_2
    +r'_1\otimes r'_2\otimes 1)
    +(r_1\otimes 1\otimes r_2)(r'_1\otimes r'_2\otimes 1),
\end{align*}
where we denoted $\mathcal{R}=1\otimes 1+\hbar r_1\otimes r_2
+\hbar^2v_1\otimes v_2+\mathcal{O}(\hbar^3)$. The first three terms on
both sides of the equation cancel each other and the remaining terms
give $\mathrm{CYB}(r_1\otimes r_2)=0$ after applying the relation
$xy-yx=[x,y]$ for $x,y\in\mathfrak{g}$, which holds in
$\mathscr{U}\mathfrak{g}$. In particular, the skew-symmetrization
$r$ of $\tilde{r}=r_1\otimes r_2$ satisfies $\llbracket r,r\rrbracket=0$,
i.e. $r$ is a classical $r$-matrix.
\end{proof}
Note that we even proved a stronger statement: the first order of any
universal $\mathcal{R}$-matrix is a solution of the classical
Yang-Baxter equation on $\mathfrak{g}^{\otimes 3}$.
The correspondence of the classical Yang-Baxter equation and quantum
Yang-Baxter equation is further discussed in \cite{Dr85}.
Since any Drinfel'd twist $\mathcal{F}$ on a cocommutative Hopf algebra
leads to a universal $\mathcal{R}$-matrix $\mathcal{F}_{21}\mathcal{F}^{-1}$
on the twisted Hopf algebra, a result similar to Theorem~\ref{thm06} holds for
the first order of a Drinfel'd twist on $\mathscr{U}\mathfrak{g}[[\hbar]]$.
\begin{corollary}[\cite{Dr83}~Thm.~5(a)]
Let $\mathcal{F}=1\otimes 1+\hbar\tilde{r}+\mathcal{O}(\hbar^2)$
be a Drinfel'd twist on $\mathscr{U}\mathfrak{g}[[\hbar]]$. Then
\begin{equation}\label{eq45}
    r=\tilde{r}_{21}-\tilde{r}\in\Lambda^2\mathfrak{g}
\end{equation}
is a classical $r$-matrix on $\mathfrak{g}$.
\end{corollary}
\begin{proof}
Since $\mathcal{R}=\mathcal{F}_{21}\mathcal{F}^{-1}$ is a universal
$\mathcal{R}$-matrix, Theorem~\ref{thm06} implies that the 
skew-symmetrization of its first order is a classical $r$-matrix.
Let $\mathcal{F}=\sum_{n=0}^\infty\hbar^nF_1^n\otimes F_2^n$
and $\mathcal{F}^{-1}=\sum_{n=0}^\infty\hbar^n\overline{F}_1^n
\otimes\overline{F}_2^n$.
Then, the first order in $\hbar$ of 
$$
\mathcal{R}
=\sum_{n=0}^\infty\hbar^n\sum_{m=0}^n
(F_2^m\otimes F_1^m)(\overline{F}_1^{n-m}\otimes\overline{F}_1^{n-m})
$$
is $F_2^1\otimes F_1^1+\overline{F}_1^1\otimes\overline{F}_2^1
=\tilde{r}_{21}-\tilde{r}
\in\mathfrak{g}\otimes\mathfrak{g}$, where we used that the first order
of the inverse of $\mathcal{F}$ is given by $-\tilde{r}$.
\end{proof}
In the next remark we resume the thought of the introduction of this section
and identify twists with invariant star products.
Furthermore, we discuss several important constructions and
classifications of Drinfel'd twist on universal enveloping algebras.
\begin{remark}\label{remark01}
\begin{enumerate}
\item[i.)] Consider a $G$-invariant star product $\star$ on a Lie group
$G$, where the Lie group action is given by left multiplication
$\ell_g\colon G\ni h\mapsto gh\in G$ on $G$. It
corresponds to a Drinfel'd twist $\mathcal{F}=1\otimes 1+\mathcal{O}(\hbar)
\in(\mathscr{U}\mathfrak{g}\otimes\mathscr{U}\mathfrak{g})[[\hbar]]$, where
$\mathfrak{g}$ is the Lie algebra corresponding to $G$ \cite{BBG1998,BCG1997}.
To see this, recall that for every $k>0$ there is an isomorphism
$$
\Gamma^\infty(TG^{\otimes k})^G\ni X\mapsto X(e)
\in T_eG^{\otimes k}
$$
between $G$-invariant sections $\Gamma^\infty(TG^{\otimes k})^G$
and the tangent space at the unit element $e\in G$. The inverse is given by
$T_eG^{\otimes k}\ni\xi\mapsto X_\xi\in\Gamma^\infty(TG^{\otimes k})^G$,
where
$
X_\xi(g)=(T_e\ell_g)^{\otimes k}\xi
$
for all $g\in G$ and $T_e\ell_g\colon T_eG\rightarrow T_gG$ denotes the
tangent map of $\ell_g$ at $e$. Since this isomorphism is in fact a Lie algebra
isomorphism, i.e. $[X_\xi,X_\eta]=X_{[\xi,\eta]}$ for all 
$\xi,\eta\in\mathfrak{g}$, it extends to an isomorphism of
$G$-invariant bidifferential operators and elements in 
$\mathscr{U}\mathfrak{g}^{\otimes 2}$, leading to the stated correspondence
$\star\mapsto\mathcal{F}$. Since $f\star g=f\cdot g
+\mathcal{O}(\hbar)$ it follows that $\mathcal{F}=1\otimes 1+\mathcal{O}(\hbar)$,
while $1\star f=f=f\star 1$ leads to $(\epsilon\otimes\mathrm{id})(\mathcal{F})
=1=(\mathrm{id}\otimes\epsilon)(\mathcal{F})$ and the associativity of
$\star$ corresponds to the $2$-cocycle condition of $\mathcal{F}$. It
is clear that the data of a $G$-invariant star product on $G$ and
a Drinfel'd twist $\mathcal{F}=1\otimes 1+\mathcal{O}(\hbar)$ on
$\mathscr{U}\mathfrak{g}$ are equivalent. In the same spirit, classical
$r$-matrices on $\mathfrak{g}$ correspond to $G$-invariant Poisson bivectors
on $G$.
Using the right multiplication one proves that right invariant Poisson
bivectors on $G$ also correspond to classical $r$-matrices and that the
difference of an induced right invariant and induced left invariant Poisson bivector
structures $G$ as a Poisson-Lie group (see e.g. \cite{ES2010}~Prop.~3.1);

\item[ii.)] In \cite{Dr83}~Thm.~6 it is proven that for any 
(non-degenerate) classical $r$-matrix $r\in\mathfrak{g}\wedge\mathfrak{g}$
there exists a
twist $\mathcal{F}=1\otimes 1+\frac{\hbar}{2}r+\mathcal{O}(\hbar^2)$
on $\mathscr{U}\mathfrak{g}[[\hbar]]$. This is done by extending the
Lie algebra $\mathfrak{g}$ via the inverse of $r$, constructing the
Gutt star product on the dual Lie algebra of the extension and restricting
this product to an affine subspace, which is locally diffeomorphic to $G$.
In the end, the identification $i.)$ is employed;

\item[iii.)] In \cite{Jonas2017} there is an alternative proof of
\cite{Dr83}~Thm.~6, using a Fedosov-like construction. Furthermore,
in \cite{Jonas2017}~Thm.~5.6 the authors give a classification of twists
on $\mathscr{U}\mathfrak{g}[[\hbar]]$ which quantize a fixed non-degenerate
classical $r$-matrix $r\in\mathfrak{g}\wedge\mathfrak{g}$. Namely,
the twists on $\mathscr{U}\mathfrak{g}[[\hbar]]$ quantizing $r$ are in
bijection with the second Chevalley-Eilenberg cohomology
$\mathrm{H}^2_{\mathrm{CE}}(\mathfrak{g})[[\hbar]]$. Note that this
classification is undertaken up to equivalence, where two
twists $\mathcal{F}=1\otimes 1+\mathcal{O}(\hbar)$ and
$\mathcal{F}'=1\otimes 1+\mathcal{O}(\hbar)$ are said to be equivalent
if there exists an element $S=1+\mathcal{O}(\hbar)
\in\mathscr{U}\mathfrak{g}[[\hbar]]$, such that $\epsilon(S)=1$ and
$
\Delta(S)\mathcal{F}'
=\mathcal{F}(S\otimes S).
$
This implies that the twist of Example~\ref{ex02}~iv.)
and the skew-symmetrization of the Jordanian twist of
Example~\ref{ex02}~v.) are equivalent, since the second
Chevalley-Eilenberg cohomology of the underlying Lie algebra is trivial
(c.f. \cite{Jonas2017}~Ex.~5.7);

\item[iv.)] The quantization of classical $r$-matrices is closely related to
the quantization of Lie bialgebras. In \cite{EK1996}~Thm~6.1,
the authors give a functor construction to quantize any quasi-triangular
Lie bialgebra. In particular, this construction provides a universal $R$-matrix
for every classical $r$-matrix;

\item[v.)] There is a more general notion of classical dynamical $r$-matrices,
which admit a quantization for reductive Lie algebras. This was proven in
\cite{Cal2006} by making use of formality theory;
\end{enumerate}
\end{remark}
Inspired by Remark~\ref{remark01}~i.) and Proposition~\ref{prop02}
we state the following definitions. The idea is to induce a star product
on a manifold $M$ from a Drinfel'd twist on $\mathscr{U}\mathfrak{g}$
via a Lie algebra action $\mathfrak{g}\rightarrow\Gamma^\infty(TM)$.
\begin{definition}[\cite{Thomas2016} Twist Star Product]
Let $(M,\pi)$ be a Poisson manifold.
\begin{enumerate}
\item[i.)] A (formal) Drinfel'd twist on $\mathscr{U}\mathfrak{g}$ is a
twist $\mathcal{F}=1\otimes 1+\mathcal{O}(\hbar)\in
(\mathscr{U}\mathfrak{g}\otimes\mathscr{U}\mathfrak{g})[[\hbar]]$
on $\mathscr{U}\mathfrak{g}[[\hbar]]$.

\item[ii.)] A symmetry $\mathfrak{g}$ of $M$ is a Lie algebra, acting on
$M$ by derivations, i.e. a Lie algebra anti-homomorphism 
$\phi\colon\mathfrak{g}\rightarrow\Gamma^\infty(TM)$.

\item[iii.)] A twist star product on $(M,\pi)$ is a star product $\star$
on $(M,\pi)$ together with a symmetry $\mathfrak{g}$ of $M$ and a formal
Drinfel'd twist $\mathcal{F}$ on $\mathscr{U}\mathfrak{g}$ such that
\begin{equation}
    f\star g
    =(\mathcal{F}_1^{-1}\rhd f)\cdot(\mathcal{F}_2^{-1}\rhd g)
\end{equation}
for all $f,g\in\mathscr{C}^\infty(M)$, where $\rhd$ denotes the
Hopf algebra
action induced by the symmetry.
\end{enumerate}
\end{definition}
By Proposition~\ref{prop02}, for any symmetry $\mathfrak{g}$ on a
manifold $M$ and any (formal) Drinfel'd twist $\mathcal{F}$ on
$\mathscr{U}\mathfrak{g}$ there is a twist star product
$f\star_\mathcal{F}g=(\mathcal{F}_1^{-1}\rhd f)\cdot(\mathcal{F}_2^{-1}\rhd g)$
on $M$ with Poisson bracket given by the skew-symmetrization of the
first order of $\mathcal{F}^{-1}$. It is the strategy of the rest of the
section to investigate such Poisson brackets and find obstructions for their
existence, leading to obstructions for twist star products. Before continuing,
we first give an example of a twist star product.
\begin{example}\label{example04}
Consider $\mathbb{R}^2$ with coordinates $(x,y)$. The $2$-torus $\mathbb{T}^2$
is the quotient of $\mathbb{R}^2$ under the identification
$(x,y)\sim(x+1,y)\sim(x,y+1)$. It is a symplectic manifold with respect
to the Poisson bracket
$$
\{f,g\}
=\frac{\partial f}{\partial x}\frac{\partial g}{\partial y}
-\frac{\partial f}{\partial y}\frac{\partial g}{\partial x}
$$
for all $f,g\in\mathscr{C}^\infty(\mathbb{T}^2)$.
The Lie algebra $\mathbb{R}^2$ endowed with
the zero Lie bracket is a symmetry of $\mathbb{T}^2$ with respect to the
usual action of coordinate vector fields $\frac{\partial}{\partial x}$
and $\frac{\partial}{\partial y}$. Consequently, $\mathscr{U}\mathfrak{g}$
can be identified with the symmetric algebra $\mathrm{S}^\bullet\mathfrak{g}$
of $\mathfrak{g}$. Then, the Moyal-Weyl product
\begin{equation}
    (f\star g)(x,y)
    =\exp\bigg(\hbar\frac{\partial}{\partial x}\frac{\partial}{\partial y}\bigg)
    f(x,v)g(u,y)|_{(u,v)=(x,y)}
\end{equation}
on $\mathbb{T}^2$, where $f,g\in\mathscr{C}^\infty(\mathbb{T}^2)$, is
a twist star product corresponding to the (formal) Drinfel'd twist
\begin{equation}\label{TwistTorus}
    \mathcal{F}=\exp\bigg(-\hbar\frac{\partial}{\partial x}
    \otimes\frac{\partial}{\partial y}\bigg)
    \in(\mathscr{U}\mathfrak{g}\otimes\mathscr{U}\mathfrak{g})[[\hbar]]
\end{equation}
on $\mathscr{U}\mathfrak{g}$.
Similarly, the Moyal-Weyl product on $\mathbb{R}^n$
given in Example~\ref{example03}~i.) is a twist star product with 
respect to the Drinfel'd twist (\ref{TwistTorus}).
\end{example}
A twist star product puts quite strict requirements on the underlying manifold.
As a first instance of this, we prove that the Poisson bracket
corresponding to a twist star product is induced by a classical $r$-matrix
(c.f. \cite{Thomas2016}~Lem.~2.7).
\begin{lemma}\label{lemma18}
Let $\star_\mathcal{F}$ be a twist star product on a Poisson manifold
$(M,\{\cdot,\cdot\})$. Then
\begin{equation}
    \{f,g\}
    =(r_1\rhd f)\cdot(r_2\rhd g)
\end{equation}
for all $f,g\in\mathscr{C}^\infty(M)$, where $r=r_1\wedge r_2$ denotes the
$r$-matrix corresponding to $\mathcal{F}$, $\cdot$ the pointwise product of
functions and $\rhd$ the induced Hopf algebra action.
\end{lemma}
\begin{proof}
Since $\star_\mathcal{F}$ is a twist star product
$$
\sum_{n=0}^\infty\hbar^n(\overline{F}_1^n\rhd f)
\cdot(\overline{F}_2^n\rhd g)
=(\mathcal{F}_1^{-1}\rhd f)\cdot(\mathcal{F}_2^{-1}\rhd g)
=f\star_\mathcal{F} g
=\sum_{n=0}^\infty\hbar^nC_n(f,g)
$$
holds for all $f,g\in\mathscr{C}^\infty(M)$, where $C_n$ are the
bidifferential operators corresponding to $\star_\mathcal{F}$ and
$\mathcal{F}^{-1}=\sum_{n=0}^\infty\hbar^n
\overline{F}_1^n\otimes\overline{F}_2^n$. In particular,
$C_1(f,g)=(\overline{F}_1^1\rhd f)
\cdot(\overline{F}_2^1\rhd g)$ and the Poisson bracket is
\begin{align*}
    \{f,g\}
    =&C_1(f,g)-C_1(g,f)\\
    =&(\overline{F}_1^1\rhd f)
    \cdot(\overline{F}_2^1\rhd g)
    -(\overline{F}_1^1\rhd g)
    \cdot(\overline{F}_2^1\rhd f)\\
    =&\mu_0(
    (\overline{F}_1^1\otimes\overline{F}_2^1
    -\overline{F}_2^1\otimes\overline{F}_1^1)
    \rhd(f\otimes g))\\
    =&\mu_0(r\rhd(f\otimes g))\\
    =&(r_1\rhd f)\cdot(r_2\rhd g),
\end{align*}
where we used that the pointwise product $\mu_0$ is commutative.
\end{proof}
More generally, the same statement holds true if a formal deformation is
braided commutative (see \cite{Thomas2016}~Prop.~2.10).
\begin{proposition}
Let $\star$ be a star product on a Poisson manifold $(M,\{\cdot,\cdot\})$
such that $(\mathscr{C}^\infty(M)[[\hbar]],\star)$ is braided
commutative with respect to a universal $\mathcal{R}$-matrix
$\mathcal{R}=1\otimes 1+\hbar r+\mathcal{O}(\hbar^2)$. Then
$r=r_1\wedge r_2$ is a classical $r$-matrix on $\mathfrak{g}$ and
\begin{equation}
    \{f,g\}
    =(r_1\rhd f)\cdot(r_2\rhd g)
\end{equation}
for all $f,g\in\mathscr{C}^\infty(M)$.
\end{proposition}
\begin{proof}
Let $f,g\in\mathscr{C}^\infty(M)$.
For $f\star g=\sum_{n=0}^\infty\hbar^nC_n(f,g)$ we obtain
\begin{align*}
    \hbar\{f,g\}+\mathcal{O}(\hbar^2)
    =&f\star g-g\star f\\
    =&f\star g-(\mathcal{R}_1^{-1}\rhd f)\star(\mathcal{R}_2^{-1}\rhd g)\\
    =&\hbar(C_1(f,g)-C_1(f,g)+(r_1\rhd f)\cdot(r_2\rhd g))
    +\mathcal{O}(\hbar^2)\\
    =&\hbar(r_1\rhd f)\cdot(r_2\rhd g)+\mathcal{O}(\hbar^2),
\end{align*}
which implies the statement. Note that we used that the first order in
$\hbar$ of $\mathcal{R}^{-1}$ is $-r$. In particular $r$ is 
skew-symmetric since the Poisson bracket is.
\end{proof}
If $\mathfrak{g}$ is a finite-dimensional Lie algebra over $\mathbb{C}$
with basis $e_1,\ldots,e_n\in\mathfrak{g}$, any $r$-matrix $r$ on $\mathfrak{g}$
reads
$$
r=\sum_{i<j}r^{ij}e_i\wedge e_j
$$
with coefficients $r^{ij}=-r^{ji}\in\mathbb{C}$. The $r$-matrix $r$ is said to be
\textit{non-degenerate} if the matrix $(r^{ij})_{ij}\in M_n(\mathbb{C})$
is of full rank. There is always a finite-dimensional Lie subalgebra of $\mathfrak{g}$
on which $r$ is non-degenerate.
\begin{lemma}[\cite{Thomas2016}~Prop.~2.5]\label{lemma22}
Let $r\in\mathfrak{g}$ be an $r$-matrix on $\mathfrak{g}$. Then
\begin{equation}
    \mathfrak{g}_r
    =\{(\alpha\otimes\mathrm{id})(r)\in\mathfrak{g}
    ~|~\alpha\in\mathfrak{g}^*\}
    \subseteq\mathfrak{g}
\end{equation}
is a finite-dimensional Lie subalgebra. Furthermore, $r$ is a non-degenerate
$r$-matrix on $\mathfrak{g}_r$. In particular, $r\in\mathfrak{g}_r$.
\end{lemma}
\begin{proof}
We first prove that $\mathfrak{g}_r$ is a Lie subalgebra. Let 
$\alpha,\beta\in\mathfrak{g}^*$. Then
\begin{align*}
    [(\alpha\otimes\mathrm{id})(r),(\beta\otimes\mathrm{id})(r)]
    =&\alpha(r_1)\beta(r'_1)[r_2,r'_2]\\
    =&(\alpha\otimes\beta\otimes\mathrm{id})
    (r_1\otimes r'_1\otimes[r_2,r'_2])\\
    =&(\alpha\otimes\beta\otimes\mathrm{id})
    (-[r_1,r'_1]\otimes r_2\otimes r'_2
    -r_1\otimes[r_2,r'_1]\otimes r'_2)\\
    =&((-\beta(r_2)\alpha\circ\mathrm{ad}_{r_1}
    -\alpha(r_1)\beta\circ\mathrm{ad}_{r_2})
    \otimes\mathrm{id})(r)\in\mathfrak{g}_r,
\end{align*}
since $r$ is a solution of the classical Yang-Baxter equation.
Furthermore, $r$ spans the Lie subalgebra $\mathfrak{g}_r$ and it follows that
 $r\in\mathfrak{g}_r$ and $r$ is non-degenerate in $\mathfrak{g}_r$.
\end{proof}
The Lie algebra $\mathfrak{g}_r$ is said to be the \textit{symplectic leaf}
corresponding to $r$. It is also employed in \cite{ES2010}~Sec.~3.5.
We prove another condition which twist star products
enforce on their underlying manifold in the symplectic case.
\begin{proposition}[\cite{Thomas2016}~Thm.~3.3]\label{prop15}
Let $(M,\{\cdot,\cdot\})$ be a connected symplectic manifold allowing
for a twist star product $\star_\mathcal{F}$, where $\mathcal{F}\in
(\mathscr{U}\mathfrak{g}\otimes\mathscr{U}\mathfrak{g})[[\hbar]]$.
If the corresponding Lie algebra action $\phi\colon\mathfrak{g}\rightarrow
\Gamma^\infty(TM)$ integrates to a Lie group action 
$G\times M\rightarrow M$, $M$ is a homogeneous $G$-space.
\end{proposition}
\begin{proof}
In Lemma~\ref{lemma18} it was proven that 
$\{\cdot,\cdot\}=(r_1\rhd\cdot)(r_2\rhd\cdot)$ in this situation.
Equivalently, the corresponding Poisson bivector reads
$\pi=\sum_{i<j}r^{ij}\phi(e_i)\wedge\phi(e_j)$. Since $\pi$ is symplectic,
the musical homomorphism $\pi^\sharp\colon T^*M\rightarrow TM$ is surjective
on fibers. Namely, for any $p\in M$ and $v_p\in T_pM$ there is an
$\alpha_p\in T_p^*M$ such that
\begin{align*}
    v_p
    =&-\pi^\sharp|_p(\alpha_p)
    =-\pi_p(\cdot,\alpha_p)
    =\pi_p(\alpha_p,\cdot)
    =(\alpha_p\otimes\mathrm{id})\pi_p\\
    =&\sum_{i<j}r^{ij}\alpha_p(\phi(e_i)|_p)\phi(e_j)|_p\\
    =&\phi\bigg(\sum_{i<j}r^{ij}\alpha_p(\phi(e_i)|_p)e_j\bigg)\bigg|_p.
\end{align*}
This proves that $\phi\colon\mathfrak{g}\rightarrow\Gamma^\infty(TM)$ is
locally transitive, i.e. that $\phi|_p\colon\mathfrak{g}\rightarrow
T_pM$ is transitive for all $p\in M$. By assumption $\phi$ integrates to a
Lie group action $\Phi\colon G\times M\rightarrow M$. This means that
$G$ is the connected and simply connected Lie group corresponding to
$\mathfrak{g}$ and
$$
\frac{\mathrm{d}}{\mathrm{d}t}\bigg|_{t=0}
\Phi(\exp(t\xi),p)
=\phi|_p(\xi)\in T_pM
$$
for all $p\in M$ and $\xi\in\mathfrak{g}$. Since $\phi$ is locally transitive
and $M$ connected, it follows that $\Phi\colon G\times M\rightarrow M$
is a transitive group action (c.f. \cite{Thomas2016}~Lem.~3.2).
In other words, $M$ is a homogeneous
$G$-space. This concludes the proof of the proposition.
\end{proof}
If $M$ is a compact manifold, the additional integration assumption is
redundant, leading us to the first examples of star products which
can not be induced by Drinfel'd twists.
\begin{corollary}\label{cor02}
Let $(M,\{\cdot,\cdot\})$ be a connected compact symplectic manifold allowing
for a twist star product $\star_\mathcal{F}$, where $\mathcal{F}\in
(\mathscr{U}\mathfrak{g}\otimes\mathscr{U}\mathfrak{g})[[\hbar]]$.
Then $M$ is a homogeneous $G$-space.
\end{corollary}
\begin{proof}
By Palais' theorem \cite{Palais}, a Lie algebra action integrates to a Lie group 
action if all
corresponding fundamental vector fields have complete flow. This
is the case, since $M$ is compact (c.f. \cite{Lee2003}~Thm.~12.12).
In particular, we can apply Proposition~\ref{prop15} to obtain the result.
\end{proof}
\begin{example}[\cite{Thomas2016}~Ex.~3.9]
The connected orientable Riemann surfaces $T(g)$ of genus $g>1$ are not homogeneous.
This follows from theorems \cite{MostowI,MostowII}
of Mostow, which say that connected compact homogeneous spaces have
non-negative Euler characteristic. Of course the Euler characteristic of
$T(g)$ is $\chi(T(g))=2(1-g)$. On the other hand, the canonical symplectic
structure on $T(g)$, which we discussed in Example~\ref{example06}, admits 
a star product quantization according to Proposition~\ref{prop19}. By
Corollary~\ref{cor02} we conclude that those star products can not be
induced by a Drinfel'd twist on a universal enveloping algebra.
\end{example}
In Example~\ref{example04} we recognized that the Moyal-Weyl product on the
symplectic $2$-torus $T(1)=\mathbb{T}^2$
is a twist star product. The question remains if there are twist star products
on the symplectic $2$-sphere $T(0)=\mathbb{S}^2$. We address this question
in the rest of this section.
\begin{lemma}[\cite{Thomas2016}~Prop.~3.6]\label{lemma19}
Let $\Phi\colon G\times M\rightarrow M$ be a Lie group action with
corresponding Lie algebra action
$\phi\colon\mathfrak{g}\rightarrow\Gamma^\infty(TM)$. Then the following
statements hold.
\begin{enumerate}
\item[i.)] If $\Phi$ is transitive, the induced Lie group action
$$
\Psi\colon G/\ker\Phi\otimes M\ni([g],p)\mapsto\Phi(g,p)\in M
$$
is transitive and effective.

\item[ii.)] If $r\in\Lambda^2\mathfrak{g}$ is a classical $r$-matrix, so is
$[r]\in\Lambda^2\mathfrak{g}/\ker\phi$.

\item[iii.)] If $\phi(r):=\sum_{i<j}r^{ij}\phi(e_1)\wedge\phi(e_j)$ is a
Poisson bivector on $M$, so is $\phi([r])$. Moreover $\phi(r)=\phi([r])$.
\end{enumerate}
\end{lemma}
\begin{proof}
For the first statement recall that $\ker\Phi=\{g\in G~|~\Phi_g=\mathrm{id}_M\}$
is a normal Lie subgroup of $G$, i.e. $\ker\Phi$ is a Lie group and
$ghg^{-1}\in\ker\Phi$ if $h\in\ker\Phi$ and $g\in G$. The latter is the case
since $\Phi_{ghg^{-1}}=\Phi_g\circ\Phi_h\circ\Phi_{g^{-1}}
=\Phi_g\circ\Phi_{g^{-1}}
=\mathrm{id}_M$. It follows that $G/\ker\Phi$ is a Lie group and there is
a homomorphism $G\rightarrow G/\ker\Phi$ of Lie groups with Lie algebra homomorphism
$\psi\colon\mathfrak{g}\rightarrow\mathfrak{g}/\ker\phi$. Transitivity and
effectiveness of $\Psi$ are easily verified. If $r\in\Lambda^2\mathfrak{g}$
is a classical $r$-matrix, i.e. $\llbracket r,r\rrbracket=0$, so is
$[r]=\psi(r)\in\Lambda^2\mathfrak{g}/\ker\phi$, since Lie algebra homomorphisms
respect the Lie bracket and consequently the Schouten-Nijenhuis bracket, i.e.
$$
\llbracket[r],[r]\rrbracket
=\llbracket\psi(r),\psi(r)\rrbracket
=\psi(\llbracket r,r\rrbracket)
=0.
$$
Then, also the last statement follows.
\end{proof}
Collecting the former results, we obtain more and more information on
the Lie group action corresponding to a twist star product.
\begin{corollary}\label{cor03}
If there is a twist star product on a connected compact symplectic
manifold $M$, there exists a non-degenerate classical $r$-matrix on a Lie algebra
$\mathfrak{g}$ such that the corresponding connected and simply connected Lie
group $G$ acts transitively and effectively on $M$.
\end{corollary}
\begin{proof}
By Corollary~\ref{cor02} there is a classical $r$-matrix 
$\tilde{r}\in\Lambda^2\tilde{\mathfrak{g}}$ on a Lie algebra $\tilde{\mathfrak{g}}$
such that the corresponding connected and simply connected Lie group $\tilde{G}$
acts transitively on $M$. Denote this group action by $\tilde{\Phi}$ and the
corresponding Lie algebra action by $\tilde{\phi}$.
Performing the quotient $G'=\tilde{G}/\ker\tilde{\Phi}$
we obtain a Lie group action $\Phi'\colon G'\times M\rightarrow M$,
which is well-defined, transitive and effective according to
Lemma~\ref{lemma19}, such that the induced $r$-matrix $r\in\Lambda^2\mathfrak{g}'$
induces the Poisson structure on $M$. This $r$-matrix is non-degenerate
in the Lie subalgebra $\mathfrak{g}_r\subseteq\mathfrak{g}'$ according to 
Lemma~\ref{lemma22}. The Lie group action of the corresponding Lie group
$G_r$ is still transitive and effective, which can be checked by repeating
the proof of Proposition~\ref{prop15} for $\mathfrak{g}_r$.
This concludes the proof of the corollary.
\end{proof}
Finally we are able to conclude the main obstruction of this section.
\begin{theorem}[\cite{Thomas2016}~Cor.~3.12]
There is no twist star product deforming the symplectic $2$-sphere.
\end{theorem}
\begin{proof}
Assume there is a twist star product on symplectic $\mathbb{S}^2$. Then
there is a non-degenerate classical $r$-matrix $r$ on a Lie algebra $\mathfrak{g}$
such that $G$ acts transitively and effectively on $\mathbb{S}^2$, according to
Corollary~\ref{cor03}. In particular, $\mathfrak{g}$ is semisimple according to
Onishchik \cite{OnishchikI,OnishchikII,OnishchikIII}. This gives a contradiction 
since there are no non-degenerate $r$-matrices on semisimple Lie algebras
(c.f. \cite{ES2010}~Prop.~5.2).
\end{proof}

\section{Morita Equivalence of Algebras}\label{SecObs3}

A weaker notion of equivalence of algebras than isomorphism is given by
Morita equivalence. It identifies two algebras if their categories of
representations are equivalent. There are several properties which are
preserved under Morita equivalence, like e.g. algebraic $K$-Theory
(c.f. \cite{Quillen}). Those invariants are said to be
\textit{Morita invariants}. We will see that commutativity is not one of them.
Following \cite{Bursztyn2004}~Sec.~3 and
\cite{Bursztyn2002}~Sec.~2. we recall some basic concepts of Morita theory,
which we apply to star product algebras and in particular to twist
star product algebras in Section~\ref{SecObs4}. For a general introduction
to Morita equivalence of rings we refer to \cite{Lam1999}.

Fix an associative unital algebra $\mathcal{A}$ over a commutative unital
ring $\Bbbk$ in the following. As usual, its category of representations
if denoted by ${}_\mathcal{A}\mathcal{M}$. If
$\mathcal{B}$ is another associative unital algebra isomorphic to $\mathcal{A}$,
it follows that ${}_\mathcal{A}\mathcal{M}$ and ${}_\mathcal{B}\mathcal{M}$
are equivalent categories. This means that there are two
functors $F\colon{}_\mathcal{A}\mathcal{M}\rightarrow{}_\mathcal{B}\mathcal{M}$
and $G\colon{}_\mathcal{B}\mathcal{M}\rightarrow{}_\mathcal{A}\mathcal{M}$,
as well as two natural isomorphisms 
$\mathrm{id}_{{}_\mathcal{A}\mathcal{M}}\rightarrow G\circ F$ and
$F\circ G\rightarrow\mathrm{id}_{{}_\mathcal{B}\mathcal{M}}$.
The following example shows that being isomorphic is not a
necessary condition for algebras to have equivalent categories of modules.
\begin{example}\label{example01}
Fix a natural number $n>0$. The set $M_n(\mathcal{A})$ of $n\times n$-matrices
with entries in $\mathcal{A}$ is itself an associative unital algebra
with product given by
matrix multiplication and unit being the $n\times n$-matrix $1_n$
with units on the diagonal and zeros in every other entry. An object
in ${}_{M_n(\mathcal{A})}\mathcal{M}$ can be identified with $n$ left
$\mathcal{A}$-modules $\mathcal{M}_i$, such that $\sum_{i=1}^na_{ji}\cdot m_i
\in\mathcal{M}_j$ for all $a_{ji}\in\mathcal{A}$ and $m_i\in\mathcal{M}_i$,
where $1\leq j\leq n$. It follows that $M_i=M_1$ for all $1\leq i\leq n$.
This means that any left $M_n(\mathcal{A})$-module is of the form $\mathcal{M}^n$
for a left $\mathcal{A}$-module $\mathcal{M}$.
Similarly one proves that any left $M_n(\mathcal{A})$-homomorphism
$\Phi\colon\mathcal{M}^n\rightarrow\mathcal{N}^n$ is of the form
$\Phi=\phi^n$ for a left $\mathcal{A}$-module homomorphism
$\phi\colon\mathcal{M}\rightarrow\mathcal{N}$. It is easy to prove that
the assignments $\mathcal{M}\mapsto\mathcal{M}^n$ and
$\phi\mapsto\phi^n$ define an invertible functor with inverse given by
the projection to the first component. This proves that
${}_\mathcal{A}\mathcal{M}$ and ${}_{M_n(\mathcal{A})}\mathcal{M}$
are equivalent categories. However, $\mathcal{A}$ and 
$M_n(\mathcal{A})$ are not isomorphic as algebras in general, since
$M_n(\mathcal{A})$ might be noncommutative even if $\mathcal{A}$ is
commutative.
\end{example}
Motivated from this example we state the following definition.
\begin{definition}
Two associative unital algebras $\mathcal{A}$ and $\mathcal{B}$ are said to be
Morita equivalent if ${}_\mathcal{A}\mathcal{M}$ and
${}_\mathcal{B}\mathcal{M}$ are equivalent categories.
\end{definition}
It follows from the previous discussion that two isomorphic algebras
are Morita equivalent and $\mathcal{A}$ is Morita equivalent to
$M_n(\mathcal{A})$ for every $n>0$.
There are several characterizations of Morita equivalence.
For instance, we note that the functor from
Example~\ref{example01}, which assigns to any left $\mathcal{A}$-module
$\mathcal{M}$ the $M_n(\mathcal{A})$-module $\mathcal{M}^n$ and to any
left $\mathcal{A}$-module homomorphism $\phi\colon\mathcal{M}\rightarrow
\mathcal{N}$ the left $M_n(\mathcal{A})$-module homomorphism
$\phi^n\colon\mathcal{M}^n\rightarrow\mathcal{N}^n$, can be represented
as the tensor product with the $M_n(\mathcal{A})$-$\mathcal{A}$-bimodule
${}_{M_n(\mathcal{A})}\mathcal{A}^n_\mathcal{A}$. The latter is defined
as $\mathcal{A}^n$, which is a left $M_n(\mathcal{A})$-module via
matrix-vector multiplication and a right $\mathcal{A}$-module by
component-wise multiplication from the right. Clearly the actions commute.
Furthermore, for any left $\mathcal{A}$-module $\mathcal{M}$ the
tensor product ${}_{M_n(\mathcal{A})}\mathcal{A}^n_\mathcal{A}
\otimes_\mathcal{A}\mathcal{M}$ over $\mathcal{A}$ is isomorphic to the left
$M_n(\mathcal{A})$-module $\mathcal{M}^n$. The left 
$M_n(\mathcal{A})$-module isomorphism is given by
$$
\mathcal{M}^n\ni
\begin{pmatrix}
m_1 \\
\vdots \\
m_n
\end{pmatrix}
\mapsto
\begin{pmatrix}
1 \\
0 \\
\vdots \\
0
\end{pmatrix}
\otimes_\mathcal{A}m_1
+\ldots+
\begin{pmatrix}
0 \\
\vdots \\
0 \\
1
\end{pmatrix}
\otimes_\mathcal{A}m_n
\in{}_{M_n(\mathcal{A})}\mathcal{A}^n_\mathcal{A}
\otimes_\mathcal{A}\mathcal{M},
$$
with inverse
$$
{}_{M_n(\mathcal{A})}\mathcal{A}^n_\mathcal{A}
\otimes_\mathcal{A}\mathcal{M}\ni
\begin{pmatrix}
a_1 \\
\vdots \\
a_n
\end{pmatrix}
\otimes_\mathcal{A}m\mapsto
\begin{pmatrix}
a_1\cdot m \\
\vdots \\
a_n\cdot m
\end{pmatrix}
\in\mathcal{M}^n.
$$
One easily proves that these assignments respect the left
$M_n(\mathcal{A})$-module actions. It turns out that this is not
a specific ramification of the Morita equivalence of $\mathcal{A}$
and $M_n(\mathcal{A})$ but rather a general construction underlying
every Morita equivalence.
\begin{proposition}[\cite{Bursztyn2002}~Cor.~2.4]
Two associative unital algebras $\mathcal{A}$ and $\mathcal{B}$ are Morita equivalent
if and only if there is a $\mathcal{B}$-$\mathcal{A}$-bimodule
${}_\mathcal{B}\mathcal{E}_\mathcal{A}$ and an
$\mathcal{A}$-$\mathcal{B}$-bimodule
${}_\mathcal{A}\mathcal{E}_\mathcal{B}$ such that
\begin{equation}\label{eq38}
    {}_\mathcal{A}\mathcal{E}_\mathcal{B}\otimes_\mathcal{B}
    {}_\mathcal{B}\mathcal{E}_\mathcal{A}\cong
    {}_\mathcal{A}\mathcal{A}_\mathcal{A}
    \text{ and }
    {}_\mathcal{B}\mathcal{E}_\mathcal{A}\otimes_\mathcal{A}
    {}_\mathcal{A}\mathcal{E}_\mathcal{B}
    \cong{}_\mathcal{B}\mathcal{B}_\mathcal{B}
\end{equation}
as bimodules.
\end{proposition}
The bimodules ${}_\mathcal{B}\mathcal{E}_\mathcal{A}$ and
${}_\mathcal{A}\mathcal{E}_\mathcal{B}$ are said to be \textit{Morita equivalence
bimodules}.
We want to point out that there is an interpretation of Morita equivalence
bimodules as invertible morphisms in the following category: objects are defined
as associative unital $\Bbbk$-algebras and morphisms are isomorphism classes of
bimodules. Namely, let $\mathcal{A}$, $\mathcal{B}$ and $\mathcal{C}$ be algebras
and fix an $\mathcal{A}$-$\mathcal{B}$-bimodule ${}_\mathcal{A}\mathcal{E}_\mathcal{B}$
and a $\mathcal{B}$-$\mathcal{C}$-bimodule ${}_\mathcal{B}\mathcal{E}_\mathcal{C}$.
Then $\otimes_\mathcal{B}$ is associative up to isomorphism, with
${}_\mathcal{A}\mathcal{A}_\mathcal{A}$ and ${}_\mathcal{B}\mathcal{B}_\mathcal{B}$
functioning as unit morphisms. A morphism ${}_\mathcal{A}\mathcal{E}_\mathcal{B}$
is invertible if and only if there is a morphism ${}_\mathcal{B}
\mathcal{E}_\mathcal{A}$ such that (\ref{eq38}) holds, i.e. if and only if
it is a Morita equivalence bimodule. This constitutes a (large)
groupoid $\mathrm{Pic}$ with objects being associative unital algebras and
(invertible) morphisms being Morita equivalence bimodules, called \textit{Picard
groupoid}. The set of $\mathcal{A}$-$\mathcal{B}$ Morita equivalence bimodules
is denoted by $\mathrm{Pic}(\mathcal{A},\mathcal{B})$, while the set of self-Morita
equivalence classes of $\mathcal{A}$, the \textit{Picard group}, is denoted by
$\mathrm{Pic}(\mathcal{A})$. The Morita equivalence class of an associative unital
algebra $\mathcal{A}$ is defined to be the orbit of $\mathcal{A}$ in $\mathrm{Pic}$
and $\mathrm{Pic}(\mathcal{A})$ measures in how many ways $\mathcal{A}$ is
Morita equivalent to another algebra $\mathcal{B}$ in its orbit. We refer to
\cite{Benabou} for more information on this interpretation of Morita
equivalence.

There is another characterization of Morita equivalence, affirming the
similarity to Example~\ref{example01}.
Recall that an element $P\in M_n(\mathcal{A})$ is said to be an \textit{idempotent}
if $P^2=P$ and it is said to be \textit{full} if the span of elements of the form
$MPN\in M_n(\mathcal{A})$ for $M,N\in M_n(\mathcal{A})$ equals $M_n(\mathcal{A})$.
\begin{theorem}\label{thm05}
Two associative unital $\Bbbk$-algebras $\mathcal{A}$ and $\mathcal{B}$
are Morita equivalent if and only if one of the following statements holds.
\begin{enumerate}
\item[i.)] there is an equivalence ${}_\mathcal{B}\mathcal{M}
\rightarrow{}_\mathcal{A}\mathcal{M}$ of categories;

\item[ii.)] there is an $\mathcal{A}$-$\mathcal{B}$-bimodule
${}_\mathcal{A}\mathcal{E}_\mathcal{B}$ such that 
$$
\mathcal{F}_{{}_\mathcal{A}\mathcal{E}_\mathcal{B}}
\colon{}_\mathcal{B}\mathcal{M}\ni{}_\mathcal{B}\mathcal{E}\mapsto
{}_\mathcal{A}\mathcal{E}_\mathcal{B}\otimes_\mathcal{B}{}_\mathcal{B}\mathcal{E}
\in{}_\mathcal{A}\mathcal{M}
$$
is an equivalence of categories;

\item[iii.)] there is an $\mathcal{A}$-$\mathcal{B}$-bimodule
${}_\mathcal{A}\mathcal{E}_\mathcal{B}$, which is finitely generated projective
as left $\mathcal{A}$-module and right $\mathcal{B}$-module such that
$$
\mathcal{A}\cong\mathrm{End}_\mathcal{B}({}_\mathcal{A}\mathcal{E}_\mathcal{B})
\text{ and }
\mathcal{B}\cong\mathrm{End}_\mathcal{A}({}_\mathcal{A}\mathcal{E}_\mathcal{B})
$$
are isomorphisms of algebras, where the first endomorphisms are right 
$\mathcal{B}$-linear and the latter left $\mathcal{A}$-linear;

\item[iv.)] there is an $n>0$ and a full idempotent $P\in M_n(\mathcal{A})$
such that 
$$
\mathcal{B}\cong
\mathrm{End}_\mathcal{A}(P\mathcal{A}^n)
=PM_n(\mathcal{A})P
$$
is an isomorphism of algebras;
\end{enumerate}
\end{theorem}
This follows from \cite{Bursztyn2002}~Thm.~2.6 and Thm.~2.8.
In a next step we examine that the center of an algebra is a Morita
invariant, where we proceed as in \cite{Bursztyn2004}~Sec.~3.1.
\begin{corollary}
If $\mathcal{A}$ and $\mathcal{B}$ are Morita equivalent, there is an isomorphism
$\mathcal{Z}(\mathcal{A})\cong\mathcal{Z}(\mathcal{B})$ of algebras.
In particular, two commutative algebras are Morita equivalent if and only if they are
isomorphic.
\end{corollary}
In other words, for every element of the Picard group $\mathrm{Pic}(\mathcal{A})$
there is an automorphism of the center of $\mathcal{A}$. This determines
a map $h\colon\mathrm{Pic}(\mathcal{A})
\rightarrow\mathrm{Aut}(\mathcal{Z}(\mathcal{A}))$. On the other hand,
every automorphism of $\mathcal{A}$ can be interpreted as a self-Morita
equivalence bimodule, defining a map
$j\colon\mathrm{Aut}(\mathcal{A})\rightarrow\mathrm{Pic}(\mathcal{A})$.
\begin{proposition}
There are two group homomorphisms
$$
j\colon\mathrm{Aut}(\mathcal{A})\rightarrow\mathrm{Pic}(\mathcal{A})
\text{ and }
h\colon\mathrm{Pic}(\mathcal{A})\rightarrow\mathrm{Aut}(\mathcal{Z}(\mathcal{A})).
$$
If $\mathcal{A}$ is commutative $h\circ j=\mathrm{id}_{\mathrm{Pic}(\mathcal{A})}$
and $\mathrm{Pic}(\mathcal{A})
=\mathrm{Aut}(\mathcal{A})\ltimes\ker h$.
\end{proposition}
In the above proposition, the action of an automorphism
$\Phi\in\mathrm{Aut}(\mathcal{A})$ on a self-Morita equivalence bimodule
$\mathcal{E}\in\ker(h)$ in the semidirect product
$\mathrm{Aut}(\mathcal{A})\ltimes\ker h$, is defined as the self-Morita
equivalence bimodule $\mathcal{E}^\Phi$, with left and right
$\mathcal{A}$-actions given by $a\cdot e\cdot b
=\Phi(a)\cdot e\cdot\Phi(b)$
for all $a,b\in\mathcal{A}$ and $e\in\mathcal{E}$.

Since we are mainly interested in the algebra of smooth functions on a manifold,
we are going to discuss its Morita equivalence classes and Picard group
in detail in the next example (c.f. \cite{Bursztyn2004}~Ex.~3.5).
\begin{example}\label{ex03}
Consider the associative unital algebra $\mathcal{A}=\mathscr{C}^\infty(M)$ of smooth
complex-valued functions on a smooth manifold $M$. If $E\rightarrow M$ is a smooth
complex vector bundle, its space $\Gamma^\infty(E)$ of smooth sections is a
Morita equivalence bimodule between $\Gamma^\infty(\mathrm{End}(E))$ and
$\mathcal{A}$. Moreover, every finitely generated projective $\mathcal{A}$-module is of
that form according to the Serre-Swan theorem (c.f. \cite{Jet2003}~Thm.~11.32).
In particular, we recover the Morita equivalence
of $\mathcal{A}$ and $M_n(\mathcal{A})$ for any $n>0$, by employing the trivial bundle
$E=\mathbb{C}^n\times M\rightarrow M$. The kernel of $h$ equals the group
$\mathrm{Pic}(M)$ of isomorphism classes of smooth complex line bundles on $M$.
Further remark that there is an isomorphism $c_1\colon\mathrm{Pic}(M)\rightarrow
H^2(M,\mathbb{Z})$, given by the Chern class map (see \cite{Hirze}~Sec.~3.8)
and
$\mathrm{Aut}(\mathcal{A})=\mathrm{Diff}(M)$. Summing up we obtain
$$
\mathrm{Pic}(\mathscr{C}^\infty(M))
=\mathrm{Diff}(M)\ltimes H^2(M,\mathbb{Z}),
$$
where $\mathrm{Diff}(M)$ acts on $H^2(M,\mathbb{Z})$ via pull-back.
\end{example}
In the next section we observe that the deformation theory is another
Morita invariant. This will have particular consequences for twist star
products.

\section{Twist Star Products and Morita Equivalence}\label{SecObs4}

We are mainly interested in Morita equivalence of star product algebras
(see \cite{Waldmann2011} for a review on this topic). An obvious question is,
if Morita equivalence bimodules can be deformed relative to a deformation
of the corresponding algebras. Even more, if $\mathcal{A}$ and $\mathcal{B}$
are Morita equivalent algebras and $\boldsymbol{\mathcal{A}}$ a formal deformation
of $\mathcal{A}$, is there a formal deformation $\boldsymbol{\mathcal{B}}$
of $\mathcal{B}$, such that $\boldsymbol{\mathcal{A}}$ and
$\boldsymbol{\mathcal{B}}$ are Morita equivalent? Following
\cite{Bursztyn2002}~Sec.~2.3 and \cite{BursztynWaldmann2000}, we give a
positive answer to this question. We further refer to
\cite{BursztynWaldmann2001,BursztynWaldmann2002}.
Afterwards we focus on formal deformations of
Morita equivalence  bimodules of twist star product algebras,
reviewing the results of \cite{dAWe17}.

Fix an associative unital algebra $\mathcal{A}$ and a left $\mathcal{A}$-module
$\mathcal{M}$ for the moment. Assume that there is a formal deformation $\star$ of
$\mathcal{A}$. The natural question arises, if there exists a left
$\boldsymbol{\mathcal{A}}=(\mathcal{A}[[\hbar]],\star)$-module structure $\bullet$
on the $\mathbb{K}[[\hbar]]$-module $\mathcal{M}[[\hbar]]$, such that
$a\bullet m=a\cdot m+\mathcal{O}(\hbar)$ for all $a\in\mathcal{A}$ and
$m\in\mathcal{M}$, where $\cdot$ denotes the left $\mathcal{A}$-module action
on $\mathcal{M}$. If there are $\mathbb{K}$-bilinear maps
$\lambda_r\colon\mathcal{A}\times\mathcal{M}\rightarrow\mathcal{M}$ such that
$$
a\bullet m
=a\cdot m+\sum_{r>0}\hbar^r\lambda_r(a,m),
$$
we call $\boldsymbol{\mathcal{M}}=(\mathcal{M}[[\hbar]],\bullet)$ a 
\textit{formal deformation}
of $(\mathcal{M},\cdot)$ with respect to $\boldsymbol{\mathcal{A}}$.
Two formal deformations $\boldsymbol{\mathcal{M}}=(\mathcal{M}[[\hbar]],\bullet)$ and
$\boldsymbol{\mathcal{M}}'=(\mathcal{M}[[\hbar]],\bullet')$ of $(\mathcal{M},\cdot)$
with respect to $\boldsymbol{\mathcal{A}}$
are said to be \textit{equivalent} if there are $\mathbb{K}$-linear maps 
$T_r\colon\mathcal{M}\rightarrow\mathcal{M}$ such that
$\boldsymbol{T}=\mathrm{id}_\mathcal{M}
+\sum_{r>0}\hbar^rT_r$ extends to an $\boldsymbol{\mathcal{A}}$-module
isomorphism $\boldsymbol{T}\colon\boldsymbol{\mathcal{M}}
\rightarrow\boldsymbol{\mathcal{M}}'$.
\begin{lemma}[c.f. \cite{FedosovIndex}]
Let $\boldsymbol{\mathcal{A}}=(\mathcal{A}[[\hbar]],\star)$ be a formal deformation
of $\mathcal{A}$. Then $M_n(\boldsymbol{\mathcal{A}})\cong M_n(\mathcal{A})[[\hbar]]$
as $\mathbb{K}[[\hbar]]$-modules and $M_n(\boldsymbol{\mathcal{A}})$ is a formal
deformation of $M_n(\mathcal{A})$. Furthermore, if $P\in M_n(\mathcal{A})$ is an
idempotent,
\begin{equation}
    \boldsymbol{P}
    =\frac{1}{2}+\bigg(P-\frac{1}{2}\bigg)
    \star\frac{1}{\sqrt[\star]{1+4(P\star P-P)}}
\end{equation}
defines an idempotent on $M_n(\boldsymbol{\mathcal{A}})$ such that
$\boldsymbol{P}=P+\mathcal{O}(\hbar)$ and $P$ is full if and only if
$\boldsymbol{P}$ is full.
\end{lemma}
In particular, for any finitely generated projective left $\mathcal{A}$-module
${}_\mathcal{A}\mathcal{E}$
there is a finitely generated projective left $\boldsymbol{\mathcal{A}}$-module
${}_{\boldsymbol{\mathcal{A}}}\boldsymbol{\mathcal{E}}$
deforming ${}_\mathcal{A}\mathcal{E}$, which is unique up to equivalence.
This answers the question we stated at the introduction of this section,
leading to the following theorem (c.f. \cite{Bursztyn2002}~Prop.~2.21).
\begin{theorem}\label{thm04}
Let $\mathcal{A}$ and $\mathcal{B}$ be two Morita equivalent algebras and
$\boldsymbol{\mathcal{A}}=(\mathcal{A}[[\hbar]],\star)$ a formal deformation
of $\mathcal{A}$. Then, there is a formal deformation $\boldsymbol{\mathcal{B}}
=(\mathcal{B}[[\hbar]],\star')$ of $\mathcal{B}$ such that 
$\boldsymbol{\mathcal{A}}$ and $\boldsymbol{\mathcal{B}}$ are Morita equivalent.
Furthermore, there is a bijection
$\mathrm{Def}(\mathcal{A})\cong\mathrm{Der}(\mathcal{B})$ given by the
Morita equivalence bimodules which deform ${}_\mathcal{A}\mathcal{E}_\mathcal{B}$.
\end{theorem}
Applying this to Example~\ref{ex03} we obtain a result on the action of
the Picard group of a symplectic manifold on equivalence classes of 
symplectic star products.
\begin{corollary}[\cite{BursztynWaldmann2002}~Thm.~3.1]\label{cor04}
Let $(M,\omega)$ be a symplectic manifold. The action of the Picard group
$\mathrm{Pic}(M)$ on star products on $(M,\omega)$ is given by
\begin{equation}
    [\star]\mapsto[\star]+2\pi ic_1(L),
\end{equation}
where $c_1(L)$ is the first Chern class of the corresponding line bundle
$L\rightarrow M$. In particular, the obtained star product $\star'$ is equivalent
to $\star$ if and only if the first Chern class of the line bundle $L$ is trivial.
\end{corollary}
A generalization of this corollary to Poisson manifold is proven in
\cite{BuDoWa2012}~Thm.~3.11. The authors have to employ
Kontsevich's quantization map \cite{Kont2003} and the curvature
$2$-form of the line bundle in addition.

In the following lines we include a Lie group symmetry into the picture.
The strategy is to use a
Drinfel'd twist on the universal enveloping algebra of the corresponding
Lie algebra to deform the Morita equivalence bimodule
corresponding to an equivariant line bundle in the end.
Consider a \textit{$G$-equivariant line bundle} $\mathrm{pr}\colon L\rightarrow M$
for a Lie group $G$. This means that there are $G$-actions on $L$ and $M$,
respectively, such that
$$
g\rhd\colon L_p\rightarrow L_{g\rhd p}
$$
for all $g\in G$ and $p\in M$, where $L_p=\mathrm{pr}^{-1}(\{p\})\subseteq L$ is the
fiber of $p$. In other words, we require the Lie group action to be linear on fibers.
In particular, this implies
$$
\mathrm{pr}(g\rhd q)
=g\rhd\mathrm{pr}(q)
$$
for all $g\in G$ and $q\in L$.
\begin{example}\label{ex04}
Let $n>0$ and consider the complex projective space $\mathbb{CP}^n$, which is
defined as the set of all orbits of the Lie group action
$$
\Phi\colon\mathbb{C}_\times\times\mathbb{C}^{n+1}_\times\ni(\lambda,z)
\mapsto\lambda\cdot z\in\mathbb{C}^{n+1}_\times,
$$
where $\mathbb{C}_\times=\mathbb{C}\setminus\{0\}$. Since $\Phi$ is free and proper,
$\mathbb{CP}^n$ can be structured as a smooth manifold. On $\mathbb{CP}^n$ there
is the tautological line bundle
\begin{equation}
    L=\{(\ell,z)\in\mathbb{CP}^n\times\mathbb{C}^{n+1}~|~z\in\ell\}
\end{equation}
with projection $\pi\colon L\ni(\ell,z)\mapsto\ell\in\mathbb{CP}^n$,
where $z\in\ell$ means that $z$ is an element of the orbit $\ell$.
It is equivariant with respect to the $\mathrm{GL}_{n+1}(\mathbb{C})$-action
given by matrix-vector multiplication. In fact, 
$\mathrm{GL}_{n+1}(\mathbb{C})\times\mathbb{C}^{n+1}\rightarrow\mathbb{C}^{n+1}$
descends to a Lie group action on $\mathbb{CP}^n$, since it commutes with $\Phi$.
Then, the projection $\pi$ intertwines the diagonal action on $L$ and $\mathbb{CP}^n$,
i.e. 
$$
\pi(A\rhd(\ell,z))
=\pi((A\rhd\ell),(A\rhd z))
=A\rhd\ell
=A\rhd(\pi(\ell,z))
$$
for all $A\in\mathrm{GL}_{n+1}(\mathbb{C})$ and $(\ell,z)\in L$. This proves that
$\pi\colon L\rightarrow\mathbb{CP}^n$ is indeed 
$\mathrm{GL}_{n+1}(\mathbb{C})$-equivariant.
\end{example}
An equivariant line bundle naturally induces Lie group actions on the smooth
functions on the manifold and the smooth sections of the line bundle.
\begin{lemma}\label{lemma16}
Let $\mathrm{pr}\colon L\rightarrow M$ be a $G$-equivariant line bundle. Then
\begin{equation}\label{eq60}
\begin{split}
    (g\rhd f)(p)
    =&f(g^{-1}\rhd p),\\
    (g\rhd s)(p)
    =&g\rhd(s(g^{-1}\rhd p)),
\end{split}
\end{equation}
where $g\in G$, $f\in\mathscr{C}^\infty(M)$, $p\in M$ and $s\in\Gamma^\infty(L)$,
define $G$-actions on $\mathscr{C}^\infty(M)$ and $\Gamma^\infty(L)$.
Moreover the $G$-actions respect the $\mathscr{C}^\infty(M)$-bimodule structure
of $\Gamma^\infty(L)$, i.e.
\begin{equation}\label{eq42}
    g\rhd(f\cdot s)
    =(g\rhd f)\cdot(g\rhd s)
\end{equation}
for all $g\in G$, $f\in\mathscr{C}^\infty(M)$ and $s\in\Gamma^\infty(L)$.
In other words, $G$ acts by group-like elements.
\end{lemma}
\begin{proof}
Consider a section $s\in\Gamma^\infty(L)$ and an element $g\in G$.
Then $g\rhd s\in\Gamma^\infty(L)$ since
$$
\mathrm{pr}((g\rhd s)(p))
=\mathrm{pr}(g\rhd s(g^{-1}\rhd p))
=g\rhd\mathrm{pr}(s(g^{-1}\rhd p))
=g\rhd(g^{-1}\rhd p)
=p
$$
for all $p\in M$. It is an easy exercise to verify that (\ref{eq60}) are 
$G$-actions. 
In fact $(e\rhd f)(p)=f(p)$, $(e\rhd s)(p)=e\rhd s(e^{-1}\rhd p)=s(p)$,
\begin{align*}
    (g\rhd(h\rhd f))(p)
    =f((h^{-1}g^{-1})\rhd p)
    =f((gh)^{-1}\rhd p)
    =((gh)\rhd f)(p)
\end{align*}
and $(g\rhd(h\rhd s))(p)
=(gh)\rhd s((h^{-1}g^{-1})\rhd p)
=((gh)\rhd s)(p)$
for all $g,h\in G$, $p\in M$, $f\in\mathscr{C}^\infty(M)$ and 
$s\in\Gamma^\infty(L)$.
Furthermore, eq.(\ref{eq42}) follows since
\begin{align*}
    (g\rhd(f\cdot s))(p)
    =&g\rhd((f\cdot s)(g^{-1}\rhd p))\\
    =&g\rhd(f(g^{-1}\rhd p)s(g^{-1}\rhd p))\\
    =&f(g^{-1}\rhd p)g\rhd(s(g^{-1}\rhd p))\\
    =&((g\rhd f)\cdot(g\rhd s))(p)
\end{align*}
for all $g\in G$, $p\in M$ $f\in\mathscr{C}^\infty(M)$ and $s\in\Gamma^\infty(L)$,
where we used that the $G$-actions are $\mathbb{C}$-linear.
\end{proof}
The Lie group actions on functions and sections of the
line bundle induce Lie algebra actions of the corresponding Lie algebra. They can 
be extended to Hopf algebra actions of the universal enveloping algebra.
We prove that those actions respect the bimodule structure of sections.
\begin{corollary}\label{cor05}
If $\mathrm{pr}\colon L\rightarrow M$ is a $G$-equivariant line bundle, then
$\Gamma^\infty(L)$ is a $\mathscr{U}\mathfrak{g}$-equivariant symmetric
$\mathscr{C}^\infty(M)$-bimodule, where $\mathfrak{g}$ is the Lie algebra 
corresponding to $G$.
\end{corollary}
\begin{proof}
By Lemma~\ref{lemma16} there is a Lie group action $G\times\Gamma^\infty(L)
\rightarrow\Gamma^\infty(L)$. Recall that the corresponding Lie algebra
action $\mathfrak{g}\rightarrow\mathrm{End}_\mathbb{C}(\Gamma^\infty(L))$
is defined by
\begin{equation}\label{eq41}
    \xi\rhd s
    =\frac{\mathrm{d}}{\mathrm{d}t}\bigg|_{t=0}\exp(t\cdot\xi)\rhd s
\end{equation}
for all $\xi\in\mathfrak{g}$ and $s\in\Gamma^\infty(L)$. Note that the triangle on
the right hand side of (\ref{eq41}) denotes the $G$-action. This extends uniquely
to a left Hopf algebra action $\mathscr{U}\mathfrak{g}\otimes\Gamma^\infty(L)
\rightarrow\Gamma^\infty(L)$.
It remains to prove
that the $\mathscr{U}\mathfrak{g}$-action respects the $\mathscr{C}^\infty(M)$-module
action. This is the case, since
\begin{align*}
    \xi\rhd(f\cdot s)
    =&\frac{\mathrm{d}}{\mathrm{d}t}\bigg|_{t=0}\exp(t\cdot\xi)\rhd(f\cdot s)\\
    =&\frac{\mathrm{d}}{\mathrm{d}t}\bigg|_{t=0}\bigg((\exp(t\cdot\xi)\rhd f)\cdot
    (\exp(t\cdot\xi)\rhd s)\bigg)\\
    =&\bigg(\frac{\mathrm{d}}{\mathrm{d}t}(\exp(t\cdot\xi)\rhd f)\bigg)
    \cdot(\exp(t\cdot\xi)\rhd s)\bigg|_{t=0}\\
    &+(\exp(t\cdot\xi)\rhd f)\cdot
    \bigg(\frac{\mathrm{d}}{\mathrm{d}t}(\exp(t\cdot\xi)\rhd s)\bigg)\bigg|_{t=0}\\
    =&(\xi\rhd f)\cdot s
    +f\cdot(\xi\rhd s)\\
    =&(\xi_{(1)}\rhd f)\cdot(\xi_{(2)}\rhd s)
\end{align*}
for all $\xi\in\mathfrak{g}$, $f\in\mathscr{C}^\infty(M)$ and 
$s\in\Gamma^\infty(L)$, where we used (\ref{eq42}). It is sufficient to
prove equivariance on primitive elements.
\end{proof}
After those general considerations we return to obstructions for twist star products.
Before proving the main theorem of this section we argue that any equivariant line
bundle on a manifold can be twist deformed into a self Morita equivalence
bimodule of the twisted functions if there is a Drinfel'd twist on the 
corresponding universal enveloping algebra.
\begin{lemma}\label{lemma17}
Let $L\rightarrow M$ be a $G$-equivariant line bundle and $\star_\mathcal{F}$ be a
twist star product on $M$ with Drinfel'd twist based on $\mathscr{U}\mathfrak{g}$.
Then there is an algebra isomorphism
\begin{equation}
    \psi\colon(\mathscr{C}^\infty(M)[[\hbar]],\star_\mathcal{F})
    \rightarrow\mathrm{End}_{(\mathscr{C}^\infty(M)[[\hbar]],\star_\mathcal{F})}
    (\Gamma^\infty(L)[[\hbar]],\cdot_\mathcal{F})
\end{equation}
given by the twisted left $\mathscr{C}^\infty(M)$-module action on
$\Gamma^\infty(L)$.
\end{lemma}
\begin{proof}
$\psi$ is an algebra homomorphism since $f\cdot_\mathcal{F}(g\cdot_\mathcal{F}s)
=(f\star_\mathcal{F}g)\cdot_\mathcal{F}s$ for all $f,g\in\mathscr{C}^\infty(M)$
and $s\in\Gamma^\infty(L)$ by the $2$-cocycle property. It is an isomorphism
since $\Gamma^\infty(L)$ is a self Morita equivalence bimodule for the 
undeformed algebra. Namely
$\Gamma^\infty(\mathrm{End}(L))\cong\mathscr{C}^\infty(M)$ since
$\mathrm{End}(L)\cong M\times\mathbb{C}$ is the trivial line bundle, which implies
that
$$
(\mathscr{C}^\infty(M),\cdot)
\rightarrow\mathrm{End}_{(\mathscr{C}^\infty(M),\cdot)}(\Gamma^\infty(L),\cdot)
$$
is an algebra isomorphism. It follows that $\psi$ is invertible in the $\hbar$-adic
topology since $\psi(f)(s)=f\cdot s+\mathcal{O}(\hbar)$ 
for all $f\in\mathscr{C}^\infty(M)$ and $s\in\Gamma^\infty(L)$. 
\end{proof}
Now it is clear how to obtain an obstruction in the symplectic setting:
if the line bundle has non-trivial
Chern class it follows that its twist deformation is a Morita equivalence bimodule
between the twisted functions and another star product which is not isomorphic
to the twist star product. This is a contradiction to Lemma~\ref{lemma17},
which proves that the induced star product has to be the twisted product itself.
\begin{theorem}[\cite{dAWe17}~Thm.~11]\label{thm08}
Let $(M,\omega)$ be a symplectic manifold such that $M$ is a $G$-space in addition.
The following two properties are mutually exclusive.
\begin{enumerate}
\item[i.)] There is a $G$-equivariant smooth complex line bundle on $M$
with non-trivial Chern class;

\item[ii.)] There is a twist star product on $(M,\omega)$ for a Drinfel'd twist
based on $\mathscr{U}\mathfrak{g}[[\hbar]]$, where $\mathfrak{g}$ is the Lie
algebra corresponding to $G$;
\end{enumerate}
\end{theorem}
\begin{proof}
Assume that $i.)$ and $ii.)$ both hold on the same symplectic manifold
and $G$-space $(M,\omega)$. In particular, there are two algebra isomorphisms
$$
(\mathscr{C}^\infty(M)[[\hbar]],\star_\mathcal{F})
\xrightarrow{\psi}
\mathrm{End}_{(\mathscr{C}^\infty(M)[[\hbar]],\star_\mathcal{F})}
(\Gamma^\infty(L)[[\hbar]],\cdot_\mathcal{F})
\rightarrow
(\mathscr{C}^\infty(M)[[\hbar]],\star'),
$$
the first according to Lemma~\ref{lemma17} and the second according to
Theorem~\ref{thm04} and Theorem~\ref{thm05}~$iii.)$. More precisely,
$\star_\mathcal{F}$ denotes the twist star product on $(M,\omega)$ that exists
by assumption $ii.)$, while $\star'$ denotes the star product which exists
corresponding to the deformation $\star_\mathcal{F}$ of $\mathscr{C}^\infty(M)$
and $\cdot_\mathcal{F}$ of $\Gamma^\infty(L)$. According to 
Corollary~\ref{cor04} this implies
$c_1(L)=0$ in contradiction to assumption $i.)$ and we conclude the proof
of the theorem.
\end{proof}
Note that Theorem~\ref{thm08} only gives obstructions for Drinfel'd twists
on $\mathscr{U}\mathfrak{g}$ if $\mathfrak{g}$ is the Lie algebra corresponding
to the symmetry $G$ of the line bundle. Nevertheless this might rule out a huge
class of candidates, like it is depicted in the following example.
\begin{example}
Consider the tautological line bundle $L$ on $\mathbb{CP}^n$ discussed in
Example~\ref{ex04}. It is $G=\mathrm{GL}_{n+1}(\mathbb{C})$-equivariant
and according to Corollary~\ref{cor05} this implies that the sections
$\Gamma^\infty(L)$ of $L$ are a 
$\mathfrak{g}=\mathfrak{gl}_{n+1}(\mathbb{C})$-equivariant
$\mathscr{C}^\infty(\mathbb{CP}^n)$-bimodule. It is well-known that
$c_1(L)\neq 0$ and that the Fubini-Study $2$-form
$\omega_{\mathrm{FS}}$ is a symplectic
structure on $\mathbb{CP}^n$. According to Theorem~\ref{thm08} there is no
twist star product on $(\mathbb{CP}^n,\omega_{\mathrm{FS}})$ with a twist
based on $\mathscr{U}\mathfrak{g}$. On the other hand, according to
Proposition~\ref{prop19}, there are star product quantizing
$(\mathbb{CP}^n,\omega_{\mathrm{FS}})$. A particular class of them,
given by \textit{Berezin-Toeplitz} quantization, is described in 
\cite{Schlichenmaier}.
\end{example}

%% file: chapters/chapter03.tex
After recalling the well-known theory of quasi-triangular Hopf algebras
and their representations we proceed by elaborating some original
results of the author. The main statement, i.e.
the construction of the braided Cartan calculus
on an arbitrary braided commutative algebra, can be
found in \cite{Weber2019}~Section~3, while we add more details, proofs
and lemmas in the following sections. To give some insight on
how several formulas were deduced as natural generalizations, we start by
recalling the construction of the Cartan calculus on a commutative algebra
over a commutative unital ring in Section~\ref{Sec3.1}
(c.f. \cite{Ko60,Ri63}). Note that Henri
Cartan laid down the foundation of the Cartan calculus on a manifold in
\cite{Ca50} and with it all of its fundamental ingredients. The study
of multivector fields on a manifold and the generalization of the
Lie bracket of vector fields to a Gerstenhaber bracket (c.f. \cite{Ge63})
goes back to early works \cite{Ni55,Sc40,Sc54} of Schouten and Nijenhuis.
In the context of noncommutative geometry \cite{Connes94}
the Schouten-Nijenhuis bracket
is discussed in \cite{D-VM94}. The notion of Cartan calculi on
noncommutative algebras goes back to Woronowicz.
In \cite{Woronowicz1989} he generalized the de Rham differential
of a smooth manifold to a noncommutative setting, focusing on its
algebraic properties. Namely, a first order differential calculus over
a Hopf algebra $\mathcal{A}$ is given as a bimodule $\Gamma$ over the
algebra, together with a $\Bbbk$-linear map
$\mathrm{d}\colon\mathcal{A}\rightarrow\Gamma$, which satisfies
a Leibniz rule $\mathrm{d}(a\cdot b)
=\mathrm{d}(a)\cdot b+a\cdot\mathrm{d}(b)$ for all $a,b\in\mathcal{A}$.
In this context, $\Gamma$ represents the bimodule of differential forms
on $\mathcal{A}$.
Furthermore, he assumes that $\Gamma=\mathcal{A}\cdot\mathrm{d}\mathcal{A}$
is generated by $\mathcal{A}$ and $\mathrm{d}$. If the first order
calculus is \textit{bicovariant}, i.e. if
$\sum_ka_k\mathrm{d}b_k=0$ implies 
$\sum_k\Delta(a_k)(\mathrm{id}\otimes\mathrm{d})\Delta(b_k)=0$
and $\sum_k\Delta(a_k)(\mathrm{d}\otimes\mathrm{id})\Delta(b_k)=0$,
it admits an extension to the exterior algebra of $\Gamma$, which
we identify with higher order differential forms. Noncommutative
differential calculi based on derivations rather than generalizations
of differential forms are discussed in
\cite{D-VM96,D-VM94,Schupp1994,Schupp1993}. It is a necessity
to employ modules over the center of a noncommutative algebra in this
case, since derivations are only a central bimodule.
Following the approach of \cite{Schenkel2015,Schenkel2016}, we
construct a braided derivation based calculus via bimodules over the whole
algebra for the huge class of braided commutative algebras, in
accordance with \cite{Woronowicz1989}.
In the Sections~\ref{Sec3.2}-\ref{Sec3.4}
we repeat the construction of Section~\ref{Sec3.1} in the
category of equivariant braided symmetric bimodules of a braided commutative
algebra. Namely, we shape the braided Gerstenhaber algebra of
braided multivector fields with the braided Schouten-Nijenhuis 
bracket in Section~\ref{Sec3.2}, while we construct the braided
Graßmann algebra of braided differential forms in Section~\ref{Sec3.3}.
In particular we define a braided analogue of the de Rham differential
in the latter section via a generalization of the Chevalley-Eilenberg
formula. Then, in Section~\ref{Sec3.4}, our arrangements culminate in
the definition of the braided Cartan calculus. The corresponding
braided Lie derivative, insertion and de Rham differential are related
by graded braided commutators and resemble the formulas known from differential
geometry. It would be interesting to compare the braided Cartan calculus to the
noncommutative calculus \cite{Tsygan2000,Tsygan2012}
and generalize our construction to Lie-Rinehart pairs (see \cite{Hue1998}).
As an application we prove in Section~\ref{Sec3.5} that the
notion of covariant derivative naturally generalizes to the braided
Cartan calculus. In particular we compute the curvature and torsion of
an equivariant covariant derivative, we prove that a given equivariant covariant
derivative on the algebra extends to braided differential forms and
multivector fields and we give an existence and uniqueness theorem for
an equivariant Levi-Civita covariant derivative of a given non-degenerate 
equivariant metric.
Covariant derivatives on noncommutative algebras are also discussed in e.g.
\cite{Arnlind2019II,Arnlind2019,Aschieri2010,AsSh14,Schenkel2016,Jyotishman2016,Jyotishman2019,D-VM96,Fiore2000,LaMa2012,Peterka2017}.
In general one has to distinguish between left and
right covariant derivatives. However, we prove that these notions coincide in
the equivariant case.
Finally, in Section~\ref{Sec3.6} we clarify that Drinfel'd twist
gauge equivalence is compatible with the construction of the braided
Cartan calculus and the notion of equivariant covariant derivative.
For a given twist we describe a deformation of the data of a braided
Cartan calculus and prove that it is isomorphic to the braided Cartan
calculus with respect to the twisted algebra and twisted triangular structure.
This recovers the well-known twisted Cartan calculus and integrates it in the
setting of braided Cartan calculi. Furthermore,
the twist deformation of an equivariant covariant derivative can be viewed
as an equivariant covariant derivative with respect to the twisted
universal $\mathcal{R}$-matrix on the twisted algebra via the same 
isomorphism.

\section{The Cartan Calculus on a Commutative Algebra}\label{Sec3.1}

Recalling differential geometry, the Cartan calculus on a smooth manifold $M$
is based on the commutative algebra $\mathscr{C}^\infty(M)$ of smooth functions.
The Gerstenhaber algebra 
$(\mathfrak{X}^\bullet(M),\wedge,\llbracket\cdot,\cdot\rrbracket)$ of
multivector fields 
and the Graßmann algebra $(\Omega^\bullet(M),\wedge)$
of differential forms are graded symmetric $\mathscr{C}^\infty(M)$-bimodules.
Using the de Rham differential $\mathrm{d}$ and the insertion $\mathrm{i}_X$
of vector fields $X\in\mathfrak{X}^1(M)$ into the first slot of a
differential form, one defines the Lie derivative
$\mathscr{L}_X=[\mathrm{i}_X,\mathrm{d}]$, the bracket denoting the
commutator of endomorphisms. On factorizing multivector fields
$X=X_1\wedge\cdots\wedge X_k\in\mathfrak{X}^k(M)$ one defines
$\mathrm{i}_X=\mathrm{i}_{X_1}\cdots\mathrm{i}_{X_k}$ and the Lie derivative by 
$$
\mathscr{L}_X
=\mathrm{i}_X\mathrm{d}
-(-1)^{k-1}\mathrm{d}\mathrm{i}_X.
$$
Linear extension leads to homogeneous maps
$\mathrm{d}\colon\Omega^\bullet(M)\rightarrow\Omega^{\bullet+1}(M)$,
$\mathrm{i}_X\colon\Omega^\bullet(M)\rightarrow\Omega^{\bullet-k}(M)$
and $\mathscr{L}_X\colon\Omega^\bullet(M)\rightarrow\Omega^{\bullet-(k-1)}(M)$,
where $X\in\mathfrak{X}^k(M)$. In particular we can define the Lie derivative
$\mathscr{L}_X=[\mathrm{i}_X,\mathrm{d}]$
of any multivector field $X\in\mathfrak{X}^\bullet(M)$ using the
graded commutator $[\cdot,\cdot]$. One proves that 
$\mathrm{i}\colon\mathfrak{X}^\bullet(M)\times\Omega^\bullet(M)
\rightarrow\Omega^\bullet(M)$ is $\mathscr{C}^\infty(M)$-bilinear. If
$X\in\mathfrak{X}^1(M)$ we obtain derivations $\mathrm{i}_X$ and
$\mathscr{L}_X$ of $(\Omega^\bullet(M),\wedge)$. It is the aim of this subsection
to reduce the Cartan calculus to its algebraic properties and
reconstruct them for any commutative algebra.
In the next subsections this will be generalized to braided commutative
algebras.

Fix a commutative algebra $\mathcal{A}$ in the following. An
endomorphism $\Phi\colon\mathcal{A}\rightarrow\mathcal{A}$ is said to be a
\textit{derivation} of $\mathcal{A}$ if $\Phi(ab)=\Phi(a)b+a\Phi(b)$
for all $a,b\in\mathcal{A}$. We denote the $\Bbbk$-module of all derivations of
$\mathcal{A}$ by $\mathrm{Der}(\mathcal{A})$.
\begin{lemma}
The derivations $\mathrm{Der}(\mathcal{A})$ of $\mathcal{A}$ form a Lie
algebra with Lie bracket given by the commutator $[\cdot,\cdot]$ of
endomorphisms. Furthermore, $\mathrm{Der}(\mathcal{A})$ is a symmetric
$\mathcal{A}$-bimodule with left and right $\mathcal{A}$-module actions
defined on $X\in\mathrm{Der}(\mathcal{A})$ by
$$
(a\cdot X)(b)=aX(b)=(X\cdot a)(b)
$$
for all $a,b\in\mathcal{A}$.
\end{lemma}
\begin{proof}
For any $a\in\mathcal{A}$ and $X\in\mathrm{Der}(\mathcal{A})$,
$a\cdot X$ is a derivation of $\mathcal{A}$, since
$$
(a\cdot X)(bc)
=aX(bc)
=a(X(b)c+bX(c))
=((a\cdot X)(b))c+b(a\cdot X)(c)
$$
holds for all $b,c\in\mathcal{A}$ by the commutativity of $\mathcal{A}$.
Therefore we obtain well-defined $\mathcal{A}$-module actions by the
associativity of $\mathcal{A}$. They are symmetric by definition.
The endomorphisms of $\mathcal{A}$ are a Lie algebra with respect to the
commutator. It remains to prove that $[\cdot,\cdot]$ closes in
$\mathrm{Der}(\mathcal{A})$. Let $X,Y\in\mathrm{Der}(\mathcal{A})$.
Then
\begin{align*}
    [X,Y](ab)
    =&X(Y(ab))-Y(X(ab))\\
    =&X(Y(a)b+aY(b))-Y(X(a)b+aX(b))\\
    =&X(Y(a))b+Y(a)X(b)+X(a)Y(b)+aX(Y(b))\\
    &-Y(X(a))b-X(a)Y(b)-Y(a)X(b)-aY(X(b))\\
    =&(X(Y(a))-Y(X(a)))b+a(X(Y(b))-Y(X(b)))\\
    =&([X,Y](a))b+a([X,Y](b))
\end{align*}
for all $a,b\in\mathcal{A}$, which means that $[X,Y]\in\mathrm{Der}(\mathcal{A})$.
\end{proof}
Next we consider the \textit{tensor algebra} 
$$
\mathrm{T}^\bullet\mathrm{Der}(\mathcal{A})
=\mathcal{A}\oplus\mathrm{Der}(\mathcal{A})
\oplus(\mathrm{Der}(\mathcal{A})\otimes_\mathcal{A}\mathrm{Der}(\mathcal{A}))
\oplus\cdots
$$
of the symmetric $\mathcal{A}$-bimodule $\mathrm{Der}(\mathcal{A})$. It is a
graded noncommutative algebra with multiplication given by the tensor product
$\otimes_\mathcal{A}$ over $\mathcal{A}$. The left and right
$\mathcal{A}$-actions, defined on factorizing elements
$X=X_1\otimes_\mathcal{A}\cdots\otimes_\mathcal{A}X_k\in
\mathrm{T}^k\mathrm{Der}(\mathcal{A})$ and $a\in\mathcal{A}$ by
$$
a\cdot X
=(a\cdot X_1)\otimes_\mathcal{A}\cdots\otimes_\mathcal{A}X_k
=X\cdot a,
$$
structure $\mathrm{T}^\bullet\mathrm{Der}(\mathcal{A})$ as a
symmetric $\mathcal{A}$-bimodule.
The quotient of $\mathrm{T}^\bullet\mathrm{Der}(\mathcal{A})$ with the ideal
generated by expressions $X\otimes_\mathcal{A}Y-(-1)^{k\ell}Y\otimes_\mathcal{A}X$,
where $X\in\mathrm{T}^k\mathrm{Der}(\mathcal{A})$
and $Y\in\mathrm{T}^\ell\mathrm{Der}(\mathcal{A})$, is the \textit{Graßmann algebra}
or \textit{exterior algebra} $\Lambda^\bullet\mathrm{Der}(\mathcal{A})$
of $\mathrm{Der}(\mathcal{A})$. The induced
product, the \textit{wedge product}, is denoted by $\wedge$. Since
the ideal respects the structure of the tensor algebra it follows that
$\Lambda^\bullet\mathrm{Der}(\mathcal{A})$ is a graded algebra and a
symmetric $\mathcal{A}$-bimodule. In particular, there are symmetric
left and right $\mathcal{A}$-module actions such that on
factorizing elements $X=X_1\wedge\cdots\wedge X_k\in
\Lambda^k\mathrm{Der}(\mathcal{A})$
$$
a\cdot X
=a\wedge X=(a\cdot X_1)\wedge\cdots\wedge X_k
=X\wedge a=X\cdot a
$$
holds for all $a\in\mathcal{A}$. Note that
$\Lambda^\bullet\mathrm{Der}(\mathcal{A})$ is \textit{graded commutative}
in addition, i.e.
$$
X\wedge Y=(-1)^{k\ell}Y\wedge X
$$
for all $X\in\Lambda^k\mathrm{Der}(\mathcal{A})$
and $Y\in\Lambda^\ell\mathrm{Der}(\mathcal{A})$. In the following we write
$\mathfrak{X}^\bullet(\mathcal{A})$ instead of
$\Lambda^\bullet\mathrm{Der}(\mathcal{A})$ and call it the
\textit{multivector fields} of $\mathcal{A}$. We define the
\textit{Schouten-Nijenhuis bracket} $\llbracket\cdot,\cdot\rrbracket
\colon\mathfrak{X}^\bullet(\mathcal{A})\times\mathfrak{X}^\bullet(\mathcal{A})
\rightarrow\mathfrak{X}^\bullet(\mathcal{A})$ on factorizing elements
$X=X_1\wedge\cdots\wedge X_k\in\mathfrak{X}^k(\mathcal{A})$ and
$Y=Y_1\wedge\cdots\wedge X_\ell\in\mathfrak{X}^\ell(\mathcal{A})$ by
\begin{equation}\label{eq25}
    \llbracket X,Y\rrbracket
    =\sum_{i=1}^k\sum_{j=1}^\ell(-1)^{i+j}[X_i,X_j]\wedge X_1\wedge\cdots
    \wedge\widehat{X_i}\wedge\cdots X_k\wedge Y_1\wedge\cdots\wedge\widehat{Y_j}\wedge
    \cdots Y_\ell
\end{equation}
if $k,\ell>0$, where $\widehat{X_i}$ and $\widehat{Y_j}$ are omitted in the above wedge
product. We further set $\llbracket a,b\rrbracket=0$ for $a,b\in\mathcal{A}$
and $\llbracket X,a\rrbracket=X(a)=-\llbracket a,X\rrbracket$ for all 
$X\in\mathfrak{X}^1(\mathcal{A})$ and $a\in\mathcal{A}$ and extend 
$\llbracket\cdot,\cdot\rrbracket$ $\Bbbk$-bilinearly.
\begin{proposition}
The Schouten-Nijenhuis bracket structures $\mathfrak{X}^\bullet(\mathcal{A})$
as a Gerstenhaber algebra. Namely,
$
\llbracket\cdot,\cdot\rrbracket\colon
\mathfrak{X}^k(\mathcal{A})\times\mathfrak{X}^\ell(\mathcal{A})
\rightarrow\mathfrak{X}^{k+\ell-1}(\mathcal{A})
$
is a graded (with respect to the degree shifted by $1$) Lie bracket, i.e.
it is graded skew-symmetric
$$
\llbracket Y,X\rrbracket
=-(-1)^{(k-1)(\ell-1)}\llbracket X,Y\rrbracket
$$
and satisfies the graded Jacobi identity
$$
\llbracket X,\llbracket Y,Z\rrbracket\rrbracket
=\llbracket\llbracket X,Y\rrbracket,Z\rrbracket
+(-1)^{(k-1)(\ell-1)}\llbracket Y,\llbracket X,Z\rrbracket\rrbracket,
$$
such that the graded Leibniz rule
$$
\llbracket X,Y\wedge Z\rrbracket
=\llbracket X,Y\rrbracket\wedge Z
+(-1)^{(k-1)\ell}Y\wedge\llbracket X,Z\rrbracket
$$
holds in addition,
where $X\in\mathfrak{X}^k(\mathcal{A})$, $Y\in\mathfrak{X}^\ell(\mathcal{A})$
and $Z\in\mathfrak{X}^\bullet(\mathcal{A})$.
It is the unique Gerstenhaber bracket on
$\mathfrak{X}^\bullet(\mathcal{A})$ such that
$$
\llbracket X,a\rrbracket=X(a)
\text{ and }
\llbracket X,Y\rrbracket=[X,Y]
$$
for all $a\in\mathcal{A}$ and $X,Y\in\mathfrak{X}^1(\mathcal{A})$.
\end{proposition}
\begin{proof}
By counting degrees we see that $\llbracket\cdot,\cdot\rrbracket\colon
\mathfrak{X}^k(\mathcal{A})\times\mathfrak{X}^\ell(\mathcal{A})
\rightarrow\mathfrak{X}^{k+\ell-1}(\mathcal{A})$. This means that
$\llbracket\cdot,\cdot\rrbracket$ is homogeneous with respect to the degree
shifted by $1$. Then, using the defining formula (\ref{eq25}) it is an
exercise to verify that $\llbracket\cdot,\cdot\rrbracket$ is a Gerstenhaber
bracket. In fact it is sufficient to prove this in degree $0$ and $1$ since
afterwards an inductive argument, using the graded Leibniz rule, implies that
the identities are valid in any degree. Remark that this is how (\ref{eq25})
was engineered: one defines $\llbracket\cdot,\cdot\rrbracket$ in degree
$0$ and $1$, extends $\llbracket X,\cdot\rrbracket\colon
\mathfrak{X}^\bullet(\mathcal{A})\rightarrow\mathfrak{X}^\bullet(\mathcal{A})$ 
as a derivation of the wedge product for all $X\in\mathfrak{X}^1(\mathcal{A})$
and imposes the graded Leibniz rule
and graded skew-symmetry afterwards. Similarly one extends
$\llbracket a,\cdot\rrbracket\colon
\mathfrak{X}^\bullet(\mathcal{A})\rightarrow\mathfrak{X}^{\bullet-1}(\mathcal{A})$
as a graded derivation for all $a\in\mathcal{A}$.
In particularly this clarifies the uniqueness by
controlling the values of the Gerstenhaber bracket in degree $0$ and $1$.
\end{proof}
A preliminary stage to differential forms is given by the dual space
$\underline{\Omega}^1(\mathcal{A})
=\mathrm{Hom}_\mathcal{A}(\mathrm{Der}(\mathcal{A}),\mathcal{A})$
corresponding to $\mathrm{Der}(\mathcal{A})$, i.e. by the $\mathcal{A}$-linear
maps $\mathrm{Der}(\mathcal{A})\rightarrow\mathcal{A}$. They form a
symmetric $\mathcal{A}$-bimodule with left and right $\mathcal{A}$-action
given for $a\in\mathcal{A}$ and $\omega\in\underline{\Omega}^1(\mathcal{A})$ by
$$
(a\cdot\omega)(X)=a\cdot\omega(X)=(\omega\cdot a)(X)
$$
for all $X\in\mathrm{Der}(\mathcal{A})$. The corresponding Graßmann algebra
is denoted by $(\underline{\Omega}^\bullet(\mathcal{A}),\wedge)$. There is a
non-degenerate
$\mathcal{A}$-bilinear dual pairing $\langle\cdot,\cdot\rangle\colon
\underline{\Omega}^1(\mathcal{A})\otimes\mathrm{Der}(\mathcal{A})
\rightarrow\mathcal{A}$ defined by $\langle\omega,X\rangle=\omega(X)$
for all $\omega\in\underline{\Omega}^1(\mathcal{A})$ and
$X\in\mathrm{Der}(\mathcal{A})$. We can also view this as inserting
a derivation $X$ into the functional $\omega\in\underline{\Omega}^1(\mathcal{A})$.
In this picture it is customized to write
$\mathrm{i}_X\omega=\langle\omega,X\rangle$. Extending $\mathrm{i}_X\colon
\underline{\Omega}^\bullet(\mathcal{A})
\rightarrow\underline{\Omega}^{\bullet-1}(\mathcal{A})$ as a graded derivation
of degree $-1$ of the wedge product, we obtain a homogeneous map of degree $-1$.
On factorizing multivector fields
$X=X_1\wedge\cdots\wedge X_k\in\mathfrak{X}^k(\mathcal{A})$ we set
$$
\mathrm{i}_X=\mathrm{i}_{X_1}\cdots\mathrm{i}_{X_k}.
$$
This determines the $\mathcal{A}$-bilinear \textit{insertion of multivector fields}
$\mathrm{i}\colon\mathfrak{X}^\bullet(\mathcal{A})
\times\underline{\Omega}^\bullet(\mathcal{A})
\rightarrow\underline{\Omega}^\bullet(\mathcal{A})$, such that
$\mathrm{i}_X\colon\underline{\Omega}^\bullet(\mathcal{A})
\rightarrow\underline{\Omega}^{\bullet-k}(\mathcal{A})$
is homogeneous of degree $-k$ for any $X\in\mathfrak{X}^k(\mathcal{A})$.
We further define a $\Bbbk$-linear map
$\mathrm{d}\colon\underline{\Omega}^\bullet(\mathcal{A})
\rightarrow\underline{\Omega}^{\bullet+1}(\mathcal{A})$ on
$\underline{\Omega}^\bullet(\mathcal{A})$ by setting
$\mathrm{i}_X(\mathrm{d}a)=X(a)$,
$$
(\mathrm{d}\alpha)(X,Y)
=X(\mathrm{i}_Y\alpha)-Y(\mathrm{i}_X\alpha)-\mathrm{i}_{[X,Y]}\alpha
$$
on $a\in\mathcal{A}$ and $\alpha\in\underline{\Omega}^1(\mathcal{A})$,
for all $X,Y\in\mathfrak{X}^1(\mathcal{A})$ and extending $\mathrm{d}$
to higher factorizing elements as a graded derivation of degree $1$, i.e.
by imposing
$$
\mathrm{d}(\omega\wedge\eta)
=\mathrm{d}\omega\wedge\eta
+(-1)^k\omega\wedge\mathrm{d}\eta
$$
for all $\omega\in\underline{\Omega}^k(\mathcal{A})$ and
$\eta\in\underline{\Omega}^\bullet(\mathcal{A})$. Alternatively one
can directly axiomatise the \textit{Chevalley-Eilenberg formula}
\begin{align*}
    (\mathrm{d}\omega)(X_0,\ldots,X_k)
    =&\sum_{i=0}^k(-1)^iX_i(\omega(X_0,\cdots,
    \widehat{X_i},\cdots,X_k))\\
    &+\sum_{i<j}(-1)^{i+j}\omega([X_i,X_j],X_{0},\cdots
    \widehat{X_i},\cdots,\widehat{X_j},\cdots,X_k)
\end{align*}
for the differential of a homogeneous element $\omega\in\underline{\Omega}^k
(\mathcal{A})$ and all $X_0,\ldots,X_k\in\mathfrak{X}^1(\mathcal{A})$.
\begin{lemma}\label{lemma06}
The derivation $\mathrm{d}$ of degree $1$ is a differential on
$(\underline{\Omega}^\bullet(\mathcal{A}),\wedge)$.
\end{lemma}
\begin{proof}
The homogeneity is clear. Let $X,Y,Z\in\mathfrak{X}^1(\mathcal{A})$.
Then
\begin{align*}
    (\mathrm{d}^2a)(X,Y)
    =&X(\mathrm{i}_Y(\mathrm{d}a))-Y(\mathrm{i}_X(\mathrm{d}a))
    -\mathrm{i}_{[X,Y]}(\mathrm{d}a)\\
    =&X(Y(a))-Y(X(a))-[X,Y](a)\\
    =&0
\end{align*}
for all $a\in\mathcal{A}$ by the definition of the Lie bracket of vector fields.
Furthermore
\begin{allowdisplaybreaks}
\begin{align*}
    (\mathrm{d}^2\omega)(X,Y,Z)
    =&X((\mathrm{d}\omega)(Y,Z))
    -Y((\mathrm{d}\omega)(X,Z))
    +Z((\mathrm{d}\omega)(X,Y))\\
    &-(\mathrm{d}\omega)([X,Y],Z)
    +(\mathrm{d}\omega)([X,Z],Y)
    -(\mathrm{d}\omega)([Y,Z],X)\\
    =&X\bigg(
    Y(\mathrm{i}_Z\omega)-Z(\mathrm{i}_Y\omega)
    -\mathrm{i}_{[Y,Z]}\omega
    \bigg)\\
    &-Y\bigg(
    X(\mathrm{i}_Z\omega)-Z(\mathrm{i}_X\omega)
    -\mathrm{i}_{[X,Z]}\omega
    \bigg)\\
    &+Z\bigg(
    X(\mathrm{i}_Y\omega)-Y(\mathrm{i}_X\omega)
    -\mathrm{i}_{[X,Y]}\omega
    \bigg)\\
    &-\bigg(
    [X,Y](\mathrm{i}_Z\omega)-Z(\mathrm{i}_{[X,Y]}\omega)
    -\mathrm{i}_{[[X,Y],Z]}\omega
    \bigg)\\
    &+[X,Z](\mathrm{i}_Y\omega)-Y(\mathrm{i}_{[X,Z]}\omega)
    -\mathrm{i}_{[[X,Z],Y]}\omega\\
    &-\bigg(
    [Y,Z](\mathrm{i}_X\omega)-X(\mathrm{i}_{[Y,Z]}\omega)
    -\mathrm{i}_{[[Y,Z],X]}\omega
    \bigg)\\
    =&0
\end{align*}
\end{allowdisplaybreaks}
for all $\omega\in\underline{\Omega}^1(\mathcal{A})$, where we used the
Jacobi identity of the Lie bracket of vector fields. Since $\mathrm{d}^2$
is a graded derivation it is sufficient to verify $\mathrm{d}^2\omega=0$
for $\omega\in\underline{\Omega}^k(\mathcal{A})$ with $k<2$ in order to
conclude $\mathrm{d}^2=0$.
\end{proof}
Differential forms on $\mathcal{A}$ are generated by $\mathcal{A}$ and the
image of the differential $\mathrm{d}$ via the wedge product.
\begin{definition}
The smallest differential graded subalgebra $\Omega^\bullet(\mathcal{A})
\subseteq\underline{\Omega}^\bullet(\mathcal{A})$ such that
$\mathcal{A}\subseteq\Omega^\bullet(\mathcal{A})$ is said to be the
Graßmann algebra of differential forms on $\mathcal{A}$.
\end{definition}
Since the former operations on $\underline{\Omega}^\bullet(\mathcal{A})$
respect the wedge product $\wedge$, they also close in 
$\Omega^\bullet(\mathcal{A})$.
\begin{lemma}
The differential forms $(\Omega^\bullet(\mathcal{A}),\wedge)$ are
a graded commutative algebra and a symmetric $\mathcal{A}$-bimodule.
Every element $\omega\in\Omega^k(\mathcal{A})$ is a finite sum of expressions
of the form $a_0\mathrm{d}a_1\wedge\cdots\wedge\mathrm{d}a_k$, where
$a_0,\ldots,a_k\in\mathcal{A}$. The restrictions of the insertion of
multivector fields and differential to $\Omega^\bullet(\mathcal{A})$ are
well-defined.
\end{lemma}
The last operation missing to describe the Cartan calculus is the
\textit{Lie derivative}. It is defined on any
$X\in\mathfrak{X}^\bullet(\mathcal{A})$ by
$$
\mathscr{L}_X=[\mathrm{i}_X,\mathrm{d}],
$$
where the bracket denotes the graded commutator. If
$X\in\mathfrak{X}^k(\mathcal{A})$, we obtain a homogeneous map
$\mathscr{L}_X\colon\Omega^\bullet(\mathcal{A})
\rightarrow\Omega^{\bullet-(k-1)}(\mathcal{A})$ of degree $-(k-1)$. For $k=1$,
$\mathscr{L}_X$ is a derivation with respect to the wedge product.
The first statement holds because $\mathrm{i}_X$ is homogeneous of
degree $-k$ and $\mathrm{d}$ is homogeneous of degree $1$. For the second
statement let $X\in\mathfrak{X}^1(\mathcal{A})$. Then
\begin{align*}
    \mathscr{L}_X(\omega\wedge\eta)
    =&\mathrm{i}_X\mathrm{d}(\omega\wedge\eta)
    -(-1)^{(-1)\cdot 1}\mathrm{d}\mathrm{i}_X(\omega\wedge\eta)\\
    =&\mathrm{i}_X(\mathrm{d}\omega\wedge\eta
    +(-1)^k\omega\wedge\mathrm{d}\eta)
    +\mathrm{d}(\mathrm{i}_X\omega\wedge\eta
    +(-1)^k\omega\wedge\mathrm{i}_X\eta)\\
    =&\mathscr{L}_X\omega\wedge\eta
    +(-1)^{k+1}\mathrm{d}\omega\wedge\mathrm{i}_X\eta
    +(-1)^k\mathrm{i}_X\omega\wedge\mathrm{d}\eta
    +(-1)^{2k}\omega\wedge\mathrm{i}_X\mathrm{d}\eta\\
    &+(-1)^{k-1}\mathrm{i}_X\omega\wedge\mathrm{d}\eta
    +(-1)^k\mathrm{d}\omega\wedge\mathrm{i}_X\eta
    +(-1)^{2k}\omega\wedge\mathrm{d}\mathrm{i}_X\eta\\
    =&\mathscr{L}_X\omega\wedge\eta
    +\omega\wedge\mathscr{L}_X\eta
\end{align*}
for all $\omega\in\Omega^k(\mathcal{A})$ and $\eta\in\Omega^\bullet(\mathcal{A})$, 
using that $\mathrm{i}_X$ and $\mathrm{d}$ are graded derivations of the wedge
product of degree $-1$ and $1$, respectively. Two helpful relations are
proven in the following lemma.
\begin{lemma}\label{lemma07}
One has
$$
\mathscr{L}_a\omega
=-(\mathrm{d}a)\wedge\omega
\text{ and }
\mathscr{L}_{X\wedge Y}
=\mathrm{i}_X\mathscr{L}_Y+(-1)^{\ell}\mathscr{L}_X\mathrm{i}_Y
$$
for all $a\in\mathcal{A}$, $X\in\mathfrak{X}^\bullet(\mathcal{A})$,
$Y\in\mathfrak{X}^\ell(\mathcal{A})$ and $\omega\in\Omega^\bullet(\mathcal{A})$. If
$X,Y\in\mathfrak{X}^1(\mathcal{A})$
$$
[\mathscr{L}_X,\mathrm{i}_Y]=\mathrm{i}_{[X,Y]}
$$
holds.
\end{lemma}
\begin{proof}
By the definition of the Lie derivative and since $\mathrm{d}$ is a graded
derivation of degree $1$
\begin{allowdisplaybreaks}
\begin{align*}
    \mathscr{L}_a\omega
    =&[\mathrm{i}_a,\mathrm{d}]\omega\\
    =&\mathrm{i}_a\mathrm{d}\omega-(-1)^{0\cdot(-1)}\mathrm{d}\mathrm{i}_a\omega\\
    =&a\wedge\mathrm{d}\omega-\mathrm{d}(a\wedge\omega)\\
    =&a\wedge\mathrm{d}\omega
    -(\mathrm{d}a\wedge\omega+(-1)^0a\wedge\mathrm{d}\omega)\\
    =&-\mathrm{d}a\wedge\omega
\end{align*}
\end{allowdisplaybreaks}
follows. The relation
\begin{align*}
    \mathscr{L}_{X\wedge Y}
    =&[\mathrm{i}_{X\wedge Y},\mathrm{d}]
    =[\mathrm{i}_X\mathrm{i}_Y,\mathrm{d}]
    =\mathrm{i}_X[\mathrm{i}_Y,\mathrm{d}]
    +(-1)^{(-\ell)\cdot 1}[\mathrm{i}_X,\mathrm{d}]\mathrm{i}_Y\\
    =&\mathrm{i}_X\mathscr{L}_Y+(-1)^{\ell}\mathscr{L}_X\mathrm{i}_Y
\end{align*}
is obtained from the graded Leibniz rule of the graded commutator. The
missing formula trivially holds on differential forms of degree $0$,
while for $\omega\in\Omega^1(\mathcal{A})$ one obtains
\begin{align*}
    [\mathscr{L}_X,\mathrm{i}_Y]\omega
    =&\mathscr{L}_X\mathrm{i}_Y\omega
    -(-1)^{0\cdot 1}\mathrm{i}_Y\mathscr{L}_X\omega\\
    =&(\mathrm{i}_X\mathrm{d}+\mathrm{d}\mathrm{i}_X)\mathrm{i}_Y\omega
    -\mathrm{i}_Y(\mathrm{i}_X\mathrm{d}+\mathrm{d}\mathrm{i}_X)\omega\\
    =&X(\mathrm{i}_Y\omega)+0-(\mathrm{d}\omega)(X,Y)-Y(\mathrm{i}_X\omega)\\
    =&\mathrm{i}_{[X,Y]}\omega
\end{align*}
for all $X,Y\in\mathfrak{X}^1(\mathcal{A})$. Since
$[\mathscr{L}_X,\mathrm{i}_Y]$ is a graded derivation
this is all we have to prove.
\end{proof}
It follows the main theorem
of this subsection, where we relate the operations given by the Lie derivative,
the insertion and the de Rham differential via graded commutators. It
describes the \textit{Cartan calculus} on $\mathcal{A}$.
\begin{theorem}[Cartan Calculus]
Let $\mathcal{A}$ be a commutative algebra. Then
\begin{align*}
\begin{split}
    [\mathscr{L}_X,\mathscr{L}_Y]
    =&\mathscr{L}_{\llbracket X,Y\rrbracket},\\
    [\mathscr{L}_X,\mathrm{i}_Y]
    =&\mathrm{i}_{\llbracket X,Y\rrbracket},\\
    [\mathscr{L}_X,\mathrm{d}]
    =&0,
\end{split}
\begin{split}
    [\mathrm{i}_X,\mathrm{i}_Y]
    =&0,\\
    [\mathrm{i}_X,\mathrm{d}]
    =&\mathscr{L}_X,\\
    [\mathrm{d},\mathrm{d}]
    =&0
\end{split}
\end{align*}
hold for all $X,Y\in\mathfrak{X}^\bullet(\mathcal{A})$.
\end{theorem}
\begin{proof}
We are going to prove all formulas in reversed order. First,
$[\mathrm{d},\mathrm{d}]=2\mathrm{d}^2=0$ follows since $\mathrm{d}$ is a
differential, which we proved in Lemma~\ref{lemma06}. The next formula,
$[\mathrm{i}_X,\mathrm{d}]=\mathscr{L}_X$, holds by definition of the Lie
derivative for all $X\in\mathfrak{X}^\bullet(\mathcal{A})$. Let
$X\in\mathfrak{X}^k(\mathcal{A})$ and $Y\in\mathfrak{X}^\ell(\mathcal{A})$.
Then
$$
[\mathrm{i}_X,\mathrm{i}_Y]
=\mathrm{i}_X\mathrm{i}_Y-(-1)^{k\cdot\ell}\mathrm{i}_Y\mathrm{i}_X
=\mathrm{i}_{X\wedge Y}-(-1)^{k\cdot\ell}\mathrm{i}_{Y\wedge X}
=0
$$
by the graded commutativity of the wedge product. For the next equation
we utilize the graded Jacobi identity of the graded commutator to
conclude
\begin{align*}
    [\mathscr{L}_X,\mathrm{d}]
    =[[\mathrm{i}_X,\mathrm{d}],\mathrm{d}]
    =[\mathrm{i}_X,[\mathrm{d},\mathrm{d}]]
    +(-1)^{1\cdot 1}[[\mathrm{i}_X,\mathrm{d}],\mathrm{d}]
    =0-[\mathscr{L}_X,\mathrm{d}]
\end{align*}
for all $X\in\mathfrak{X}^\bullet(\mathcal{A})$, which implies
$[\mathscr{L}_X,\mathrm{d}]=0$. Recall that for any
$a\in\mathcal{A}$, $X\in\mathfrak{X}^1(\mathcal{A})$ and any
homogeneous element $Y=Y_1\wedge\cdots\wedge Y_\ell
\in\mathfrak{X}^\ell(\mathcal{A})$
$$
\llbracket a,Y\rrbracket
=\sum_{j=1}^\ell(-1)^j
Y_j(a)Y_1\wedge\cdots\wedge\widehat{Y_j}\wedge\cdots\wedge Y_\ell
$$
and
$$
\llbracket X,Y\rrbracket
=\sum_{j=1}Y_1\wedge\cdots\wedge Y_{j-1}
\wedge[X,Y_j]\wedge Y_{j+1}\wedge\cdots\wedge Y_\ell
$$
hold. Together with Lemma~\ref{lemma07} this implies for all
$\omega\in\Omega^\bullet(\mathcal{A})$
\begin{allowdisplaybreaks}
\begin{align*}
    \mathrm{i}_{\llbracket a,Y\rrbracket}\omega
    =&\sum_{j=1}^\ell(-1)^j
    \mathrm{i}_{Y_j(a)Y_1\wedge\cdots\wedge\widehat{Y_j}\wedge\cdots Y_\ell}
    \omega\\
    =&-\mathrm{d}a\wedge\mathrm{i}_Y\omega
    +(-1)^\ell\bigg(
    \sum_{j=1}^\ell(-1)^{\ell-j}
    \mathrm{i}_{Y_j}(\mathrm{d}a)\mathrm{i}_{Y_1}
    \cdots\widehat{\mathrm{i}_{Y_j}}\cdots\mathrm{i}_{Y_\ell}\omega
    +(-1)^\ell\mathrm{d}a\wedge\mathrm{i}_Y\omega
    \bigg)\\
    =&-\mathrm{d}a\wedge\mathrm{i}_Y\omega
    +(-1)^\ell\bigg(
    \sum_{j=2}^{\ell}(-1)^{\ell-j}
    \mathrm{i}_{Y_j}(\mathrm{d}a)\mathrm{i}_{Y_1}
    \cdots\widehat{\mathrm{i}_{Y_j}}\cdots\mathrm{i}_{Y_\ell}\omega\\
    &+(-1)^{\ell-1}\mathrm{i}_{Y_1}(\mathrm{d}a
    \wedge\mathrm{i}_{Y_2}\cdots\mathrm{i}_{Y_{\ell}}\omega)
    \bigg)\\
    =&\cdots\\
    =&-\mathrm{d}a\wedge\mathrm{i}_Y\omega
    +(-1)^\ell\bigg(
    \mathrm{i}_{Y_\ell}(\mathrm{d}a)\mathrm{i}_{Y_1}
    \cdots\mathrm{i}_{Y_{\ell-1}}\omega
    -\mathrm{i}_{Y_1}\cdots\mathrm{i}_{Y_{\ell-1}}(\mathrm{d}a
    \wedge\mathrm{i}_{Y_{\ell}}\omega)
    \bigg)\\
    =&-\mathrm{d}a\wedge\mathrm{i}_Y\omega
    +(-1)^\ell\mathrm{i}_{Y_1}\cdots\mathrm{i}_{Y_{\ell-1}}\bigg(
    \mathrm{i}_{Y_\ell}(\mathrm{d}a)\wedge\omega
    -\mathrm{d}a\wedge\mathrm{i}_{Y_\ell}\omega
    \bigg)\\
    =&-\mathrm{d}a\wedge\mathrm{i}_Y\omega
    +(-1)^\ell\mathrm{i}_Y(\mathrm{d}a\wedge\omega)\\
    =&\mathscr{L}_a\mathrm{i}_Y\omega
    -(-1)^\ell\mathrm{i}_Y\mathscr{L}_a\omega\\
    =&[\mathscr{L}_a,\mathrm{i}_Y]\omega,
\end{align*}
\end{allowdisplaybreaks}
and
\begin{align*}
    [\mathscr{L}_X,\mathrm{i}_Y]
    =&[\mathscr{L}_X,\mathrm{i}_{Y_1}]\mathrm{i}_{Y_2\wedge\cdots\wedge Y_\ell}
    +(-1)^{0\cdot 1}\mathrm{i}_{Y_1}
    [\mathscr{L}_X,\mathrm{i}_{Y_2\wedge\cdots\wedge Y_\ell}]\\
    =&\mathrm{i}_{[X,Y_1]}\mathrm{i}_{Y_2\wedge\cdots\wedge Y_\ell}
    +\mathrm{i}_{Y_1}([\mathscr{L}_X,\mathrm{i}_{Y_2}]
    \mathrm{i}_{Y_3\wedge\cdots\wedge Y_\ell}
    +\mathrm{i}_{Y_2}[\mathscr{L}_X,\mathrm{i}_{Y_3\wedge\cdots\wedge Y_\ell}])\\
    =&\cdots\\
    =&\sum_{j=1}^\ell\mathrm{i}_{Y_1\wedge\cdots\wedge Y_{j-1}\wedge[X,Y_j]
    \wedge Y_{j+1}\wedge\cdots\wedge Y_\ell}\\
    =&\mathrm{i}_{\llbracket X,Y\rrbracket},
\end{align*}
where we used the graded Leibniz rule of the graded commutator and
Lemma~\ref{lemma07}. This proves $[\mathscr{L}_X,\mathrm{i}_Y]
=\mathrm{i}_{\llbracket X,Y\rrbracket}$ for all $X\in\mathfrak{X}^k(\mathcal{A})$ 
with $k<2$ and all $Y\in\mathfrak{X}^\bullet(\mathcal{A})$. We assume now
inductively that this formula holds for a fixed $k>0$. Let $X\in\mathfrak{X}^k
(\mathcal{A})$, $Y\in\mathfrak{X}^1(\mathcal{A})$ and
$Z\in\mathfrak{X}^\ell(\mathcal{A})$ for an arbitrary $\ell\in\mathbb{N}$.
Then
\begin{allowdisplaybreaks}
\begin{align*}
    [\mathscr{L}_{X\wedge Y},\mathrm{i}_Z]
    =&[\mathrm{i}_X\mathscr{L}_Y-\mathscr{L}_X\mathrm{i}_Y,\mathrm{i}_Z]\\
    =&\mathrm{i}_X[\mathscr{L}_Y,\mathrm{i}_Z]
    +(-1)^{0\cdot\ell}[\mathrm{i}_X,\mathrm{i}_Z]\mathscr{L}_Y
    -\mathscr{L}_X[\mathrm{i}_Y,\mathrm{i}_Z]
    -(-1)^{1\cdot\ell}[\mathscr{L}_X,\mathrm{i}_Z]\mathrm{i}_Y\\
    =&\mathrm{i}_X\mathrm{i}_{\llbracket Y,Z\rrbracket}+0-0
    +(-1)^{\ell-1}\mathrm{i}_{\llbracket X,Z\rrbracket}\mathrm{i}_Y\\
    =&\mathrm{i}_{\llbracket X\wedge Y,Z\rrbracket},
\end{align*}
\end{allowdisplaybreaks}
where we utilized the graded Leibniz rules of $\llbracket\cdot,\cdot\rrbracket$
and $[\cdot,\cdot]$ as well as Lemma~\ref{lemma07}. We conclude that
$[\mathscr{L}_X,\mathrm{i}_Y]=\mathrm{i}_{\llbracket X,Y\rrbracket}$
holds for all $X,Y\in\mathfrak{X}^\bullet(\mathcal{A})$. 
For the remaining formula we note that
\begin{align*}
    [\mathscr{L}_X,\mathscr{L}_Y]
    =&[\mathscr{L}_X,[\mathrm{i}_Y,\mathrm{d}]]\\
    =&[[\mathscr{L}_X,\mathrm{i}_Y],\mathrm{d}]
    +(-1)^{(k-1)\cdot\ell}[\mathrm{i}_Y,[\mathscr{L}_X,\mathrm{d}]]\\
    =&[\mathrm{i}_{\llbracket X,Y\rrbracket},\mathrm{d}]+0\\
    =&\mathscr{L}_{\llbracket X,Y\rrbracket}
\end{align*}
holds for all $X\in\mathfrak{X}^k(\mathcal{A})$ and
$Y\in\mathfrak{X}^\ell(\mathcal{A})$. This concludes the proof of the theorem.
\end{proof}
The motivating example is of course the Cartan calculus on
$\mathcal{A}=\mathscr{C}^\infty(M)$ for a smooth manifold $M$.
In the next sections we are going to repeat the construction of the
Cartan calculus, however, in a more general setting. The commutative
algebra is replaced by a braided commutative one. Consistently,
the braided symmetry has to be transferred to every involved object
and morphism. In the words of Section~\ref{Sec2.5} we have to work in the
symmetric braided monoidal category of equivariant braided symmetric bimodules.

\section{Braided Multivector Fields}\label{Sec3.2}

In this section we intend to motivate why vector fields on braided
commutative algebras should be represented by braided
derivations rather than by usual derivations. The braided 
derivations are an equivariant braided symmetric bimodule
with respect to the adjoint Hopf algebra action and a braided Lie algebra
with braided Lie bracket given by the braided commutator. In this sense
the braided derivations are the "correct" generalization of vector
fields in the category of equivariant braided symmetric bimodules.
Keeping track of the braided symmetry we construct the braided
Graßmann algebra of braided derivations and furthermore extend
the braided commutator via a graded braided Leibniz rule to
higher order multivector fields. The latter turns out to be a braided
Gerstenhaber bracket, structuring the braided multivector fields as
a braided Gerstenhaber algebra.

Fix a triangular Hopf algebra $(H,\mathcal{R})$ and a braided commutative
algebra $\mathcal{A}$. Recall that this means that $\mathcal{A}$ is an
associative unital left $H$-module algebra such that
$b\cdot a=(\mathcal{R}_1^{-1}\rhd a)\cdot(\mathcal{R}_2^{-1}\rhd b)$
for all $a,b\in\mathcal{A}$. It is well known that the \textit{endomorphisms}
$\mathrm{End}_\Bbbk(\mathcal{A})$ of $\mathcal{A}$ can be structured
as a left $H$-module algebra via the \textit{adjoint action}
$$
(\xi\rhd\Phi)(a)
=\xi_{(1)}\rhd(\Phi(S(\xi_{(2)})\rhd a))
$$
for all $a\in\mathcal{A}$, where $\Phi\in\mathrm{End}_\Bbbk(\mathcal{A})$.
The action on $\mathrm{End}_\Bbbk(\mathcal{A})$ is well-defined since
the $H$-action on $\mathcal{A}$ is $\Bbbk$-linear, it respects the
product, which is the concatenation, because
\begin{align*}
    ((\xi_{(1)}\rhd\Phi)(\xi_{(2)}\rhd\Psi))(a)
    =&\xi_{(1)}\rhd(\Phi((S(\xi_{(2)})\xi_{(3)})\rhd(
    \Psi(S(\xi_{(4)})\rhd a))))\\
    =&\xi_{(1)}\rhd(\Phi(\Psi(S(\xi_{(2)})\rhd a)))\\
    =&(\xi\rhd(\Phi\circ\Psi))(a)
\end{align*}
for all $\Phi,\Psi\in\mathrm{End}_\Bbbk(\mathcal{A})$ and $a\in\mathcal{A}$
and it respects the unit endomorphism $\mathrm{id}_\mathcal{A}$, since
$$
(\xi\rhd\mathrm{id}_\mathcal{A})(a)
=\xi_{(1)}\rhd(\mathrm{id}_\mathcal{A}(S(\xi_{(2)})\rhd a))
=(\xi_{(1)}S(\xi_{(2)}))\rhd a
=\epsilon(\xi)\mathrm{id}_\mathcal{A}(a)
$$
for all $a\in\mathcal{A}$ and $\xi\in H$. It strikes that we did not
utilize the triangular structure of $H$ or the braided commutativity of
$\mathcal{A}$ at all. In fact, it follows by the same computation that 
$\mathrm{End}_\Bbbk(\mathcal{A})$ is a module algebra for any Hopf algebra
and any module algebra $\mathcal{A}$. However, the triangular structure
matters if one wants to restrict the $H$-module action of
endomorphisms to derivations. This is not possible in general
unless $H$ is cocommutative.
Moreover, the Lie bracket of derivations is not $H$-equivariant in general
unless $H$ is cocommutative. Also, for all
$a,b,c\in\mathcal{A}$ and $X\in\mathrm{Der}(\mathcal{A})$
$$
(a\cdot X)(b\cdot c)
=(a\cdot X)(b)\cdot c
+(\mathcal{R}_1^{-1}\rhd b)\cdot((\mathcal{R}_2^{-1}\rhd a)\cdot X)(c)
$$
further suggests that $\mathrm{Der}(\mathcal{A})$ is not an interesting
object in this setting since it is not a left $\mathcal{A}$-module
for the canonical module action in general unless $\mathcal{A}$ is
commutative and $H$ is cocommutative, i.e. unless $\mathcal{R}=1\otimes 1$.
To fix all of these problems we define \textit{braided derivations} as the 
$\Bbbk$-linear endomorphisms $X\colon\mathcal{A}\rightarrow\mathcal{A}$ of
$\mathcal{A}$ which satisfy
\begin{equation}
    X(a\cdot b)
    =X(a)\cdot b
    +(\mathcal{R}_1^{-1}\rhd a)\cdot(\mathcal{R}_2^{-1}\rhd X)(b)
\end{equation}
for all $a,b\in\mathcal{A}$. The $\Bbbk$-module of braided derivations on
$\mathcal{A}$ is denoted by $\mathrm{Der}_\mathcal{R}(\mathcal{A})$.
\begin{lemma}\label{lemma08}
The braided derivations on $\mathcal{A}$ form a braided Lie algebra with
respect to the braided commutator of endomorphisms
$$
[X,Y]_\mathcal{R}
=XY-(\mathcal{R}_1^{-1}\rhd Y)(\mathcal{R}_2^{-1}\rhd X),
$$
where $X,Y\in\mathrm{Der}_\mathcal{R}(\mathcal{A})$. Furthermore,
$\mathrm{Der}_\mathcal{R}(\mathcal{A})$ is
an $H$-equivariant braided symmetric $\mathcal{A}$-bimodule with respect to 
the adjoint Hopf algebra action and the left and right $\mathcal{A}$-module
actions defined by
\begin{equation}\label{eq26}
    (a\cdot X)(b)=a\cdot X(b)
    \text{ and }
    (X\cdot a)(b)
    =X(\mathcal{R}_1^{-1}\rhd b)\cdot(\mathcal{R}_2^{-1}\rhd a)
\end{equation}
for all $a,b\in\mathcal{A}$ and $X\in\mathrm{Der}_\mathcal{R}(\mathcal{A})$.
\end{lemma}
\begin{proof}
Let $\xi\in H$, $a,b,c\in\mathcal{A}$ and
$X,Y,Z\in\mathrm{Der}_\mathcal{R}(\mathcal{A})$. We split the proof
of the lemma in two parts.
\begin{enumerate}
\item[i.)] \textbf{$(\mathrm{Der}_\mathcal{R}(\mathcal{A}),
[\cdot,\cdot]_\mathcal{R})$ is a braided Lie algebra: }
first of all, the adjoint $H$-action is well-defined on 
$\mathrm{Der}_\mathcal{R}(\mathcal{A})$, since
\begin{allowdisplaybreaks}
\begin{align*}
    (\xi\rhd X)(a\cdot b)
    =&\xi_{(1)}\rhd\bigg(X((S(\xi_{(2)})_{(1)}\rhd a)
    \cdot(S(\xi_{(2)})_{(1)}\rhd a))\bigg)\\
    =&\xi_{(1)}\rhd\bigg(
    X(S(\xi_{(3)})\rhd a)\cdot(S(\xi_{(2)})\rhd b)\\
    &+((\mathcal{R}_1^{-1}S(\xi_{(3)}))\rhd a)
    (\mathcal{R}_2^{-1}\rhd X)(S(\xi_{(2)})\rhd b)\bigg)\\
    =&(\xi_{(1)}\rhd X)((\xi_{(2)}S(\xi_{(5)}))\rhd a)
    \cdot((\xi_{(3)}S(\xi_{(4)}))\rhd b)\\
    &+((\xi_{(1)}\mathcal{R}_1^{-1}S(\xi_{(5)}))\rhd a)
    ((\xi_{(2)}\mathcal{R}_2^{-1})\rhd X)((\xi_{(3)}S(\xi_{(4)}))\rhd b)\\
    =&(\xi_{(1)}\rhd X)((\xi_{(2)}S(\xi_{(3)}))\rhd a)
    \cdot b\\
    &+((\mathcal{R}_1^{-1}\xi_{(2)}S(\xi_{(3)}))\rhd a)
    ((\mathcal{R}_2^{-1}\xi_{(1)})\rhd X)(b)\\
    =&(\xi\rhd X)(a)\cdot b
    +(\mathcal{R}_1^{-1}\rhd a)
    (\mathcal{R}_2^{-1}\rhd(\xi\rhd X))(b)
\end{align*}
\end{allowdisplaybreaks}
by the quasi-cocommutativity of $\Delta$ and the anti-coalgebra property
of $S$. It remains to prove that $[X,Y]_\mathcal{R}
\in\mathrm{Der}_\mathcal{R}(\mathcal{A})$, since the braided commutator
of endomorphisms is clearly a braided Lie bracket. Using the hexagon
relations we see that the difference of
\begin{align*}
    X(Y(a\cdot b))
    =&X(Y(a)\cdot b
    +(\mathcal{R}_1^{-1}\rhd a)(\mathcal{R}_2^{-1}\rhd Y)(b))\\
    =&X(Y(a))\cdot b
    +(\mathcal{R}_1^{-1}\rhd(Y(a)))\cdot(\mathcal{R}_2^{-1}\rhd X)(b)\\
    &+X(\mathcal{R}_1^{-1}\rhd a)\cdot(\mathcal{R}_2^{-1}\rhd Y)(b))\\
    &+((\mathcal{R}_1^{'-1}\mathcal{R}_1^{-1})\rhd a)
    \cdot(\mathcal{R}_2^{'-1}\rhd X)((\mathcal{R}_2^{-1}\rhd Y)(b))
\end{align*}
and
\begin{align*}
    (\mathcal{R}_1^{-1}\rhd Y)((\mathcal{R}_2^{-1}\rhd X)(a\cdot b))
    =&(\mathcal{R}_1^{-1}\rhd Y)((\mathcal{R}_2^{-1}\rhd X)(a)\cdot b\\
    &+(\mathcal{R}_1^{'-1}\rhd a)
    \cdot((\mathcal{R}_2^{'-1}\mathcal{R}_2^{-1})\rhd X)(b))\\
    =&(\mathcal{R}_1^{-1}\rhd Y)((\mathcal{R}_2^{-1}\rhd X)(a))\cdot b\\
    &+(\mathcal{R}_1^{'-1}\rhd(\mathcal{R}_2^{-1}\rhd X)(a))
    \cdot((\mathcal{R}_2^{'-1}\mathcal{R}_1^{-1})\rhd Y)(b)\\
    &+(\mathcal{R}_1^{-1}\rhd Y)(\mathcal{R}_1^{'-1}\rhd a)
    \cdot((\mathcal{R}_2^{'-1}\mathcal{R}_2^{-1})\rhd X)(b)\\
    &+((\mathcal{R}_1^{''-1}\mathcal{R}_1^{'-1})\rhd a)\\
    &\cdot((\mathcal{R}_2^{''-1}\mathcal{R}_1^{-1})\rhd Y)(
    ((\mathcal{R}_2^{'-1}\mathcal{R}_2^{-1})\rhd X)(b))
\end{align*}
equals
\begin{align*}
    [X,Y]_\mathcal{R}(a)\cdot b
    +(\mathcal{R}_1\rhd a)\cdot(\mathcal{R}_2^{-1}\rhd[X,Y]_\mathcal{R})(b),
\end{align*}
where the second and third terms cancel.

\item[ii.)] \textbf{$\mathrm{Der}_\mathcal{R}(\mathcal{A})$ is an
$H$-equivariant braided symmetric $\mathcal{A}$-bimodule: }
the $\Bbbk$-linear maps in (\ref{eq26}) are $H$-equivariant, because
\begin{align*}
    (\xi\rhd(a\cdot X))(b)
    =&\xi_{(1)}\rhd((a\cdot X)(S(\xi_{(2)})\rhd b))\\
    =&\xi_{(1)}\rhd(a\cdot X(S(\xi_{(2)})\rhd b))\\
    =&(\xi_{(1)}\rhd a)\cdot(\xi_{(2)}\rhd X)
    ((\xi_{(3)}S(\xi_{(4)}))\rhd b)\\
    =&((\xi_{(1)}\rhd a)\cdot(\xi_{(2)}\rhd X))(b)
\end{align*}
and
\begin{align*}
    (\xi\rhd(X\cdot a))(b)
    =&\xi_{(1)}\rhd((X\cdot a)(S(\xi_{(2)})\rhd b))\\
    =&\xi_{(1)}\rhd(X((\mathcal{R}_1^{-1}S(\xi_{(2)}))\rhd b)
    \cdot(\mathcal{R}_2^{-1}\rhd a))\\
    =&(\xi_{(1)}\rhd X)((\xi_{(2)}\mathcal{R}_1^{-1}S(\xi_{(4)}))\rhd b)
    \cdot((\xi_{(3)}\mathcal{R}_2^{-1})\rhd a)\\
    =&(\xi_{(1)}\rhd X)((\mathcal{R}_1^{-1}\xi_{(3)}S(\xi_{(4)}))\rhd b)
    \cdot((\mathcal{R}_2^{-1}\xi_{(2)})\rhd a)\\
    =&((\xi_{(1)}\rhd X)\cdot(\xi_{(2)}\rhd a))(b),
\end{align*}
using the adjoint $H$-module action on $\Bbbk$-linear endomorphisms.
Next, we note that $a\cdot X$ and $X\cdot a$ are in fact braided derivations,
since
\begin{align*}
    (a\cdot X)(b\cdot c)
    =&a\cdot X(b\cdot c)\\
    =&a\cdot(X(b)\cdot c
    +(\mathcal{R}_1^{-1}\rhd b)\cdot(\mathcal{R}_2^{-1}\rhd X)(c))\\
    =&(a\cdot X)(b)\cdot c
    +((\mathcal{R}_1^{'-1}\mathcal{R}_1^{-1})\rhd b)
    \cdot(\mathcal{R}_2^{'-1}\rhd a)
    \cdot(\mathcal{R}_2^{-1}\rhd X)(c)\\
    =&(a\cdot X)(b)\cdot c
    +(\mathcal{R}_1^{-1}\rhd b)\cdot(\mathcal{R}_2^{-1}\rhd(a\cdot X))(c)
\end{align*}
and
\begin{align*}
    (X\cdot a)(b\cdot c)
    =&X((\mathcal{R}_{1(1)}^{-1}\rhd b)\cdot(\mathcal{R}_{1(2)}^{-1}\rhd c))
    \cdot(\mathcal{R}_2^{-1}\rhd a)\\
    =&\bigg(X(\mathcal{R}_{1(1)}^{-1}\rhd b)
    \cdot(\mathcal{R}_{1(2)}^{-1}\rhd c)\\
    &+((\mathcal{R}_1^{'-1}\mathcal{R}_{1(1)}^{-1})\rhd b)
    \cdot(\mathcal{R}_2^{'-1}\rhd X)(\mathcal{R}_{1(2)}^{-1}\rhd c)\bigg)
    \cdot(\mathcal{R}_2^{-1}\rhd a)\\
    =&X(\mathcal{R}_{1(1)}^{-1}\rhd b)
    \cdot((\mathcal{R}_1^{'-1}\mathcal{R}_2^{-1})\rhd a)
    \cdot((\mathcal{R}_2^{'-1}\mathcal{R}_{1(2)}^{-1})\rhd c)\\
    &+((\mathcal{R}_1^{'-1}\mathcal{R}_{1}^{-1})\rhd b)
    \cdot(\mathcal{R}_2^{'-1}\rhd X)(\mathcal{R}_{1}^{''-1}\rhd c)
    \cdot((\mathcal{R}_2^{''-1}\mathcal{R}_2^{-1})\rhd a)\\
    =&X(\mathcal{R}_{1}^{-1}\rhd b)
    \cdot((\mathcal{R}_1^{'-1}\mathcal{R}_2^{''-1}\mathcal{R}_2^{-1})\rhd a)
    \cdot((\mathcal{R}_2^{'-1}\mathcal{R}_{1}^{''-1})\rhd c)\\
    &+((\mathcal{R}_1^{'-1}\mathcal{R}_{1}^{-1})\rhd b)
    \cdot((\mathcal{R}_2^{'-1}\rhd X)\cdot(\mathcal{R}_2^{-1}\rhd a))(c)\\
    =&(X\cdot a)(b)\cdot c
    +(\mathcal{R}_1^{-1}\rhd b)\cdot
    (\mathcal{R}_2^{-1}\rhd(X\cdot a))(c).
\end{align*}
One has $(a\cdot(b\cdot X))(c)=a\cdot b\cdot X(c)=((a\cdot b)\cdot X)(c)$
and
\begin{align*}
    ((X\cdot a)\cdot b)(c)
    =&(X\cdot a)(\mathcal{R}_1^{-1}\rhd c)\cdot(\mathcal{R}_2^{-1}\rhd b)\\
    =&X((\mathcal{R}_1^{'-1}\mathcal{R}_1^{-1})\rhd c)
    \cdot(\mathcal{R}_2^{'-1}\rhd a)
    \cdot(\mathcal{R}_2^{-1}\rhd b)\\
    =&X(\mathcal{R}_1^{-1}\rhd c)
    \cdot(\mathcal{R}_{2(1)}^{-1}\rhd a)
    \cdot(\mathcal{R}_{2(2)}^{-1}\rhd b)\\
    =&X(\mathcal{R}_1^{-1}\rhd c)
    \cdot(\mathcal{R}_2^{-1}\rhd(a\cdot b))\\
    =&(X\cdot(a\cdot b))(c)
\end{align*}
showing that (\ref{eq26}) define $\mathcal{A}$-module actions.
They commute because
\begin{align*}
    ((a\cdot X)\cdot b)(c)
    =&(a\cdot X)(\mathcal{R}_1^{-1}\rhd c)\cdot(\mathcal{R}_2^{-1}\rhd b)\\
    =&a\cdot X(\mathcal{R}_1^{-1}\rhd c)\cdot(\mathcal{R}_2^{-1}\rhd b)\\
    =&a\cdot(X\cdot b)(c)\\
    =&(a\cdot(X\cdot b))(c)
\end{align*}
and they are braided symmetric, since
\begin{align*}
    (X\cdot a)(b)
    =&X(\mathcal{R}_1^{-1}\rhd b)\cdot(\mathcal{R}_2^{-1}\rhd a)\\
    =&((\mathcal{R}_1^{'-1}\mathcal{R}_2^{-1})\rhd a)
    \cdot(\mathcal{R}_{2(1)}^{'-1}\rhd X)
    ((\mathcal{R}_{2(2)}^{'-1}\mathcal{R}_1^{-1})\rhd b)\\
    =&((\mathcal{R}_1^{''-1}\mathcal{R}_1^{'-1}\mathcal{R}_2^{-1})\rhd a)
    \cdot(\mathcal{R}_{2}^{''-1}\rhd X)
    ((\mathcal{R}_{2}^{'-1}\mathcal{R}_1^{-1})\rhd b)\\
    =&((\mathcal{R}_1^{-1}\rhd a)\cdot(\mathcal{R}_2^{-1}\rhd X))(b).
\end{align*}
\end{enumerate}
This concludes the proof of the lemma.
\end{proof}
Since $\mathrm{Der}_\mathcal{R}(\mathcal{A})$ is an
$\mathcal{A}$-bimodule we can build the tensor algebra
$$
\mathrm{T}^\bullet\mathrm{Der}_\mathcal{R}(\mathcal{A})
=\mathcal{A}\oplus\mathrm{Der}_\mathcal{R}(\mathcal{A})
\oplus(\mathrm{Der}_\mathcal{R}(\mathcal{A})\otimes_\mathcal{A}
\mathrm{Der}_\mathcal{R}(\mathcal{A}))\oplus\cdots
$$
of $\mathrm{Der}_\mathcal{R}(\mathcal{A})$ with respect to the
tensor product $\otimes_\mathcal{A}$ over $\mathcal{A}$. It is a
braided symmetric $H$-equivariant $\mathcal{A}$-bimodule with
module actions defined on factorizing elements
$X_1\otimes_\mathcal{A}\cdots\otimes_\mathcal{A}X_k\in
\mathrm{T}^k\mathrm{Der}_\mathcal{R}(\mathcal{A})$ by
\begin{align}
    \xi\rhd(X_1\otimes_\mathcal{A}\cdots\otimes_\mathcal{A}X_k)
    =&(\xi_{(1)}\rhd X_1)\otimes_\mathcal{A}\cdots\otimes_\mathcal{A}
    (\xi_{(k)}\rhd X_k),\label{eq33}\\
    a\cdot(X_1\otimes_\mathcal{A}\cdots\otimes_\mathcal{A}X_k)
    =&(a\cdot X_1)\otimes_\mathcal{A}\cdots\otimes_\mathcal{A}X_k
    ,\label{eq34}\\
    (X_1\otimes_\mathcal{A}\cdots\otimes_\mathcal{A}X_k)\cdot a
    =&X_1\otimes_\mathcal{A}\cdots\otimes_\mathcal{A}(X_k\cdot a)\label{eq35}
\end{align}
for all $\xi\in H$ and $a\in\mathcal{A}$. Note however that
$(\mathrm{T}^\bullet\mathrm{Der}_\mathcal{R}(\mathcal{A}),
\otimes_\mathcal{A})$ is \textit{not} (graded) braided commutative in
general. There is an ideal $I$ in
$(\mathrm{T}^\bullet\mathrm{Der}_\mathcal{R}(\mathcal{A}),
\otimes_\mathcal{A})$ generated by elements
$
X_1\otimes_\mathcal{A}\cdots\otimes_\mathcal{A}X_k
\in\mathrm{T}^k\mathrm{Der}_\mathcal{R}(\mathcal{A})
$
which equal
\begin{align*}
    X_1\otimes_\mathcal{A}&\cdots\otimes_\mathcal{A}X_{i-1}
    \otimes_\mathcal{A}\bigg(\mathcal{R}_1^{'-1}\rhd\bigg(
    (\mathcal{R}_1^{-1}\rhd X_j)\otimes_\mathcal{A}
    (\mathcal{R}_2^{-1}\rhd(X_{i+1}\otimes_\mathcal{A}
    \cdots\otimes_\mathcal{A}X_{j-1}))\bigg)\bigg)\\
    &\otimes_\mathcal{A}(\mathcal{R}_2^{'-1}\rhd X_i)
    \otimes_\mathcal{A}X_{j+1}\otimes_\mathcal{A}\cdots
    \otimes_\mathcal{A}X_k
\end{align*}
for a pair $(i,j)$ such that $1\leq i<j\leq k$. We illustrate this
for small tensor powers: if $k=2$ the elements of $I$ are those
$X\otimes_\mathcal{A}Y\in\mathrm{T}^2\mathrm{Der}_\mathcal{R}(\mathcal{A})$
satisfying $X\otimes_\mathcal{A}Y=(\mathcal{R}_1^{-1}\rhd Y)
\otimes_\mathcal{A}(\mathcal{R}_2^{-1}\rhd X)$, if $k=3$ elements
$X\otimes_\mathcal{A}Y\otimes_\mathcal{A}Z\in
\mathrm{T}^3\mathrm{Der}_\mathcal{R}(\mathcal{A})$ have to equal
$(\mathcal{R}_1^{-1}\rhd Y)\otimes_\mathcal{A}(\mathcal{R}_2^{-1}\rhd X)
\otimes_\mathcal{A}Z$, $X\otimes_\mathcal{A}(\mathcal{R}_1^{-1}\rhd Z)
\otimes_\mathcal{A}(\mathcal{R}_2^{-1}\rhd Y)$ or
$$
((\mathcal{R}_{1(1)}^{'-1}\mathcal{R}_1^{-1})\rhd Z)\otimes_\mathcal{A}
((\mathcal{R}_{1(2)}^{'-1}\mathcal{R}_2^{-1})\rhd Y)
\otimes_\mathcal{A}(\mathcal{R}_2^{'-1}\rhd X)
$$
in order to belong to $I$. According to Lemma~\ref{lemma14} from the
appendix, the $H$-module action (\ref{eq33}) and the $\mathcal{A}$-bimodule
actions (\ref{eq34}) and (\ref{eq35}) respect the ideal $I$.
We denote the quotient
$\mathrm{T}^\bullet\mathrm{Der}_\mathcal{R}(\mathcal{A})/I$ by
$\mathfrak{X}^\bullet_\mathcal{R}(\mathcal{A})$ and call them
\textit{braided multivector fields} on $\mathcal{A}$. The induced
product $\wedge_\mathcal{R}$ is said to be the
\textit{braided wedge product}. By Proposition~\ref{prop14} we conclude
the following statement.
\begin{proposition}
The braided multivector fields
$(\mathfrak{X}^\bullet_\mathcal{R}(\mathcal{A}),\wedge_\mathcal{R})$
are a braided Graßmann algebra. Namely,
$\mathfrak{X}^\bullet_\mathcal{R}(\mathcal{A})$ is a braided symmetric
$H$-equivariant $\mathcal{A}$-bimodule such that
$H\rhd\mathfrak{X}^k_\mathcal{R}(\mathcal{A})
\subseteq\mathfrak{X}^k_\mathcal{R}(\mathcal{A})$ and
$$
\wedge_\mathcal{R}\colon\mathfrak{X}^k_\mathcal{R}(\mathcal{A})
\times\mathfrak{X}^\ell_\mathcal{R}(\mathcal{A})
\rightarrow\mathfrak{X}^{k+\ell}_\mathcal{R}(\mathcal{A})
$$
is associative, $H$-equivariant and graded braided commutative, i.e.
$$
Y\wedge_\mathcal{R}X
=(-1)^{k\cdot\ell}(\mathcal{R}_1^{-1}\rhd X)\wedge_\mathcal{R}
(\mathcal{R}_2^{-1}\rhd Y)
$$
for all $X\in\mathfrak{X}^k_\mathcal{R}(\mathcal{A})$ and
$Y\in\mathfrak{X}^\ell_\mathcal{R}(\mathcal{A})$.
\end{proposition}
We are defining a $\Bbbk$-bilinear operation
$\llbracket\cdot,\cdot\rrbracket_\mathcal{R}\colon
\mathfrak{X}^k_\mathcal{R}(\mathcal{A})\times
\mathfrak{X}^\ell_\mathcal{R}(\mathcal{A})
\rightarrow\mathfrak{X}^{k+\ell-1}_\mathcal{R}(\mathcal{A})$
in the following. If $a,b\in\mathcal{A}$ we set 
$\llbracket a,b\rrbracket_\mathcal{R}=0$. For $a\in\mathcal{A}$ and a
factorizing element
$X=X_1\wedge_\mathcal{R}\cdots\wedge_\mathcal{R}X_k
\in\mathfrak{X}^k_\mathcal{R}(\mathcal{A})$ where $k>0$, we define
\begin{align*}
    \llbracket X,a\rrbracket_\mathcal{R}
    =&\sum_{i=1}^k(-1)^{k-i}
    X_1\wedge_\mathcal{R}\cdots\wedge_\mathcal{R}X_{i-1}
    \wedge_\mathcal{R}(X_i(\mathcal{R}_1^{-1}\rhd a))\\
    &\wedge_\mathcal{R}\bigg(\mathcal{R}_2^{-1}\rhd\bigg(
    X_{i+1}\wedge_\mathcal{R}\cdots\wedge_\mathcal{R}X_k
    \bigg)\bigg)
\end{align*}
and
\begin{align*}
    \llbracket a,X\rrbracket_\mathcal{R}
    =&\sum_{i=1}^k(-1)^i\bigg(\mathcal{R}_{1(1)}^{-1}\rhd\bigg(
    X_1\wedge_\mathcal{R}\cdots\wedge_\mathcal{R}X_{i-1}\bigg)\bigg)
    \wedge_\mathcal{R}((\mathcal{R}_{1(2)}^{-1}\rhd X_i)
    (\mathcal{R}_2^{-1}\rhd a))\\
    &\wedge_\mathcal{R}
    X_{i+1}\wedge_\mathcal{R}\cdots\wedge_\mathcal{R}X_k.
\end{align*}
Furthermore, on factorizing elements
$X=X_1\wedge_\mathcal{R}\cdots\wedge_\mathcal{R}X_k
\in\mathfrak{X}^k_\mathcal{R}(\mathcal{A})$ and
$Y=Y_1\wedge_\mathcal{R}\cdots\wedge_\mathcal{R}Y_\ell
\in\mathfrak{X}^\ell_\mathcal{R}(\mathcal{A})$, where
$k,\ell>0$, we define
\begin{align*}
    \llbracket X&,Y\rrbracket_\mathcal{R}
    =\sum_{i=1}^k\sum_{j=1}^\ell(-1)^{i+j}
    [\mathcal{R}_1^{-1}\rhd X_i,\mathcal{R}_1^{'-1}\rhd Y_j]_\mathcal{R}\\
    &\wedge_\mathcal{R}
    \bigg(\mathcal{R}_2^{'-1}\rhd\bigg(
    \bigg(\mathcal{R}_2^{-1}\rhd(X_1\wedge_\mathcal{R}\cdots
    \wedge_\mathcal{R}X_{i-1})\bigg)
    \wedge_\mathcal{R}\widehat{X_i}\wedge_\mathcal{R}X_{i+1}
    \wedge_\mathcal{R}\cdots\wedge_\mathcal{R}X_k\\
    &Y_1\wedge_\mathcal{R}\cdots\wedge_\mathcal{R}Y_{j-1}\bigg)\bigg)
    \wedge_\mathcal{R}\widehat{Y_j}\wedge_\mathcal{R}Y_{j+1}
    \wedge_\mathcal{R}\cdots\wedge_\mathcal{R}Y_\ell,
\end{align*}
where $[\cdot,\cdot]_\mathcal{R}$ denotes the braided commutator.
The operation $\llbracket\cdot,\cdot\rrbracket_\mathcal{R}$
is said to be the \textit{braided Schouten-Nijenhuis bracket}.
In practice the bracket is
determined on $\mathfrak{X}^k_\mathcal{R}(\mathcal{A})$
for $k<2$ if we impose the \textit{graded braided Leibniz rules}
\begin{equation*}
    \llbracket X,Y\wedge_\mathcal{R}Z\rrbracket_\mathcal{R}
    =\llbracket X,Y\rrbracket_\mathcal{R}\wedge_\mathcal{R}Z
    +(-1)^{(k-1)\cdot\ell}(\mathcal{R}_1^{-1}\rhd Y)\wedge_\mathcal{R}
    \llbracket\mathcal{R}_2^{-1}\rhd X,Z\rrbracket_\mathcal{R}
\end{equation*}
and
\begin{equation*}
    \llbracket X\wedge_\mathcal{R}Y,Z\rrbracket_\mathcal{R}
    =X\wedge_\mathcal{R}\llbracket Y,Z\rrbracket_\mathcal{R}
    +(-1)^{\ell\cdot(m-1)}
    \llbracket X,\mathcal{R}_1^{-1}\rhd Z\rrbracket_\mathcal{R}
    \wedge_\mathcal{R}(\mathcal{R}_2^{-1}\rhd Y)
\end{equation*}
for all $X\in\mathfrak{X}^k_\mathcal{R}(\mathcal{A})$,
$Y\in\mathfrak{X}^\ell_\mathcal{R}(\mathcal{A})$ and
$Z\in\mathfrak{X}^m_\mathcal{R}(\mathcal{A})$. In those low orders
one obtains
\begin{align*}
    \llbracket a,b\rrbracket_\mathcal{R}=&0,\\
    \llbracket a,X\rrbracket_\mathcal{R}=&
    -(\mathcal{R}_1^{-1}\rhd X)(\mathcal{R}_2^{-1}\rhd a),\\
    \llbracket X,a\rrbracket_\mathcal{R}=&X(a),\\
    \llbracket X,Y\rrbracket_\mathcal{R}=&[X,Y]_\mathcal{R}
\end{align*}
for all $a,b\in\mathcal{A}$ and $X,Y\in\mathfrak{X}^1_\mathcal{R}(\mathcal{A})$.
As a consequence, we conclude the following result.
\begin{proposition}
The braided Schouten-Nijenhuis bracket
$\llbracket\cdot,\cdot\rrbracket_\mathcal{R}$ structures
the braided Graßmann algebra
$(\mathfrak{X}^\bullet_\mathcal{R}(\mathcal{A}),\wedge_\mathcal{R})$
of braided multivector fields
as a braided Gerstenhaber algebra. Namely,
$\llbracket\cdot,\cdot\rrbracket_\mathcal{R}\colon
\mathfrak{X}^k_\mathcal{R}(\mathcal{A})\times
\mathfrak{X}^\ell_\mathcal{R}(\mathcal{A})
\rightarrow\mathfrak{X}^{k+\ell-1}_\mathcal{R}(\mathcal{A})$ is
graded with respect to the degree shifted by $1$,
$H$-equivariant, graded braided skewsymmetric, i.e.
$$
\llbracket Y,X\rrbracket_\mathcal{R}
=-(-1)^{(k-1)\cdot(\ell-1)}\llbracket\mathcal{R}_1^{-1}\rhd X,
\mathcal{R}_2^{-1}\rhd Y\rrbracket_\mathcal{R},
$$
satisfies the graded braided Jacobi identity
$$
\llbracket X,\llbracket Y,Z\rrbracket_\mathcal{R}\rrbracket_\mathcal{R}
=\llbracket\llbracket X,Y\rrbracket_\mathcal{R},Z\rrbracket_\mathcal{R}
+(-1)^{(k-1)\cdot(\ell-1)}\llbracket\mathcal{R}_1^{-1}\rhd Y,
\llbracket\mathcal{R}_2^{-1}\rhd X,Z\rrbracket_\mathcal{R}
\rrbracket_\mathcal{R}
$$
and the graded braided Leibniz rule
$$
\llbracket X,Y\wedge_\mathcal{R}Z\rrbracket_\mathcal{R}
=\llbracket X,Y\rrbracket_\mathcal{R}\wedge_\mathcal{R}Z
+(-1)^{(k-1)\cdot\ell}(\mathcal{R}_1^{-1}\rhd Y)\wedge_\mathcal{R}
\llbracket\mathcal{R}_2^{-1}\rhd X,Z\rrbracket_\mathcal{R},
$$
where $X\in\mathfrak{X}^k_\mathcal{R}(\mathcal{A})$,
$Y\in\mathfrak{X}^\ell_\mathcal{R}(\mathcal{A})$
and
$Z\in\mathfrak{X}^\bullet_\mathcal{R}(\mathcal{A})$.
Furthermore, $\llbracket\cdot,\cdot\rrbracket_\mathcal{R}$ is the unique
braided Gerstenhaber bracket on $(\mathfrak{X}^\bullet_\mathcal{R}(\mathcal{A}),
\wedge_\mathcal{R})$ such that 
$$
\llbracket X,a\rrbracket_\mathcal{R}=X(a)
\text{ and }
\llbracket X,Y\rrbracket_\mathcal{R}=[X,Y]_\mathcal{R}
$$
hold for all $a\in\mathcal{A}$ and $X,Y\in\mathfrak{X}^1_\mathcal{R}(\mathcal{A})$.
\end{proposition}

\section{Braided Differential Forms}\label{Sec3.3}

Considering $\mathfrak{X}^1_\mathcal{R}(\mathcal{A})$ as a module
over $\mathcal{A}$, its dual space consists of $\Bbbk$-linear maps
$\mathfrak{X}^1_\mathcal{R}(\mathcal{A})$ which are right $\mathcal{A}$-linear
in addition. It can be structured as an equivariant braided symmetric
$\mathcal{A}$-bimodule and consequently its braided exterior algebra is
a well-defined braided Graßmann algebra. The duality with braided
multivector fields becomes explicit if one considers the braided insertion
of braided multivector fields. We prove that the insertion is an
equivariant map, which is left $\mathcal{A}$-linear in the first and
right $\mathcal{A}$-linear in the second argument. It is
a morphism in the category of equivariant braided symmetric bimodules.
In analogy to the de Rham differential we define a differential in low
degrees and extend it as a braided derivation of the braided wedge product to
higher orders. Since this differential turns out to be equivariant
it is actually a derivation and a braided derivation at the
same time. Afterwards we define the braided differential forms to be the
smallest differential graded subalgebra with respect to the previous
differential, which shelters $\mathcal{A}$ (c.f. \cite{D-VM96}). In other words,
the braided differential forms are generated by $\mathcal{A}$ and the
differential.

Consider the $\Bbbk$-module $\underline{\Omega}^1_\mathcal{R}(\mathcal{A})$
of $\Bbbk$-linear maps
$\omega\colon\mathrm{Der}_\mathcal{R}(\mathcal{A})\rightarrow\mathcal{A}$
such that $\omega(X\cdot a)=\omega(X)\cdot a$ for all $a\in\mathcal{A}$ and
$X\in\mathrm{Der}_\mathcal{R}(\mathcal{A})$.
\begin{lemma}
$\underline{\Omega}^1_\mathcal{R}(\mathcal{A})$ is an $H$-equivariant
braided symmetric $\mathcal{A}$-bimodule with respect to the $H$-adjoint action and
left and right $\mathcal{A}$-module actions given by
$$
(a\cdot\omega)(X)=a\cdot\omega(X)
\text{ and }
(\omega\cdot a)(X)=\omega(\mathcal{R}_1^{-1}\rhd X)\cdot(\mathcal{R}_2^{-1}\rhd a)
$$
for all $a\in\mathcal{A}$ and $X\in\mathrm{Der}_\mathcal{R}(\mathcal{A})$.
\end{lemma}
\begin{proof}
Let $a,b,c\in\mathcal{A}$, $\omega\in\underline{\Omega}^1_\mathcal{R}(\mathcal{A})$
and $X\in\mathrm{Der}_\mathcal{R}(\mathcal{A})$.
First of all, the $H$-module action and the $\mathcal{A}$-module actions are
well-defined, since
\begin{align*}
    (\xi\rhd\omega)(X\cdot a)
    =&\xi_{(1)}\rhd(\omega((S(\xi_{(2)})_{(1)}\rhd X)
    \cdot(S(\xi_{(2)})_{(2)}\rhd a)))\\
    =&\xi_{(1)}\rhd(\omega(S(\xi_{(3)})\rhd X)
    \cdot(S(\xi_{(2)})\rhd a))\\
    =&(\xi_{(1)}\rhd\omega(S(\xi_{(4)})\rhd X))
    \cdot((\xi_{(2)}S(\xi_{(3)}))\rhd a)\\
    =&((\xi_{(1)}\rhd\omega)((\xi_{(2)}S(\xi_{(3)}))\rhd X))\cdot a\\
    =&((\xi\rhd\omega)(X))\cdot a,
\end{align*}
$(a\cdot\omega)(X\cdot b)=a\cdot\omega(X\cdot b)=a\cdot\omega(X)\cdot b
=(a\cdot\omega)(X)\cdot b$ and
\begin{align*}
    (\omega\cdot a)(X\cdot b)
    =&\omega((\mathcal{R}_{1(1)}^{-1}\rhd X)\cdot(\mathcal{R}_{1(2)}^{-1}\rhd b))
    \cdot(\mathcal{R}_2^{-1}\rhd a)\\
    =&\omega(\mathcal{R}_{1(1)}^{-1}\rhd X)
    \cdot((\mathcal{R}_1^{'-1}\mathcal{R}_2^{-1})\rhd a)
    \cdot((\mathcal{R}_2^{'-1}\mathcal{R}_{1(2)}^{-1})\rhd b)\\
    =&\omega(\mathcal{R}_{1}^{-1}\rhd X)
    \cdot((\mathcal{R}_1^{'-1}\mathcal{R}_2^{''-1}\mathcal{R}_2^{-1})\rhd a)
    \cdot((\mathcal{R}_2^{'-1}\mathcal{R}_{1}^{''-1})\rhd b)\\
    =&\omega(\mathcal{R}_{1}^{-1}\rhd X)
    \cdot(\mathcal{R}_2^{-1}\rhd a)\cdot b\\
    =&(\omega\cdot a)(X)\cdot b
\end{align*}
hold by the hexagon relations and the bialgebra anti-homomorphism properties of
$S$.
By the associativity of the product on $\mathcal{A}$ and the hexagon relations
those assignments are left and right $\mathcal{A}$-module actions, respectively.
They commute since $((a\cdot\omega)\cdot b)(c)
=(a\cdot\omega)(\mathcal{R}_1^{-1}\rhd c)\cdot(\mathcal{R}_2^{-1}\rhd b)
=(a\cdot(\omega\cdot b))(c)$. Furthermore
\begin{allowdisplaybreaks}
\begin{align*}
    (\xi\rhd(a\cdot\omega\cdot b))(c)
    =&\xi_{(1)}\rhd((a\cdot\omega\cdot b)(S(\xi_{(2)})\rhd c))\\
    =&(\xi_{(1)}\rhd a)\cdot(\xi_{(2)}\rhd\omega)
    ((\xi_{(3)}\mathcal{R}_1^{-1}S(\xi_{(5)}))\rhd c)
    \cdot((\xi_{(4)}\mathcal{R}_2^{-1})\rhd b)\\
    =&(\xi_{(1)}\rhd a)\cdot(\xi_{(2)}\rhd\omega)
    ((\mathcal{R}_1^{-1}\xi_{(4)}S(\xi_{(5)}))\rhd c)
    \cdot((\mathcal{R}_2^{-1}\xi_{(3)})\rhd b)\\
    =&((\xi_{(1)}\rhd a)\cdot(\xi_{(2)}\rhd\omega)\cdot(\xi_{(3)}\rhd b))(c)
\end{align*}
\end{allowdisplaybreaks}
proves that $\underline{\Omega}^1_\mathcal{R}(\mathcal{A})$ is an
$H$-equivariant $\mathcal{A}$-bimodule. It is braided symmetric because
\begin{align*}
    (\omega\cdot a)(b)
    =&\omega(\mathcal{R}_1^{-1}\rhd b)\cdot(\mathcal{R}_2^{-1}\rhd a)\\
    =&((\mathcal{R}_1^{'-1}\mathcal{R}_2^{-1})\rhd a)
    \cdot(\mathcal{R}_{2(1)}^{'-1}\rhd\omega)(
    (\mathcal{R}_{2(2)}^{'-1}\mathcal{R}_1^{-1})\rhd b)\\
    =&(\mathcal{R}_1^{-1}\rhd a)\cdot(\mathcal{R}_2^{-1}\rhd\omega)(b)\\
    =&((\mathcal{R}_1^{-1}\rhd a)\cdot(\mathcal{R}_2^{-1}\rhd\omega))(b).
\end{align*}
This concludes the proof.
\end{proof}
It follows from Proposition~\ref{prop14} and Lemma~\ref{lemma09} that
the braided exterior algebra $
\underline{\Omega}^\bullet_\mathcal{R}(\mathcal{A})$ 
of $\underline{\Omega}^1_\mathcal{R}(\mathcal{A})$ is an $H$-equivariant
braided symmetric $\mathcal{A}$-bimodule and a graded braided
commutative algebra. In the following lines we show that it is also
compatible with the braided evaluation.
For two elements $\omega,\eta\in\underline{\Omega}^1_\mathcal{R}(\mathcal{A})$
we define a $\Bbbk$-bilinear map
$\omega\wedge_\mathcal{R}\eta\colon
\mathrm{Der}_\mathcal{R}(\mathcal{A})\times\mathrm{Der}_\mathcal{R}(\mathcal{A})
\rightarrow\mathcal{A}$ by
$$
(\omega\wedge_\mathcal{R}\eta)(X,Y)
=\omega(\mathcal{R}_1^{-1}\rhd X)\cdot(\mathcal{R}_2^{-1}\rhd\eta)(Y)
-\omega(\mathcal{R}_1^{-1}\rhd Y)\cdot(\mathcal{R}_2^{-1}\rhd(\eta(X)))
$$
for all $X,Y\in\mathrm{Der}_\mathcal{R}(\mathcal{A})$.
One proves that
$$
-(\omega\wedge_\mathcal{R}\eta)
(\mathcal{R}_1^{-1}\rhd Y,\mathcal{R}_1^{-1}\rhd X)
=(\omega\wedge_\mathcal{R}\eta)(X,Y)
=-((\mathcal{R}_1^{-1}\rhd\eta)\wedge_\mathcal{R}
(\mathcal{R}_2^{-1}\rhd\omega))(X,Y)
$$
and that
\begin{align*}
    (\omega\wedge_\mathcal{R}\eta)(X,Y\cdot a)
    =&((\omega\wedge_\mathcal{R}\eta)(X,Y))\cdot a,\\
    (\omega\wedge_\mathcal{R}\eta)(a\cdot X,Y)
    =&(\mathcal{R}_1^{-1}\rhd a)\cdot
    ((\mathcal{R}_2^{-1}\rhd(\omega\wedge_\mathcal{R}\eta))(X,Y)),\\
    \xi\rhd((\omega\wedge_\mathcal{R}\eta)(X,Y))
    =&((\xi_{(1)}\rhd\omega)\wedge_\mathcal{R}(\xi_{(2)}\rhd\eta))
    (\xi_{(3)}\rhd X,\xi_{(4)}\rhd Y)
\end{align*}
hold for all $\xi\in H$, $\omega,\eta\in\underline{\Omega}^1_\mathcal{R}
(\mathcal{A})$, $a\in\mathcal{A}$ and
$X,Y\in\mathrm{Der}_\mathcal{R}(\mathcal{A})$.
The evaluations of the $H$-action and $\mathcal{A}$-module actions read
\begin{align*}
    (\xi\rhd(\omega\wedge_\mathcal{R}\eta))(X,Y)
    =&\xi_{(1)}\rhd((\omega\wedge_\mathcal{R}\eta)
    (S(\xi_{(3)})\rhd X,S(\xi_{(2)})\rhd Y)),\\
    (a\cdot(\omega\wedge_\mathcal{R}\eta))(X,Y)
    =&a\cdot((\omega\wedge_\mathcal{R}\eta)(X,Y)),\\
    ((\omega\wedge_\mathcal{R}\eta)\cdot a)(X,Y)
    =&((\omega\wedge_\mathcal{R}\eta)
    (\mathcal{R}_{1(1)}^{-1}\rhd X,\mathcal{R}_{1(2)}^{-1}\rhd Y))
    \cdot(\mathcal{R}_2^{-1}\rhd a).
\end{align*}
Inductively one defines the evaluation of higher wedge products.
Explicitly, the evaluated module actions on factorizing elements
$\omega_1\wedge_\mathcal{R}\ldots\wedge_\mathcal{R}\omega_k
\in\underline{\Omega}^k_\mathcal{R}(\mathcal{A})$ read
\begin{align*}
    (\xi\rhd&(\omega_1\wedge_\mathcal{R}\ldots\wedge_\mathcal{R}\omega_k))
    (X_1,\ldots,X_k)\\
    =&\xi_{(1)}\rhd((\omega_1\wedge_\mathcal{R}\ldots\wedge_\mathcal{R}\omega_k)
    (S(\xi_{(k+1)})\rhd X_1,\ldots,S(\xi_{(2)})\rhd X_k)),
\end{align*}
$$
(a\cdot(\omega_1\wedge_\mathcal{R}\ldots\wedge_\mathcal{R}\omega_k))
    (X_1,\ldots,X_k)
    =a\cdot((\omega_1\wedge_\mathcal{R}\ldots\wedge_\mathcal{R}\omega_k)
    (X_1,\ldots,X_k))
$$
and
\begin{align*}
    ((\omega_1&\wedge_\mathcal{R}\ldots\wedge_\mathcal{R}\omega_k)\cdot a)
    (X_1,\ldots,X_k)\\
    =&((\omega_1\wedge_\mathcal{R}\ldots\wedge_\mathcal{R}\omega_k)
    (\mathcal{R}_{1(1)}^{-1}\rhd X_1,\ldots,\mathcal{R}_{1(k)}^{-1}\rhd X_k))
    \cdot(\mathcal{R}_2^{-1}\rhd a).
\end{align*}
for all $X_1,\ldots,X_1\in\mathrm{Der}_\mathcal{R}(\mathcal{A})$,
$a\in\mathcal{A}$ and $\xi\in H$.
It is useful to further define the insertion $\mathrm{i}^\mathcal{R}_X\colon
\underline{\Omega}^\bullet_\mathcal{R}(\mathcal{A})
\rightarrow\underline{\Omega}^{\bullet-1}_\mathcal{R}(\mathcal{A})$ of an element
$X\in\mathrm{Der}_\mathcal{R}(\mathcal{A})$ into the \textit{last} slot of
an element $\omega\in\underline{\Omega}^k_\mathcal{R}(\mathcal{A})$ by
$$
\mathrm{i}^\mathcal{R}_X\omega
=(-1)^{k-1}(\mathcal{R}_1^{-1}\rhd\omega)(\cdot,\ldots,\cdot,
\mathcal{R}_2^{-1}\rhd X).
$$
More general, we inductively define 
$$
\mathrm{i}^\mathcal{R}_{X\wedge_\mathcal{R}Y}
=\mathrm{i}^\mathcal{R}_{X}
\mathrm{i}^\mathcal{R}_{Y}
$$
for all $X,Y\in\mathfrak{X}^\bullet_\mathcal{R}(\mathcal{A})$.
\begin{lemma}
$(\underline{\Omega}^\bullet_\mathcal{R}(\mathcal{A}),\wedge_\mathcal{R})$
is a graded braided commutative associative unital algebra 
and an $H$-equivariant braided symmetric $\mathcal{A}$-bimodule. The insertion 
$$
\mathrm{i}^\mathcal{R}\colon\mathfrak{X}^k_\mathcal{R}(\mathcal{A})\otimes
\underline{\Omega}^\bullet_\mathcal{R}(\mathcal{A})
\rightarrow\underline{\Omega}^{\bullet-k}_\mathcal{R}(\mathcal{A})
$$
of braided multivector field
is $H$-equivariant such that $\mathrm{i}^\mathcal{R}_X$ is a
right $\mathcal{A}$-linear and braided left
$\mathcal{A}$-linear homogeneous map of degree $-k$
for all $X\in\mathfrak{X}^k_\mathcal{R}(\mathcal{A})$. Furthermore,
$\mathrm{i}^\mathcal{R}_X$ is
left $\mathcal{A}$-linear and braided right $\mathcal{A}$-linear in $X$.
For $k=1$ we obtain a
graded braided derivation $\mathrm{i}^\mathcal{R}_X$ of degree $-1$.
\end{lemma}
\begin{proof}
We already proved that 
$\underline{\Omega}^\bullet_\mathcal{R}(\mathcal{A})$ is an
$H$-equivariant braided symmetric $\mathcal{A}$-bimodule
and that it is compatible with braided evaluation.
We prove that $\mathrm{i}^\mathcal{R}_X$
is a braided graded derivation of the wedge product for $X\in\mathrm{Der}_\mathcal{R}
(\mathcal{A})$. Let $\omega,\eta\in\underline{\Omega}^1_\mathcal{R}(\mathcal{A})$.
Then
\begin{align*}
    \mathrm{i}^\mathcal{R}_X(\omega\wedge_\mathcal{R}\eta)
    =&(-1)^{2-1}((\mathcal{R}_{1(1)}^{-1}\rhd\omega)
    \wedge_\mathcal{R}(\mathcal{R}_{1(2)}^{-1}\rhd\eta))
    (\cdot,\mathcal{R}_2^{-1}\rhd X)\\
    =&-(\mathcal{R}_{1(1)}^{-1}\rhd\omega)
    (\mathcal{R}_{1(2)}^{-1}\rhd\eta)(\mathcal{R}_2^{-1}\rhd X)\\
    &+(\mathcal{R}_{1(1)}^{-1}\rhd\omega)
    ((\mathcal{R}_1^{'-1}\mathcal{R}_2^{-1})\rhd X)
    ((\mathcal{R}_2^{'-1}\mathcal{R}_{1(2)}^{-1})\rhd\eta)\\
    =&\mathrm{i}^\mathcal{R}_X(\omega)\wedge_\mathcal{R}\eta
    +(-1)^{1\cdot 1}(\mathcal{R}_1^{-1}\rhd\omega)\wedge_\mathcal{R}
    \mathrm{i}^\mathcal{R}_{\mathcal{R}_2^{-1}\rhd X}\eta.
\end{align*}
In particular this implies $\xi\rhd(\mathrm{i}^\mathcal{R}_X
(\omega\wedge_\mathcal{R}\eta))
=\mathrm{i}^\mathcal{R}_{\xi_{(1)}\rhd
X}((\xi_{(2)}\rhd\omega)\wedge_\mathcal{R}(\xi_{(3)}\rhd\eta))$
for all $\xi\in H$.
Inductively, one shows
\begin{align*}
    \mathrm{i}^\mathcal{R}_X(\omega\wedge_\mathcal{R}\eta)
    =(\mathrm{i}^\mathcal{R}_X\omega)\wedge_\mathcal{R}\eta
    +(-1)^k(\mathcal{R}_1^{-1}\rhd\omega)\wedge_\mathcal{R}
    \mathrm{i}^\mathcal{R}_{\mathcal{R}_2^{-1}\rhd X}\eta
\end{align*}
and $\xi\rhd(\mathrm{i}^\mathcal{R}_X\eta)
=\mathrm{i}^\mathcal{R}_{\xi_{(1)}\rhd X}(\xi_{(2)}\rhd\eta)$
for all $\xi\in H$, $X\in\mathrm{Der}_\mathcal{R}(\mathcal{A})$, $\omega\in
\underline{\Omega}^k_\mathcal{R}(\mathcal{A})$ and $\eta\in
\underline{\Omega}^\bullet_\mathcal{R}(\mathcal{A})$.
For factorizing elements $X_1\wedge_\mathcal{R}
X_2\in\mathfrak{X}^2_\mathcal{R}(\mathcal{A})$ this implies
\begin{align*}
    \xi\rhd\mathrm{i}^\mathcal{R}_{X_1\wedge_\mathcal{R}X_2}\omega
    =&\xi\rhd(\mathrm{i}^\mathcal{R}_{X_1}
    \mathrm{i}^\mathcal{R}_{X_2}\omega)
    =\mathrm{i}^\mathcal{R}_{\xi_{(1)}\rhd X_1}
    \mathrm{i}^\mathcal{R}_{\xi_{(2)}\rhd X_2}(\xi_{(3)}\rhd\omega)
    =\mathrm{i}^\mathcal{R}_{\xi_{(1)}\rhd(X_1\wedge_\mathcal{R}X_2)}
    (\xi_{(2)}\rhd\omega)
\end{align*}
for all $\xi\in H$ and $\omega\in\underline{\Omega}^\bullet_\mathcal{R}(\mathcal{A})$
and inductively one obtains $\xi\rhd(\mathrm{i}^\mathcal{R}_X\omega)
=\mathrm{i}^\mathcal{R}_{\xi_{(1)}\rhd X}(\xi_{(2)}\rhd\omega)$ for any
$X\in\underline{\Omega}^\bullet_\mathcal{R}(\mathcal{A})$. It is easy to verify
that $\mathrm{i}^\mathcal{R}_X$ inherits the linearity properties
\begin{align*}
    \mathrm{i}^\mathcal{R}_{a\cdot X}\omega
    =&a\cdot(\mathrm{i}^\mathcal{R}_X\omega),\\
    \mathrm{i}^\mathcal{R}_{X\cdot a}\omega
    =&(\mathrm{i}^\mathcal{R}_X(\mathcal{R}_1^{-1}\rhd\omega))
    \cdot(\mathcal{R}_2^{-1}\rhd a),\\
    \mathrm{i}^\mathcal{R}_X(\omega\cdot a)
    =&(\mathrm{i}^\mathcal{R}_X\omega)\cdot a,\\
    \mathrm{i}^\mathcal{R}_X(a\cdot\omega)
    =&(\mathcal{R}_1^{-1}\rhd a)\cdot
    (\mathrm{i}^\mathcal{R}_{\mathcal{R}_2^{-1}\rhd X}\omega)
\end{align*}
for all $X\in\mathfrak{X}^\bullet_\mathcal{R}(\mathcal{A})$,
$a\in\mathcal{A}$ and $\omega\in\underline{\Omega}^\bullet_\mathcal{R}
(\mathcal{A})$.
\end{proof}
One defines a $\Bbbk$-linear map
$\mathrm{d}\colon\underline{\Omega}_\mathcal{R}^\bullet(\mathcal{A})
\rightarrow\underline{\Omega}_\mathcal{R}^{\bullet+1}(\mathcal{A})$
on $a\in\mathcal{A}$ by $\mathrm{i}^\mathcal{R}_X(\mathrm{d}a)=X(a)$ for all
$X\in\mathrm{Der}_\mathcal{R}(\mathcal{A})$, on 
$\omega\in\underline{\Omega}^1_\mathcal{R}(\mathcal{A})$ by
$$
(\mathrm{d}\omega)(X,Y)
=(\mathcal{R}_1^{-1}\rhd X)((\mathcal{R}_2^{-1}\rhd\omega)(Y))
-(\mathcal{R}_1^{-1}\rhd Y)(\mathcal{R}_2^{-1}\rhd(\omega(X)))
-\omega([X,Y]_\mathcal{R})
$$
for all $X,Y\in\mathrm{Der}_\mathcal{R}(\mathcal{A})$
and extends $\mathrm{d}$
to higher wedge powers by demanding it to be a
graded derivation with respect to $\wedge_\mathcal{R}$, i.e.
$$
\mathrm{d}(\omega_1\wedge_\mathcal{R}\omega_2)
=(\mathrm{d}\omega_1)\wedge_\mathcal{R}\omega_2
+(-1)^k\omega_1\wedge_\mathcal{R}(\mathrm{d}\omega_2)
$$
for $\omega_1\in\underline{\Omega}_\mathcal{R}^k(\mathcal{A})$
and $\omega_2\in\underline{\Omega}_\mathcal{R}^\bullet(\mathcal{A})$.
One can also directly define
$\mathrm{d}\omega\in\underline{\Omega}^{k+1}_\mathcal{R}(\mathcal{A})$ for any
$\omega\in\underline{\Omega}^{k}_\mathcal{R}(\mathcal{A})$ by
\begin{allowdisplaybreaks}
\begin{align*}
    (\mathrm{d}\omega)(X_0,\ldots,X_k)
    =&\sum_{i=0}^k(-1)^i
    (\mathcal{R}_1^{-1}\rhd X_i)
    \bigg((\mathcal{R}_{2(1)}^{-1}\rhd\omega)\bigg(\\
    &\mathcal{R}_{2(2)}^{-1}\rhd X_0,\ldots,
    \mathcal{R}_{2(i+1)}^{-1}\rhd X_{i-1},
    \widehat{X_i},X_{i+1},\ldots,X_k\bigg)\bigg)\\
    &+\sum_{i<j}(-1)^{i+j}
    \omega\bigg([\mathcal{R}_1^{-1}\rhd X_i,
    \mathcal{R}_1^{'-1}\rhd X_j]_\mathcal{R},\\
    &(\mathcal{R}_{2(1)}^{'-1}\mathcal{R}_{2(1)}^{-1})\rhd X_{0},
    \ldots,(\mathcal{R}_{2(i)}^{'-1}\mathcal{R}_{2(i)}^{-1})\rhd X_{i-1},
    \widehat{X_i},\\
    &\mathcal{R}_{2(i+1)}^{'-1}\rhd X_{i+1},\ldots,
    \mathcal{R}_{2(j-1)}^{'-1}\rhd X_{j-1},\widehat{X_j},
    X_{j+1},\ldots,X_k\bigg)
\end{align*}
\end{allowdisplaybreaks}
for all $X_0,\ldots,X_k\in\mathrm{Der}_\mathcal{R}(\mathcal{A})$.
Using the above formula and the fact that the braided
commutator is $H$-equivariant, it immediately follows that $\mathrm{d}$
commutes with the left $H$-module action. In other words, $\mathrm{d}$
is an \textit{integral} for the adjoint action, i.e.
$$
(\xi\rhd\mathrm{d})\omega
=\xi_{(1)}\rhd(\mathrm{d}(S(\xi_{(2)})\rhd\omega))
=(\xi_{(1)}S(\xi_{(2)}))\rhd(\mathrm{d}\omega)
=\epsilon(\xi)\mathrm{d}\omega
$$
for all $\xi\in H$ and $\omega\in\Omega^\bullet_\mathcal{R}(\mathcal{A})$.
\begin{lemma}
The map
$\mathrm{d}\colon\underline{\Omega}_\mathcal{R}^\bullet(\mathcal{A})
\rightarrow\underline{\Omega}_\mathcal{R}^{\bullet+1}(\mathcal{A})$
is a differential, i.e. $\mathrm{d}^2=0$.
\end{lemma}
\begin{proof}
It is sufficient to prove $\mathrm{d}^2=0$ on $\underline{\Omega}^k_\mathcal{R}
(\mathcal{A})$ for $k<2$, since $\mathrm{d}^2$ is a graded braided derivation.
Let $X,Y,Z\in\mathrm{Der}_\mathcal{R}(\mathcal{A})$. For $a\in\mathcal{A}$
we obtain
\begin{allowdisplaybreaks}
\begin{align*}
    (\mathrm{d}^2a)(X,Y)
    =&(\mathcal{R}_1^{-1}\rhd X)((\mathcal{R}_2^{-1}\rhd(\mathrm{d}a))(Y))
    -(\mathcal{R}_1^{-1}\rhd Y)(\mathcal{R}_2^{-1}\rhd((\mathrm{d}a)(X)))\\
    &-(\mathrm{d}a)([X,Y]_\mathcal{R})\\
    =&(\mathcal{R}_1^{-1}\rhd X)
    ((\mathcal{R}_1^{'-1}\rhd Y)((\mathcal{R}_2^{'-1}\mathcal{R}_2^{-1})\rhd a))\\
    &-(\mathcal{R}_1^{-1}\rhd Y)
    (\mathcal{R}_2^{-1}\rhd((\mathcal{R}_1^{'-1}\rhd X)
    (\mathcal{R}_2^{'-1}\rhd a)))\\
    &-(\mathcal{R}_1^{-1}\rhd[X,Y]_\mathcal{R})(\mathcal{R}_2^{-1}\rhd a)\\
    =&(\mathcal{R}_{1(1)}^{-1}\rhd X)
    ((\mathcal{R}_{1(2)}^{-1}\rhd Y)(\mathcal{R}_2^{-1}\rhd a))\\
    &-((\mathcal{R}_1^{''-1}\mathcal{R}_1^{-1})\rhd Y)
    (((\mathcal{R}_2^{''-1}\mathcal{R}_1^{'-1})\rhd X)
    ((\mathcal{R}_2^{-1}\mathcal{R}_2^{'-1})\rhd a))\\
    &-[\mathcal{R}_{1(1)}^{-1}\rhd X,\mathcal{R}_{1(2)}^{-1}\rhd Y]_\mathcal{R}
    (\mathcal{R}_2^{-1}\rhd a)\\
    =&(\mathcal{R}_{1(1)}^{-1}\rhd X)
    ((\mathcal{R}_{1(2)}^{-1}\rhd Y)(\mathcal{R}_2^{-1}\rhd a))\\
    &-((\mathcal{R}_1^{''-1}\mathcal{R}_{1(2)}^{-1})\rhd Y)
    (((\mathcal{R}_2^{''-1}\mathcal{R}_{1(1)}^{-1})\rhd X)
    (\mathcal{R}_2^{-1}\rhd a))\\
    &-[\mathcal{R}_{1(1)}^{-1}\rhd X,\mathcal{R}_{1(2)}^{-1}\rhd Y]_\mathcal{R}
    (\mathcal{R}_2^{-1}\rhd a)\\
    =&0
\end{align*}
\end{allowdisplaybreaks}
by the definition of the braided commutator. For $\omega\in
\underline{\Omega}^1_\mathcal{R}(\mathcal{A})$
\begin{allowdisplaybreaks}
\begin{align*}
    (\mathrm{d}^2\omega)(X,Y,Z)
    =&(\mathcal{R}_1^{-1}\rhd X)
    ((\mathrm{d}(\mathcal{R}_2^{-1}\rhd\omega))(Y,Z))\\
    &-(\mathcal{R}_1^{-1}\rhd Y)
    ((\mathrm{d}(\mathcal{R}_{2(1)}^{-1}\rhd\omega))
    (\mathcal{R}_{2(2)}^{-1}\rhd X,Z))\\
    &+(\mathcal{R}_1^{-1}\rhd Z)
    ((\mathrm{d}(\mathcal{R}_{2(1)}^{-1}\rhd\omega))
    (\mathcal{R}_{2(2)}^{-1}\rhd X,\mathcal{R}_{2(3)}^{-1}\rhd Y))\\
    &-\mathrm{d}\omega([X,Y]_\mathcal{R},Z)
    +\mathrm{d}\omega
    ([X,\mathcal{R}_1^{-1}\rhd Z]_\mathcal{R},\mathcal{R}_2^{-1}\rhd Y)\\
    &-\mathrm{d}\omega
    ([\mathcal{R}_{1(1)}^{-1}\rhd Y,\mathcal{R}_{1(2)}^{-1}\rhd Z]_\mathcal{R},
    \mathcal{R}_2^{-1}\rhd X)\\
    =&(\mathcal{R}_1^{-1}\rhd X)\bigg(
    (\mathcal{R}_1^{'-1}\rhd Y)
    ((\mathcal{R}_2^{'-1}\mathcal{R}_2^{-1})\rhd\omega)(Z)\\
    &-(\mathcal{R}_1^{'-1}\rhd Z)
    ((\mathcal{R}_{2(1)}^{'-1}\mathcal{R}_2^{-1})\rhd\omega)
    (\mathcal{R}_{2(2)}^{'-1}\rhd Y)\\
    &-(\mathcal{R}_2^{-1}\rhd\omega)([Y,Z]_\mathcal{R})
    \bigg)\\
    &-(\mathcal{R}_1^{-1}\rhd Y)\bigg(
    ((\mathcal{R}_1^{'-1}\mathcal{R}_{2(2)}^{-1})\rhd X)
    ((\mathcal{R}_2^{'-1}\mathcal{R}_{2(1)}^{-1})\rhd\omega)(Z)\\
    &-(\mathcal{R}_1^{'-1}\rhd Z)
    ((\mathcal{R}_{2(1)}^{'-1}\mathcal{R}_{2(1)}^{-1})\rhd\omega)
    ((\mathcal{R}_{2(2)}^{'-1}\mathcal{R}_{2(2)}^{-1})\rhd X)\\
    &-(\mathcal{R}_{2(1)}^{-1}\rhd\omega)
    ([\mathcal{R}_{2(2)}^{-1}\rhd X,Z]_\mathcal{R})
    \bigg)\\
    &+(\mathcal{R}_1^{-1}\rhd Z)\bigg(
    ((\mathcal{R}_1^{'-1}\mathcal{R}_{2(2)}^{-1})\rhd X)
    ((\mathcal{R}_2^{'-1}\mathcal{R}_{2(1)}^{-1})\rhd\omega)
    (\mathcal{R}_{2(3)}^{-1}\rhd Y)\\
    &-((\mathcal{R}_1^{'-1}\mathcal{R}_{2(3)}^{-1})\rhd Y)
    ((\mathcal{R}_{2(1)}^{'-1}\mathcal{R}_{2(1)}^{-1})\rhd\omega)
    ((\mathcal{R}_{2(2)}^{'-1}\mathcal{R}_{2(2)}^{-1})\rhd X)\\
    &-(\mathcal{R}_{2(1)}^{-1}\rhd\omega)
    (\mathcal{R}_{2(2)}^{-1}\rhd[X,Y]_\mathcal{R})
    \bigg)\\
    &-(\mathcal{R}_1^{-1}\rhd[X,Y]_\mathcal{R})
    (\mathcal{R}_2^{-1}\rhd\omega)(Z)\\
    &+(\mathcal{R}_1^{-1}\rhd Z)(\mathcal{R}_{2(1)}^{-1}\rhd\omega)
    (\mathcal{R}_{2(2)}^{-1}\rhd[X,Y]_\mathcal{R})\\
    &+\omega([[X,Y]_\mathcal{R},Z]_\mathcal{R})\\
    &+(\mathcal{R}_1^{'-1}\rhd[X,\mathcal{R}_1^{-1}\rhd Z]_\mathcal{R})
    (\mathcal{R}_2^{'-1}\rhd\omega)(\mathcal{R}_2^{-1}\rhd Y)\\
    &-((\mathcal{R}_1^{'-1}\mathcal{R}_2^{-1})\rhd Y)
    (\mathcal{R}_{2(1)}^{'-1}\rhd\omega)
    (\mathcal{R}_{2(2)}^{'-1}\rhd[X,\mathcal{R}_1^{-1}\rhd Z]_\mathcal{R})\\
    &-\omega([[X,\mathcal{R}_1^{-1}\rhd Z]_\mathcal{R},
    \mathcal{R}_2^{-1}\rhd Y]_\mathcal{R})\\
    &-((\mathcal{R}_1^{'-1}\mathcal{R}_1^{-1})\rhd[Y,Z]_\mathcal{R})
    (\mathcal{R}_2^{'-1}\rhd\omega)(\mathcal{R}_2^{-1}\rhd X)\\
    &+((\mathcal{R}_1^{'-1}\mathcal{R}_2^{-1})\rhd X)
    (\mathcal{R}_{2(1)}^{'-1}\rhd\omega)
    ((\mathcal{R}_{2(2)}^{'-1}\mathcal{R}_1^{-1})\rhd[Y,Z]_\mathcal{R})\\
    &+\omega([[\mathcal{R}_{1(1)}^{-1}\rhd Y,
    \mathcal{R}_{1(2)}^{-1}\rhd Z]_\mathcal{R},
    \mathcal{R}_2^{-1}\rhd X]_\mathcal{R})\\
    =&0,
\end{align*}
\end{allowdisplaybreaks}
where those 18 terms cancel in the following way: 12,15,18 because of the
braided Jacobi identity, 1,4,10 and 2,7,13 and 5,8,16 by the definition of
the braided commutator, while 3,17 and 6,14 and 9,11 simply cancel each other.
\end{proof}
We define the \textit{braided differential forms}
$\Omega^\bullet_\mathcal{R}(\mathcal{A})$ on $\mathcal{A}$ to be the smallest
differential graded subalgebra of
$\underline{\Omega}^\bullet_\mathcal{R}(\mathcal{A})$
such that $\mathcal{A}\subseteq\Omega^\bullet_\mathcal{R}(\mathcal{A})$.
Every 
element in $\Omega^k_\mathcal{R}(\mathcal{A})$
can be written as a finite sum of elements of the form 
$a_0\mathrm{d}a_1\wedge_\mathcal{R}\ldots\wedge_\mathcal{R}\mathrm{d}a_k$, for
$a_i\in\mathcal{A}$.

\section{The Braided Cartan Calculus}\label{Sec3.4}

We enter the main section of this chapter. In the following lines we
construct a noncommutative Cartan calculus for any braided commutative
algebra $\mathcal{A}$. Building on Section~\ref{Sec3.2} and
Section~\ref{Sec3.3}, we define the braided Lie derivative as the
graded braided commutator of the braided insertion and the braided
de Rham differential. It is a well-defined morphism in the category of
equivariant braided symmetric bimodules and it completes the data of the
braided Cartan calculus. Using graded braided commutators and the
braided Schouten-Nijenhuis bracket we relate
the braided Lie derivative, the braided insertion and the braided
de Rham differential to each other. The result is a generalization
of the usual relations of the Cartan calculus in the category
of equivariant braided symmetric bimodules. Since there is no choice
involved in our construction, as in differential geometry, it seems
justified to call the resulting data \textit{the} braided Cartan calculus
of $\mathcal{A}$. It is a noncommutative Cartan calculus with the difference
that we are not restricted to incorporate modules of the center of
$\mathcal{A}$ but are free to work with modules over the whole algebra
instead.

The \textit{graded braided commutator}
of two homogeneous maps
$\Phi,\Psi\colon\mathfrak{G}^\bullet\rightarrow\mathfrak{G}^\bullet$
of degree $k$ and $\ell$ between braided Graßmann algebras is defined by
$$
[\Phi,\Psi]_\mathcal{R}
=\Phi\circ\Psi-(-1)^{k\ell}(\mathcal{R}_1^{-1}\rhd\Psi)
\circ(\mathcal{R}_2^{-1}\rhd\Phi).
$$
If $\Phi$ or $\Psi$ is equivariant, the graded braided
commutator coincides with the graded commutator.
If $\Phi,\Psi\colon\mathfrak{X}^\bullet_\mathcal{R}(\mathcal{A})
\otimes\mathfrak{G}^\bullet\rightarrow\mathfrak{G}^\bullet$
are two $H$-equivariant maps such that
$\Phi_X,\Psi_Y\colon\mathfrak{G}^\bullet\rightarrow\mathfrak{G}^\bullet$
are homogeneous of degree $k$ and $\ell$
for any $X\in\mathfrak{X}^k_\mathcal{R}(\mathcal{A})$ and
$Y\in\mathfrak{X}^\ell_\mathcal{R}(\mathcal{A})$, respectively, 
the graded braided commutator of $\Phi_X$ and $\Psi_Y$ reads
$$
[\Phi_X,\Psi_Y]_\mathcal{R}
=\Phi_X\Psi_Y
-(-1)^{k\ell}\Psi_{\mathcal{R}_1^{-1}\rhd Y}\Phi_{\mathcal{R}_2^{-1}\rhd X}.
$$
For any $X\in\mathfrak{X}^k_\mathcal{R}(\mathcal{A})$ we define the
\textit{braided Lie derivative} $\mathscr{L}^\mathcal{R}_X\colon
\Omega^\bullet_\mathcal{R}(\mathcal{A})
\rightarrow\Omega^{\bullet-(k-1)}_\mathcal{R}(\mathcal{A})$ by
$\mathscr{L}^\mathcal{R}_X=[\mathrm{i}^\mathcal{R}_X,\mathrm{d}]_\mathcal{R}$.
It is a homogeneous map of degree $-(k-1)$ and a braided derivation of 
$\Omega^\bullet_\mathcal{R}(\mathcal{A})$ for $k=1$. Moreover, it is easy to check
that
$$
\mathscr{L}^\mathcal{R}\colon\mathfrak{X}^\bullet_\mathcal{R}(\mathcal{A})
\otimes\Omega^\bullet_\mathcal{R}(\mathcal{A})
\rightarrow\Omega^\bullet_\mathcal{R}(\mathcal{A})
$$
is $H$-equivariant. We prove an auxiliary lemma in analogy to
Lemma~\ref{lemma07}.
\begin{lemma}\label{lemma10}
One has
$$
\mathscr{L}^\mathcal{R}_a\omega=-(\mathrm{d}a)\wedge_\mathcal{R}\omega
\text{ and }
\mathscr{L}^\mathcal{R}_{X\wedge_\mathcal{R}Y}
=\mathrm{i}^\mathcal{R}_X\mathscr{L}^\mathcal{R}_Y
+(-1)^\ell\mathscr{L}^\mathcal{R}_X\mathrm{i}^\mathcal{R}_Y
$$
for all $a\in\mathcal{A}$, 
$\omega\in\Omega^\bullet_\mathcal{R}(\mathcal{A})$,
$X\in\mathfrak{X}^\bullet_\mathcal{R}(\mathcal{A})$ and
$Y\in\mathfrak{X}^\ell_\mathcal{R}(\mathcal{A})$. If
$X,Y\in\mathfrak{X}^1_\mathcal{R}(\mathcal{A})$
$$
[\mathscr{L}^\mathcal{R}_X,\mathrm{i}^\mathcal{R}_Y]_\mathcal{R}
=\mathrm{i}^\mathcal{R}_{[X,Y]_\mathcal{R}}
$$
holds.
\end{lemma}
\begin{proof}
By the very definition of the braided Lie derivative
\begin{align*}
    \mathscr{L}^\mathcal{R}_a\omega
    =&\mathrm{i}^\mathcal{R}_a\mathrm{d}\omega
    -(-1)^{0\cdot 1}\mathrm{d}(\mathrm{i}^\mathcal{R}_a\omega)
    =a\wedge_\mathcal{R}\mathrm{d}\omega
    -((\mathrm{d}a)\wedge_\mathcal{R}\omega
    +(-1)^{0}a\wedge_\mathcal{R}\mathrm{d}\omega)\\
    =&-(\mathrm{d}a)\wedge_\mathcal{R}\omega
\end{align*}
follows. From the graded braided Leibniz rule of the graded braided commutator
we obtain
\begin{align*}
    \mathscr{L}^\mathcal{R}_{X\wedge_\mathcal{R}Y}
    =&[\mathrm{i}^\mathcal{R}_{X\wedge_\mathcal{R}Y},\mathrm{d}]_\mathcal{R}
    =[\mathrm{i}^\mathcal{R}_{X}\mathrm{i}^\mathcal{R}_{Y},
    \mathrm{d}]_\mathcal{R}
    =\mathrm{i}^\mathcal{R}_{X}[\mathrm{i}^\mathcal{R}_{Y},
    \mathrm{d}]_\mathcal{R}
    +(-1)^{-1\cdot\ell}[\mathrm{i}^\mathcal{R}_{X},
    \mathrm{d}]_\mathcal{R}\mathrm{i}^\mathcal{R}_{Y}\\
    =&\mathrm{i}^\mathcal{R}_X\mathscr{L}^\mathcal{R}_Y
    +(-1)^\ell\mathscr{L}^\mathcal{R}_X\mathrm{i}^\mathcal{R}_Y.
\end{align*}
The missing formula trivially holds on braided differential forms of degree $0$,
while for $\omega\in\Omega^1_\mathcal{R}(\mathcal{A})$
\begin{align*}
    [\mathscr{L}^\mathcal{R}_X,\mathrm{i}^\mathcal{R}_Y]_\mathcal{R}\omega
    =&\mathscr{L}^\mathcal{R}_X\mathrm{i}^\mathcal{R}_Y\omega
    -(-1)^{0\cdot 1}\mathrm{i}^\mathcal{R}_{\mathcal{R}_1^{-1}\rhd Y}
    \mathscr{L}^\mathcal{R}_{\mathcal{R}_2^{-1}\rhd X}\omega\\
    =&(\mathrm{i}^\mathcal{R}_X\mathrm{d}+\mathrm{d}\mathrm{i}^\mathcal{R}_X)
    \mathrm{i}^\mathcal{R}_Y\omega
    -\mathrm{i}^\mathcal{R}_{\mathcal{R}_1^{-1}\rhd Y}
    (\mathrm{i}^\mathcal{R}_{\mathcal{R}_2^{-1}\rhd X}\mathrm{d}
    +\mathrm{d}\mathrm{i}^\mathcal{R}_{\mathcal{R}_2^{-1}\rhd X})\omega\\
    =&X(\mathrm{i}^\mathcal{R}_Y\omega)+0
    +(\mathrm{d}((\mathcal{R}_1^{''-1}\mathcal{R}_1^{'-1})\rhd\omega))
    ((\mathcal{R}_2^{''-1}\mathcal{R}_1^{-1})\rhd Y,
    (\mathcal{R}_2^{'-1}\mathcal{R}_2^{-1})\rhd X)\\
    &-(\mathcal{R}_1^{-1}\rhd Y)
    (\mathrm{i}^\mathcal{R}_{\mathcal{R}_2^{-1}\rhd X}\omega)\\
    =&\mathrm{i}^\mathcal{R}_{[X,Y]}\omega
\end{align*}
for all $X,Y\in\mathfrak{X}^1_\mathcal{R}(\mathcal{A})$. Since
$[\mathscr{L}^\mathcal{R}_X,\mathrm{i}^\mathcal{R}_Y]_\mathcal{R}$
is a graded braided derivation this is sufficient to conclude the proof
of the lemma.
\end{proof}
Now we are prepared to prove the main theorem of this section.
It assigns to any braided commutative left $H$-module algebra
$\mathcal{A}$ a noncommutative Cartan calculus, which we call
\textit{the braided Cartan calculus} of $\mathcal{A}$ in the
following.
\begin{theorem}[Braided Cartan calculus]\label{ThmBraidedCC}
Let $\mathcal{A}$ be a braided commutative left $H$-module algebra
and consider the braided differential forms 
$(\Omega^\bullet_\mathcal{R}(\mathcal{A}),\wedge_\mathcal{R},\mathrm{d})$ 
and braided multivector fields
$(\mathfrak{X}^\bullet_\mathcal{R}(\mathcal{A}),\wedge_\mathcal{R},
\llbracket\cdot,\cdot\rrbracket_\mathcal{R})$ on $\mathcal{A}$.
The homogeneous maps
$$
\mathscr{L}^\mathcal{R}_X\colon\Omega^\bullet_\mathcal{R}(\mathcal{A})
\rightarrow\Omega^{\bullet-(k-1)}_\mathcal{R}(\mathcal{A})
\text{ and }
\mathrm{i}^\mathcal{R}_X\colon\Omega^\bullet_\mathcal{R}(\mathcal{A})
\rightarrow\Omega^{\bullet-k}_\mathcal{R}(\mathcal{A}),
$$
where $X\in\mathfrak{X}^k_\mathcal{R}(\mathcal{A})$,
and $\mathrm{d}\colon
\Omega^\bullet_\mathcal{R}(\mathcal{A})
\rightarrow\Omega^{\bullet+1}_\mathcal{R}(\mathcal{A})$ satisfy
\begin{equation}\label{BraidedCC}
\begin{split}
        [\mathscr{L}^\mathcal{R}_X,\mathscr{L}^\mathcal{R}_Y]_\mathcal{R}
        =&\mathscr{L}^\mathcal{R}_{\llbracket X,Y\rrbracket_\mathcal{R}},\\
        [\mathscr{L}^\mathcal{R}_X,\mathrm{i}^\mathcal{R}_Y]_\mathcal{R}
        =&\mathrm{i}^\mathcal{R}_{\llbracket X,Y\rrbracket_\mathcal{R}},\\
        [\mathscr{L}^\mathcal{R}_X,\mathrm{d}]_\mathcal{R}=&0,
\end{split}
\hspace{1.5cm}
\begin{split}
        [\mathrm{i}^\mathcal{R}_X,\mathrm{i}^\mathcal{R}_Y]_\mathcal{R}=&0,\\
        [\mathrm{i}^\mathcal{R}_X,\mathrm{d}]_\mathcal{R}
        =&\mathscr{L}^\mathcal{R}_X,\\
        [\mathrm{d},\mathrm{d}]_\mathcal{R}=&0,
\end{split}
\end{equation}
for all $X,Y\in\mathfrak{X}^\bullet_\mathcal{R}(\mathcal{A})$.
\end{theorem}
\begin{proof}
We are going to prove the above formulas in reversed order. Since
$\mathrm{d}$ is a differential it follows that
$[\mathrm{d},\mathrm{d}]_\mathcal{R}=2\mathrm{d}^2=0$. Recall that
there is no braiding appearing here since $\mathrm{d}$ is equivariant.
By the definition of the braided Lie derivative 
$[\mathrm{i}^\mathcal{R}_X,\mathrm{d}]_\mathcal{R}
=\mathscr{L}^\mathcal{R}_X$ holds for all $X\in\mathfrak{X}^\bullet_\mathcal{R}
(\mathcal{A})$. Let $X\in\mathfrak{X}^k_\mathcal{R}(\mathcal{A})$
and $Y\in\mathfrak{X}^\ell_\mathcal{R}(\mathcal{A})$. Then
$$
[\mathrm{i}^\mathcal{R}_X,\mathrm{i}^\mathcal{R}_Y]_\mathcal{R}
=\mathrm{i}^\mathcal{R}_X\mathrm{i}^\mathcal{R}_Y
-(-1)^{k\ell}\mathrm{i}^\mathcal{R}_{\mathcal{R}_1^{-1}\rhd Y}
\mathrm{i}^\mathcal{R}_{\mathcal{R}_2^{-1}\rhd X}
=\mathrm{i}^\mathcal{R}_{X\wedge_\mathcal{R}Y
-(-1)^{k\ell}(\mathcal{R}_1^{-1}\rhd Y)
\wedge_\mathcal{R}(\mathcal{R}_2^{-1}\rhd X)}
=0
$$
follows by the very definition of
$\mathrm{i}^\mathcal{R}_{X\wedge_\mathcal{R}Y}
=\mathrm{i}^\mathcal{R}_X\mathrm{i}^\mathcal{R}_Y$.
Using the graded braided Jacobi identity of the graded braided commutator we obtain
\begin{align*}
    [[\mathrm{i}^\mathcal{R}_X,\mathrm{d}]_\mathcal{R},\mathrm{d}]_\mathcal{R}
    =[\mathrm{i}^\mathcal{R}_X,[\mathrm{d},\mathrm{d}]_\mathcal{R}]_\mathcal{R}
    +(-1)^{1\cdot 1}
    [[\mathrm{i}^\mathcal{R}_X,\mathrm{d}]_\mathcal{R},\mathrm{d}]_\mathcal{R}
    =-[[\mathrm{i}^\mathcal{R}_X,
    \mathrm{d}]_\mathcal{R},\mathrm{d}]_\mathcal{R}
\end{align*}
for all $X\in\mathfrak{X}^\bullet_\mathcal{R}(\mathcal{A})$, which implies
$[\mathscr{L}^\mathcal{R}_X,\mathrm{d}]_\mathcal{R}=0$.
Again, there is no braiding appearing since $\mathrm{d}$ is equivariant.
Recall that the braided Schouten-Nijenhuis bracket of a homogeneous element
$Y=Y_1\wedge_\mathcal{R}\cdots\wedge_\mathcal{R}Y_\ell\in
\mathfrak{X}^\ell_\mathcal{R}(\mathcal{A})$ with $a\in\mathcal{A}$ and
$X\in\mathfrak{X}^1_\mathcal{R}(\mathcal{A})$ read
\begin{align*}
    \llbracket a,Y\rrbracket_\mathcal{R}
    =&\sum_{j=1}^\ell(-1)^{j+1}(\mathcal{R}_{1(1)}^{-1}\rhd Y_1)\wedge_\mathcal{R}
    \cdots\wedge_\mathcal{R}(\mathcal{R}_{1(j-1)}^{-1}\rhd Y_{j-1})\\
    &\wedge_\mathcal{R}\llbracket\mathcal{R}_2^{-1}\rhd a,Y_j\rrbracket_\mathcal{R}
    \wedge_\mathcal{R}Y_{j+1}\wedge_\mathcal{R}\cdots\wedge_\mathcal{R}Y_\ell
\end{align*}
and
\begin{align*}
    \llbracket X,Y\rrbracket_\mathcal{R}
    =&\sum_{j=1}^\ell(\mathcal{R}_{1(1)}^{-1}\rhd Y_1)\wedge_\mathcal{R}
    \cdots\wedge_\mathcal{R}(\mathcal{R}_{1(j-1)}^{-1}\rhd Y_{j-1})\\
    &\wedge_\mathcal{R}[\mathcal{R}_2^{-1}\rhd X,Y_j]_\mathcal{R}
    \wedge_\mathcal{R}Y_{j+1}\wedge_\mathcal{R}\cdots\wedge_\mathcal{R}Y_\ell,
\end{align*}
respectively. If $\ell=1$ we obtain
\begin{align*}
    [\mathscr{L}^\mathcal{R}_a,\mathrm{i}^\mathcal{R}_Y]_\mathcal{R}\omega
    =&(\mathscr{L}^\mathcal{R}_a\mathrm{i}^\mathcal{R}_Y
    -(-1)^{(-1)\cdot 1}\mathrm{i}^\mathcal{R}_{\mathcal{R}_1^{-1}\rhd Y}
    \mathscr{L}^\mathcal{R}_{\mathcal{R}_2^{-1}\rhd a})\omega\\
    =&-\mathrm{d}a\wedge_\mathcal{R}\mathrm{i}^\mathcal{R}_Y\omega
    -\mathrm{i}^\mathcal{R}_{\mathcal{R}_1^{-1}\rhd Y}(
    \mathrm{d}(\mathcal{R}_2^{-1}\rhd a)\wedge_\mathcal{R}\omega)\\
    =&-\mathrm{d}a\wedge_\mathcal{R}\mathrm{i}^\mathcal{R}_Y\omega
    -(\mathcal{R}_1^{-1}\rhd Y)(\mathcal{R}_2^{-1}\rhd a)\wedge_\mathcal{R}\omega\\
    &+\mathrm{d}((\mathcal{R}_1^{'-1}\mathcal{R}_2^{-1})\rhd a)\wedge_\mathcal{R}
    \mathrm{i}^\mathcal{R}_{(\mathcal{R}_2^{'-1}\mathcal{R}_1^{-1})\rhd Y}\omega\\
    =&\mathrm{i}^\mathcal{R}_{\llbracket a,Y\rrbracket_\mathcal{R}}\omega
\end{align*}
for all $\omega\in\Omega^\bullet_\mathcal{R}(\mathcal{A})$
by Lemma~\ref{lemma10}. Using the graded braided Leibniz rule
this extends to any $\ell>1$, namely
\begin{align*}
    [\mathscr{L}^\mathcal{R}_a,
    \mathrm{i}^\mathcal{R}_{Y_1\wedge_\mathcal{R}\cdots
    \wedge_\mathcal{R}Y_\ell}]_\mathcal{R}
    =&[\mathscr{L}^\mathcal{R}_a,\mathrm{i}^\mathcal{R}_{Y_1}]_\mathcal{R}
    \mathrm{i}^\mathcal{R}_{Y_2\wedge_\mathcal{R}\cdots\wedge_\mathcal{R}Y_\ell}
    +(-1)^{(-1)\cdot 1}\mathrm{i}^\mathcal{R}_{\mathcal{R}_1^{-1}\rhd Y_1}
    [\mathscr{L}^\mathcal{R}_{\mathcal{R}_2^{-1}\rhd a},
    \mathrm{i}^\mathcal{R}_{Y_2\wedge_\mathcal{R}\cdots\wedge_\mathcal{R}Y_\ell}]\\
    =&\mathrm{i}^\mathcal{R}_{\llbracket a,Y_1\rrbracket_\mathcal{R}
    \wedge_\mathcal{R}Y_2\wedge_\mathcal{R}\cdots\wedge_\mathcal{R}Y_\ell}
    -\mathrm{i}^\mathcal{R}_{\mathcal{R}_1^{-1}\rhd Y_1}
    [\mathscr{L}^\mathcal{R}_{\mathcal{R}_2^{-1}\rhd a},
    \mathrm{i}^\mathcal{R}_{Y_2\wedge_\mathcal{R}\cdots\wedge_\mathcal{R}Y_\ell}]\\
    =&\cdots
    =\mathrm{i}^\mathcal{R}_{\llbracket a,Y\rrbracket_\mathcal{R}}.
\end{align*}
Again by Lemma~\ref{lemma10} we know that $[\mathscr{L}^\mathcal{R}_X,
\mathrm{i}^\mathcal{R}_Y]_\mathcal{R}=\mathrm{i}^\mathcal{R}_{[X,Y]_\mathcal{R}}$
holds for $\ell=1$ and $X\in\mathfrak{X}^1_\mathcal{R}(\mathcal{A})$.
Using the graded braided Leibniz rule this extends to all 
$Y\in\mathfrak{X}^\bullet_\mathcal{R}(\mathcal{A})$.
Assume now that $[\mathscr{L}^\mathcal{R}_X,\mathrm{i}^\mathcal{R}_Z]_\mathcal{R}
=\mathrm{i}^\mathcal{R}_{\llbracket X,Z\rrbracket_\mathcal{R}}$
holds for all $X\in\mathfrak{X}^k_\mathcal{R}(\mathcal{A})$ and
$Z\in\mathfrak{X}^\bullet_\mathcal{R}(\mathcal{A})$ for a fixed
$k>0$. Then, for all $X\in\mathfrak{X}^k_\mathcal{R}(\mathcal{A})$,
$Y\in\mathfrak{X}^1_\mathcal{R}(\mathcal{A})$ and $Z\in
\mathfrak{X}^m_\mathcal{R}(\mathcal{A})$ it follows that
\begin{align*}
    [\mathscr{L}^\mathcal{R}_{X\wedge_\mathcal{R}Y},\mathrm{i}^\mathcal{R}_Z
    ]_\mathcal{R}
    =&[\mathrm{i}^\mathcal{R}_X\mathscr{L}^\mathcal{R}_Y
    -\mathscr{L}^\mathcal{R}_X\mathrm{i}^\mathcal{R}_Y,
    \mathrm{i}^\mathcal{R}_Z]_\mathcal{R}\\
    =&\mathrm{i}^\mathcal{R}_X[\mathscr{L}^\mathcal{R}_Y,
    \mathrm{i}^\mathcal{R}_Z]_\mathcal{R}
    +[\mathrm{i}^\mathcal{R}_X,
    \mathrm{i}^\mathcal{R}_{\mathcal{R}_1^{-1}\rhd Z}]_\mathcal{R}
    \mathscr{L}^\mathcal{R}_{\mathcal{R}_2^{-1}\rhd Y}\\
    &-\mathscr{L}^\mathcal{R}_X[\mathrm{i}^\mathcal{R}_Y,
    \mathrm{i}^\mathcal{R}_Z]_\mathcal{R}
    -(-1)^m[\mathscr{L}^\mathcal{R}_X,
    \mathrm{i}^\mathcal{R}_{\mathcal{R}_1^{-1}\rhd Z}]_\mathcal{R}
    \mathrm{i}^\mathcal{R}_{\mathcal{R}_2^{-1}\rhd Y}\\
    =&\mathrm{i}^\mathcal{R}_X[\mathscr{L}^\mathcal{R}_Y,
    \mathrm{i}^\mathcal{R}_Z]_\mathcal{R}
    -(-1)^m[\mathscr{L}^\mathcal{R}_X,
    \mathrm{i}^\mathcal{R}_{\mathcal{R}_1^{-1}\rhd Z}]_\mathcal{R}
    \mathrm{i}^\mathcal{R}_{\mathcal{R}_2^{-1}\rhd Y}\\
    =&\mathrm{i}^\mathcal{R}_{X}\mathrm{i}^\mathcal{R}_{
    \llbracket Y,Z\rrbracket_\mathcal{R}}
    -(-1)^{m}\mathrm{i}^\mathcal{R}_{\llbracket X,
    \mathcal{R}_1^{-1}\rhd Z\rrbracket_\mathcal{R}}
    \mathrm{i}^\mathcal{R}_{(\mathcal{R}_2^{-1}\rhd Y)}\\
    =&\mathrm{i}^\mathcal{R}_{X\wedge_\mathcal{R}
    \llbracket Y,Z\rrbracket_\mathcal{R}}
    +(-1)^{m-1}\mathrm{i}^\mathcal{R}_{\llbracket X,
    \mathcal{R}_1^{-1}\rhd Z\rrbracket_\mathcal{R}\wedge_\mathcal{R}
    (\mathcal{R}_2^{-1}\rhd Y)}\\
    =&\mathrm{i}^\mathcal{R}_{\llbracket X\wedge_\mathcal{R}Y,
    Z\rrbracket_\mathcal{R}}
\end{align*}
for all $X\in\mathfrak{X}^k_\mathcal{R}(\mathcal{A})$,
$Y\in\mathfrak{X}^1_\mathcal{R}(\mathcal{A})$ and
$Z\in\mathfrak{X}^m_\mathcal{R}(\mathcal{A})$
using Lemma~\ref{lemma10}. By induction
$[\mathscr{L}^\mathcal{R}_X,\mathrm{i}^\mathcal{R}_Y]_\mathcal{R}
=\mathrm{i}^\mathcal{R}_{\llbracket X,Y\rrbracket_\mathcal{R}}$ holds
for all $X,Y\in\mathfrak{X}^\bullet_\mathcal{R}(\mathcal{A})$.
The remaining formula is verified via
\begin{align*}
    [\mathscr{L}^\mathcal{R}_X,\mathscr{L}^\mathcal{R}_Y]_\mathcal{R}
    =&[\mathscr{L}^\mathcal{R}_X,[\mathrm{i}^\mathcal{R}_Y,
    \mathrm{d}]_\mathcal{R}]_\mathcal{R}\\
    =&[[\mathscr{L}^\mathcal{R}_X,\mathrm{i}^\mathcal{R}_Y]_\mathcal{R},
    \mathrm{d}]_\mathcal{R}
    +(-1)^{(k-1)\ell}[\mathrm{i}^\mathcal{R}_{\mathcal{R}_1^{-1}\rhd Y},
    [\mathscr{L}^\mathcal{R}_{\mathcal{R}_2^{-1}\rhd X},
    \mathrm{d}]_\mathcal{R}]_\mathcal{R}\\
    =&[\mathrm{i}^\mathcal{R}_{\llbracket X,Y\rrbracket_\mathcal{R}},
    \mathrm{d}]_\mathcal{R}+0\\
    =&\mathscr{L}^\mathcal{R}_{\llbracket X,Y\rrbracket_\mathcal{R}}
\end{align*}
for all $X\in\mathfrak{X}^k_\mathcal{R}(\mathcal{A})$ and
$Y\in\mathfrak{X}^\ell_\mathcal{R}(\mathcal{A})$.
This concludes the proof of the theorem.
\end{proof}
In particular, the Cartan calculus on a smooth manifold $M$ is a braided
Cartan calculus with respect to the trivial action of any cocommutative Hopf
algebra and $\mathcal{R}=1\otimes 1$. The equations (\ref{BraidedCC}) reduce
to the usual formulas of the classical Cartan calculus in this case.
In Section~\ref{Sec3.6} we study another class of examples of braided Cartan
calculi, namely twisted Cartan calculi.

\section{Equivariant Covariant Derivatives}\label{Sec3.5}

Having the braided Cartan calculus for any braided commutative algebra
$\mathcal{A}$ at hand it is nearby to ask if other fundamental
constructions of differential geometry can be generalized to this
setting. Focusing on the algebraic properties of covariant derivatives,
namely their linearity in the first argument and that they satisfy
a Leibniz rule with respect to the second argument, we define
equivariant covariant derivatives on equivariant braided symmetric bimodules.
Note that there are
several notions of covariant derivatives in the context of noncommutative
algebras in the literature, as already mentioned in the introduction 
to this chapter. While in the general noncommutative
case one has to distinguish between left and right covariant derivatives
these two notions coincide for braided commutative algebras if one also
requires equivariance of the covariant derivative. We
further introduce curvature and torsion of an equivariant covariant derivative
and prove that they have the expected symmetry properties. Given an equivariant
covariant derivative on $\mathcal{A}$ we construct an equivariant covariant
derivative on braided differential $1$-forms by employing the
dual pairing and in a next step we extend both to braided multivector
fields and differential forms. We introduce equivariant metrics
on $\mathcal{A}$ and prove the existence and uniqueness of an equivariant
Levi-Civita covariant derivative with respect to a fixed non-degenerate
equivariant metric.
Let $(H,\mathcal{R})$ be a triangular Hopf algebra and $\mathcal{A}$ a
braided commutative left $H$-module algebra in the following.
\begin{definition}[Equivariant covariant derivative]
Consider an $H$-equivariant braided symmetric
$\mathcal{A}$-bimodule $\mathcal{M}$. An $H$-equivariant map
$
\nabla^\mathcal{R}\colon\mathfrak{X}^1_\mathcal{R}(\mathcal{A})
\otimes\mathcal{M}\rightarrow\mathcal{M}
$
is said to be an equivariant covariant derivative on $\mathcal{M}$
(with respect to $\mathcal{R}$), if for all
$a\in\mathcal{A}$, $X\in\mathfrak{X}^1_\mathcal{R}(\mathcal{A})$ and
$s\in\mathcal{M}$ one has $\nabla^\mathcal{R}_{a\cdot X}s
=a\cdot(\nabla^\mathcal{R}_Xs)$
and
\begin{equation}
    \nabla^\mathcal{R}_X(a\cdot s)
    =(\mathscr{L}^\mathcal{R}_Xa)\cdot s
    +(\mathcal{R}_1^{-1}\rhd a)\cdot
    (\nabla^\mathcal{R}_{\mathcal{R}_2^{-1}\rhd X}s).
\end{equation}
The \textit{curvature} of an equivariant covariant derivative $\nabla^\mathcal{R}$
on $\mathcal{M}$ is defined by
\begin{equation}
    R^{\nabla^\mathcal{R}}(X,Y)
    =\nabla^\mathcal{R}_X\nabla^\mathcal{R}_Y
    -\nabla^\mathcal{R}_{\mathcal{R}_1^{-1}\rhd Y}
    \nabla^\mathcal{R}_{\mathcal{R}_2^{-1}\rhd X}
    -\nabla^\mathcal{R}_{[X,Y]_\mathcal{R}}
\end{equation}
for $X,Y\in\mathfrak{X}^1_\mathcal{R}(\mathcal{A})$.
If $\mathcal{M}=\mathfrak{X}^1_\mathcal{R}(\mathcal{A})$
we can further define the \textit{torsion} of $\nabla^\mathcal{R}$ by
\begin{equation}
    \mathrm{Tor}^{\nabla^\mathcal{R}}(X,Y)
    =\nabla^\mathcal{R}_XY
    -\nabla^\mathcal{R}_{\mathcal{R}_1^{-1}\rhd Y}(\mathcal{R}_2^{-1}\rhd X)
    -[X,Y]_\mathcal{R},
\end{equation}
for all $X,Y\in\mathfrak{X}^1_\mathcal{R}(\mathcal{A})$. An equivariant
covariant derivative $\nabla^\mathcal{R}$ is \textit{flat} if
$R^{\nabla^\mathcal{R}}=0$ and \textit{torsion-free} if
$\mathrm{Tor}^{\nabla^\mathcal{R}}=0$.
\end{definition}
Recall that the $H$-equivariance of an equivariant covariant derivative 
$\nabla^\mathcal{R}$ on $\mathcal{M}$ reads
$\xi\rhd(\nabla^\mathcal{R}_Xs)
=\nabla^\mathcal{R}_{\xi_{(1)}\rhd X}(\xi_{(2)}\rhd s)$
on elements $\xi\in H$, $X\in\mathfrak{X}^1_\mathcal{R}(\mathcal{A})$
and $s\in\mathcal{M}$.
In the next lemma we are going to investigate the linearity properties and
symmetries of the curvature and torsion. In short, we prove that
$R^{\nabla^\mathcal{R}}$ descends to a map
$\mathfrak{X}^2_\mathcal{R}(\mathcal{R})\otimes\mathcal{M}\rightarrow
\mathcal{M}$ and, in the case it is defined,
$\mathrm{Tor}^{\nabla^\mathcal{R}}$ to a map
$\mathfrak{X}^2_\mathcal{R}(\mathcal{A})
\rightarrow\mathfrak{X}^1_\mathcal{R}(\mathcal{A})$.
\begin{lemma}
Let $\nabla^\mathcal{R}$ be an equivariant covariant derivative on an
$H$-equivariant braided symmetric $\mathcal{A}$-bimodule $\mathcal{M}$.
Then
\begin{align*}
    \nabla^\mathcal{R}_{X\cdot a}s
    =&\nabla^\mathcal{R}_X(\mathcal{R}_1^{-1}\rhd s)
    \cdot(\mathcal{R}_2^{-1}\rhd a),\\
    \nabla^\mathcal{R}_X(s\cdot a)
    =&(\nabla^\mathcal{R}_Xs)\cdot a
    +(\mathcal{R}_1^{-1}\rhd s)
    \cdot(\mathscr{L}^\mathcal{R}_{\mathcal{R}_2^{-1}\rhd X}a)
\end{align*}
hold for all $a\in\mathcal{A}$, $X\in\mathfrak{X}^1_\mathcal{R}(\mathcal{A})$
and $s\in\mathcal{M}$. Furthermore,
\begin{align*}
    R^{\nabla^\mathcal{R}}(Y,X)
    =&-R^{\nabla^\mathcal{R}}(\mathcal{R}_1^{-1}\rhd X,
    \mathcal{R}_1^{-1}\rhd Y),\\
    R^{\nabla^\mathcal{R}}(a\cdot X,Y\cdot b)s
    =&a\cdot R^{\nabla^\mathcal{R}}(X,Y)(\mathcal{R}_1^{-1}\rhd s)
    \cdot(\mathcal{R}_2^{-1}\rhd b),\\
    R^{\nabla^\mathcal{R}}(X\cdot a,Y)s
    =&R^{\nabla^\mathcal{R}}(X,\mathcal{R}_{1(1)}^{-1}\rhd Y)
    (\mathcal{R}_{1(2)}^{-1}\rhd s)\cdot(\mathcal{R}_2^{-1}\rhd a)\\
    =&R^{\nabla^\mathcal{R}}(X,a\cdot Y)s
\end{align*}
for all $a,b\in\mathcal{A}$, $X,Y\in\mathfrak{X}^1_\mathcal{R}(\mathcal{A})$
and $s\in\mathcal{M}$.
In the case $\mathcal{M}=\mathfrak{X}^1_\mathcal{R}(\mathcal{A})$ we obtain
\begin{align*}
    \mathrm{Tor}^{\nabla^\mathcal{R}}(Y,X)
    =&-\mathrm{Tor}^{\nabla^\mathcal{R}}(\mathcal{R}_1^{-1}\rhd X,
    \mathcal{R}_1^{-1}\rhd Y),\\
    \mathrm{Tor}^{\nabla^\mathcal{R}}(a\cdot X,Y\cdot b)
    =&a\cdot\mathrm{Tor}^{\nabla^\mathcal{R}}(X,Y)\cdot b,\\
    \mathrm{Tor}^{\nabla^\mathcal{R}}(X\cdot a,Y)
    =&\mathrm{Tor}^{\nabla^\mathcal{R}}(X,\mathcal{R}_1^{-1}\rhd Y)
    \cdot(\mathcal{R}_2^{-1}\rhd a)\\
    =&\mathrm{Tor}^{\nabla^\mathcal{R}}(X,a\cdot Y)
\end{align*}
for all $a,b\in\mathcal{A}$ and $X,Y\in\mathfrak{X}^1_\mathcal{R}(\mathcal{A})$.
\end{lemma}
\begin{proof}
Fix elements $a\in\mathcal{A}$, $X,Y\in\mathfrak{X}^1_\mathcal{R}(\mathcal{A})$
and $s\in\mathcal{M}$. Then
\begin{allowdisplaybreaks}
\begin{align*}
    \nabla^\mathcal{R}_{X\cdot a}s
    =&\nabla^\mathcal{R}_{(\mathcal{R}_1^{-1}\rhd a)
    \cdot(\mathcal{R}_2^{-1}\rhd X)}s\\
    =&(\mathcal{R}_1^{-1}\rhd a)
    \cdot(\nabla^\mathcal{R}_{\mathcal{R}_2^{-1}\rhd X}s)\\
    =&(\mathcal{R}_1^{'-1}\rhd(\nabla^\mathcal{R}_{\mathcal{R}_2^{-1}\rhd X}s))
    \cdot((\mathcal{R}_2^{'-1}\mathcal{R}_1^{-1})\rhd a)\\
    =&\nabla^\mathcal{R}_{(\mathcal{R}_{1(1)}^{'-1}\mathcal{R}_2^{-1})
    \rhd X}(\mathcal{R}_{1(2)}^{'-1}\rhd s)
    \cdot((\mathcal{R}_2^{'-1}\mathcal{R}_1^{-1})\rhd a)\\
    =&\nabla^\mathcal{R}_{(\mathcal{R}_{1}^{'-1}\mathcal{R}_2^{-1})
    \rhd X}(\mathcal{R}_{1}^{''-1}\rhd s)
    \cdot((\mathcal{R}_2^{''-1}\mathcal{R}_2^{'-1}\mathcal{R}_1^{-1})\rhd a)\\
    =&\nabla^\mathcal{R}_{X}(\mathcal{R}_{1}^{-1}\rhd s)
    \cdot(\mathcal{R}_2^{-1}\rhd a)
\end{align*}
\end{allowdisplaybreaks}
and
\begin{allowdisplaybreaks}
\begin{align*}
    \nabla^\mathcal{R}_X(s\cdot a)
    =&\nabla^\mathcal{R}_X((\mathcal{R}_1^{-1}\rhd a)
    \cdot(\mathcal{R}_2^{-1}\rhd s))\\
    =&\mathscr{L}^\mathcal{R}_X(\mathcal{R}_1^{-1}\rhd a)
    \cdot(\mathcal{R}_2^{-1}\rhd s)
    +((\mathcal{R}_1^{'-1}\mathcal{R}_1^{-1})\rhd a)
    \cdot\nabla^\mathcal{R}_{\mathcal{R}_2^{'-1}\rhd X}
    (\mathcal{R}_2^{-1}\rhd s)\\
    =&((\mathcal{R}_1^{'-1}\mathcal{R}_2^{-1})\rhd s)\cdot
    \mathscr{L}^\mathcal{R}_{\mathcal{R}_{2(1)}^{'-1}\rhd X}
    ((\mathcal{R}_{2(2)}^{'-1}\mathcal{R}_1^{-1})\rhd a)\\
    &+\nabla^\mathcal{R}_{(\mathcal{R}_{1(1)}^{''-1}\mathcal{R}_2^{'-1})
    \rhd X}
    ((\mathcal{R}_{1(1)}^{''-1}\mathcal{R}_2^{-1})\rhd s)
    \cdot((\mathcal{R}_2^{''-1}\mathcal{R}_1^{'-1}\mathcal{R}_1^{-1})\rhd a)\\
    =&((\mathcal{R}_1^{''-1}\mathcal{R}_1^{'-1}\mathcal{R}_2^{-1})\rhd s)\cdot
    \mathscr{L}^\mathcal{R}_{\mathcal{R}_{2}^{''-1}\rhd X}
    ((\mathcal{R}_{2}^{'-1}\mathcal{R}_1^{-1})\rhd a)\\
    &+\nabla^\mathcal{R}_{(\mathcal{R}_{1}^{''-1}\mathcal{R}_2^{'-1})
    \rhd X}
    ((\mathcal{R}_{1}^{'''-1}\mathcal{R}_2^{-1})\rhd s)
    \cdot((\mathcal{R}_2^{'''-1}\mathcal{R}_2^{''-1}
    \mathcal{R}_1^{'-1}\mathcal{R}_1^{-1})\rhd a)\\
    =&\nabla^\mathcal{R}_Xs\cdot a
    +(\mathcal{R}_1^{-1}\rhd s)
    \cdot\mathscr{L}^\mathcal{R}_{\mathcal{R}_2^{-1}\rhd X}a,
\end{align*}
\end{allowdisplaybreaks}
proving the first two equations. By the braided skew-symmetry of
$[\cdot,\cdot]_\mathcal{R}$ we observe that
\begin{align*}
    R^{\nabla^\mathcal{R}}(\mathcal{R}_1^{-1}\rhd X,\mathcal{R}_2^{-1}\rhd Y)
    =&\nabla^\mathcal{R}_{\mathcal{R}_1^{-1}\rhd X}
    \nabla^\mathcal{R}_{\mathcal{R}_2^{-1}\rhd Y}
    -\nabla^\mathcal{R}_{(\mathcal{R}_1^{'-1}\mathcal{R}_2^{-1})\rhd Y}
    \nabla^\mathcal{R}_{(\mathcal{R}_2^{'-1}\mathcal{R}_1^{-1})\rhd X}\\
    &-\nabla^\mathcal{R}_{[\mathcal{R}_1^{-1}\rhd X,
    \mathcal{R}_2^{-1}\rhd Y]_\mathcal{R}}\\
    =&-(\nabla^\mathcal{R}_{Y}\nabla^\mathcal{R}_{X}
    -\nabla^\mathcal{R}_{\mathcal{R}_1^{-1}\rhd X}
    \nabla^\mathcal{R}_{\mathcal{R}_2^{-1}\rhd Y}
    -\nabla^\mathcal{R}_{[Y,X]_\mathcal{R}})\\
    =&-R^{\nabla^\mathcal{R}}(Y,X)
\end{align*}
and using the equivariancy of $\nabla^\mathcal{R}$ it follows that
\begin{allowdisplaybreaks}
\begin{align*}
    R^{\nabla^\mathcal{R}}(a\cdot X,Y)
    =&\nabla^\mathcal{R}_{a\cdot X}\nabla^\mathcal{R}_Y
    -\nabla^\mathcal{R}_{\mathcal{R}_1^{-1}\rhd Y}
    \nabla^\mathcal{R}_{\mathcal{R}_2^{-1}\rhd(a\cdot X)}
    -\nabla^\mathcal{R}_{[a\cdot X,Y]_\mathcal{R}}\\
    =&a\cdot\nabla^\mathcal{R}_{X}\nabla^\mathcal{R}_Y
    -\nabla^\mathcal{R}_{\mathcal{R}_1^{-1}\rhd Y}(
    (\mathcal{R}_{2(1)}^{-1}\rhd a)\cdot
    \nabla^\mathcal{R}_{\mathcal{R}_{2(2)}^{-1}\rhd X})\\
    &-\nabla^\mathcal{R}_{
    a\cdot[X,Y]_\mathcal{R}+(((\mathcal{R}_1^{'-1}\mathcal{R}_1^{-1})
    \rhd Y)(\mathcal{R}_2^{'-1}\rhd a))\cdot(\mathcal{R}_2^{-1} X)}\\
    =&a\cdot\nabla^\mathcal{R}_{X}\nabla^\mathcal{R}_Y
    -\mathscr{L}^\mathcal{R}_{\mathcal{R}_1^{-1}\rhd Y}
    (\mathcal{R}_{2(1)}^{-1}\rhd a)\cdot
    \nabla^\mathcal{R}_{\mathcal{R}_{2(2)}^{-1}\rhd X}\\
    &-((\mathcal{R}_1^{'-1}\mathcal{R}_{2(1)}^{-1})\rhd a)\cdot
    \nabla^\mathcal{R}_{(\mathcal{R}_2^{'-1}\mathcal{R}_1^{-1})\rhd Y}
    \nabla^\mathcal{R}_{\mathcal{R}_{2(2)}^{-1}\rhd X}\\
    &-\nabla^\mathcal{R}_{
    a\cdot[X,Y]_\mathcal{R}}
    +(((\mathcal{R}_1^{'-1}\mathcal{R}_1^{-1})
    \rhd Y)(\mathcal{R}_2^{'-1}\rhd a))
    \cdot\nabla^\mathcal{R}_{\mathcal{R}_2^{-1} X}\\
    =&a\cdot\nabla^\mathcal{R}_{X}\nabla^\mathcal{R}_Y
    -((\mathcal{R}_1^{'-1}\mathcal{R}_{2}^{''-1})\rhd a)\cdot
    \nabla^\mathcal{R}_{(\mathcal{R}_2^{'-1}\mathcal{R}_1^{''-1}
    \mathcal{R}_1^{-1})\rhd Y}
    \nabla^\mathcal{R}_{\mathcal{R}_{2}^{-1}\rhd X}\\
    &-a\cdot\nabla^\mathcal{R}_{[X,Y]_\mathcal{R}}\\
    =&a\cdot R^{\nabla^\mathcal{R}}(X,Y)
\end{align*}
\end{allowdisplaybreaks}
and
\begin{allowdisplaybreaks}
\begin{align*}
    R^{\nabla^\mathcal{R}}(X,Y\cdot a)s
    =&\nabla^\mathcal{R}_{X}\nabla^\mathcal{R}_{Y\cdot a}s
    -\nabla^\mathcal{R}_{\mathcal{R}_1^{-1}\rhd(Y\cdot a)}
    \nabla^\mathcal{R}_{\mathcal{R}_2^{-1}\rhd X}s
    -\nabla^\mathcal{R}_{[X,Y\cdot a]_\mathcal{R}}s\\
    =&\nabla^\mathcal{R}_{X}(\nabla^\mathcal{R}_{Y}(\mathcal{R}_1^{-1}\rhd s)
    \cdot(\mathcal{R}_2^{-1}\rhd a))\\
    &-\nabla^\mathcal{R}_{\mathcal{R}_{1(1)}^{-1}\rhd Y}
    \nabla^\mathcal{R}_{(\mathcal{R}_{1(1)}^{'-1}\mathcal{R}_2^{-1})\rhd X}
    (\mathcal{R}_{1(2)}^{'-1}\rhd s)
    \cdot((\mathcal{R}_2^{'-1}\mathcal{R}_{1(2)}^{-1})\rhd a)\\
    &-\nabla^\mathcal{R}_{[X,Y]_\mathcal{R}\cdot a
    +(\mathcal{R}_1^{-1}\rhd Y)\cdot((\mathcal{R}_2^{-1}\rhd X)(a))}s\\
    =&\nabla^\mathcal{R}_{X}\nabla^\mathcal{R}_{Y}(\mathcal{R}_1^{-1}\rhd s)
    \cdot(\mathcal{R}_2^{-1}\rhd a)\\
    &+\nabla^\mathcal{R}_{\mathcal{R}_{1(1)}^{'-1}\rhd Y}
    ((\mathcal{R}_{1(2)}^{'-1}\mathcal{R}_1^{-1})\rhd s)
    \cdot\mathscr{L}^\mathcal{R}_{\mathcal{R}_2^{'-1}\rhd X}
    (\mathcal{R}_2^{-1}\rhd a)\\
    &-\nabla^\mathcal{R}_{\mathcal{R}_{1}^{-1}\rhd Y}
    \nabla^\mathcal{R}_{\mathcal{R}_2^{-1}\rhd X}
    (\mathcal{R}_{1}^{'-1}\rhd s)
    \cdot(\mathcal{R}_2^{'-1}\rhd a)\\
    &-\nabla^\mathcal{R}_{[X,Y]_\mathcal{R}\cdot a}s
    -\nabla^\mathcal{R}_{(\mathcal{R}_1^{-1}\rhd Y)
    \cdot((\mathcal{R}_2^{-1}\rhd X)(a))}s\\
    =&\nabla^\mathcal{R}_{X}\nabla^\mathcal{R}_{Y}(\mathcal{R}_1^{-1}\rhd s)
    \cdot(\mathcal{R}_2^{-1}\rhd a)\\
    &+\nabla^\mathcal{R}_{\mathcal{R}_{1}^{'-1}\rhd Y}
    ((\mathcal{R}_{1}^{''-1}\mathcal{R}_1^{-1})\rhd s)
    \cdot\mathscr{L}^\mathcal{R}_{(\mathcal{R}_2^{''-1}\mathcal{R}_2^{'-1})
    \rhd X}(\mathcal{R}_2^{-1}\rhd a)\\
    &-\nabla^\mathcal{R}_{\mathcal{R}_{1}^{-1}\rhd Y}
    \nabla^\mathcal{R}_{\mathcal{R}_2^{-1}\rhd X}
    (\mathcal{R}_{1}^{'-1}\rhd s)
    \cdot(\mathcal{R}_2^{'-1}\rhd a)\\
    &-\nabla^\mathcal{R}_{[X,Y]_\mathcal{R}}(\mathcal{R}_1^{-1}\rhd s)
    \cdot(\mathcal{R}_2^{-1}\rhd a)\\
    &-\nabla^\mathcal{R}_{\mathcal{R}_1^{-1}\rhd Y}(\mathcal{R}_1^{'-1}\rhd s)
    \cdot(((\mathcal{R}_{2(1)}^{'-1}\mathcal{R}_2^{-1})\rhd X)
    (\mathcal{R}_{2(2)}^{'-1}\rhd a))\\
    =&R^{\nabla^\mathcal{R}}(X,Y)(\mathcal{R}_1^{-1}\rhd s)
    \cdot(\mathcal{R}_2^{-1}\rhd a)
\end{align*}
\end{allowdisplaybreaks}
hold. Furthermore
\begin{allowdisplaybreaks}
\begin{align*}
    R^{\nabla^\mathcal{R}}(X\cdot a,Y)s
    =&\nabla^\mathcal{R}_{X\cdot a}\nabla^\mathcal{R}_Ys
    -\nabla^\mathcal{R}_{\mathcal{R}_1^{-1}\rhd Y}
    \nabla^\mathcal{R}_{\mathcal{R}_2^{-1}\rhd(X\cdot a)}s
    -\nabla^\mathcal{R}_{[X\cdot a,Y]_\mathcal{R}}s\\
    =&\nabla^\mathcal{R}_{X}\nabla^\mathcal{R}_{\mathcal{R}_{1(1)}^{-1}\rhd Y}
    (\mathcal{R}_{1(2)}^{-1}\rhd s)\cdot(\mathcal{R}_2^{-1}\rhd a)\\
    &-\nabla^\mathcal{R}_{\mathcal{R}_1^{-1}\rhd Y}(
    \nabla^\mathcal{R}_{\mathcal{R}_{2(1)}^{-1}\rhd X}
    (\mathcal{R}_1^{'-1}\rhd s)
    \cdot((\mathcal{R}_2^{'-1}\mathcal{R}_{2(2)}^{-1})\rhd a))\\
    &-\nabla^\mathcal{R}_{[X,\mathcal{R}_1^{-1}\rhd Y]_\mathcal{R}
    \cdot(\mathcal{R}_2^{-1}\rhd a)
    -X\cdot((\mathcal{R}_1^{-1}\rhd Y)(\mathcal{R}_2^{-1}\rhd a))}s\\
    =&\nabla^\mathcal{R}_{X}\nabla^\mathcal{R}_{\mathcal{R}_{1(1)}^{-1}\rhd Y}
    (\mathcal{R}_{1(2)}^{-1}\rhd s)\cdot(\mathcal{R}_2^{-1}\rhd a)\\
    &-\nabla^\mathcal{R}_{\mathcal{R}_1^{-1}\rhd Y}
    \nabla^\mathcal{R}_{\mathcal{R}_{2(1)}^{-1}\rhd X}
    (\mathcal{R}_1^{'-1}\rhd s)
    \cdot((\mathcal{R}_2^{'-1}\mathcal{R}_{2(2)}^{-1})\rhd a)\\
    &-\nabla^\mathcal{R}_{(\mathcal{R}_{1(1)}^{''-1}\mathcal{R}_{2(1)}^{-1})
    \rhd X}
    ((\mathcal{R}_{1(2)}^{''-1}\mathcal{R}_1^{'-1})\rhd s)\\
    &\cdot\mathscr{L}^\mathcal{R}_{(\mathcal{R}_2^{''-1}\mathcal{R}_1^{-1})
    \rhd Y}((\mathcal{R}_2^{'-1}\mathcal{R}_{2(2)}^{-1})\rhd a)\\
    &-\nabla^\mathcal{R}_{[X,\mathcal{R}_1^{-1}\rhd Y]_\mathcal{R}
    \cdot(\mathcal{R}_2^{-1}\rhd a)}s\\
    &+\nabla^\mathcal{R}_{X}(\mathcal{R}_1^{'-1}\rhd s)
    \cdot(((\mathcal{R}_{2(1)}^{'-1}\mathcal{R}_1^{-1})\rhd Y)
    ((\mathcal{R}_{2(2)}^{'-1}\mathcal{R}_2^{-1})\rhd a))\\
    =&\nabla^\mathcal{R}_{X}\nabla^\mathcal{R}_{\mathcal{R}_{1}^{-1}\rhd Y}
    (\mathcal{R}_{1}^{'-1}\rhd s)
    \cdot((\mathcal{R}_2^{'-1}\mathcal{R}_2^{-1})\rhd a)\\
    &-\nabla^\mathcal{R}_{(\mathcal{R}_1^{''-1}\mathcal{R}_1^{-1})\rhd Y}
    \nabla^\mathcal{R}_{\mathcal{R}_{2}^{''-1}\rhd X}
    (\mathcal{R}_1^{'-1}\rhd s)
    \cdot((\mathcal{R}_2^{'-1}\mathcal{R}_{2}^{-1})\rhd a)\\
    &-\nabla^\mathcal{R}_{[X,\mathcal{R}_1^{-1}\rhd Y]_\mathcal{R}}
    (\mathcal{R}_1^{'-1}\rhd s)
    \cdot((\mathcal{R}_2^{'-1}\mathcal{R}_2^{-1})\rhd a)\\
    =&R^{\nabla^\mathcal{R}}(X,\mathcal{R}_{1(1)}^{-1}\rhd Y)
    (\mathcal{R}_{1(2)}^{-1}\rhd s)\cdot(\mathcal{R}_2^{-1}\rhd a)\\
    =&R^{\nabla^\mathcal{R}}(X,(\mathcal{R}_{1(1)}^{-1}\rhd Y)
    \cdot((\mathcal{R}_1^{'-1}\mathcal{R}_2^{-1})\rhd a))
    ((\mathcal{R}_2^{'-1}\mathcal{R}_{1(2)}^{-1})\rhd s)\\
    =&R^{\nabla^\mathcal{R}}(X,a\cdot Y)s.
\end{align*}
\end{allowdisplaybreaks}
It remains to discuss the properties of torsion. We obtain
\begin{align*}
    \mathrm{Tor}^{\nabla^\mathcal{R}}(\mathcal{R}_1^{-1}\rhd X,
    \mathcal{R}_2^{-1}\rhd Y)
    =&\nabla^\mathcal{R}_{\mathcal{R}_1^{-1}\rhd X}(\mathcal{R}_2^{-1}\rhd Y)
    -\nabla^\mathcal{R}_{(\mathcal{R}_1^{'-1}\mathcal{R}_2^{-1})\rhd Y}
    ((\mathcal{R}_2^{'-1}\mathcal{R}_1^{-1})\rhd X)\\
    &-[\mathcal{R}_1^{-1}\rhd X,\mathcal{R}_2^{-1}\rhd Y]_\mathcal{R}\\
    =&-\nabla^\mathcal{R}_YX+\nabla^\mathcal{R}_{\mathcal{R}_1^{-1}\rhd X}(
    \mathcal{R}_2^{-1}\rhd Y)
    +[Y,X]_\mathcal{R}\\
    =&-\mathrm{Tor}^{\nabla^\mathcal{R}}(Y,X),
\end{align*}
as well as
\begin{allowdisplaybreaks}
\begin{align*}
    \mathrm{Tor}^{\nabla^\mathcal{R}}(a\cdot X,Y\cdot b)
    =&\nabla^\mathcal{R}_{a\cdot X}(Y\cdot b)
    -\nabla^\mathcal{R}_{\mathcal{R}_1^{-1}\rhd(Y\cdot b)}
    (\mathcal{R}_2^{-1}\rhd(a\cdot X))
    -[a\cdot X,Y\cdot b]_\mathcal{R}\\
    =&a\cdot\nabla^\mathcal{R}_XY\cdot b
    +a\cdot(\mathcal{R}_1^{-1}\rhd Y)\cdot(\mathcal{R}_2^{-1}\rhd X)(b)\\
    &-\nabla^\mathcal{R}_{\mathcal{R}_{1(1)}^{-1} Y}
    ((\mathcal{R}_1^{'-1}\mathcal{R}_2^{-1})\rhd(a\cdot X))
    \cdot((\mathcal{R}_2^{'-1}\mathcal{R}_{1(2)}^{-1})\rhd b)\\
    &-a\cdot[X,Y\cdot b]_\mathcal{R}
    +((\mathcal{R}_1^{'-1}\mathcal{R}_1^{-1})\rhd(Y\cdot b))
    (\mathcal{R}_2^{'-1}\rhd a)
    \cdot(\mathcal{R}_2^{-1}\rhd X)\\
    =&a\cdot\nabla^\mathcal{R}_XY\cdot b
    +a\cdot(\mathcal{R}_1^{-1}\rhd Y)\cdot(\mathcal{R}_2^{-1}\rhd X)(b)\\
    &-(\mathcal{R}_{1(1)}^{-1}\rhd Y)
    ((\mathcal{R}_{1(1)}^{'-1}\mathcal{R}_{2(1)}^{-1})\rhd a)\cdot
    ((\mathcal{R}_{1(2)}^{'-1}\mathcal{R}_{2(2)}^{-1})\rhd X)\\
    &\cdot((\mathcal{R}_2^{'-1}\mathcal{R}_{1(2)}^{-1})\rhd b)\\
    &-((\mathcal{R}_1^{''-1}\mathcal{R}_{1(1)}^{'-1}\mathcal{R}_{2(1)}^{-1})
    \rhd a)\cdot\nabla^\mathcal{R}_{
    (\mathcal{R}_2^{''-1}\mathcal{R}_{1(1)}^{-1})\rhd Y}
    ((\mathcal{R}_{1(2)}^{'-1}\mathcal{R}_{2(2)}^{-1})\rhd X)\\
    &\cdot((\mathcal{R}_2^{'-1}\mathcal{R}_{1(2)}^{-1})\rhd b)\\
    &-a\cdot[X,Y]_\mathcal{R}\cdot b
    -a\cdot(\mathcal{R}_1^{-1}\rhd Y)\cdot(\mathcal{R}_2^{-1}\rhd X)(a)\\
    &+((\mathcal{R}_1^{'-1}\mathcal{R}_1^{-1})\rhd(Y\cdot b))
    (\mathcal{R}_2^{'-1}\rhd a)
    \cdot(\mathcal{R}_2^{-1}\rhd X)\\
    =&a\cdot\nabla^\mathcal{R}_XY\cdot b-a\cdot[X,Y]_\mathcal{R}\cdot b\\
    &-((\mathcal{R}_1^{''-1}\mathcal{R}_{1}^{'-1}\mathcal{R}_{2(1)}^{-1})
    \rhd a)\cdot\nabla^\mathcal{R}_{
    (\mathcal{R}_2^{''-1}\mathcal{R}_{1(1)}^{-1})\rhd Y}
    ((\mathcal{R}_{1}^{'''-1}\mathcal{R}_{2(2)}^{-1})\rhd X)\\
    &\cdot((\mathcal{R}_2^{'''-1}\mathcal{R}_2^{'-1}\mathcal{R}_{1(2)}^{-1})\rhd b)\\
    &-(\mathcal{R}_{1(1)}^{-1}\rhd Y)
    ((\mathcal{R}_{1(1)}^{'-1}\mathcal{R}_{2(1)}^{-1})\rhd a)\cdot
    ((\mathcal{R}_{1(2)}^{'-1}\mathcal{R}_{2(2)}^{-1})\rhd X)\\
    &\cdot((\mathcal{R}_2^{'-1}\mathcal{R}_{1(2)}^{-1})\rhd b)\\
    &+((\mathcal{R}_{1(1)}^{'-1}\mathcal{R}_{1(1)}^{-1})\rhd Y)
    ((\mathcal{R}_1^{''-1}\mathcal{R}_2^{'-1})\rhd a)\\
    &\cdot((\mathcal{R}_2^{''-1}\mathcal{R}_{1(2)}^{'-1}
    \mathcal{R}_{1(2)}^{-1})\rhd b)
    \cdot(\mathcal{R}_2^{-1}\rhd X)\\
    =&a\cdot\nabla^\mathcal{R}_XY\cdot b-a\cdot[X,Y]_\mathcal{R}\cdot b\\
    &-((\mathcal{R}_1^{''-1}\mathcal{R}_{2(1)}^{-1})
    \rhd a)\cdot\nabla^\mathcal{R}_{
    (\mathcal{R}_2^{''-1}\mathcal{R}_{1}^{-1})\rhd Y}
    (\mathcal{R}_{2(2)}^{-1}\rhd X)\cdot b\\
    &-(\mathcal{R}_{1(1)}^{-1}\rhd Y)
    ((\mathcal{R}_{1(1)}^{'-1}\mathcal{R}_{2(1)}^{-1})\rhd a)\cdot
    ((\mathcal{R}_{1(2)}^{'-1}\mathcal{R}_{2(2)}^{-1})\rhd X)\\
    &\cdot((\mathcal{R}_2^{'-1}\mathcal{R}_{1(2)}^{-1})\rhd b)\\
    &+((\mathcal{R}_{1(1)}^{'-1}\mathcal{R}_{1(1)}^{-1})\rhd Y)
    ((\mathcal{R}_1^{''-1}\mathcal{R}_2^{'-1})\rhd a)
    \cdot((\mathcal{R}_1^{'''-1}\mathcal{R}_2^{-1})\rhd X)\\
    &\cdot((\mathcal{R}_2^{'''-1}\mathcal{R}_2^{''-1}\mathcal{R}_{1(2)}^{'-1}
    \mathcal{R}_{1(2)}^{-1})\rhd b)\\
    =&a\cdot\mathrm{Tor}^{\nabla^\mathcal{R}}(X,Y)\cdot b\\
    &-(\mathcal{R}_{1(1)}^{-1}\rhd Y)
    ((\mathcal{R}_{1}^{'-1}\mathcal{R}_{2(1)}^{-1})\rhd a)\cdot
    ((\mathcal{R}_{1}^{''-1}\mathcal{R}_{2(2)}^{-1})\rhd X)\\
    &\cdot((\mathcal{R}_2^{''-1}\mathcal{R}_2^{'-1}
    \mathcal{R}_{1(2)}^{-1})\rhd b)\\
    &+((\mathcal{R}_{1}^{'-1}\mathcal{R}_{1}^{-1})\rhd Y)
    (\mathcal{R}_2^{'-1}\rhd a)
    \cdot(\mathcal{R}_2^{-1}\rhd X)\cdot b\\
    =&a\cdot\mathrm{Tor}^{\nabla^\mathcal{R}}(X,Y)\cdot b
\end{align*}
\end{allowdisplaybreaks}
and
\begin{allowdisplaybreaks}
\begin{align*}
    \mathrm{Tor}^{\nabla^\mathcal{R}}(X\cdot a,Y)
    =&\nabla^\mathcal{R}_{X\cdot a}Y
    -\nabla^\mathcal{R}_{\mathcal{R}_1^{-1}\rhd Y}
    ((\mathcal{R}_{2(1)}^{-1}\rhd X)\cdot(\mathcal{R}_{2(2)}^{-1}\rhd a))
    -[X\cdot a,Y]_\mathcal{R}\\
    =&\nabla^\mathcal{R}_X(\mathcal{R}_1^{-1}\rhd Y)\cdot(\mathcal{R}_2^{-1}\rhd a)
    -\nabla^\mathcal{R}_{\mathcal{R}_1^{-1}\rhd Y}
    (\mathcal{R}_{2(1)}^{-1}\rhd X)\cdot(\mathcal{R}_{2(2)}^{-1}\rhd a)\\
    &-((\mathcal{R}_1^{'-1}\mathcal{R}_{2(1)}^{-1})\rhd X)
    \cdot((\mathcal{R}_2^{'-1}\mathcal{R}_1^{-1})\rhd Y)
    (\mathcal{R}_{2(2)}^{-1}\rhd a)\\
    &-[X,\mathcal{R}_1^{-1}\rhd Y]\cdot(\mathcal{R}_2^{-1}\rhd a)
    +X\cdot(\mathcal{R}_1^{-1}\rhd Y)(\mathcal{R}_2^{-1}\rhd a)\\
    =&\mathrm{Tor}^{\nabla^\mathcal{R}}(X,\mathcal{R}_1^{-1}\rhd Y)
    \cdot(\mathcal{R}_2^{-1}\rhd a)\\
    &-((\mathcal{R}_1^{'-1}\mathcal{R}_{2}^{''-1})\rhd X)
    \cdot((\mathcal{R}_2^{'-1}\mathcal{R}_1^{''-1}\mathcal{R}_1^{-1})\rhd Y)
    (\mathcal{R}_{2}^{-1}\rhd a)\\
    &+X\cdot(\mathcal{R}_1^{-1}\rhd Y)(\mathcal{R}_2^{-1}\rhd a)\\
    =&\mathrm{Tor}^{\nabla^\mathcal{R}}(X,\mathcal{R}_1^{-1}\rhd Y)
    \cdot(\mathcal{R}_2^{-1}\rhd a)\\
    =&\mathrm{Tor}^{\nabla^\mathcal{R}}(X,(\mathcal{R}_1^{-1}\rhd Y)
    \cdot(\mathcal{R}_2^{-1}\rhd a))\\
    =&\mathrm{Tor}^{\nabla^\mathcal{R}}(X,a\cdot Y).
\end{align*}
\end{allowdisplaybreaks}
This concludes the proof of the lemma.
\end{proof}
There are natural extensions of an equivariant covariant
derivative $\nabla^\mathcal{R}$ on $\mathcal{A}$ to
braided multivector fields and differential forms in analogy
to differential geometry. We define the \textit{braided dual
pairing} $\langle\cdot,\cdot\rangle_\mathcal{R}\colon
\Omega^1_\mathcal{R}(\mathcal{A})\otimes
\mathfrak{X}^1_\mathcal{R}(\mathcal{R})\rightarrow\mathcal{A}$ by
$\langle\omega,X\rangle_\mathcal{R}=\omega(X)$
for all $\omega\in\Omega^1_\mathcal{R}
(\mathcal{A})$ and $X\in\mathfrak{X}^1_\mathcal{R}(\mathcal{A})$. It is
$H$-equivariant,
left $\mathcal{A}$-linear in the first and right $\mathcal{A}$-linear in the
second argument.
\begin{proposition}
An equivariant covariant derivative $\nabla^\mathcal{R}$ on
$\mathcal{A}$ induces an equivariant covariant derivative
$\tilde{\nabla}^\mathcal{R}$ on $\Omega^1_\mathcal{R}(\mathcal{A})$ via
\begin{equation}
    \langle\tilde{\nabla}^\mathcal{R}_X\omega,Y\rangle_\mathcal{R}
    =\mathscr{L}^\mathcal{R}_X\langle\omega,Y\rangle_\mathcal{R}
    -\langle\mathcal{R}_1^{-1}\rhd\omega,
    \nabla^\mathcal{R}_{\mathcal{R}_2^{-1}\rhd X}Y\rangle_\mathcal{R}
\end{equation}
for all $X,Y\in\mathfrak{X}^1_\mathcal{R}(\mathcal{A})$ and
$\omega\in\Omega^1_\mathcal{R}(\mathcal{A})$. Moreover, $\nabla^\mathcal{R}$
and $\tilde{\nabla}^\mathcal{R}$ extend
as braided derivations to equivariant covariant
derivatives
$$
\nabla^\mathcal{R}\colon\mathfrak{X}^1_\mathcal{R}(\mathcal{A})
\otimes\mathfrak{X}^\bullet_\mathcal{R}(\mathcal{A})
\rightarrow\mathfrak{X}^\bullet_\mathcal{R}(\mathcal{A})
\text{ and }
\tilde{\nabla}^\mathcal{R}\colon\mathfrak{X}^1_\mathcal{R}(\mathcal{A})
\otimes\Omega^\bullet_\mathcal{R}(\mathcal{A})
\rightarrow\Omega^\bullet_\mathcal{R}(\mathcal{A}),
$$
respectively. Namely, we inductively set
$\nabla^\mathcal{R}_Xa=X(a)=\tilde{\nabla}^\mathcal{R}_Xa$,
$$
\nabla^\mathcal{R}_X(Y\wedge_\mathcal{R}Z)
=\nabla^\mathcal{R}_XY\wedge_\mathcal{R}Z
+(\mathcal{R}_1^{-1}\rhd Y)\wedge_\mathcal{R}
\nabla^\mathcal{R}_{\mathcal{R}_2^{-1}\rhd X}Z
$$
and
$$
\tilde{\nabla}^\mathcal{R}_X(\omega\wedge_\mathcal{R}\eta)
=\tilde{\nabla}^\mathcal{R}_X\omega\wedge_\mathcal{R}\eta
+(\mathcal{R}_1^{-1}\rhd\omega)\wedge_\mathcal{R}
\tilde{\nabla}^\mathcal{R}_{\mathcal{R}_2^{-1}\rhd X}\eta
$$
for all $a\in\mathcal{A}$, $X\in\mathfrak{X}^1_\mathcal{R}(\mathcal{A})$,
$Y,Z\in\mathfrak{X}^\bullet_\mathcal{R}(\mathcal{A})$
and $\omega,\eta\in\Omega^\bullet_\mathcal{R}(\mathcal{A})$.
\end{proposition}
\begin{proof}
Let $\xi\in H$, $X\in\mathfrak{X}^1_\mathcal{R}(\mathcal{A})$,
$Y,Z\in\mathfrak{X}^\bullet_\mathcal{R}(\mathcal{A})$,
$\omega\in\Omega^1_\mathcal{R}(\mathcal{A})$ and
$\eta,\chi\in\Omega^\bullet_\mathcal{R}(\mathcal{A})$ be arbitrary. Then
$\tilde{\nabla}^\mathcal{R}$ is well-defined on $\Omega^1_\mathcal{R}(\mathcal{A})$,
since
\begin{align*}
    \mathscr{L}^\mathcal{R}_X\langle\omega,Y\cdot a\rangle_\mathcal{R}
    &-\langle\mathcal{R}_1^{-1}\rhd\omega,
    \nabla^\mathcal{R}_{\mathcal{R}_2^{-1}\rhd X}(Y\cdot a)\rangle_\mathcal{R}\\
    =&\mathscr{L}^\mathcal{R}_X\langle\omega,Y\rangle_\mathcal{R}\cdot a
    +\langle\mathcal{R}_{1(1)}^{-1}\rhd\omega,
    \mathcal{R}_{1(2)}^{-1}\rhd Y\rangle_\mathcal{R}
    \cdot(\mathcal{R}_2^{-1}\rhd X)(a)\\
    &-\langle\mathcal{R}_1^{-1}\rhd\omega,
    \nabla^\mathcal{R}_{\mathcal{R}_2^{-1}\rhd X}Y\rangle_\mathcal{R}\cdot a\\
    &-\langle\mathcal{R}_1^{-1}\rhd\omega,
    (\mathcal{R}_1^{'-1}\rhd Y)
    \cdot((\mathcal{R}_2^{'-1}\mathcal{R}_2^{-1})\rhd X)(a)\rangle_\mathcal{R}\\
    =&(\mathscr{L}^\mathcal{R}_X\langle\omega,Y\rangle_\mathcal{R}
    -\langle\mathcal{R}_1^{-1}\rhd\omega,
    \nabla^\mathcal{R}_{\mathcal{R}_2^{-1}\rhd X}Y\rangle_\mathcal{R})\cdot a.
\end{align*}
It is an equivariant covariant derivative because
\begin{align*}
    \langle\tilde{\nabla}^\mathcal{R}_{\xi_{(1)}\rhd X}(\xi_{(2)}\rhd\omega),
    Y\rangle_\mathcal{R}
    =&\mathscr{L}^\mathcal{R}_{\xi_{(1)}\rhd X}\langle\xi_{(2)}\rhd\omega,
    Y\rangle_\mathcal{R}
    -\langle(\mathcal{R}_1^{-1}\xi_{(2)})\rhd\omega,
    \nabla^\mathcal{R}_{(\mathcal{R}_2^{-1}\xi_{(1)})\rhd X}Y\rangle_\mathcal{R}\\
    =&\xi_{(1)}\rhd(\mathscr{L}^\mathcal{R}_{X}
    \langle\omega,S(\xi_{(2)})\rhd Y\rangle_\mathcal{R})\\
    &-\xi_{(1)}\rhd\langle\mathcal{R}_1^{-1}\rhd\omega,
    \nabla^\mathcal{R}_{\mathcal{R}_2^{-1}\rhd X}(S(\xi_{(2)})\rhd Y)
    \rangle_\mathcal{R}\\
    =&\xi_{(1)}\rhd\langle\tilde{\nabla}^\mathcal{R}_X\omega,
    S(\xi_{(2)})\rhd Y\rangle_\mathcal{R}\\
    =&\langle\xi\rhd(\tilde{\nabla}^\mathcal{R}_X\omega),Y\rangle_\mathcal{R}
\end{align*}
shows that it is $H$-equivariant, while
\begin{align*}
    \langle\tilde{\nabla}^\mathcal{R}_{a\cdot X}\omega,Y\rangle_\mathcal{R}
    =&\mathscr{L}^\mathcal{R}_{a\cdot X}\langle\omega,Y\rangle_\mathcal{R}
    -\langle\mathcal{R}_1^{-1}\rhd\omega,
    \nabla^\mathcal{R}_{(\mathcal{R}_{2(1)}^{-1}\rhd a)
    \cdot(\mathcal{R}_{2(2)}^{-1}\rhd X)}Y\rangle_\mathcal{R}\\
    =&a\cdot\mathscr{L}^\mathcal{R}_X\langle\omega,Y\rangle_\mathcal{R}
    -((\mathcal{R}_1^{'-1}\mathcal{R}_{2(1)}^{-1})\rhd a)
    \cdot\langle(\mathcal{R}_2^{'-1}\mathcal{R}_1^{-1})\rhd\omega,
    \nabla^\mathcal{R}_{\mathcal{R}_{2(2)}^{-1}\rhd X}Y\rangle_\mathcal{R}\\
    =&a\cdot\mathscr{L}^\mathcal{R}_X\langle\omega,Y\rangle_\mathcal{R}
    -a\cdot\langle\mathcal{R}_1^{-1}\rhd\omega,
    \nabla^\mathcal{R}_{\mathcal{R}_{2}^{-1}\rhd X}Y\rangle_\mathcal{R}\\
    =&a\cdot\langle\tilde{\nabla}^\mathcal{R}_X\omega,Y\rangle_\mathcal{R}\\
    =&\langle a\cdot\tilde{\nabla}^\mathcal{R}_X\omega,Y\rangle_\mathcal{R}
\end{align*}
and
\begin{align*}
    \langle\tilde{\nabla}^\mathcal{R}_X(a\cdot\omega),Y\rangle_\mathcal{R}
    =&\mathscr{L}^\mathcal{R}_X\langle a\cdot\omega,Y\rangle_\mathcal{R}
    -\langle(\mathcal{R}_{1(1)}^{-1}\rhd a)\cdot(\mathcal{R}_{1(2)}^{-1}\rhd\omega),
    \nabla^\mathcal{R}_{\mathcal{R}_2^{-1}\rhd X}Y\rangle_\mathcal{R}\\
    =&X(a)\cdot\langle\omega,Y\rangle_\mathcal{R}
    +(\mathcal{R}_1^{-1}\rhd a)
    \cdot\mathscr{L}^\mathcal{R}_{\mathcal{R}_2^{-1}\rhd X}
    \langle\omega,Y\rangle_\mathcal{R}\\
    &-(\mathcal{R}_1^{-1}\rhd a)
    \cdot\langle\mathcal{R}_1^{'-1}\rhd\omega,
    \nabla^\mathcal{R}_{(\mathcal{R}_2^{'-1}\mathcal{R}_2^{-1})\rhd X}
    Y\rangle_\mathcal{R}\\
    =&\langle\mathscr{L}^\mathcal{R}_X(a)\cdot\omega
    +(\mathcal{R}_1^{-1}\rhd a)
    \cdot(\tilde{\nabla}^\mathcal{R}_{\mathcal{R}_2^{-1}\rhd X}\omega),
    Y\rangle_\mathcal{R}
\end{align*}
prove that $\tilde{\nabla}^\mathcal{R}$ provides the correct linearity
properties via the non-degeneracy of the braided dual pairing.
To prove that the extension of $\nabla^\mathcal{R}$ to
$\mathfrak{X}^\bullet_\mathcal{R}(\mathcal{A})$ is well-defined it is
sufficient to show that
\begin{align*}
    \nabla^\mathcal{R}_XY\wedge_\mathcal{R}Z
    &+(\mathcal{R}_1^{-1}\rhd Y)\wedge_\mathcal{R}
    \nabla^\mathcal{R}_{\mathcal{R}_2^{-1}\rhd X}Z\\
    &-(-1)^{k\cdot\ell}\bigg(
    \nabla^\mathcal{R}_X(\mathcal{R}_1^{-1}\rhd Z)\wedge_\mathcal{R}
    (\mathcal{R}_2^{-1}\rhd Y)\\
    &+((\mathcal{R}_1^{-1}\mathcal{R}_1^{'-1})\rhd Z)\wedge_\mathcal{R}
    \nabla^\mathcal{R}_{\mathcal{R}_2^{-1}\rhd X}(\mathcal{R}_2^{'-1}\rhd Y)\bigg)\\
    =&0
\end{align*}
where $k$ and $\ell$ are the degrees of $Y$ and $Z$, respectively.
Starting from $k=\ell=0$ this follows inductively by the braided
commutativity of $\wedge_\mathcal{R}$ and the equivariance of $\nabla^\mathcal{R}$.
Assume that $\nabla^\mathcal{R}_{a\cdot X}Y=a\cdot\nabla^\mathcal{R}_XY$ and
$\nabla^\mathcal{R}_X(a\cdot Y)=\mathscr{L}^\mathcal{R}_Xa\cdot Y
+(\mathcal{R}_1^{-1}\rhd a)\cdot\nabla^\mathcal{R}_{\mathcal{R}_2^{-1}\rhd X}Y$
for a fixed degree $k>0$ of $Y$. Let the degree of $Z$ be $1$. Then
\begin{align*}
    \nabla^\mathcal{R}_{a\cdot X}(Y\wedge_\mathcal{R}Z)
    =&\nabla^\mathcal{R}_{a\cdot X}Y\wedge_\mathcal{R}Z
    +(\mathcal{R}_1^{-1}\rhd Y)\wedge_\mathcal{R}
    \nabla^\mathcal{R}_{(\mathcal{R}_{2(1)}^{-1}\rhd a)
    \cdot(\mathcal{R}_{2(2)}^{-1}\rhd X)}Z\\
    =&a\cdot\nabla^\mathcal{R}_XY\wedge_\mathcal{R}Z
    +((\mathcal{R}_1^{'-1}\mathcal{R}_{2(1)}^{-1})\rhd a)
    \cdot((\mathcal{R}_2^{'-1}\mathcal{R}_1^{-1})\rhd Y)\wedge_\mathcal{R}
    \nabla^\mathcal{R}_{\mathcal{R}_{2(2)}^{-1}\rhd X}Z\\
    =&a\cdot\nabla^\mathcal{R}_XY\wedge_\mathcal{R}Z
    +a\cdot(\mathcal{R}_1^{-1}\rhd Y)
    \wedge_\mathcal{R}\nabla^\mathcal{R}_{\mathcal{R}_{2}^{-1}\rhd X}Z\\
    =&a\cdot(\nabla^\mathcal{R}_X(Y\wedge_\mathcal{R}Z))
\end{align*}
and
\begin{align*}
    \nabla^\mathcal{R}_X(a\cdot Y\wedge_\mathcal{R}Z)
    =&\nabla^\mathcal{R}_X(a\cdot Y)\wedge_\mathcal{R}Z
    +(\mathcal{R}_{1(1)}^{-1}\rhd a)\cdot(\mathcal{R}_{1(2)}^{-1}\rhd Y)
    \wedge_\mathcal{R}\nabla^\mathcal{R}_{\mathcal{R}_2^{-1}\rhd X}Z\\
    =&\mathscr{L}^\mathcal{R}_X(a)\cdot Y\wedge_\mathcal{R}Z
    +(\mathcal{R}_1^{-1}\rhd a)\cdot(\nabla^\mathcal{R}_{\mathcal{R}_2^{-1}\rhd X}Y)
    \wedge_\mathcal{R}Z\\
    &+(\mathcal{R}_{1}^{-1}\rhd a)\cdot(\mathcal{R}_{1}^{'-1}\rhd Y)
    \wedge_\mathcal{R}\nabla^\mathcal{R}_{(\mathcal{R}_2^{'-1}\mathcal{R}_2^{-1})
    \rhd X}Z\\
    =&\mathscr{L}^\mathcal{R}_X(a)\cdot Y\wedge_\mathcal{R}Z
    +(\mathcal{R}_1^{-1}\rhd a)\cdot
    \nabla^\mathcal{R}_{\mathcal{R}_2^{-1}\rhd X}
    (Y\wedge_\mathcal{R}Z)
\end{align*}
show that $\nabla^\mathcal{R}$ has the correct linearity properties on
elements of degree $k+1$. Inductively this shows $\nabla^\mathcal{R}$ is an
equivariant covariant derivative on $\mathfrak{X}^\bullet_\mathcal{R}(\mathcal{A})$.
Similarly one proves that $\tilde{\nabla}^\mathcal{R}$ is an equivariant
covariant derivative on $\Omega^\bullet_\mathcal{R}(\mathcal{A})$.
\end{proof}
In Riemannian geometry, covariant derivatives are always consider together
with a Riemannian metric. We want to generalize them to the braided symmetric
setting. 
\begin{definition}[Equivariant Metric]
For a triangular Hopf algebra $(H,\mathcal{R})$ and a braided
commutative left $H$-module algebra $\mathcal{A}$ we define a $\Bbbk$-linear map
$$
{\bf g}\colon\mathfrak{X}^1_\mathcal{R}(\mathcal{A})
\otimes_\mathcal{A}\mathfrak{X}^1_\mathcal{R}(\mathcal{A})
\rightarrow\mathcal{A}
$$
to be an \textit{equivariant metric} if it is
left $\mathcal{A}$-linear in the first argument, $H$-equivariant
and \textit{braided symmetric},
i.e. if ${\bf g}(Y,X)
={\bf g}(\mathcal{R}_1^{-1}\rhd X,\mathcal{R}_2^{-1}\rhd Y)$ for all
$X,Y\in\mathfrak{X}^1_\mathcal{R}(\mathcal{A})$. An equivariant metric 
${\bf g}$ is said to be
\begin{enumerate}
\item[i.)] \textit{non-degenerate} if
${\bf g}(X,Y)=0$ for all $Y\in\mathfrak{X}^1_\mathcal{R}(\mathcal{A})$ implies
$X=0$;

\item[ii.)] \textit{strongly non-degenerate} if
${\bf g}(X,X)\neq 0$ for $X\neq 0$;

\item[iii.)] \textit{Riemannian} if ${\bf g}$ is strongly non-degenerate such that
${\bf g}(X,X)\geq 0$ for all $X\in\mathfrak{X}^1_\mathcal{R}(\mathcal{A})$ and
a partial order $\geq$ on $\mathcal{A}$;
\end{enumerate}
An equivariant covariant derivative $\nabla^\mathcal{R}
\colon
\mathfrak{X}^1_\mathcal{R}(\mathcal{A})\otimes\mathfrak{X}^1_\mathcal{R}
(\mathcal{A})\rightarrow\mathfrak{X}^1_\mathcal{R}(\mathcal{A})
$ 
on $\mathcal{A}$ is said to be a
\textit{metric equivariant covariant derivative} with respect to an equivariant
metric ${\bf g}$ if
\begin{equation}
    \mathscr{L}^\mathcal{R}_X({\bf g}(Y,Z))
    ={\bf g}(\nabla^\mathcal{R}_XY,Z)
    +{\bf g}(\mathcal{R}_1^{-1}\rhd Y,\nabla^\mathcal{R}_{\mathcal{R}_2^{-1}\rhd X}Z)
\end{equation}
holds for all $X,Y,Z\in\mathfrak{X}^1_\mathcal{R}(\mathcal{A})$.
\end{definition}
%
Note that equivariance of a metric is a quite strong requirement. Similar approaches
which omit this condition are e.g. \cite{Aschieri2010,Aschieri2006,GaetanoThomas19}.
\begin{lemma}
Any equivariant metric is braided right $\mathcal{A}$-linear in the first argument
as well as right $\mathcal{A}$-linear and braided left $\mathcal{A}$-linear in
the second argument.
\end{lemma}
In the subsequent proposition we prove the existence and uniqueness of
a Levi-Civita covariant derivative in braided geometry: for
every non-degenerate equivariant metric there exists a unique metric torsion-free
equivariant covariant derivative.
We are even going to prove this for an arbitrary value
of the torsion. Note that the non-degeneracy is
crucial in the proof since the equivariant covariant derivative
is defined implicitly in terms of the equivariant metric.
\begin{proposition}[Levi-Civita]
Let $\bf{g}$ be a non-degenerate equivariant metric on $\mathcal{A}$.
Then there is a unique metric equivariant covariant derivative on $\mathcal{A}$
with fixed torsion
$T\colon\mathfrak{X}^2_\mathcal{R}(\mathcal{A})
\rightarrow\mathfrak{X}^1_\mathcal{R}(\mathcal{A})$.
\end{proposition}
\begin{proof}
Fix an equivariant metric ${\bf g}$ on $\mathcal{A}$.
Any equivariant covariant
derivative $\nabla^\mathcal{R}$ on $\mathcal{A}$ which is metric with
respect to ${\bf g}$ satisfies
\begin{equation}\label{eq57}
\begin{split}
    2{\bf g}(\nabla^\mathcal{R}_XY,Z)
    =&X({\bf g}(Y,Z))
    +(\mathcal{R}_{1(1)}^{-1}\rhd Y)({\bf g}(\mathcal{R}_{1(2)}^{-1}\rhd Z,
    \mathcal{R}_2^{-1}\rhd X))\\
    &-(\mathcal{R}_{1}^{-1}\rhd Z)({\bf g}(\mathcal{R}_{2(1)}^{-1}\rhd X,
    \mathcal{R}_{2(2)}^{-1}\rhd Y))\\
    &-{\bf g}(X,[Y,Z]_\mathcal{R})
    +{\bf g}(\mathcal{R}_{1(1)}^{-1}\rhd Y,
    [\mathcal{R}_{1(2)}^{-1}\rhd Z,\mathcal{R}_2^{-1}\rhd X]_\mathcal{R})\\
    &+{\bf g}(\mathcal{R}_{1}^{-1}\rhd Z,
    [\mathcal{R}_{2(1)}^{-1}\rhd X,\mathcal{R}_{2(2)}^{-1}\rhd Y]_\mathcal{R})\\
    &-{\bf g}(X,\mathrm{Tor}^{\nabla^\mathcal{R}}(Y,Z))
    -{\bf g}(\mathcal{R}_1^{-1}\rhd Y,
    \mathrm{Tor}^{\nabla^\mathcal{R}}(\mathcal{R}_2^{-1}\rhd X,Z))\\
    &+{\bf g}(\mathcal{R}_1^{-1}\rhd Z,\mathrm{Tor}^{\nabla^\mathcal{R}}
    (\mathcal{R}_{2(1)}^{-1}\rhd X,\mathcal{R}_{2(2)}^{-1}\rhd Y))
\end{split}
\end{equation}
for all $X,Y,Z\in\mathfrak{X}^1_\mathcal{R}(\mathcal{A})$. To see this,
we first note that
\begin{align*}
    X({\bf g}(Y,Z))
    =&{\bf g}(\nabla^\mathcal{R}_XY,Z)
    +{\bf g}(\mathcal{R}_1^{-1}\rhd Y,\nabla^\mathcal{R}_{\mathcal{R}_2^{-1}\rhd X}Z)\\
    =&{\bf g}(\nabla^\mathcal{R}_XY,Z)
    +{\bf g}(\mathcal{R}_1^{-1}\rhd Y,
    \nabla^\mathcal{R}_{\mathcal{R}_1^{'-1}\rhd Z}
    ((\mathcal{R}_2^{'-1}\mathcal{R}_2^{-1})\rhd X))\\
    &+{\bf g}(\mathcal{R}_1^{-1}\rhd Y,[\mathcal{R}_2^{-1}\rhd X,Z]_\mathcal{R})
    +{\bf g}(\mathcal{R}_1^{-1}\rhd Y,\mathrm{Tor}^{\nabla^\mathcal{R}}(
    \mathcal{R}_2^{-1}\rhd X,Z)).
\end{align*}
Then
\begin{allowdisplaybreaks}
\begin{align*}
    X({\bf g}(Y,Z))
    &+(\mathcal{R}_{1(1)}^{-1}\rhd Y)({\bf g}(\mathcal{R}_{1(2)}^{-1}\rhd Z,
    \mathcal{R}_2^{-1}\rhd X))\\
    &-(\mathcal{R}_1^{-1}\rhd Z)({\bf g}(\mathcal{R}_{2(1)}^{-1}\rhd X,
    \mathcal{R}_{2(2)}^{-1}\rhd Y))\\
    =&{\bf g}(\nabla^\mathcal{R}_XY,Z)
    +{\bf g}(\mathcal{R}_1^{-1}\rhd Y,
    \nabla^\mathcal{R}_{\mathcal{R}_1^{'-1}\rhd Z}
    ((\mathcal{R}_2^{'-1}\mathcal{R}_2^{-1})\rhd X))\\
    &+{\bf g}(\mathcal{R}_1^{-1}\rhd Y,[\mathcal{R}_2^{-1}\rhd X,Z]_\mathcal{R})
    +{\bf g}(\mathcal{R}_1^{-1}\rhd Y,\mathrm{Tor}^{\nabla^\mathcal{R}}(
    \mathcal{R}_2^{-1}\rhd X,Z))\\
    &+{\bf g}(\nabla^\mathcal{R}_{\mathcal{R}_{1(1)}^{'-1}\rhd Y}
    (\mathcal{R}_{1(2)}^{'-1}\rhd Z),\mathcal{R}_{2}^{'-1}\rhd X)\\
    &+{\bf g}((\mathcal{R}_1^{-1}\mathcal{R}_{1(2)}^{'-1})\rhd Z,
    \nabla^\mathcal{R}_{(\mathcal{R}_1^{''-1}\mathcal{R}_{2}^{'-1})\rhd X}
    ((\mathcal{R}_2^{''-1}\mathcal{R}_2^{-1}\mathcal{R}_{1(1)}^{'-1})\rhd Y))\\
    &+{\bf g}((\mathcal{R}_1^{-1}\mathcal{R}_{1(2)}^{'-1})\rhd Z,
    [(\mathcal{R}_2^{-1}\mathcal{R}_{1(1)}^{'-1})\rhd Y,
    \mathcal{R}_{2}^{'-1}\rhd X]_\mathcal{R})\\
    &+{\bf g}((\mathcal{R}_1^{-1}\mathcal{R}_{1(2)}^{'-1}\rhd Z,
    \mathrm{Tor}^{\nabla^\mathcal{R}}(
    (\mathcal{R}_2^{-1}\mathcal{R}_{1(1)}^{'-1})\rhd Y,
    \mathcal{R}_{2}^{'-1}\rhd X))\\
    &-{\bf g}(\nabla^\mathcal{R}_{\mathcal{R}_1^{'-1}\rhd Z}
    (\mathcal{R}_{2(1)}^{'-1}\rhd X),\mathcal{R}_{2(2)}^{'-1}\rhd Y)\\
    &-{\bf g}((\mathcal{R}_1^{-1}\mathcal{R}_{2(1)}^{''-1})\rhd X,
    \nabla^\mathcal{R}_{(\mathcal{R}_1^{'-1}\mathcal{R}_{2(2)}^{''-1})\rhd Y}
    ((\mathcal{R}_2^{'-1}\mathcal{R}_2^{-1}\mathcal{R}_{1}^{''-1})\rhd Z))\\
    &-{\bf g}((\mathcal{R}_1^{-1}\mathcal{R}_{2(1)}^{'-1})\rhd X,
    [(\mathcal{R}_2^{-1}\mathcal{R}_{1}^{'-1})\rhd Z,
    \mathcal{R}_{2(2)}^{'-1}\rhd Y]_\mathcal{R})\\
    &-{\bf g}((\mathcal{R}_1^{-1}\mathcal{R}_{2(1)}^{'-1})\rhd X,
    \mathrm{Tor}^{\nabla^\mathcal{R}}(
    (\mathcal{R}_2^{-1}\mathcal{R}_{1}^{'-1})\rhd Z,
    \mathcal{R}_{2(2)}^{'-1}\rhd Y))\\
    =&2{\bf g}(\nabla^\mathcal{R}_XY,Z)\\
    &+{\bf g}(\mathcal{R}_1^{-1}\rhd Y,[\mathcal{R}_2^{-1}\rhd X,Z]_\mathcal{R})
    +{\bf g}(\mathcal{R}_1^{-1}\rhd Y,\mathrm{Tor}^{\nabla^\mathcal{R}}(
    \mathcal{R}_2^{-1}\rhd X,Z))\\
    &+{\bf g}((\mathcal{R}_1^{-1}\mathcal{R}_{1(2)}^{'-1})\rhd Z,
    [(\mathcal{R}_2^{-1}\mathcal{R}_{1(1)}^{'-1})\rhd Y,
    \mathcal{R}_{2}^{'-1}\rhd X]_\mathcal{R})\\
    &+{\bf g}((\mathcal{R}_1^{-1}\mathcal{R}_{1(2)}^{'-1}\rhd Z,
    \mathrm{Tor}^{\nabla^\mathcal{R}}(
    (\mathcal{R}_2^{-1}\mathcal{R}_{1(1)}^{'-1})\rhd Y,
    \mathcal{R}_{2}^{'-1}\rhd X))\\
    &-{\bf g}((\mathcal{R}_1^{-1}\mathcal{R}_{2(1)}^{'-1})\rhd X,
    [(\mathcal{R}_2^{-1}\mathcal{R}_{1}^{'-1})\rhd Z,
    \mathcal{R}_{2(2)}^{'-1}\rhd Y]_\mathcal{R})\\
    &-{\bf g}((\mathcal{R}_1^{-1}\mathcal{R}_{2(1)}^{'-1})\rhd X,
    \mathrm{Tor}^{\nabla^\mathcal{R}}(
    (\mathcal{R}_2^{-1}\mathcal{R}_{1}^{'-1})\rhd Z,
    \mathcal{R}_{2(2)}^{'-1}\rhd Y))
\end{align*}
\end{allowdisplaybreaks}
gives the proposed formula. Since the equivariant metric is
non-degenerate this proves that a metric equivariant covariant derivative
with a fixed torsion is unique. To also provide existence we show that
the above formula actually defines a metric equivariant covariant derivative
with torsion $T$. Replacing $X$ in (\ref{eq57}) by $a\cdot X$ for an arbitrary
$a\in\mathcal{A}$ we observe that the terms including torsion are linear with
respect to $a$, as well as the first and the fourth term on the right-hand-side.
The additional non-$\mathcal{A}$-linear terms of the second and third summand cancel with
the corresponding additional terms of the sixth and fifth summand, respectively.
By the non-degeneracy of ${\bf g}$ this proves that $\nabla^\mathcal{R}$ is
left $\mathcal{A}$-linear in the first argument. Similarly one demonstrates the
braided Leibniz rule in the second argument.
\end{proof}
The unique torsion-free metric equivariant covariant derivative on
$(\mathcal{A},{\bf g})$ is said to be the 
\textit{equivariant Levi-Civita covariant derivative}.

\section{Gauge Equivalent Braided Cartan Calculi and Covariant 
Derivatives}\label{Sec3.6}

In this last section we prove that the gauge equivalence induced by
a Drinfel'd twist is compatible with the construction of the braided
Cartan calculus and the notion of equivariant covariant derivative.
Namely, we describe how the Drinfel'd functor together with the
corresponding natural transformation discussed in Section~\ref{Sec2.4}
deforms the involved structure. For a given braided
commutative algebra $\mathcal{A}$ we define the twisted multivector
fields and twisted differential forms, as well as the twisted
Gerstenhaber bracket, Lie derivative, insertion and de Rham
differential. If the universal $\mathcal{R}$-matrix is trivial this
recovers the well-known twisted Cartan calculus (see e.g. \cite{Aschieri2006}).
Even in the general case, we prove that the twisted Cartan calculus
is isomorphic to the braided Cartan calculus with respect to the twisted
algebra and twisted triangular structure. In particular, this
isomorphism (taken from \cite{Aschieri2006,Fiore1997,GurMaj1994})
intertwines the braided Lie derivative, insertion and
de Rham differential. In a similar fashion we prove that a
twisted covariant derivative can be interpreted as an equivariant
covariant derivative on the twisted algebra with respect to the twisted
triangular structure, by applying the same isomorphism.
The concepts of equivariant metric and equivariant Levi-Civita covariant
derivative are respected as well.
We refer to \cite{Schenkel2016}~Chap.~3-4 for a similar discussion.

Recall that for any Drinfel'd twist $\mathcal{F}$ on a triangular Hopf algebra $(H,\mathcal{R})$
and any braided commutative left $H$-module algebra $\mathcal{A}$ the
twisted product
$$
a\cdot_\mathcal{F}b
=(\mathcal{F}_1^{-1}\rhd a)\cdot(\mathcal{F}_2^{-1}\rhd b),
$$
where $a,b\in\mathcal{A}$, structures $\mathcal{A}$ as a 
left $H_\mathcal{F}$-module algebra which is braided commutative
with respect to $\mathcal{R}_\mathcal{F}=\mathcal{F}_{21}\mathcal{R}
\mathcal{F}^{-1}$. As usual we write $\mathcal{A}_\mathcal{F}=
(\mathcal{A},\cdot_\mathcal{F})$. More general, the Drinfel'd functor
$$
\mathrm{Drin}_\mathcal{F}\colon
{}_\mathcal{A}^H\mathcal{M}\rightarrow
{}_{\mathcal{A}_\mathcal{F}}^{H_\mathcal{F}}\mathcal{M},
$$
which is the identity on morphisms and assigns to any $H$-equivariant
left $\mathcal{A}$-module $\mathcal{M}$ the same left $H$-module, however
with left $\mathcal{A}_\mathcal{F}$-module action
$$
a\cdot_\mathcal{F}m
=(\mathcal{F}_1^{-1}\rhd a)\cdot(\mathcal{F}_2^{-1}\rhd m),
$$
where $a\in\mathcal{A}$ and $m\in\mathcal{M}$ is an isomorphism of categories.
As usual we denote $\mathcal{M}_\mathcal{F}=(\mathcal{M},\cdot_\mathcal{F})$.
On $H$-equivariant braided symmetric $\mathcal{A}$-bimodules the Drinfel'd
functor
$$
\mathrm{Drin}_\mathcal{F}\colon
({}_\mathcal{A}^H\mathcal{M}_\mathcal{A}^\mathcal{R},\otimes_\mathcal{A})\rightarrow
({}_{\mathcal{A}_\mathcal{F}}^{H_\mathcal{F}}\mathcal{M}_{
\mathcal{A}_\mathcal{F}}^{\mathcal{R}_\mathcal{F}},\otimes_{\mathcal{A}_\mathcal{F}})
$$
is a braided monoidal equivalence of braided monoidal categories
with braided monoidal natural transformation given on objects
$\mathcal{M}$ and $\mathcal{M}'$ of
${}_\mathcal{A}^H\mathcal{M}^\mathcal{R}_\mathcal{A}$ by
$$
\varphi_{\mathcal{M},\mathcal{M}'}\colon
\mathcal{M}_\mathcal{F}\otimes_{\mathcal{A}_\mathcal{F}}\mathcal{M}'_\mathcal{F}
\ni(m\otimes_{\mathcal{A}_\mathcal{F}}m')
\mapsto(\mathcal{F}_1^{-1}\rhd m)\otimes_\mathcal{A}(\mathcal{F}_2^{-1}\rhd m')
\in(\mathcal{M}\otimes_\mathcal{A}\mathcal{M}')_\mathcal{F}.
$$
Fix a Drinfel'd twist $\mathcal{F}$ on $H$ in the following.
\begin{lemma}\label{lemma23}
The assignment
$$
(\mathfrak{X}^1_\mathcal{R}(\mathcal{A}))_\mathcal{F}
\ni X\mapsto X^\mathcal{F}
\in\mathfrak{X}^1_{\mathcal{R}_\mathcal{F}}(\mathcal{A}_\mathcal{F}),
$$
where $X^\mathcal{F}(a)=(\mathcal{F}_1^{-1}\rhd X)(\mathcal{F}_2^{-1}\rhd a)$
for all $a\in\mathcal{A}_\mathcal{F}$,
is an isomorphism of $H_\mathcal{F}$-equivariant braided symmetric
$\mathcal{A}_\mathcal{F}$-bimodules. Namely,
\begin{align*}
    \xi\rhd X^\mathcal{F}=&(\xi\rhd X)^\mathcal{F},\\
    a\cdot_{\mathcal{R}_\mathcal{F}}X^\mathcal{F}
    =&(a\cdot_\mathcal{F}X)^\mathcal{F},\\
    X^\mathcal{F}\cdot_{\mathcal{R}_\mathcal{F}}a
    =&(X\cdot_\mathcal{F}a)^\mathcal{F}
\end{align*}
for all $\xi\in H$, $a\in\mathcal{A}$ and
$X\in\mathfrak{X}^1_\mathcal{R}(\mathcal{A})$, where we denoted the
$\mathcal{A}_\mathcal{F}$-module actions on 
$\mathfrak{X}^1_{\mathcal{R}_\mathcal{F}}(\mathcal{A}_\mathcal{F})$ by
$\cdot_{\mathcal{R}_\mathcal{F}}$.
\end{lemma}
\begin{proof}
The assignment is well-defined, since for every
$X\in\mathrm{Der}_\mathcal{R}(\mathcal{A})$ one obtains an element
$X^\mathcal{F}\in\mathrm{Der}_{\mathcal{R}_\mathcal{F}}(\mathcal{A}_\mathcal{F})$
because
\begin{allowdisplaybreaks}
\begin{align*}
    X^\mathcal{F}(a\cdot_\mathcal{F}b)
    =&(\mathcal{F}_1^{-1}\rhd X)
    (((\mathcal{F}_{2(1)}^{-1}\mathcal{F}_1^{'-1})\rhd a)
    \cdot((\mathcal{F}_{2(2)}^{-1}\mathcal{F}_2^{'-1})\rhd b))\\
    =&(\mathcal{F}_1^{-1}\rhd X)
    ((\mathcal{F}_{2(1)}^{-1}\mathcal{F}_1^{'-1})\rhd a)
    \cdot((\mathcal{F}_{2(2)}^{-1}\mathcal{F}_2^{'-1})\rhd b)\\
    &+((\mathcal{R}_1^{-1}\mathcal{F}_{2(1)}^{-1}\mathcal{F}_1^{'-1})\rhd a)
    \cdot((\mathcal{R}_2^{-1}\mathcal{F}_1^{-1})\rhd X)
    ((\mathcal{F}_{2(2)}^{-1}\mathcal{F}_2^{'-1})\rhd b)\\
    =&X^\mathcal{F}(a)\cdot_\mathcal{F}b
    +((\mathcal{R}_1^{-1}\mathcal{F}_{1(2)}^{-1}\mathcal{F}_2^{'-1})\rhd a)
    \cdot((\mathcal{R}_2^{-1}\mathcal{F}_{1(1)}^{-1}\mathcal{F}_1^{'-1})\rhd X)
    (\mathcal{F}_{2}^{-1}\rhd b)\\
    =&X^\mathcal{F}(a)\cdot_\mathcal{F}b
    +((\mathcal{F}_{1(1)}^{-1}\mathcal{R}_1^{-1}\mathcal{F}_2^{'-1})\rhd a)
    \cdot((\mathcal{F}_{1(2)}^{-1}\mathcal{R}_2^{-1}\mathcal{F}_1^{'-1})\rhd X)
    (\mathcal{F}_{2}^{-1}\rhd b)\\
    =&X^\mathcal{F}(a)\cdot_\mathcal{F}b
    +((\mathcal{F}_{1(1)}^{-1}\mathcal{F}_1^{'-1}\mathcal{R}_{\mathcal{F}1}^{-1})
    \rhd a)
    \cdot((\mathcal{F}_{1(2)}^{-1}\mathcal{F}_2^{'-1}\mathcal{R}_{\mathcal{F}2}^{-1})
    \rhd X)(\mathcal{F}_{2}^{-1}\rhd b)\\
    =&X^\mathcal{F}(a)\cdot_\mathcal{F}b
    +((\mathcal{F}_{1}^{-1}\mathcal{R}_{\mathcal{F}1}^{-1})\rhd a)
    \cdot((\mathcal{F}_{2(1)}^{-1}\mathcal{F}_1^{'-1}\mathcal{R}_{\mathcal{F}2}^{-1})
    \rhd X)((\mathcal{F}_{2(2)}^{-1}\mathcal{F}_2^{'-1})\rhd b)\\
    =&X^\mathcal{F}(a)\cdot_\mathcal{F}b
    +(\mathcal{R}_{\mathcal{F}1}^{-1}\rhd a)
    \cdot_\mathcal{F}(\mathcal{R}_{\mathcal{F}2}^{-1}\rhd X)^\mathcal{F}(b)\\
    =&X^\mathcal{F}(a)\cdot_\mathcal{F}b
    +(\mathcal{R}_{\mathcal{F}1}^{-1}\rhd a)
    \cdot_\mathcal{F}(\mathcal{R}_{\mathcal{F}2}^{-1}\rhd X^\mathcal{F})(b)
\end{align*}
\end{allowdisplaybreaks}
for all $a,b\in\mathcal{A}$, where we used that
$(\xi\rhd X)^\mathcal{F}=\xi\rhd X^\mathcal{F}$ for all $\xi\in H$. The
latter is true because
\begin{allowdisplaybreaks}
\begin{align*}
    (\xi\rhd X^\mathcal{F})(a)
    =&\xi_{\widehat{(1)}}\rhd(X^\mathcal{F}(S_\mathcal{F}
    (\xi_{\widehat{(2)}})\rhd a))\\
    =&(\mathcal{F}_1\xi_{(1)}\mathcal{F}_1^{'-1})\rhd
    ((\mathcal{F}_1^{''-1}\rhd X)((\mathcal{F}_2^{''-1}
    \beta S(\mathcal{F}_2\xi_{(2)}\mathcal{F}_2^{'-1})
    \beta^{-1})\rhd a))\\
    =&(\mathcal{F}_1\xi_{(1)})\rhd
    ((\mathcal{F}_1^{'-1}\rhd X)((\mathcal{F}_{2(1)}^{'-1}\mathcal{F}_1^{''-1}
    \beta S(\mathcal{F}_2\xi_{(2)}\mathcal{F}_{2(2)}^{'-1}\mathcal{F}_2^{''-1})
    \beta^{-1})\rhd a))\\
    =&(\mathcal{F}_1\xi_{(1)})\rhd
    ((\mathcal{F}_1^{'-1}\rhd X)((\mathcal{F}_{2(1)}^{'-1}
    S(\mathcal{F}_2\xi_{(2)}\mathcal{F}_{2(2)}^{'-1})
    \beta^{-1})\rhd a))\\
    =&(\mathcal{F}_1\xi_{(1)})\rhd
    (X((S(\mathcal{F}_2\xi_{(2)})\beta^{-1})\rhd a))\\
    =&((\mathcal{F}_{1(1)}\xi)\rhd X)
    ((\mathcal{F}_{1(2)}S(\mathcal{F}_2)\beta^{-1})\rhd a)\\
    =&((\mathcal{F}'_1\mathcal{F}_{1(1)}\xi)\rhd X)^\mathcal{F}
    ((\mathcal{F}'_2\mathcal{F}_{1(2)}S(\mathcal{F}_2)\beta^{-1})\rhd a)\\
    =&((\mathcal{F}_{1}\xi)\rhd X)^\mathcal{F}
    ((\mathcal{F}'_1\mathcal{F}_{2(1)}
    S(\mathcal{F}'_2\mathcal{F}_{2(2)})\beta^{-1})\rhd a)\\
    =&(\xi\rhd X)^\mathcal{F}
    ((\beta\beta^{-1})\rhd a)\\
    =&(\xi\rhd X)^\mathcal{F}(a).
\end{align*}
\end{allowdisplaybreaks}
Furthermore, the left and right $\mathcal{A}$-module actions are respected
since
\begin{align*}
    (a\cdot_\mathcal{F}X)^\mathcal{F}(b)
    =&((\mathcal{F}_{1(1)}^{-1}\mathcal{F}_1^{'-1})\rhd a)
    \cdot((\mathcal{F}_{1(2)}^{-1}\mathcal{F}_2^{'-1})\rhd X)
    (\mathcal{F}_2^{-1}\rhd b)\\
    =&(\mathcal{F}_{1}^{-1}\rhd a)
    \cdot((\mathcal{F}_{2(1)}^{-1}\mathcal{F}_1^{'-1})\rhd X)
    ((\mathcal{F}_{2(2)}^{-1}\mathcal{F}_2^{'-1})\rhd b)\\
    =&a\cdot_\mathcal{F}(X^\mathcal{F}(b))\\
    =&(a\cdot_\mathcal{F}X^\mathcal{F})(b)
\end{align*}
and
\begin{allowdisplaybreaks}
\begin{align*}
    (X\cdot_\mathcal{F}a)^\mathcal{F}(b)
    =&(((\mathcal{F}_{1(1)}^{-1}\mathcal{F}_1^{'-1})\rhd X)
    \cdot((\mathcal{F}_{1(2)}^{-1}\mathcal{F}_2^{'-1})\rhd a))
    (\mathcal{F}_2^{-1}\rhd b)\\
    =&(\mathcal{F}_{1}^{-1}\rhd X)
    ((\mathcal{R}_1^{-1}\mathcal{F}_{2(2)}^{-1}\mathcal{F}_2^{'-1})\rhd b)
    \cdot((\mathcal{R}_2^{-1}\mathcal{F}_{2(1)}^{-1}\mathcal{F}_1^{'-1})\rhd a)\\
    =&(\mathcal{F}_{1}^{-1}\rhd X)
    ((\mathcal{F}_{2(1)}^{-1}\mathcal{R}_1^{-1}\mathcal{F}_2^{'-1})\rhd b)
    \cdot((\mathcal{F}_{2(2)}^{-1}\mathcal{R}_2^{-1}\mathcal{F}_1^{'-1})\rhd a)\\
    =&(\mathcal{F}_{1}^{-1}\rhd X)
    ((\mathcal{F}_{2(1)}^{-1}\mathcal{F}_1^{'-1}\mathcal{R}_{\mathcal{F}1}^{-1})
    \rhd b)
    \cdot((\mathcal{F}_{2(2)}^{-1}\mathcal{F}_2^{'-1}\mathcal{R}_{\mathcal{F}2}^{-1})
    \rhd a)\\
    =&((\mathcal{F}_{1(1)}^{-1}\mathcal{F}_1^{'-1})\rhd X)
    ((\mathcal{F}_{1(2)}^{-1}\mathcal{F}_2^{'-1}\mathcal{R}_{\mathcal{F}1}^{-1})
    \rhd b)
    \cdot((\mathcal{F}_{2}^{-1}\mathcal{R}_{\mathcal{F}2}^{-1})\rhd a)\\
    =&X^\mathcal{F}(\mathcal{R}_{\mathcal{F}1}^{-1}\rhd b)
    \cdot_\mathcal{F}(\mathcal{R}_{\mathcal{F}2}^{-1}\rhd a)\\
    =&(X^\mathcal{F}\cdot_\mathcal{F}a)(b)
\end{align*}
\end{allowdisplaybreaks}
hold for all $a,b\in\mathcal{A}$. The inverse homomorphism is given by
$$
\mathrm{Der}_{\mathcal{R}_\mathcal{F}}(\mathcal{A}_\mathcal{F})
\ni\Xi\mapsto\Xi^{\mathcal{F}^{-1}}\in\mathrm{Der}_\mathcal{R}(\mathcal{A}),
$$
where $\Xi^{\mathcal{F}^{-1}}(a)=(\mathcal{F}_1\rhd\Xi)(\mathcal{F}_2\rhd a)$
for all $a\in\mathcal{A}$.
\end{proof}
This proves that we can work with $\mathfrak{X}^1_\mathcal{R}
(\mathcal{A})_\mathcal{F}$ instead of $\mathfrak{X}^1_{\mathcal{R}_\mathcal{F}}
(\mathcal{A}_\mathcal{F})$. Applying the Drinfel'd functor
on the wedge product
$$
\wedge_\mathcal{R}\colon\mathfrak{X}^1_\mathcal{R}(\mathcal{A})\otimes_\mathcal{A}
\mathfrak{X}^1_\mathcal{R}(\mathcal{A})\rightarrow
\mathfrak{X}^2_\mathcal{R}(\mathcal{A})
$$
gives
$$
\mathrm{Drin}_\mathcal{F}(\wedge_\mathcal{R})
\colon\big(\mathfrak{X}^1_\mathcal{R}(\mathcal{A})\otimes_\mathcal{A}
\mathfrak{X}^1_\mathcal{R}(\mathcal{A})\big)_\mathcal{F}\rightarrow
\mathfrak{X}^2_\mathcal{R}(\mathcal{A})_\mathcal{F}.
$$
So if we want to interpret the image of the wedge product as an actual
product on $\mathfrak{X}^\bullet_\mathcal{R}(\mathcal{A})_\mathcal{F}$ we
have to make use of the natural transformation $\varphi$. We define
$$
\wedge_\mathcal{F}
=\mathrm{Drin}(\wedge)\circ\varphi_{\mathfrak{X}^1_\mathcal{R}(\mathcal{A}),
\mathfrak{X}^1_\mathcal{R}(\mathcal{A})}\colon
\mathfrak{X}^1_\mathcal{R}(\mathcal{A})_\mathcal{F}
\otimes_{\mathcal{A}_\mathcal{F}}
\mathfrak{X}^1_\mathcal{R}(\mathcal{A})_\mathcal{F}
\rightarrow\mathfrak{X}^2_\mathcal{R}(\mathcal{A})_\mathcal{F}
$$
and call it the \textit{twisted wedge product}.
Furthermore we extend the isomorphism
$\mathfrak{X}^1_\mathcal{R}(\mathcal{A})_\mathcal{F}\rightarrow
\mathfrak{X}^1_{\mathcal{R}_\mathcal{F}}(\mathcal{A}_\mathcal{F})$
to higher wedge powers as a homomorphism of the twisted wedge product, i.e.
$$
(X\wedge_\mathcal{F}Y)^\mathcal{F}
=X^\mathcal{F}\wedge_{\mathcal{R}_\mathcal{F}}Y^\mathcal{F}
$$
for all $X,Y\in\mathfrak{X}^\bullet_\mathcal{R}(\mathcal{A})_\mathcal{F}$,
where $\wedge_\mathcal{F}=\mathrm{Drin}_\mathcal{F}(\wedge_\mathcal{R})
\circ\varphi_{\mathfrak{X}^\bullet_\mathcal{R}(\mathcal{A}),
\mathfrak{X}^\bullet_\mathcal{R}(\mathcal{A})}$.
By Lemma~\ref{lemma23}
this is well-defined. Inductively this leads to an isomorphism 
$
\mathfrak{X}^\bullet_\mathcal{R}(\mathcal{A})_\mathcal{F}\rightarrow
\mathfrak{X}^\bullet_{\mathcal{R}_\mathcal{F}}(\mathcal{A}_\mathcal{F})
$
of $H_\mathcal{F}$-equivariant braided symmetric
$\mathcal{A}_\mathcal{F}$-bimodules.
Also the \textit{twisted Schouten-Nijenhuis bracket}
$$
\llbracket\cdot,\cdot\rrbracket_\mathcal{F}\colon
\mathrm{Drin}_\mathcal{F}(\llbracket\cdot,\cdot\rrbracket_\mathcal{R})\circ
\varphi_{\mathfrak{X}^\bullet_\mathcal{R}(\mathcal{A}),
\mathfrak{X}^\bullet_\mathcal{R}(\mathcal{A})}
\colon\mathfrak{X}^\bullet_\mathcal{R}(\mathcal{A})_\mathcal{F}
\otimes_{\mathcal{A}_\mathcal{F}}
\mathfrak{X}^\bullet_\mathcal{R}(\mathcal{A})_\mathcal{F}
\rightarrow\mathfrak{X}^\bullet_\mathcal{R}(\mathcal{A})_\mathcal{F}
$$
can be defined.
On elements
$X,Y\in\mathfrak{X}^\bullet_\mathcal{R}(\mathcal{A})_\mathcal{F}$ the twisted
structures read
$$
X\wedge_\mathcal{F}Y
=(\mathcal{F}_1^{-1}\rhd X)\wedge_\mathcal{R}(\mathcal{F}_2^{-1}\rhd Y)
$$
and
$$
\llbracket X,Y\rrbracket_\mathcal{F}
=\llbracket\mathcal{F}_1^{-1}\rhd X,\mathcal{F}_2^{-1}\rhd Y\rrbracket_\mathcal{R},
$$
respectively.
\begin{proposition}\label{prop13}
This assignment
\begin{equation}
    {}^\mathcal{F}\colon
    (\mathfrak{X}^\bullet_\mathcal{R}
    (\mathcal{A})_\mathcal{F},\wedge_\mathcal{F},
    \llbracket\cdot,\cdot\rrbracket_\mathcal{F})
    \rightarrow
    (\mathfrak{X}^\bullet_{\mathcal{R}_\mathcal{F}}(\mathcal{A}_\mathcal{F}),
    \wedge_{\mathcal{R}_\mathcal{F}},
    \llbracket\cdot,\cdot\rrbracket_{\mathcal{R}_\mathcal{F}})
\end{equation}
is an isomorphism of braided Gerstenhaber algebras.
\end{proposition}
\begin{proof}
First note that
\begin{align*}
    (X\cdot_\mathcal{F}Y)^\mathcal{F}(a)
    =&((\mathcal{F}_{1(1)}^{-1}\mathcal{F}_1^{'-1})\rhd X)
    ((\mathcal{F}_{1(2)}^{-1}\mathcal{F}_2^{'-1})\rhd Y)
    (\mathcal{F}_2^{-1}\rhd a)\\
    =&(\mathcal{F}_{1}^{-1}\rhd X)
    ((\mathcal{F}_{2(1)}^{-1}\mathcal{F}_1^{'-1})\rhd Y)
    ((\mathcal{F}_{2(2)}^{-1}\mathcal{F}_2^{'-1})\rhd a)\\
    =&(X^\mathcal{F}\cdot_{\mathcal{R}_\mathcal{F}}Y^\mathcal{F})(a)
\end{align*}
for all $X,Y\in\mathfrak{X}^1_\mathcal{R}(\mathcal{A})_\mathcal{F}$ and
$a\in\mathcal{A}$. Then
\begin{allowdisplaybreaks}
\begin{align*}
    ([X,Y]_\mathcal{F})^\mathcal{F}
    =&([\mathcal{F}_1^{-1}\rhd X,
    \mathcal{F}_2^{-1}\rhd Y]_\mathcal{R})^\mathcal{F}\\
    =&((\mathcal{F}_1^{-1}\rhd X)
    \cdot_\mathcal{R}(\mathcal{F}_2^{-1}\rhd Y))^\mathcal{F}\\
    &-(((\mathcal{R}_1^{-1}\mathcal{F}_2^{-1})\rhd Y)
    \cdot_\mathcal{R}((\mathcal{R}_2^{-1}\mathcal{F}_1^{-1})\rhd X))^\mathcal{F}\\
    =&(X\cdot_\mathcal{F}Y)^\mathcal{F}
    -((\mathcal{R}_{\mathcal{F}1}^{-1}\rhd Y)
    \cdot_\mathcal{F}(\mathcal{R}_{\mathcal{F}2}^{-1}\rhd X))^\mathcal{F}\\
    =&X^\mathcal{F}\cdot_{\mathcal{R}_\mathcal{F}}Y^\mathcal{F}
    -(\mathcal{R}_{\mathcal{F}1}^{-1}\rhd Y)^\mathcal{F}
    \cdot_{\mathcal{R}_\mathcal{F}}
    (\mathcal{R}_{\mathcal{F}1}^{-1}\rhd Y)^\mathcal{F}\\
    =&X^\mathcal{F}\cdot_{\mathcal{R}_\mathcal{F}}Y^\mathcal{F}
    -(\mathcal{R}_{\mathcal{F}1}^{-1}\rhd Y^\mathcal{F})
    \cdot_{\mathcal{R}_\mathcal{F}}
    (\mathcal{R}_{\mathcal{F}1}^{-1}\rhd Y^\mathcal{F})\\
    =&[X^\mathcal{F},Y^\mathcal{F}]_{\mathcal{R}_\mathcal{F}}.
\end{align*}
\end{allowdisplaybreaks}
Using the defining formula (see Section~\ref{Sec3.2}) of the braided
Schouten-Nijenhuis bracket, this implies that
$$
(\llbracket X,Y\rrbracket_\mathcal{F})^\mathcal{F}
=\llbracket X^\mathcal{F},Y^\mathcal{F}\rrbracket_{\mathcal{R}_\mathcal{F}}
$$
for all $X,Y\in\mathfrak{X}^\bullet_\mathcal{R}(\mathcal{A})_\mathcal{F}$.
\end{proof}
Similarly we define an isomorphism
${}^\mathcal{F}\colon\Omega^\bullet_\mathcal{R}(\mathcal{A})_\mathcal{F}\rightarrow
\Omega^\bullet_{\mathcal{R}_\mathcal{F}}(\mathcal{A}_\mathcal{F})$ of
$H_\mathcal{F}$-equivariant braided symmetric $\mathcal{A}_\mathcal{F}$-bimodules.
On elements $\omega\in\Omega^1_\mathcal{R}(\mathcal{A})_\mathcal{F}$ it
reads
$$
\omega^\mathcal{F}(X^\mathcal{F})
=(\mathcal{F}_1^{-1}\rhd\omega)(\mathcal{F}_2^{-1}\rhd X)
$$
for all $X\in\mathfrak{X}^1_\mathcal{R}(\mathcal{A})_\mathcal{F}$. In fact
$\omega^\mathcal{F}$ is an element of
$\Omega^1_{\mathcal{R}_\mathcal{F}}(\mathcal{A}_\mathcal{F})$ since
\begin{align*}
    \omega^\mathcal{F}(X^\mathcal{F}\cdot_{\mathcal{R}_\mathcal{F}}a)
    =&\omega^\mathcal{F}((X\cdot_{\mathcal{F}}a)^\mathcal{F})\\
    =&(\mathcal{F}_1^{-1}\rhd\omega)
    (((\mathcal{F}_{2(1)}^{-1}\mathcal{F}_1^{'-1})\rhd X)
    \cdot_\mathcal{R}((\mathcal{F}_{2(2)}^{-1}\mathcal{F}_2^{'-1})\rhd a))\\
    =&((\mathcal{F}_{1(1)}^{-1}\mathcal{F}_1^{'-1})\rhd\omega)
    ((\mathcal{F}_{1(2)}^{-1}\mathcal{F}_2^{'-1})\rhd X)
    \cdot(\mathcal{F}_{2}^{-1}\rhd a)\\
    =&(\mathcal{F}_1^{'-1}\rhd\omega)
    (\mathcal{F}_2^{'-1}\rhd X)
    \cdot_\mathcal{F}a\\
    =&\omega^\mathcal{F}(X^\mathcal{F})\cdot_\mathcal{F}a
\end{align*}
for all $a\in\mathcal{A}$. Furthermore we define
$$
(\omega\wedge_\mathcal{F}\eta)^\mathcal{F}
=\omega^\mathcal{F}\wedge_{\mathcal{R}_\mathcal{F}}\eta^\mathcal{F}
$$
for all $\omega,\eta\in\Omega^\bullet_\mathcal{R}(\mathcal{A})_\mathcal{F}$.
Applying the Drinfel'd functor and the natural transformation $\varphi$
on $\mathscr{L}^\mathcal{R}$ and $\mathrm{i}^\mathcal{R}$ leads to
\begin{align*}
    \mathscr{L}^\mathcal{F}&\colon
    \mathfrak{X}^\bullet_\mathcal{R}(\mathcal{A})_\mathcal{F}
    \otimes_\mathcal{F}\Omega^\bullet_\mathcal{R}(\mathcal{A})_\mathcal{F}
    \rightarrow\Omega^\bullet_\mathcal{R}(\mathcal{A})_\mathcal{F},~\\
    \mathrm{i}^\mathcal{F}&\colon
    \mathfrak{X}^\bullet_\mathcal{R}(\mathcal{A})_\mathcal{F}
    \otimes_\mathcal{F}\Omega^\bullet_\mathcal{R}(\mathcal{A})_\mathcal{F}
\rightarrow\Omega^\bullet_\mathcal{R}(\mathcal{A})_\mathcal{F},
\end{align*}
while the de Rham differential becomes
$\mathrm{d}\colon\Omega^\bullet_\mathcal{R}(\mathcal{A})_\mathcal{F}
\rightarrow\Omega^{\bullet+1}_\mathcal{R}(\mathcal{A})_\mathcal{F}$ after utilizing
the Drinfel'd functor. On elements $X\in\mathfrak{X}^\bullet_\mathcal{R}
(\mathcal{A})_\mathcal{F}$ and $\omega\in\Omega^\bullet_\mathcal{R}
(\mathcal{A})_\mathcal{F}$ the \textit{twisted Lie derivative} and
\textit{twisted insertion} read
$$
\mathscr{L}^\mathcal{F}_X\omega
=\mathscr{L}^\mathcal{R}_{\mathcal{F}_1^{-1}\rhd X}(\mathcal{F}_2^{-1}\rhd\omega)
\text{ and }
\mathrm{i}^\mathcal{F}_X\omega
=\mathrm{i}^\mathcal{R}_{\mathcal{F}_1^{-1}\rhd X}(\mathcal{F}_2^{-1}\rhd\omega),
$$
while the de Rham differential remains undeformed. We refer to
$$
(\Omega^\bullet_\mathcal{R}(\mathcal{A})_\mathcal{F},\wedge_\mathcal{F},
\mathscr{L}^\mathcal{F},\mathrm{i}^\mathcal{F},\mathrm{d})
\text{ and }
(\mathfrak{X}^\bullet_\mathcal{R}
(\mathcal{A})_\mathcal{F},\wedge_\mathcal{F},
\llbracket\cdot,\cdot\rrbracket_\mathcal{F})
$$
as the \textit{twisted Cartan calculus} (with respect to $\mathcal{F}$ and 
$\mathcal{R}$).
\begin{theorem}
The twisted Cartan calculus with respect to $\mathcal{R}$ and
$\mathcal{F}$ is isomorphic to the braided Cartan calculus on
$\mathcal{A}_\mathcal{F}$ with respect
to $\mathcal{R}_\mathcal{F}$ via the isomorphism ${}^\mathcal{F}$.
In particular
\begin{align*}
    (\mathscr{L}^\mathcal{F}_X\omega)^\mathcal{F}
    =&\mathscr{L}^{\mathcal{R}_\mathcal{F}}_{X^\mathcal{F}}\omega^\mathcal{F},\\
    (\mathrm{i}^\mathcal{F}_X\omega)^\mathcal{F}
    =&\mathrm{i}^{\mathcal{R}_\mathcal{F}}_{X^\mathcal{F}}\omega^\mathcal{F},\\
    (\mathrm{d}\omega)^\mathcal{F}
    =&\mathrm{d}\omega^\mathcal{F}
\end{align*}
for all $X\in\mathfrak{X}^\bullet_\mathcal{R}(\mathcal{A})_\mathcal{F}$
and $\omega\in\Omega^\bullet_\mathcal{R}(\mathcal{A})_\mathcal{F}$.
\end{theorem}
\begin{proof}
For $a\in\mathcal{A}$, $X,Y\in\mathfrak{X}^1_\mathcal{R}(\mathcal{A})_\mathcal{F}$
and $\omega\in\Omega^1_\mathcal{R}(\mathcal{A})_\mathcal{F}$ we obtain
\begin{allowdisplaybreaks}
\begin{align*}
    (\mathrm{i}^\mathcal{F}_X\omega)^\mathcal{F}
    =&(\mathrm{i}^\mathcal{R}_{\mathcal{F}_1^{-1}\rhd X}
    (\mathcal{F}_2^{-1}\rhd\omega))^\mathcal{F}\\
    =&(((\mathcal{R}_1^{-1}\mathcal{F}_2^{-1})\rhd\omega)
    ((\mathcal{R}_2^{-1}\mathcal{F}_1^{-1})\rhd X))^\mathcal{F}\\
    =&((\mathcal{F}_1^{-1}\mathcal{R}_{\mathcal{F}1}^{-1})\rhd\omega)
    ((\mathcal{F}_2^{-1}\mathcal{R}_{\mathcal{F}2}^{-1})\rhd X)\\
    =&(\mathcal{R}_{\mathcal{F}1}^{-1}\rhd\omega)^\mathcal{F}
    ((\mathcal{R}_{\mathcal{F}2}^{-1}\rhd X)^\mathcal{F})\\
    =&(\mathcal{R}_{\mathcal{F}1}^{-1}\rhd\omega^\mathcal{F})
    (\mathcal{R}_{\mathcal{F}2}^{-1}\rhd X^\mathcal{F})\\
    =&\mathrm{i}^{\mathcal{R}_\mathcal{F}}_{X^\mathcal{F}}\omega^\mathcal{F},
\end{align*}
\end{allowdisplaybreaks}
since ${}^\mathcal{F}$ is an isomorphism of $H_\mathcal{F}$-equivariant
braided symmetric $\mathcal{A}_\mathcal{F}$-bimodules. Similarly
\begin{allowdisplaybreaks}
\begin{align*}
    (\mathrm{d}a)^\mathcal{F}(X^\mathcal{F})
    =&(\mathcal{F}_1^{-1}\rhd(\mathrm{d}a))(\mathcal{F}_2^{-1}\rhd X)\\
    =&(\mathrm{d}(\mathcal{F}_1^{-1}\rhd a))(\mathcal{F}_2^{-1}\rhd X)\\
    =&((\mathcal{R}_1^{-1}\mathcal{F}_2^{-1})\rhd X)
    ((\mathcal{R}_2^{-1}\mathcal{F}_1^{-1})\rhd a)\\
    =&((\mathcal{F}_1^{-1}\mathcal{R}_{\mathcal{F}_1}^{-1})\rhd X)
    ((\mathcal{F}_2^{-1}\mathcal{R}_{\mathcal{F}2}^{-1})\rhd a)\\
    =&(\mathcal{R}_{\mathcal{F}1}^{-1}\rhd X)^\mathcal{F}
    (\mathcal{R}_{\mathcal{F}2}^{-1}\rhd a)\\
    =&(\mathcal{R}_{\mathcal{F}1}^{-1}\rhd X^\mathcal{F})
    (\mathcal{R}_{\mathcal{F}2}^{-1}\rhd a)\\
    =&(\mathrm{d}a)(X^\mathcal{F}),
\end{align*}
\end{allowdisplaybreaks}
follows, and since ${}^\mathcal{F}$ respects the braided commutator we obtain
\begin{allowdisplaybreaks}
\begin{align*}
    \mathrm{d}\omega^\mathcal{F}
    (X^\mathcal{F}\wedge_{\mathcal{R}_\mathcal{F}}Y^\mathcal{F})
    =&(\mathcal{R}_{\mathcal{F}1}^{-1}\rhd X^\mathcal{F})
    ((\mathcal{R}_{\mathcal{F}2}^{-1}\rhd\omega^\mathcal{F})(Y^\mathcal{F}))\\
    &-(\mathcal{R}_{\mathcal{F}1}^{-1}\rhd Y^\mathcal{F})
    (\mathcal{R}_{\mathcal{F}2}^{-1}\rhd(\omega^\mathcal{F}(X^\mathcal{F})))\\
    &-\omega^\mathcal{F}([X^\mathcal{F},Y^\mathcal{F}]_{\mathcal{R}_\mathcal{F}})\\
    =&((\mathcal{F}_1^{-1}\mathcal{R}_{\mathcal{F}1}^{-1})\rhd X)
    (((\mathcal{F}_{2(1)}^{-1}\mathcal{F}_1^{'-1}\mathcal{R}_{\mathcal{F}2}^{-1})
    \rhd\omega)((\mathcal{F}_{2(2)}^{-1}\mathcal{F}_2^{'-1})\rhd Y))\\
    &-((\mathcal{F}_1^{-1}\mathcal{R}_{\mathcal{F}1}^{-1})\rhd Y)
    ((\mathcal{F}_{2}^{-1}\mathcal{R}_{\mathcal{F}2}^{-1})
    \rhd((\mathcal{F}_1^{'-1}\rhd\omega)(\mathcal{F}_2^{'-1}\rhd X)))\\
    &-\omega^\mathcal{F}(([X,Y]_{\mathcal{F}})^\mathcal{F})\\
    =&((\mathcal{F}_{1(1)}^{-1}\mathcal{F}_1^{'-1}\mathcal{R}_{\mathcal{F}1}^{-1})\rhd X)
    (((\mathcal{F}_{1(2)}^{-1}\mathcal{F}_2^{'-1}\mathcal{R}_{\mathcal{F}2}^{-1})
    \rhd\omega)(\mathcal{F}_{2}^{-1}\rhd Y))\\
    &-((\mathcal{R}_1^{-1}\mathcal{F}_2^{-1})\rhd Y)
    ((\mathcal{R}_2^{-1}\mathcal{F}_1^{-1})
    \rhd((\mathcal{F}_1^{'-1}\rhd\omega)(\mathcal{F}_2^{'-1}\rhd X)))\\
    &-(\mathcal{F}_1^{-1}\rhd\omega)(\mathcal{F}_2^{-1}\rhd[X,Y]_{\mathcal{F}})\\
    =&((\mathcal{F}_{1(1)}^{-1}\mathcal{R}_1^{-1}\mathcal{F}_2^{'-1})\rhd X)
    (((\mathcal{F}_{1(2)}^{-1}\mathcal{R}_2^{-1}\mathcal{F}_1^{'-1})
    \rhd\omega)(\mathcal{F}_{2}^{-1}\rhd Y))\\
    &-((\mathcal{R}_1^{-1}\mathcal{F}_2^{-1})\rhd Y)
    (\mathcal{R}_2^{-1}
    \rhd(((\mathcal{F}_{1(1)}^{-1}\mathcal{F}_1^{'-1})\rhd\omega)
    ((\mathcal{F}_{1(2)}^{-1}\mathcal{F}_2^{'-1})\rhd X)))\\
    &-(\mathcal{F}_1^{-1}\rhd\omega)
    ([(\mathcal{F}_{2(1)}^{-1}\mathcal{F}_1^{'-1})\rhd X,
    (\mathcal{F}_{2(2)}^{-1}\mathcal{F}_2^{'-1})\rhd Y]_{\mathcal{R}})\\
    =&((\mathcal{R}_1^{-1}\mathcal{F}_{2(1)}^{-1}\mathcal{F}_1^{'-1})\rhd X)
    (((\mathcal{R}_2^{-1}\mathcal{F}_{1}^{-1})
    \rhd\omega)(\mathcal{F}_{2(2)}^{-1}\mathcal{F}_2^{'-1})\rhd Y))\\
    &-((\mathcal{R}_1^{-1}\mathcal{F}_{2(2)}^{-1}\mathcal{F}_2^{'-1})\rhd Y)
    (\mathcal{R}_2^{-1}
    \rhd((\mathcal{F}_{1}^{-1}\rhd\omega)
    ((\mathcal{F}_{2(1)}^{-1}\mathcal{F}_1^{'-1})\rhd X)))\\
    &-(\mathcal{F}_1^{-1}\rhd\omega)
    ([(\mathcal{F}_{2(1)}^{-1}\mathcal{F}_1^{'-1})\rhd X,
    (\mathcal{F}_{2(2)}^{-1}\mathcal{F}_2^{'-1})\rhd Y]_{\mathcal{R}})\\
    =&(\mathrm{d}\omega)^\mathcal{F}((X\wedge_\mathcal{F}Y)^\mathcal{F})\\
    =&(\mathrm{d}\omega)^\mathcal{F}
    (X^\mathcal{F}\wedge_{\mathcal{R}_\mathcal{F}}Y^\mathcal{F}).
\end{align*}
\end{allowdisplaybreaks}
Moreover, since $\mathrm{d}$ is equivariant
\begin{allowdisplaybreaks}
\begin{align*}
    \mathscr{L}^{\mathcal{R}_\mathcal{F}}_{X^\mathcal{F}}\omega^\mathcal{F}
    =&[\mathrm{i}^{\mathcal{R}_\mathcal{F}}_{X^\mathcal{F}},
    \mathrm{d}]_{\mathcal{R}_\mathcal{F}}\omega^\mathcal{F}\\
    =&\mathrm{i}^{\mathcal{R}_\mathcal{F}}_{X^\mathcal{F}}\mathrm{d}\omega^\mathcal{F}
    -\mathrm{d}
    \mathrm{i}^{\mathcal{R}_\mathcal{F}}_{X^\mathcal{F}}\omega^\mathcal{F}\\
    =&\mathrm{i}^{\mathcal{R}_\mathcal{F}}_{X^\mathcal{F}}(\mathrm{d}\omega)^\mathcal{F}
    -\mathrm{d}
    (\mathrm{i}^{\mathcal{F}}_{X}\omega)^\mathcal{F}\\
    =&(\mathrm{i}^{\mathcal{F}}_{X}\mathrm{d}\omega)^\mathcal{F}
    -(\mathrm{d}
    \mathrm{i}^{\mathcal{F}}_{X}\omega)^\mathcal{F}\\
    =&([\mathrm{i}^\mathcal{F}_X,\mathrm{d}]_\mathcal{R}\omega)^\mathcal{F}\\
    =&(\mathscr{L}^\mathcal{F}_X\omega)^\mathcal{F},
\end{align*}
\end{allowdisplaybreaks}
follows. In degree zero
there is nothing to prove and higher degrees are generated by the lowest two degrees.
Since ${}^\mathcal{F}$ respects the wedge product this concludes the proof of
the theorem.
\end{proof}
Let $\mathcal{M}$ be an $H$-equivariant braided symmetric $\mathcal{A}$-bimodule and
$\nabla^\mathcal{R}\colon\mathfrak{X}^1_\mathcal{R}(\mathcal{A})\otimes
\mathcal{M}\rightarrow\mathcal{M}$ an equivariant covariant derivative with respect
to $\mathcal{R}$. Define the \textit{twisted covariant derivative}
as
$$
\nabla^\mathcal{F}=
\mathrm{Drin}_\mathcal{F}(\nabla^\mathcal{R})\circ
\varphi_{\mathfrak{X}^1_\mathcal{R}(\mathcal{A}),
\mathcal{M}}\colon\mathfrak{X}^1_\mathcal{R}(\mathcal{A})_\mathcal{F}
\otimes_\mathcal{F}\mathcal{M}_\mathcal{F}\rightarrow\mathcal{M}_\mathcal{F}.
$$
On elements $X\in\mathfrak{X}^1_\mathcal{R}(\mathcal{A})_\mathcal{F}$ and
$s\in\mathcal{M}_\mathcal{F}$ this reads
$$
\nabla^\mathcal{F}_Xs
=\nabla^\mathcal{R}_{\mathcal{F}_1^{-1}\rhd X}(\mathcal{F}_2^{-1}\rhd s).
$$
\begin{lemma}
$\nabla^\mathcal{F}$ satisfies
$\xi\rhd(\nabla^\mathcal{F}_Xs)=\nabla^\mathcal{F}_{\xi_{(1)}\rhd X}
(\xi_{(2)}\rhd s)$,
$\nabla^\mathcal{F}_{a\cdot_\mathcal{F}X}s
=a\cdot_\mathcal{F}\nabla^\mathcal{F}_Xs$ and
$$
\nabla^\mathcal{F}_X(a\cdot_\mathcal{F}s)
=(\mathscr{L}^\mathcal{F}_Xa)\cdot_\mathcal{F}s
+(\mathcal{R}_{\mathcal{F}1}^{-1}\rhd a)
\cdot_\mathcal{F}(\nabla^\mathcal{F}_{\mathcal{R}_{\mathcal{F}2}^{-1}\rhd X}s)
$$
for all $a\in\mathcal{A}$, $X\in\mathfrak{X}^1_\mathcal{R}(\mathcal{A})_\mathcal{F}$
and $s\in\mathcal{M}_\mathcal{F}$.
\end{lemma}
\begin{proof}
Let $\xi\in H$, $a\in\mathcal{A}$, $X\in\mathfrak{X}^1_\mathcal{R}
(\mathcal{A})_\mathcal{F}$ and $s\in\mathcal{M}_\mathcal{F}$. Then
\begin{align*}
    \xi\rhd(\nabla^\mathcal{F}_Xs)
    =\nabla^\mathcal{R}_{(\xi_{(1)}\mathcal{F}_1^{-1})\rhd X}
    ((\xi_{(2)}\mathcal{F}_2^{-1})\rhd s)
    =\nabla^\mathcal{F}_{\xi_{\widehat{(1)}}\rhd X}(\xi_{\widehat{(2)}}\rhd s)
\end{align*}
shows that $\nabla^\mathcal{F}$ is $H_\mathcal{F}$-equivariant, while also
\begin{align*}
    \nabla^\mathcal{F}_{a\cdot_\mathcal{F}X}s
    =&((\mathcal{F}_{1(1)}^{-1}\mathcal{F}_1^{'-1})\rhd a)
    \cdot(\nabla^\mathcal{R}_{(\mathcal{F}_{1(2)}^{-1}\mathcal{F}_2^{'-1})\rhd X}
    (\mathcal{F}_2^{-1}\rhd s))\\
    =&(\mathcal{F}_{1}^{-1}\rhd a)
    \cdot(\nabla^\mathcal{R}_{(\mathcal{F}_{2(1)}^{-1}\mathcal{F}_1^{'-1})\rhd X}
    ((\mathcal{F}_{2(2)}^{-1}\mathcal{F}_2^{'-1})\rhd s))\\
    =&(\mathcal{F}_{1}^{-1}\rhd a)
    \cdot(\mathcal{F}_2^{-1}\rhd(
    \nabla^\mathcal{R}_{\mathcal{F}_1^{'-1}\rhd X}
    (\mathcal{F}_2^{'-1}\rhd s)))\\
    =&a\cdot_\mathcal{F}(\nabla^\mathcal{F}_Xs)
\end{align*}
and
\begin{allowdisplaybreaks}
\begin{align*}
    \nabla^\mathcal{F}_X(a\cdot_\mathcal{F}s)
    =&\nabla^\mathcal{R}_{\mathcal{F}_1^{-1}\rhd X}(
    ((\mathcal{F}_{2(1)}^{-1}\mathcal{F}_1^{'-1})\rhd a)
    \cdot((\mathcal{F}_{2(1)}^{-1}\mathcal{F}_1^{'-1})\rhd a))\\
    =&(\mathscr{L}^\mathcal{R}_{\mathcal{F}_1^{-1}\rhd X}
    ((\mathcal{F}_{2(1)}^{-1}\mathcal{F}_1^{'-1})\rhd a))
    \cdot((\mathcal{F}_{2(2)}^{-1}\mathcal{F}_2^{'-1})\rhd s)\\
    &+((\mathcal{R}_1^{-1}\mathcal{F}_{2(1)}^{-1}\mathcal{F}_1^{'-1})\rhd a)
    \cdot(\nabla^\mathcal{R}_{(\mathcal{R}_2^{-1}\mathcal{F}_1^{-1})\rhd X}
    ((\mathcal{F}_{2(2)}^{-1}\mathcal{F}_2^{'-1})\rhd s))\\
    =&(\mathcal{F}_1^{-1}\rhd(
    \mathscr{L}^\mathcal{R}_{\mathcal{F}_1^{'-1}\rhd X}
    (\mathcal{F}_2^{'-1}\rhd a)))
    \cdot(\mathcal{F}_{2}^{-1}\rhd s)\\
    &+((\mathcal{F}_{2(1)}^{-1}\mathcal{R}_1^{-1}\mathcal{F}_2^{'-1})\rhd a)
    \cdot(\nabla^\mathcal{R}_{(\mathcal{F}_{2(2)}^{-1}\mathcal{R}_2^{-1}
    \mathcal{F}_1^{'-1})\rhd X}
    (\mathcal{F}_{2}^{-1}\rhd s))\\
    =&(\mathscr{L}^\mathcal{F}_Xa)\cdot_\mathcal{F}s
    +((\mathcal{F}_{1(1)}^{-1}\mathcal{F}_1^{-1}\mathcal{R}_{\mathcal{F}1}^{-1})
    \rhd a)\cdot
    (\nabla^\mathcal{R}_{(\mathcal{F}_{1(2)}^{-1}
    \mathcal{F}_2^{-1}\mathcal{R}_{\mathcal{F}2}^{-1})
    \rhd X}(\mathcal{F}_2^{-1}\rhd s))\\
    =&(\mathscr{L}^\mathcal{F}_Xa)\cdot_\mathcal{F}s
    +(\mathcal{R}_{\mathcal{F}1}^{-1}\rhd a)\cdot_\mathcal{F}
    (\nabla^\mathcal{F}_{\mathcal{R}_{\mathcal{F}2}\rhd X}s)
\end{align*}
\end{allowdisplaybreaks}
hold.
\end{proof}
Note however that strictly speaking $\nabla^\mathcal{F}$ is
\textit{not} an equivariant covariant derivative with respect to
$\mathcal{R}_\mathcal{F}$, since it is a map
$\mathfrak{X}^1_\mathcal{R}(\mathcal{A})_\mathcal{F}
\otimes_\mathcal{F}\mathcal{M}_\mathcal{F}\rightarrow\mathcal{M}_\mathcal{F}$.
Using the isomorphism $\mathfrak{X}^1_\mathcal{R}(\mathcal{A})_\mathcal{F}
\rightarrow\mathfrak{X}^1_{\mathcal{R}_\mathcal{F}}(\mathcal{A}_\mathcal{F})$
we are able to view $\nabla^\mathcal{F}$ as an equivariant covariant derivative on
$\mathcal{M}_\mathcal{F}$
with respect to $\mathcal{R}_\mathcal{F}$ nevertheless. This is done in the next
proposition.
\begin{proposition}\label{prop12}
Let $\nabla^\mathcal{R}$ be an equivariant covariant derivative with respect to
$\mathcal{R}$ on an object $\mathcal{M}$ in ${}^H_\mathcal{A}\mathcal{M}_\mathcal{A}$.
Then we can define an equivariant covariant derivative
$$
\nabla^{\mathcal{R}_\mathcal{F}}
\colon\mathfrak{X}^1_{\mathcal{R}_\mathcal{F}}(\mathcal{A}_\mathcal{F})
\otimes_{\mathcal{F}}\mathcal{M}_\mathcal{F}
\rightarrow\mathcal{M}_\mathcal{F}
$$
with respect to $\mathcal{R}_\mathcal{F}$ via
$$
\nabla^{\mathcal{R}_\mathcal{F}}_{X^\mathcal{F}}s
=\nabla^\mathcal{F}_Xs
$$
for all $X\in\mathfrak{X}^1_\mathcal{R}(\mathcal{A})_\mathcal{F}$ and
$s\in\mathcal{M}_\mathcal{F}$. If $\nabla^\mathcal{R}$ is an equivariant covariant
derivative with respect to $\mathcal{R}$ on $\mathfrak{X}^1_\mathcal{R}(\mathcal{A})$
there is an equivariant covariant derivative
$$
\nabla^{\mathcal{R}_\mathcal{F}}\colon
\mathfrak{X}^1_{\mathcal{R}_\mathcal{F}}(\mathcal{A}_\mathcal{F})
\otimes_\mathcal{F}\mathfrak{X}^1_{\mathcal{R}_\mathcal{F}}(\mathcal{A}_\mathcal{F})
\rightarrow\mathfrak{X}^1_{\mathcal{R}_\mathcal{F}}(\mathcal{A}_\mathcal{F})
$$
with respect to $\mathcal{R}_\mathcal{F}$ defined by
$$
\nabla^{\mathcal{R}_\mathcal{F}}_{X^\mathcal{F}}Y^\mathcal{F}
=(\nabla^\mathcal{F}_XY)^\mathcal{F}
$$
for all $X,Y\in\mathfrak{X}^1_\mathcal{R}(\mathcal{A})_\mathcal{F}$.
The corresponding extensions
$$
\nabla^{\mathcal{R}_\mathcal{F}}\colon
\mathfrak{X}^1_{\mathcal{R}_\mathcal{F}}(\mathcal{A}_\mathcal{F})
\otimes_\mathcal{F}\mathfrak{X}^\bullet_{\mathcal{R}_\mathcal{F}}(\mathcal{A}_\mathcal{F})
\rightarrow\mathfrak{X}^\bullet_{\mathcal{R}_\mathcal{F}}(\mathcal{A}_\mathcal{F})
$$
and
$$
\tilde{\nabla}^{\mathcal{R}_\mathcal{F}}\colon
\mathfrak{X}^1_{\mathcal{R}_\mathcal{F}}(\mathcal{A}_\mathcal{F})
\otimes_\mathcal{F}\Omega^\bullet_{\mathcal{R}_\mathcal{F}}(\mathcal{A}_\mathcal{F})
\rightarrow\Omega^\bullet_{\mathcal{R}_\mathcal{F}}(\mathcal{A}_\mathcal{F})
$$
of the latter
to braided multivector fields and braided differential forms satisfy
$$
\nabla^{\mathcal{R}_\mathcal{F}}_{X^\mathcal{F}}Y^\mathcal{F}
=(\nabla^\mathcal{F}_XY)^\mathcal{F}
\text{ and }
\tilde{\nabla}^{\mathcal{R}_\mathcal{F}}_{X^\mathcal{F}}\omega^\mathcal{F}
=(\tilde{\nabla}^\mathcal{F}_X\omega)^\mathcal{F}
$$
for all $X\in\mathfrak{X}^1_\mathcal{R}(\mathcal{A})_\mathcal{F}$,
$Y\in\mathfrak{X}^\bullet_\mathcal{R}(\mathcal{A})_\mathcal{F}$ and
$\omega\in\Omega^\bullet_\mathcal{R}(\mathcal{A})_\mathcal{F}$.
\end{proposition}
Let ${\bf g}$ be an equivariant metric on $\mathcal{A}$. Then, the
\textit{twisted metric} ${\bf g}_\mathcal{F}$ is defined by
$$
{\bf g}_\mathcal{F}(X,Y)
={\bf g}(\mathcal{F}_1^{-1}\rhd X,\mathcal{F}_2^{-1}\rhd Y)
$$
for all $X,Y\in\mathfrak{X}^1_\mathcal{R}(\mathcal{A})$. In the next lemma
we prove that it actually deserves this name, i.e. that ${\bf g}_\mathcal{F}$
is an equivariant metric. Furthermore we show that the assignment
$\mathrm{LC}\colon {\bf g}\mapsto\nabla^\mathcal{R}$, attributing to a 
non-degenerate equivariant metric
${\bf g}$ its corresponding equivariant Levi-Civita covariant derivative 
$\nabla^\mathcal{R}$, is respected by the Drinfel'd functor, i.e.
\begin{equation*}
\begin{tikzcd}
{\bf g}
\arrow[mapsto]{r}{\mathrm{LC}}
\arrow[mapsto]{d}[swap]{\mathrm{Drin}_\mathcal{F}}
& \nabla^\mathcal{R}
\arrow[mapsto]{d}{\mathrm{Drin}_\mathcal{F}} \\
{\bf g}_\mathcal{F}
\arrow[mapsto]{r}{\mathrm{LC}}
& \nabla^\mathcal{F}
\end{tikzcd}
\end{equation*}
commutes.
Remark that ${\bf g}_\mathcal{F}$ might not be non-degenerate in general.
Nonetheless, $\nabla^\mathcal{F}$ is well-defined as twist deformation
of $\nabla^\mathcal{R}$. Furthermore it is metric with respect to 
${\bf g}_\mathcal{F}$ and torsion-free, while it might not be the unique
equivariant covariant derivative with those properties. However, any
torsion-free equivariant covariant derivative $\tilde{\nabla}$ on
$\mathcal{A}_\mathcal{F}$ which is metric with respect to ${\bf g}_\mathcal{F}$
is mapped to $\nabla^\mathcal{R}$ via the inverse twist
$\mathrm{Drin}_{\mathcal{F}^{-1}}$. This follows from the uniqueness of
the equivariant Levi-Civita covariant derivative $\nabla^\mathcal{R}$ on
$(\mathcal{A},{\bf g}=({\bf g}_\mathcal{F})_{\mathcal{F}^{-1}})$. In other words,
$\nabla^\mathcal{F}$ is the unique equivariant Levi-Civita covariant derivative
on $(\mathcal{A}_\mathcal{F},{\bf g}_\mathcal{F})$ up to the kernel of
$\mathcal{F}^{-1}\rhd\colon
\mathfrak{X}^1_\mathcal{R}(\mathcal{A})
\otimes\mathfrak{X}^1_\mathcal{R}(\mathcal{A})
\rightarrow
\mathfrak{X}^1_\mathcal{R}(\mathcal{A})
\otimes\mathfrak{X}^1_\mathcal{R}(\mathcal{A})$.
\begin{lemma}\label{lemma24}
For any equivariant metric ${\bf g}$ on $\mathcal{A}$, the twisted metric
${\bf g}_\mathcal{F}$ is an equivariant metric
with respect to $\mathcal{R}_\mathcal{F}$ on $\mathcal{A}_\mathcal{F}$.
Moreover, twisting the equivariant Levi-Civita covariant derivative with
respect to ${\bf g}$ leads to a torsion-free equivariant covariant derivative
$\nabla^\mathcal{F}$, which is metric with respect to ${\bf g}_\mathcal{F}$.
If ${\bf g}_\mathcal{F}$ is non-degenerate, $\nabla^\mathcal{F}$ is the unique
equivariant covariant derivative with those properties.
\end{lemma}
\begin{proof}
Let $X,Y\in\mathfrak{X}_\mathcal{R}^1(\mathcal{A})$. The relation
$$
\mathscr{L}^\mathcal{F}_X({\bf g}_\mathcal{F}(Y,Z))
={\bf g}_\mathcal{F}(\nabla^\mathcal{F}_XY,Z)
+{\bf g}_\mathcal{F}(\mathcal{R}_{\mathcal{F}1}^{-1}\rhd Y,
\nabla^\mathcal{F}_{\mathcal{R}_{\mathcal{F}2}^{-1}\rhd X}Z)
$$
follows from $\mathscr{L}^\mathcal{R}_X{\bf g}(Y,Z)
={\bf g}(\nabla^\mathcal{R}_XY,Z)
+{\bf g}(\mathcal{R}_1^{-1}\rhd Y,\nabla^\mathcal{R}_{\mathcal{R}_2^{-1}\rhd X}Z)$
and the
$H$-equivariance of ${\bf g}$, $\nabla^\mathcal{R}$ and $\mathscr{L}^\mathcal{R}$.
The last statement holds since $\mathrm{Tor}^{\nabla^\mathcal{F}}=0$ if
$\mathrm{Tor}^{\nabla^\mathcal{R}}=0$.
\end{proof}
Notice that Lemma~\ref{lemma24} even holds for metrics which are \textit{not}
equivariant if one assumes $\mathcal{F}$ to consist of affine Killing vector fields
(c.f. \cite{Aschieri2010}~Sec.~6.2).

%% file: chapters/chapter04.tex
We introduce the notion of submanifold algebra on a braided commutative
left $H$-module algebra $\mathcal{A}$ for a triangular Hopf algebra $(H,\mathcal{R})$,
slightly modifying the approach of \cite{Masson1995} (see \cite{Francesco2019}
for a recent discussion on submanifold algebras).
It generalizes the concept of closed embedded submanifolds from differential geometry.
In a nutshell a submanifold algebra
is given by an algebra ideal $\mathcal{C}\subseteq\mathcal{A}$
which is closed under the Hopf algebra action. In particular, the surjective
projection $\mathrm{pr}\colon\mathcal{A}\rightarrow\mathcal{A}/\mathcal{C}$
commutes with the $H$-action. In the course of this chapter we want to make sense
of the following commutative diagram
\begin{equation}\label{DiagramSubmanifolds}
\begin{tikzcd}
\text{Geometry on }\mathcal{A}
\arrow[mapsto]{r}{\mathrm{pr}}
\arrow[mapsto]{d}[swap]{\mathrm{Drin}_\mathcal{F}}
& \text{Geometry on }\mathcal{A}/\mathcal{C}
\arrow[mapsto]{d}{\mathrm{Drin}_\mathcal{F}} \\
\text{Geometry on }\mathcal{A}_\mathcal{F}
\arrow[mapsto]{r}{\mathrm{pr}}
& \text{Geometry on }\mathcal{A}_\mathcal{F}/\mathcal{C}_\mathcal{F}
=\big(\mathcal{A}/\mathcal{C}\big)_\mathcal{F}
\end{tikzcd},
\end{equation}
where $\mathcal{F}$ is a Drinfel'd twist on $H$. This vague picture should be
interpreted in the following way: first, we prove that the geometric data on
$\mathcal{A}/\mathcal{C}$, namely the braided Cartan calculus, equivariant metrics,
covariant derivatives, curvature and torsion,
are gained as projections of
the corresponding objects in $\mathcal{A}$ and secondly, we prove that
this projection commutes with the Drinfel'd functor. In
Section~\ref{Sec5.1} and Section~\ref{Sec5.2} we study the horizontal arrow,
projecting the braided Cartan calculus and equivariant covariant derivatives,
respectively, while the vertical arrow together with the commutativity of
the diagram are examined in Section~\ref{Sec5.3}. Note that we have to
accept two axioms in order to receive well-defined projected equivariant metrics
and covariant derivatives. However, since those assumptions are quite
mild we still obtain Riemannian geometry on smooth submanifolds as a special case
of our theory. Finally, in Section~\ref{Sec4.4}, we give an explicit example
of twist deformation of the $2$-sheet elliptic hyperboloid. Starting from
the commutative algebra of smooth functions with the pointwise product,
the vertical arrows of (\ref{DiagramSubmanifolds}) correspond to a 
quantization and the commutativity of (\ref{DiagramSubmanifolds}) implies that
twist deformation quantization and projection to the quadric surface commute.
A different approach to Riemannian geometry on noncommutative submanifolds, based on
the choice of a finite-dimensional Lie subalgebra $\mathfrak{g}$ of 
$\mathrm{Der}(\mathcal{A})$ and a vector space homomorphism 
$\mathfrak{g}\rightarrow\mathcal{M}$ into a right $\mathcal{A}$-bimodule $\mathcal{M}$,
is considered in \cite{Arnlind2019}.

Fix a triangular Hopf algebra $(H,\mathcal{R})$ and a braided
commutative left $H$-module algebra $\mathcal{A}$ for the rest of
this chapter.

\section{Braided Cartan Calculi on Submanifold Algebras}\label{Sec5.1}

For any algebra ideal $\mathcal{C}$ the coset space $\mathcal{A}/\mathcal{C}$
becomes an algebra with unit and product induced from $\mathcal{A}$. The elements
of $\mathcal{A}/\mathcal{C}$ are equivalence classes of elements in
$\mathcal{A}$, where $a$ and $b$ are in the same equivalence class if and only
if there exists an element $c\in\mathcal{C}$ such that $a=b+c$. Choosing an
arbitrary representative $a\in\mathcal{A}$ we denote the corresponding
equivalence class by $[a]$ or $a+\mathcal{C}$. This constitutes a surjective
projection $\mathrm{pr}\colon\mathcal{A}\rightarrow\mathcal{A}/\mathcal{C}$,
which assigns any element of $\mathcal{A}$ its equivalence class. It is easy
to verify that $\mathrm{pr}$ is an algebra homomorphism. The projection
is injective if and only if $\mathcal{C}=\{0\}$.
\begin{lemma}
Let $\mathcal{C}\subseteq\mathcal{A}$ be an algebra ideal such that
$H\rhd\mathcal{C}\subseteq\mathcal{C}$. Then
$\mathcal{A}/\mathcal{C}$ is a left $H$-module algebra and
braided commutative with respect to the same triangular structure
$\mathcal{R}$.
\end{lemma}
\begin{proof}
From general algebra we know that the coset space $\mathcal{A}/\mathcal{C}$
is an algebra with respect to the induced multiplication if and only if
$\mathcal{C}$ is an ideal. The induced left $H$-action, i.e.
$\xi\rhd\mathrm{pr}(a)=\mathrm{pr}(\xi\rhd a)$, where $a\in\mathcal{A}$ and
$\xi\in H$, is well-defined if and only if $H\rhd\mathcal{C}\subseteq\mathcal{C}$.
For the same reason the action respects the algebra structure and
$\rhd$ descends to a left $H$-module algebra action. Since the braiding is
encoded via the Hopf algebra action $\mathcal{A}/\mathcal{C}$ is braided
commutative.
\end{proof}
It is intuitive that not all braided vector fields on $\mathcal{A}$ can be
projected to a braided vector field on $\mathcal{A}/\mathcal{C}$.
In the end braided vector fields
are braided derivations and in particular endomorphisms of $\mathcal{A}$.
They only descend to an endomorphism of $\mathcal{A}/\mathcal{C}$ if
$\mathcal{C}$ is a subspace of the kernel. It turns out that this is the
only obstruction: let $X\in\mathrm{Der}_\mathcal{R}(\mathcal{A})$ such that
$X(\mathcal{C})\subseteq\mathcal{C}$, then
\begin{equation}\label{eq27}
    \mathrm{pr}(X)(\mathrm{pr}(a))
    =\mathrm{pr}(X(a)),
\end{equation}
where $a\in\mathcal{A}$, defines a braided derivation $\mathrm{pr}(X)$
on $\mathcal{A}/\mathcal{C}$ with respect to $\mathcal{R}$. In fact,
$\mathrm{pr}(X)$ is well-defined exactly because of the condition
$X(\mathcal{C})\subseteq\mathcal{C}$ and it inherits the braided
derivation property from $X$.
\begin{definition}
A braided derivation $X\in\mathrm{Der}_\mathcal{R}(\mathcal{A})$ is said to be
a braided tangent vector field (with respect to $\mathcal{C}$) if
$X(\mathcal{C})\subseteq\mathcal{C}$. We denote the $\Bbbk$-module
of braided tangent vector fields on $\mathcal{A}$ with respect to
$\mathcal{C}$ by $\mathfrak{X}^1_t(\mathcal{A})$. The
$\Bbbk$-submodule of braided tangent vector fields
$X\in\mathfrak{X}^1_t(\mathcal{A})$ satisfying $X(\mathcal{A})
\subseteq\mathcal{C}$ is denoted by $\mathfrak{X}^1_0(\mathcal{A})$.
The corresponding elements are called vanishing vector fields.
\end{definition}
The braided tangent vector fields are closed under the module actions
and the braided commutator. This is discussed in the next lemma. Note
that we are able to define a left $H$-action on the image of
$\mathrm{pr}\colon\mathfrak{X}^1_t(\mathcal{A})\rightarrow
\mathfrak{X}^1_\mathcal{R}(\mathcal{A}/\mathcal{C})$ by
\begin{equation}\label{eq31}
    \xi\rhd\mathrm{pr}(X)=\mathrm{pr}(\xi\rhd X)
\end{equation}
for all $\xi\in H$ and $X\in\mathfrak{X}^1_t(\mathcal{A})$.
With (\ref{eq31}) we have a natural candidate for a left $H$-action
on $\mathfrak{X}^1_\mathcal{R}(\mathcal{A}/\mathcal{C})$. However,
since it is only defined on the image of the projection this sets a 
further condition on the ideal $\mathcal{C}$.
\begin{definition}\label{def05}
If the $\Bbbk$-linear map 
$\mathrm{pr}\colon\mathfrak{X}^1_t(\mathcal{A})\rightarrow
\mathfrak{X}^1_\mathcal{R}(\mathcal{A}/\mathcal{C})$ defined in 
eq.(\ref{eq27}) is surjective, the braided commutative
algebra $\mathcal{A}/\mathcal{C}$ is said to be a submanifold algebra
and $\mathcal{C}$ a submanifold ideal.
\end{definition}
Fix a submanifold ideal $\mathcal{C}$ for the rest of this section.
We want to stress that by our definition this includes the property
$H\rhd\mathcal{C}\subseteq\mathcal{C}$.
\begin{lemma}\label{lemma12}
$\mathfrak{X}^1_t(\mathcal{A})$ is an $H$-equivariant braided symmetric
$\mathcal{A}$-bimodule and a braided Lie algebra, while
$\mathfrak{X}^1_0(\mathcal{A})$ is an
$H$-equivariant braided symmetric
$\mathcal{A}$-sub-bimodule and a braided Lie ideal. Furthermore,
\begin{equation}
    \mathrm{pr}\colon\mathfrak{X}^1_t(\mathcal{A})\rightarrow
    \mathfrak{X}^1_\mathcal{R}(\mathcal{A}/\mathcal{C})
\end{equation}
is a homomorphism of $H$-equivariant braided symmetric
$\mathcal{A}$-bimodules and a homomorphism of braided Lie algebras
with kernel $\mathfrak{X}^1_0(\mathcal{A})$.
\end{lemma}
\begin{proof}
According to Lemma~\ref{lemma08}
$\mathfrak{X}^1_\mathcal{R}(\mathcal{A})$ is an
$H$-equivariant braided symmetric $\mathcal{A}$-bimodule and a
braided Lie algebra with respect to the braided commutator. So
it is sufficient to prove that the module actions and the braided commutator
are closed in $\mathfrak{X}^1_t(\mathcal{A})$. Let $a,b\in\mathcal{A}$,
$X,Y\in\mathfrak{X}^1_t(\mathcal{A})$ and $\xi\in H$. Then
$(a\cdot X)(\mathcal{C})=a\cdot X(\mathcal{C})\subseteq\mathcal{C}$,
\begin{align*}
    (X\cdot a)(\mathcal{C})
    =X(\mathcal{R}_1^{-1}\rhd\mathcal{C})\cdot(\mathcal{R}_2^{-1}\rhd a)
    \subseteq\mathcal{C}
\end{align*}
and 
$
(\xi\rhd X)(\mathcal{C})
=\xi_{(1)}\rhd(X(S(\xi_{(2)})\rhd\mathcal{C}))
\subseteq\mathcal{C},
$
while
\begin{align*}
    [X,Y]_\mathcal{R}(\mathcal{C})
    =X(Y(\mathcal{C}))
    -(\mathcal{R}_1^{-1}\rhd Y)(\mathcal{R}_2^{-1}\rhd X)(\mathcal{C})
    \subseteq\mathcal{C}
\end{align*}
since $\mathcal{C}$ is an ideal and $H\rhd\mathcal{C}\subseteq\mathcal{C}$. This proves that
$\mathfrak{X}^1_t(\mathcal{A})$ is an $H$-equivariant braided symmetric
$\mathcal{A}$-sub-bimodule and a braided Lie subalgebra of
$\mathfrak{X}^1_\mathcal{R}(\mathcal{A})$.
If $X\in\mathfrak{X}^1_0(\mathcal{A})$ and
$Y\in\mathfrak{X}^1_t(\mathcal{A})$ we obtain
$(a\cdot X)(\mathcal{A})=a\cdot X(\mathcal{A})\subseteq\mathcal{C}$,
$$
(X\cdot a)(\mathcal{A})
=X(\mathcal{R}_1^{-1}\rhd\mathcal{A})\cdot(\mathcal{R}_2^{-1}\rhd a)
\subseteq\mathcal{C},
$$
$(\xi\rhd X)(\mathcal{A})=\xi_{(1)}\rhd X(S(\xi_{(2)})\rhd\mathcal{A})
\subseteq\mathcal{C}$
and
$$
[X,Y]_\mathcal{R}(\mathcal{A})
=X(Y(\mathcal{A}))
-(\mathcal{R}_1^{-1}\rhd Y)((\mathcal{R}_2^{-1}\rhd X)(\mathcal{A}))
\subseteq\mathcal{C}
$$
for all $a\in\mathcal{A}$ and $\xi\in H$, since $\mathcal{C}$ is an
ideal and $H\rhd\mathcal{C}\subseteq\mathcal{C}$.
The projection respects the $H$-module action by definition, while
it respects the left and right $\mathcal{A}$-module actions since
\begin{align*}
    \mathrm{pr}(a\cdot X)(\mathrm{pr}(b))
    =&\mathrm{pr}(a\cdot X(b))\\
    =&\mathrm{pr}(a)\cdot\mathrm{pr}(X(b))\\
    =&(\mathrm{pr}(a)\cdot\mathrm{pr}(X))(\mathrm{pr}(b))
\end{align*}
and
\begin{align*}
    \mathrm{pr}(X\cdot a)(\mathrm{pr}(b))
    =&\mathrm{pr}((X\cdot a)(b))\\
    =&\mathrm{pr}(X(\mathcal{R}_1^{-1}\rhd b)
    \cdot(\mathcal{R}_2^{-1}\rhd a))\\
    =&\mathrm{pr}(X(\mathcal{R}_1^{-1}\rhd b))
    \cdot\mathrm{pr}((\mathcal{R}_2^{-1}\rhd a))\\
    =&\mathrm{pr}(X)\mathrm{pr}((\mathcal{R}_1^{-1}\rhd b))
    \cdot\mathrm{pr}((\mathcal{R}_2^{-1}\rhd a))\\
    =&\mathrm{pr}(X)(\mathcal{R}_1^{-1}\rhd\mathrm{pr}(b))
    \cdot(\mathcal{R}_2^{-1}\rhd\mathrm{pr}(a))\\
    =&(\mathrm{pr}(X)\cdot\mathrm{pr}(a))(\mathrm{pr}(b))
\end{align*}
hold. Finally,
\begin{align*}
    \mathrm{pr}([X,Y]_\mathcal{R})(\mathrm{pr}(a))
    =&\mathrm{pr}([X,Y]_\mathcal{R}(a))\\
    =&\mathrm{pr}(X(Y(a))-(\mathcal{R}_1^{-1}\rhd Y)
    (\mathcal{R}_2^{-1}\rhd X)(a))\\
    =&\mathrm{pr}(X)(\mathrm{pr}(Y(a)))
    -\mathrm{pr}(\mathcal{R}_1^{-1}\rhd Y)
    (\mathrm{pr}((\mathcal{R}_2^{-1}\rhd X)(a)))\\
    =&\mathrm{pr}(X)(\mathrm{pr}(Y)(\mathrm{pr}(a)))
    -(\mathcal{R}_1^{-1}\rhd\mathrm{pr}(Y))(
    (\mathcal{R}_2^{-1}\rhd\mathrm{pr}(X))(\mathrm{pr}(a)))\\
    =&[\mathrm{pr}(X),\mathrm{pr}(Y)]_\mathcal{R}(\mathrm{pr}(a))
\end{align*}
proves that $\mathrm{pr}$ is a homomorphism of braided Lie algebras.
By definition $\mathfrak{X}^1_0(\mathcal{A})$ is the kernel of
the projection.
\end{proof}
In the light of Lemma~\ref{lemma12} we are able to reformulate
Definition~\ref{def05} in the following way.
\begin{definition}
Let $\mathcal{C}\subseteq\mathcal{A}$ be an algebra ideal. Then
$\mathcal{A}/\mathcal{C}$ is a submanifold algebra if
there is a short exact sequence
$$
0\rightarrow\mathfrak{X}^1_0(\mathcal{A})\rightarrow
\mathfrak{X}^1_t(\mathcal{A})\xrightarrow{\mathrm{pr}}
\mathfrak{X}^1_\mathcal{R}(\mathcal{A}/\mathcal{C})
\rightarrow 0
$$
of $H$-equivariant braided symmetric $\mathcal{A}$-bimodules and
braided Lie algebras.
\end{definition}
In particular, the braided exterior algebra 
$$
\mathfrak{X}^\bullet_t(\mathcal{A})
=\Lambda^\bullet\mathfrak{X}^1_t(\mathcal{A})
=\mathcal{A}\oplus\mathfrak{X}^1_t(\mathcal{A})
\oplus(\mathfrak{X}^1_t(\mathcal{A})\wedge_\mathcal{R}
\mathfrak{X}^1_t(\mathcal{A}))\oplus\cdots
$$
of $\mathfrak{X}^1_t(\mathcal{A})$ is an $H$-equivariant braided symmetric
$\mathcal{A}$-sub-bimodule of $\mathfrak{X}^\bullet_\mathcal{R}(\mathcal{A})$ according to
Proposition~\ref{prop14}.
If we extend $\mathrm{pr}\colon\mathfrak{X}^1_t(\mathcal{A})
\rightarrow\mathfrak{X}^1_\mathcal{R}(\mathcal{A}/\mathcal{C})$ as a
homomorphism of the braided wedge product, i.e.
$$
\mathrm{pr}(X\wedge_\mathcal{R}Y)
=\mathrm{pr}(X)\wedge_\mathcal{R}\mathrm{pr}(Y)
$$
for all $X,Y\in\mathfrak{X}^\bullet_t(\mathcal{A})$, we obtain the
following result.
\begin{proposition}\label{prop16}
There is a short exact sequence
$$
0\rightarrow\ker\mathrm{pr}\rightarrow
\mathfrak{X}^\bullet_t(\mathcal{A})\xrightarrow{\mathrm{pr}}
\mathfrak{X}^\bullet_\mathcal{R}(\mathcal{A}/\mathcal{C})
\rightarrow 0
$$
of braided Gerstenhaber algebras. In particular,
$$
\mathrm{pr}(\llbracket X,Y\rrbracket_\mathcal{R})
=\llbracket\mathrm{pr}(X),\mathrm{pr}(Y)\rrbracket_\mathcal{R}
$$
for all $X,Y\in\mathfrak{X}^\bullet_t(\mathcal{A})$.
\end{proposition}
\begin{proof}
This is a consequence of Lemma~\ref{lemma12} since the braided multivector
fields are generated in degree zero and one. To prove that the projection
respects the braided Schouten-Nijenhuis bracket one recalls the expression
of $\llbracket\cdot,\cdot\rrbracket_\mathcal{R}$ on factorizing elements
(see Section~\ref{Sec3.2}).
\end{proof}
For braided differential forms $\omega=a_0\cdot\mathrm{d}a_1\wedge_\mathcal{R}
\cdots\wedge_\mathcal{R}\mathrm{d}a_n
\in\Omega^\bullet_\mathcal{R}(\mathcal{A})$ one defines
$$
\mathrm{pr}(\omega)
=\mathrm{pr}(a_0)\mathrm{d}(\mathrm{pr}(a_1))\wedge_\mathcal{R}\cdots
\wedge_\mathcal{R}\mathrm{d}(\mathrm{pr}(a_n))
$$
and as for braided multivector fields we define
a left $H$-action on $\Omega^\bullet_\mathcal{R}(\mathcal{A}/\mathcal{C})$
by $\xi\rhd\mathrm{pr}(\omega)=\mathrm{pr}(\xi\rhd\omega)$
for all $\xi\in H$ and $\omega\in\Omega^\bullet_\mathcal{R}(\mathcal{A})$.
\begin{proposition}\label{prop18}
There is a short exact sequence of differential graded algebras
$$
0\rightarrow\mathrm{ker}~\mathrm{pr}
\rightarrow\Omega^\bullet_\mathcal{R}(\mathcal{A})
\xrightarrow{\mathrm{pr}}\Omega^\bullet_\mathcal{R}(\mathcal{A}/\mathcal{C})
\rightarrow 0,
$$
where $\mathrm{ker}~\mathrm{pr}
=\bigoplus_{k\geq 0}\mathrm{ker}~\mathrm{pr}^k$ is defined recursively
by $\mathrm{ker}~\mathrm{pr}^0=\mathcal{C}$ and
$$
\mathrm{ker}~\mathrm{pr}^{k+1}
=\{\omega\in\Omega^{k+1}_\mathcal{R}(\mathcal{A})~|~
\mathrm{i}^\mathcal{R}_X\omega\in\mathrm{ker}~\mathrm{pr}^k
\text{ for all }X\in\mathfrak{X}^1_t(\mathcal{A})\}
$$
for $k\geq 0$.
\end{proposition}
\begin{proof}
Since every braided differential form can be written as a finite sum of elements
$a_0\cdot\mathrm{d}a_1\wedge_\mathcal{R}\cdots\wedge_\mathcal{R}\mathrm{d}a_n$,
where $a_0,\ldots,a_n$ and the projection $\mathrm{pr}\colon
\mathcal{A}\rightarrow\mathcal{A}/\mathcal{C}$ is surjective, it follows that
the above sequence is exact. By definition the projection commutes
with the de Rham differential.
\end{proof}
There are braided Cartan calculi on $\mathcal{A}$ and
$\mathcal{A}/\mathcal{C}$ with respect to $\mathcal{R}$
according to Theorem~\ref{ThmBraidedCC}. In Proposition~\ref{prop16} and
Proposition~\ref{prop18} we proved that the projection respects the
algebraic structure of $\mathfrak{X}_\mathcal{R}^\bullet(\mathcal{A})$
and $\Omega^\bullet_\mathcal{R}(\mathcal{A})$. As a natural question we
ask if also the geometric data of the braided Cartan calculus is
intertwined by the projection. A positive answer is given in the following
theorem.
\begin{theorem}\label{thm03}
The braided Cartan calculus on $\mathcal{A}/\mathcal{C}$ is the projection
of the braided Cartan calculus on $\mathcal{A}$. Namely,
\begin{align*}
    \mathscr{L}^\mathcal{R}_{\mathrm{pr}(X)}\mathrm{pr}(\omega)
    =&\mathrm{pr}(\mathscr{L}^\mathcal{R}_X\omega),\\
    \mathrm{i}^\mathcal{R}_{\mathrm{pr}(X)}\mathrm{pr}(\omega)
    =&\mathrm{pr}(\mathrm{i}^\mathcal{R}_X\omega),\\
    \mathrm{d}(\mathrm{pr}(\omega))
    =&\mathrm{pr}(\mathrm{d}\omega)
\end{align*}
hold for $X\in\mathfrak{X}^\bullet_t(\mathcal{A})$ and
$\omega\in\Omega^\bullet_\mathcal{R}(\mathcal{A})$.
\end{theorem}
\begin{proof}
By definition, the projection respects the de Rham differential.
In a next step we prove that this is also the case for the insertion. Let
$X\in\mathfrak{X}^1_t(\mathcal{A})$ and $\omega\in
\Omega^1_\mathcal{R}(\mathcal{A})$. We can assume without loss of generality that
$\omega=a\mathrm{d}b$ for $a,b\in\mathcal{A}$. Then
\begin{align*}
    \mathrm{i}^\mathcal{R}_{\mathrm{pr}(X)}\mathrm{pr}(\omega)
    =&\mathrm{i}^\mathcal{R}_{\mathrm{pr}(X)}
    (\mathrm{pr}(a)\mathrm{d}(\mathrm{pr}(b)))\\
    =&\mathrm{pr}(\mathcal{R}_1^{-1}\rhd a)
    \cdot\mathrm{pr}(\mathcal{R}_2^{-1}\rhd X)(\mathrm{pr}(b))\\
    =&\mathrm{pr}(\mathcal{R}_1^{-1}\rhd a)
    \cdot\mathrm{pr}((\mathcal{R}_2^{-1}\rhd X)(b))\\
    =&\mathrm{pr}((\mathcal{R}_1^{-1}\rhd a)
    \cdot(\mathcal{R}_2^{-1}\rhd X)(b))\\
    =&\mathrm{pr}(\mathrm{i}^\mathcal{R}_X\omega)
\end{align*}
holds. Similarly one proves that the insertion of a braided vector field
into a higher order braided differential form is respected by the
projection and since $\mathrm{i}^\mathcal{R}_{X\wedge_\mathcal{R}Y}
=\mathrm{i}^\mathcal{R}_X\mathrm{i}^\mathcal{R}_Y$ for another
$Y\in\mathfrak{X}^1_t(\mathcal{A})$ the same is true for the insertion
of braided multivector fields of any degree. Since the braided Lie derivative
is the graded braided commutator of the insertion and the de Rham differential
it is also respected by the projection. Furthermore, according to
Proposition~\ref{prop16} and Proposition~\ref{prop18},
$\mathrm{pr}\colon\mathfrak{X}^\bullet_t(\mathcal{A})\rightarrow
\mathfrak{X}^\bullet_\mathcal{R}(\mathcal{A}/\mathcal{C})$ and
$\mathrm{pr}\colon\Omega^\bullet_\mathcal{R}(\mathcal{A})\rightarrow
\Omega^\bullet_\mathcal{R}(\mathcal{A}/\mathcal{C})$ are surjections.
This concludes the proof of the theorem.
\end{proof}
As a special case we recover that the classical Cartan calculus on a
closed embedded submanifold $\iota\colon N\rightarrow M$ of a smooth manifold
$M$ is obtained by the pullback 
$\iota^*\colon\Omega^\bullet(M)\rightarrow\Omega^\bullet(N)$ of differential forms
and restriction $\iota^*\colon\mathfrak{X}_t^\bullet(M)\rightarrow
\mathfrak{X}^\bullet(N)$ of tangent multivector fields to $N$.
The latter is defined for any $X\in\mathfrak{X}^1_t(M)$ as the unique vector field
$X|_N\in\mathfrak{X}^1(N)$, which is $\iota$-related to $X$, i.e.
$T_q\iota(X|_N)_q=X_{\iota(q)}$ for all $q\in N$, where
$T_q\iota\colon T_qN\rightarrow T_{\iota(q)}M$ denotes the tangent map
(c.f. \cite{Lee2003}~Lem.~5.39).
In particular,
\begin{align*}
    \mathscr{L}_{\iota^*(X)}\iota^*(\omega)
    =\iota^*(\mathscr{L}_X\omega),~~
    \mathrm{i}_{\iota^*(X)}\iota^*(\omega)
    =\iota^*(\mathrm{i}_X\omega)
    \text{ and }
    \mathrm{d}\iota^*(\omega)
    =\iota^*(\mathrm{d}\omega)
\end{align*}
for all $X\in\mathfrak{X}^\bullet(M)$ and $\omega\in\Omega^\bullet(M)$.

\section{Equivariant Covariant Derivatives on Submanifold Algebras}\label{Sec5.2}

Fix a submanifold ideal $\mathcal{C}$ of $\mathcal{A}$ and a strongly
non-degenerate equivariant metric ${\bf g}$ on $\mathcal{A}$ in the following.
There is a direct sum decomposition
$$
\mathfrak{X}^1_\mathcal{R}(\mathcal{A})
=\mathfrak{X}^1_t(\mathcal{A})\oplus\mathfrak{X}^1_n(\mathcal{A}),
$$
where $\mathfrak{X}^1_n(\mathcal{A})$ are the so-called \textit{braided normal
vector fields} with respect to $\mathcal{C}$ and ${\bf g}$, defined to be the
subspace orthogonal to $\mathfrak{X}^1_t(\mathcal{A})$ with respect to ${\bf g}$.
Namely, $X\in\mathfrak{X}^1_\mathcal{R}(\mathcal{A})$ is an element
of $\mathfrak{X}^1_n(\mathcal{A})$ if and only if
${\bf g}(X,Y)=0$ for all $Y\in\mathfrak{X}^1_t(\mathcal{A})$.
Then we define
$\mathrm{pr}_{\bf g}\colon\mathfrak{X}^1_\mathcal{R}(\mathcal{A})
\rightarrow\mathfrak{X}^1_\mathcal{R}(\mathcal{A}/\mathcal{C})$ as the
$\Bbbk$-linear map
which first projects to the first subspace in the above decomposition and applies
$\mathrm{pr}\colon\mathfrak{X}^1_t(\mathcal{A})\rightarrow
\mathfrak{X}^1_\mathcal{R}(\mathcal{A}/\mathcal{C})$ afterwards.
In particular $\mathrm{pr}_{\bf g}(X)
=\mathrm{pr}(X)$ for all $X\in\mathfrak{X}^1_t(\mathcal{A})$. 
\begin{lemma}
The braided normal vector fields $\mathfrak{X}^1_n(\mathcal{A})$
are an $H$-equivariant braided symmetric $\mathcal{A}$-sub-bimodule
of $\mathfrak{X}^1_\mathcal{R}(\mathcal{A})$ and
$\mathrm{pr}_{\bf g}\colon\mathfrak{X}^1_\mathcal{R}(\mathcal{A})
\rightarrow\mathfrak{X}^1_\mathcal{R}(\mathcal{A}/\mathcal{C})$
is a homomorphism of $H$-equivariant braided symmetric
$\mathcal{A}$-bimodules.
\end{lemma}
\begin{proof}
Corresponding the first claim,
it is sufficient to prove that the module actions
are closed in $\mathfrak{X}^1_n(\mathcal{A})$. Let $a,b\in\mathcal{A}$,
$\xi\in H$ and $X\in\mathfrak{X}^1_n(\mathcal{A})$. Then
\begin{align*}
    {\bf g}(a\cdot X\cdot b,Y)
    =a\cdot {\bf g}(X,\mathcal{R}_1^{-1}\rhd Y)\cdot(\mathcal{R}_2^{-1}\rhd b)
    =0
\end{align*}
and
\begin{align*}
    {\bf g}(\xi\rhd X,Y)
    =\xi_{(1)}\rhd {\bf g}(X,S(\xi_{(2)})\rhd Y)
    =0
\end{align*}
follow for all $Y\in\mathfrak{X}^1_t(\mathcal{A})$, since
$\mathfrak{X}^1_t(\mathcal{A})$ is an $H$-equivariant braided symmetric 
$\mathcal{A}$-sub-bimodule and ${\bf g}$ is $H$-equivariant as well as
left $\mathcal{A}$-linear and braided right $\mathcal{A}$-linear in
the first argument. In particular, this implies that $\mathrm{pr}_{\bf g}$
respects the $H$-action and $\mathcal{A}$-bimodule actions: the tangent and
normal parts are closed under the actions and $\mathrm{pr}\colon
\mathfrak{X}^1_t(\mathcal{A})\rightarrow\mathfrak{X}^1_\mathcal{R}(
\mathcal{A}/\mathcal{C})$ is a homomorphism of $H$-equivariant braided
symmetric $\mathcal{A}$-bimodules according to Proposition~\ref{prop16}.
This concludes the proof.
\end{proof}
Note that the definition of braided normal vector fields still makes sense
for non-degenerate braided metrics and that also in this case they form
an $H$-equivariant braided symmetric $\mathcal{A}$-sub-bimodule of
$\mathfrak{X}^1_\mathcal{R}(\mathcal{A})$. Nonetheless we stick to
strongly non-degenerate braided metrics in this section, which is motivated in
the following lines. Recall that the
\textit{vanishing vector fields} $\mathfrak{X}^1_0(\mathcal{A})$
were defined as the kernel of the projection
$\mathrm{pr}\colon
\mathfrak{X}^1_t(\mathcal{A})\rightarrow\mathfrak{X}^1_\mathcal{R}
(\mathcal{A}/\mathcal{C})$.
Consider the $\Bbbk$-linear map 
\begin{equation}
    {\bf g}_{\mathcal{A}/\mathcal{C}}\colon
    \mathfrak{X}^1_\mathcal{R}(\mathcal{A}/\mathcal{C})
    \otimes_{\mathcal{A}/\mathcal{C}}
    \mathfrak{X}^1_\mathcal{R}(\mathcal{A}/\mathcal{C})
    \rightarrow\mathcal{A}/\mathcal{C}
\end{equation}
which is determined by
$
{\bf g}_{\mathcal{A}/\mathcal{C}}(\mathrm{pr}_{\bf g}(X),\mathrm{pr}_{\bf g}(Y))
=\mathrm{pr}(g(X,Y))
$
for all $X,Y\in\mathfrak{X}^1_\mathcal{R}(\mathcal{A})$. It is well-defined
if the following property holds.
\begin{align*}
    \textbf{Axiom 1: }&
    \text{for every $X\in\mathfrak{X}^1_0(\mathcal{A})$ there are finitely many
    $c_i\in\mathcal{C}$ and $X^i\in\mathfrak{X}^1_t(\mathcal{A})$}\\
    &\text{such that
    $X=\sum_ic_i\cdot X^i$.}
\end{align*}
This is for example the case if $\mathfrak{X}^1_0(\mathcal{A})$ is finitely
generated as a $\mathcal{C}$-bimodule.
Note that every linear combination $c\cdot X$ with $c\in\mathcal{C}$ and
$X\in\mathfrak{X}^1_t(\mathcal{A})$ defines a vanishing vector field.
Moreover, since we assumed ${\bf g}$ to be strongly non-degenerate it follows
that ${\bf g}_{\mathcal{A}/\mathcal{C}}$ is strongly non-degenerate as well, if
the following property holds.
\begin{equation*}
    \textbf{Axiom 2: }
    \text{if $X\in\mathfrak{X}^1_t(\mathcal{A})$, then
    ${\bf g}(X,X)\in\mathcal{C}$ implies $X\in\mathfrak{X}^1_0(\mathcal{A})$.}
\end{equation*}
Note that non-degeneracy of ${\bf g}$ in combination with Axiom~2 is not sufficient
to prove non-degeneracy of ${\bf g}_{\mathcal{A}/\mathcal{C}}$. We further would
like to point out that in the case of closed embedded smooth submanifolds
both, Axiom~1 and Axiom~2 are satisfied and that any Riemannian metric
is strongly non-degenerate in our sense. Let us perform a rigorous proof 
of the above discussion.
\begin{lemma}\label{lemma15}
Let $\mathcal{C}$ be a submanifold ideal of $\mathcal{A}$ and ${\bf g}$ a
strongly non-degenerate equivariant metric on $\mathcal{A}$ such that
Axiom~1 and Axiom~2 are satisfied. Then ${\bf g}_{\mathcal{A}/\mathcal{C}}$ is a
well-defined strongly non-degenerate equivariant metric on 
$\mathcal{A}/\mathcal{C}$ and any
equivariant covariant derivative $\nabla^\mathcal{R}\colon
\mathfrak{X}^1_\mathcal{R}(\mathcal{A})
\otimes\mathfrak{X}^1_\mathcal{R}(\mathcal{A})
\rightarrow\mathfrak{X}^1_\mathcal{R}(\mathcal{A})$ on $\mathcal{A}$ projects to
an equivariant covariant derivative
$$
\nabla^{\mathcal{A}/\mathcal{C}}_{\mathrm{pr}(X)}\mathrm{pr}(Y)
=\mathrm{pr}_{\bf g}(\nabla^\mathcal{R}_XY),
$$
on $\mathcal{A}/\mathcal{C}$ with respect to $\mathcal{R}$,
where $X,Y\in\mathfrak{X}^1_t(\mathcal{A})$.
\end{lemma}
\begin{proof}
As already claimed, Axiom~1 assures ${\bf g}_{\mathcal{A}/\mathcal{C}}$ to be
well-defined, since
\begin{align*}
    {\bf g}_{\mathcal{A}/\mathcal{C}}(\mathrm{pr}_{\bf g}(X),\mathrm{pr}_{\bf g}(Y))
    =&{\bf g}_{\mathcal{A}/\mathcal{C}}\bigg(
    \mathrm{pr}_{\bf g}\bigg(\sum_ic_i\cdot X^i\bigg),
    \mathrm{pr}_{\bf g}(Y)\bigg)
    =\mathrm{pr}\bigg({\bf g}\bigg(\sum_ic_i\cdot X^i,Y\bigg)\bigg)\\
    =&\mathrm{pr}\bigg(\underbrace{\sum_ic_i\cdot 
    {\bf g}(X^i,Y)}_{\in\mathcal{C}}\bigg)
    =0
\end{align*}
and similarly ${\bf g}_{\mathcal{A}/\mathcal{C}}
(\mathrm{pr}_{\bf g}(Y),\mathrm{pr}_{\bf g}(X))=0$
for all $X\in\mathfrak{X}^1_0(\mathcal{A})$ and
$Y\in\mathfrak{X}^1_\mathcal{R}(\mathcal{A})$. We prove that
${\bf g}_{\mathcal{A}/\mathcal{C}}$ has the correct linearity and symmetry
properties. Let $a\in\mathcal{A}$, $X,Y\in\mathfrak{X}^1_\mathcal{R}(\mathcal{A})$
and $\xi\in H$. Then
\begin{allowdisplaybreaks}
\begin{align*}
    {\bf g}_{\mathcal{A}/\mathcal{C}}(\mathrm{pr}(a)\cdot\mathrm{pr}_{\bf g}(X),
    \mathrm{pr}_{\bf g}(Y))
    =&{\bf g}_{\mathcal{A}/\mathcal{C}}(\mathrm{pr}_{\bf g}(a\cdot X),\mathrm{pr}_{\bf g}(Y))\\
    =&\mathrm{pr}({\bf g}(a\cdot X,Y))\\
    =&\mathrm{pr}(a\cdot {\bf g}(X,Y))\\
    =&\mathrm{pr}(a)\cdot\mathrm{pr}({\bf g}(X,Y))\\
    =&\mathrm{pr}(a)\cdot {\bf g}_{\mathcal{A}/\mathcal{C}}(\mathrm{pr}_{\bf g}(X),
    \mathrm{pr}_{\bf g}(Y)),
\end{align*}
\end{allowdisplaybreaks}
\begin{allowdisplaybreaks}
\begin{align*}
    {\bf g}_{\mathcal{A}/\mathcal{C}}(\mathrm{pr}_{\bf g}(X)\cdot\mathrm{pr}(a),
    \mathrm{pr}_{\bf g}(Y))
    =&{\bf g}_{\mathcal{A}/\mathcal{C}}(\mathrm{pr}_{\bf g}(X\cdot a),
    \mathrm{pr}_{\bf g}(Y))\\
    =&\mathrm{pr}({\bf g}(X\cdot a,Y))\\
    =&\mathrm{pr}({\bf g}(X,a\cdot Y))\\
    =&{\bf g}_{\mathcal{A}/\mathcal{C}}(\mathrm{pr}_{\bf g}(X),
    \mathrm{pr}_{\bf g}(a\cdot Y))\\
    =&{\bf g}_{\mathcal{A}/\mathcal{C}}(\mathrm{pr}_{\bf g}(X),
    \mathrm{pr}(a)\cdot\mathrm{pr}_{\bf g}(Y))
\end{align*}
\end{allowdisplaybreaks}
and
\begin{allowdisplaybreaks}
\begin{align*}
    \xi\rhd {\bf g}_{\mathcal{A}/\mathcal{C}}(\mathrm{pr}_{\bf g}(X),
    \mathrm{pr}_{\bf g}(Y))
    =&\xi\rhd\mathrm{pr}({\bf g}(X,Y))\\
    =&\mathrm{pr}(\xi\rhd {\bf g}(X,Y))\\
    =&\mathrm{pr}({\bf g}(\xi_{(1)}\rhd X,\xi_{(2)}\rhd Y))\\
    =&{\bf g}_{\mathcal{A}/\mathcal{C}}(\mathrm{pr}_{\bf g}(\xi_{(1)}\rhd X),
    \mathrm{pr}_{\bf g}(\xi_{(2)}\rhd Y))\\
    =&{\bf g}_{\mathcal{A}/\mathcal{C}}(\xi_{(1)}\rhd\mathrm{pr}_{\bf g}(X),
    \xi_{(2)}\rhd\mathrm{pr}_{\bf g}(Y))
\end{align*}
\end{allowdisplaybreaks}
hold, proving that ${\bf g}_{\mathcal{A}/\mathcal{C}}$ is a $\Bbbk$-linear map
$\mathfrak{X}^1_\mathcal{R}(\mathcal{A}/\mathcal{C})
\otimes_{\mathcal{A}/\mathcal{C}}
\mathfrak{X}^1_\mathcal{R}(\mathcal{A}/\mathcal{C})
\rightarrow\mathcal{A}/\mathcal{C}$. It is braided symmetric since
\begin{align*}
    {\bf g}_{\mathcal{A}/\mathcal{C}}(\mathrm{pr}_{\bf g}(Y),\mathrm{pr}_{\bf g}(X))
    =&\mathrm{pr}({\bf g}(Y,X))\\
    =&\mathrm{pr}({\bf g}(\mathcal{R}_1^{-1}\rhd X,\mathcal{R}_2^{-1}\rhd Y))\\
    =&{\bf g}_{\mathcal{A}/\mathcal{C}}(\mathrm{pr}_{\bf g}(\mathcal{R}_1^{-1}\rhd X),
    \mathrm{pr}_{\bf g}(\mathcal{R}_2^{-1}\rhd Y))\\
    =&{\bf g}_{\mathcal{A}/\mathcal{C}}(\mathcal{R}_1^{-1}\rhd\mathrm{pr}_{\bf g}(X),
    \mathcal{R}_2^{-1}\rhd\mathrm{pr}_{\bf g}(Y))
\end{align*}
for all $X,Y\in\mathfrak{X}^1_\mathcal{R}(\mathcal{A})$. Using Axiom~2
we prove that ${\bf g}_{\mathcal{A}/\mathcal{C}}$
is strongly non-degenerate. Let $X\in\mathfrak{X}^1_\mathcal{R}(
\mathcal{A}/\mathcal{C})$ and choose $Y\in\mathfrak{X}^1_t(\mathcal{A})$ such 
that $\mathrm{pr}(Y)=X$. Then
\begin{align*}
    0
    ={\bf g}_{\mathcal{A}/\mathcal{C}}(X,X)
    =\mathrm{pr}({\bf g}(Y,Y))
\end{align*}
implies ${\bf g}(Y,Y)\in\mathcal{C}$, i.e. $Y\in\mathfrak{X}^1_0(\mathcal{A})$ by
Axiom~2. In other words ${\bf g}_{\mathcal{A}/\mathcal{C}}(X,X)=0$ implies
$X=0$, which is equivalent to the statement that $X\neq 0$ implies
${\bf g}_{\mathcal{A}/\mathcal{C}}(X,X)\neq 0$, i.e. strong non-degeneracy of
${\bf g}_{\mathcal{A}/\mathcal{C}}$. From Axiom~1 it follows that
$\nabla^{\mathcal{A}/\mathcal{C}}$ is well-defined. In fact, for
$X=\sum_ic_i\cdot X^i\in\mathfrak{X}^1_0(\mathcal{A})$ and
$Y\in\mathfrak{X}^1_t(\mathcal{A})$ we obtain
\begin{align*}
    \nabla^{\mathcal{A}/\mathcal{C}}_{\mathrm{pr}(X)}\mathrm{pr}(Y)
    =\mathrm{pr}_{\bf g}(\nabla^\mathcal{R}_XY)
    =\mathrm{pr}_{\bf g}\bigg(\underbrace{
    \sum_ic_i\cdot\nabla^\mathcal{R}_{X^i}Y}_{
    \in\mathfrak{X}^1_0(\mathcal{A})}\bigg)
    =0
\end{align*}
and
\begin{align*}
    \nabla^{\mathcal{A}/\mathcal{C}}_{\mathrm{pr}(Y)}\mathrm{pr}(X)
    =&\mathrm{pr}_{\bf g}\bigg(\sum_i\nabla^\mathcal{R}_Y(c_i\cdot X^i)\bigg)\\
    =&\sum_i\mathrm{pr}_{\bf g}(\underbrace{\overbrace{
    (\mathscr{L}^\mathcal{R}_Yc_i)}^{\in\mathcal{C}}\cdot X^i}_{
    \in\mathfrak{X}^1_0(\mathcal{A})}
    +\underbrace{\overbrace{(\mathcal{R}_1^{-1}\rhd c_i)}^{\in\mathcal{C}}\cdot
    \nabla^\mathcal{R}_{\mathcal{R}_2^{-1}\rhd X^i}Y}_{
    \in\mathfrak{X}^1_0(\mathcal{A})})\\
    =&0,
\end{align*}
since $\nabla^\mathcal{R}$ is left $\mathcal{A}$-linear in the first
argument and satisfies a braided Leibniz rule in the second
argument. Clearly $\nabla^{\mathcal{A}/\mathcal{C}}$ is
$\Bbbk$-linear. It is left $\mathcal{A}/\mathcal{C}$-linear in the
first argument since
\begin{align*}
    \nabla^{\mathcal{A}/\mathcal{C}}_{\mathrm{pr}(a)\cdot\mathrm{pr}(X)}
    \mathrm{pr}(Y)
    =&\nabla^{\mathcal{A}/\mathcal{C}}_{\mathrm{pr}(a\cdot X)}
    \mathrm{pr}(Y)
    =\mathrm{pr}_{\bf g}(\nabla^\mathcal{R}_{a\cdot X}Y)
    =\mathrm{pr}_{\bf g}(a\cdot\nabla^\mathcal{R}_XY)\\
    =&\mathrm{pr}(a)\cdot\mathrm{pr}_{\bf g}(\nabla^\mathcal{R}_XY)
    =\mathrm{pr}(a)\cdot
    (\nabla^{\mathcal{A}/\mathcal{C}}_{\mathrm{pr}(X)}\mathrm{pr}(Y))
\end{align*}
and for all $X,Y\in\mathfrak{X}^1_t(\mathcal{A})$ and $a\in\mathcal{A}$
it satisfies a braided Leibniz rule 
\begin{allowdisplaybreaks}
\begin{align*}
    \nabla^{\mathcal{A}/\mathcal{C}}_{\mathrm{pr}(X)}
    (\mathrm{pr}(a)\cdot\mathrm{pr}(Y))
    =&\nabla^{\mathcal{A}/\mathcal{C}}_{\mathrm{pr}(X)}
    (\mathrm{pr}(a\cdot Y))\\
    =&\mathrm{pr}_{\bf g}(\nabla^\mathcal{R}_X(a\cdot Y))\\
    =&\mathrm{pr}_{\bf g}((\mathscr{L}^\mathcal{R}_Xa)\cdot Y
    +(\mathcal{R}_1^{-1}\rhd a)\cdot
    (\nabla^\mathcal{R}_{\mathcal{R}_2^{-1}\rhd X}Y))\\
    =&\mathrm{pr}(\mathscr{L}^\mathcal{R}_Xa)\cdot\mathrm{pr}(Y)
    +\mathrm{pr}(\mathcal{R}_1^{-1}\rhd a)\cdot
    \mathrm{pr}_{\bf g}(\nabla^\mathcal{R}_{\mathcal{R}_2^{-1}\rhd X}Y)\\
    =&(\mathscr{L}^\mathcal{R}_{\mathrm{pr}(X)}\mathrm{pr}(a))
    \cdot\mathrm{pr}(Y)
    +(\mathcal{R}_1^{-1}\rhd\mathrm{pr}(a))\cdot
    (\nabla^{\mathcal{A}/\mathcal{C}}_{\mathcal{R}_2^{-1}\rhd\mathrm{pr}(X)}
    \mathrm{pr}(Y))
\end{align*}
\end{allowdisplaybreaks}
in the second argument. This concludes the proof of the lemma.
\end{proof}
The curvature and torsion of a projected equivariant covariant
derivative coincide with the projections of curvature and torsion of the
initial equivariant covariant derivative. This underlines how the concepts of
braided commutative geometry on submanifold algebras
$\mathcal{A}/\mathcal{C}$ can be obtained from the ones on $\mathcal{A}$.
\begin{corollary}
Under the assumptions of Lemma~\ref{lemma15},
the curvature $R^{\nabla^{\mathcal{A}/\mathcal{C}}}$ and the torsion
$\mathrm{Tor}^{\nabla^{\mathcal{A}/\mathcal{C}}}$ of the projected
equivariant covariant derivative
$\nabla^{\mathcal{A}/\mathcal{C}}$ are given by
$$
R^{\nabla^{\mathcal{A}/\mathcal{C}}}(\mathrm{pr}(X),\mathrm{pr}(Y))(\mathrm{pr}(Z))
=\mathrm{pr}_{\bf g}(R^{\nabla^\mathcal{R}}(X,Y)Z)
$$
and
$$
\mathrm{Tor}^{\nabla^{\mathcal{A}/\mathcal{C}}}(\mathrm{pr}(X),\mathrm{pr}(Y))
=\mathrm{pr}_{\bf g}(\mathrm{Tor}^{\nabla^\mathcal{R}}(X,Y))
$$
for all $X,Y,Z\in\mathfrak{X}^1_t(\mathcal{A})$.
If furthermore,
$\nabla^\mathcal{R}$ is the equivariant Levi-Civita covariant derivative with
respect to ${\bf g}$, $\nabla^{\mathcal{A}/\mathcal{C}}$ is the equivariant Levi-Civita
covariant derivative on $\mathcal{A}/\mathcal{C}$ with respect to
${\bf g}_{\mathcal{A}/\mathcal{C}}$.
\end{corollary}
\begin{proof}
Let $X,Y,Z\in\mathfrak{X}^1_t(\mathcal{A})$. According to Lemma~\ref{lemma15}
$\nabla^{\mathcal{A}/\mathcal{C}}$ is an equivariant
covariant derivative on $\mathcal{A}/\mathcal{C}$ with respect to $\mathcal{R}$.
In particular, its curvature and torsion
are well-defined and
\begin{allowdisplaybreaks}
\begin{align*}
    R^{\nabla^{\mathcal{A}/\mathcal{C}}}(\mathrm{pr}(X),\mathrm{pr}(Y))
    (\mathrm{pr}(Z))
    =&\nabla^{\mathcal{A}/\mathcal{C}}_{\mathrm{pr}(X)}
    \nabla^{\mathcal{A}/\mathcal{C}}_{\mathrm{pr}(Y)}\mathrm{pr}(Z)
    -\nabla^{\mathcal{A}/\mathcal{C}}_{\mathcal{R}_1^{-1}\rhd\mathrm{pr}(Y)}
    \nabla^{\mathcal{A}/\mathcal{C}}_{\mathcal{R}_2^{-1}\rhd\mathrm{pr}(X)}
    \mathrm{pr}(Z)\\
    &-\nabla^{\mathcal{A}/\mathcal{C}}_{[\mathrm{pr}(X)
    ,\mathrm{pr}(Y)]_\mathcal{R}}\mathrm{pr}(Z)\\
    =&\nabla^{\mathcal{A}/\mathcal{C}}_{\mathrm{pr}(X)}
    \nabla^{\mathcal{A}/\mathcal{C}}_{\mathrm{pr}(Y)}\mathrm{pr}(Z)
    -\nabla^{\mathcal{A}/\mathcal{C}}_{\mathrm{pr}(\mathcal{R}_1^{-1}\rhd Y)}
    \nabla^{\mathcal{A}/\mathcal{C}}_{\mathrm{pr}(\mathcal{R}_2^{-1}\rhd X)}
    \mathrm{pr}(Z)\\
    &-\nabla^{\mathcal{A}/\mathcal{C}}_{\mathrm{pr}([X,Y]_\mathcal{R})}
    \mathrm{pr}(Z)\\
    =&\mathrm{pr}_{\bf g}(\nabla^\mathcal{R}_X\nabla^\mathcal{R}_YZ
    -\nabla^\mathcal{R}_{\mathcal{R}_1^{-1}\rhd Y}
    \nabla^\mathcal{R}_{\mathcal{R}_2^{-1}\rhd X}Z
    -\nabla^\mathcal{R}_{[X,Y]_\mathcal{R}}Z)\\
    =&\mathrm{pr}_{\bf g}(R^{\nabla^\mathcal{R}}(X,Y)Z)
\end{align*}
\end{allowdisplaybreaks}
as well as
\begin{allowdisplaybreaks}
\begin{align*}
    \mathrm{Tor}^{\nabla^{\mathcal{A}/\mathcal{C}}}
    (\mathrm{pr}(X),\mathrm{pr}(Y))
    =&\nabla^{\mathcal{A}/\mathcal{C}}_{\mathrm{pr}(X)}\mathrm{pr}(Y)
    -\nabla^{\mathcal{A}/\mathcal{C}}_{\mathcal{R}_1^{-1}\rhd\mathrm{pr}(Y)}
    (\mathcal{R}_2^{-1}\rhd\mathrm{pr}(X))\\
    &-[\mathrm{pr}(X),\mathrm{pr}(Y)]_\mathcal{R}\\
    =&\mathrm{pr}_{\bf g}(\nabla^\mathcal{R}_XY
    -\nabla^\mathcal{R}_{\mathcal{R}_1^{-1}\rhd Y}(\mathcal{R}_2^{-1}\rhd X)
    -[X,Y]_\mathcal{R})\\
    =&\mathrm{pr}_{\bf g}(\mathrm{Tor}^{\nabla^\mathcal{R}}(X,Y))
\end{align*}
\end{allowdisplaybreaks}
hold. If $\nabla^\mathcal{R}$ is the Levi-Civita covariant derivative
with respect to ${\bf g}$, then
\begin{align*}
    \mathscr{L}^\mathcal{R}_X{\bf g}(Y,Z)
    ={\bf g}(\nabla^\mathcal{R}_XY,Z)
    +{\bf g}(\mathcal{R}_1^{-1}\rhd Y,\nabla^\mathcal{R}_{\mathcal{R}_2^{-1}\rhd X}Z),
\end{align*}
where $X,Y,Z\in\mathfrak{X}^1_\mathcal{R}(\mathcal{A})$, implies
\begin{allowdisplaybreaks}
\begin{align*}
    \mathscr{L}^\mathcal{R}_{\mathrm{pr}(X)}
    {\bf g}_{\mathcal{A}/\mathcal{C}}(\mathrm{pr}(Y),\mathrm{pr}(Z))
    =&\mathscr{L}^\mathcal{R}_{\mathrm{pr}(X)}
    \mathrm{pr}({\bf g}(Y,Z))\\
    =&\mathrm{pr}(\mathscr{L}^\mathcal{R}_X{\bf g}(Y,Z))\\
    =&\mathrm{pr}({\bf g}(\nabla^\mathcal{R}_XY,Z)
    +{\bf g}(\mathcal{R}_1^{-1}\rhd Y,\nabla^\mathcal{R}_{\mathcal{R}_2^{-1}\rhd X}Z))\\
    =&{\bf g}_{\mathcal{A}/\mathcal{C}}(\mathrm{pr}_{\bf g}(\nabla^\mathcal{R}_XY),\mathrm{pr}(Z))\\
    &+{\bf g}_{\mathcal{A}/\mathcal{C}}(\mathrm{pr}(\mathcal{R}_1^{-1}\rhd Y),
    \mathrm{pr}_{\bf g}(\nabla^\mathcal{R}_{\mathcal{R}_2^{-1}\rhd X}Z))\\
    =&{\bf g}_{\mathcal{A}/\mathcal{C}}(
    \nabla^{\mathcal{A}/\mathcal{C}}_{\mathrm{pr}(X)}\mathrm{pr}(Y),
    \mathrm{pr}(Z))\\
    &+{\bf g}_{\mathcal{A}/\mathcal{C}}(\mathcal{R}_1^{-1}\rhd\mathrm{pr}(Y),
    \nabla^{\mathcal{A}/\mathcal{C}}_{\mathcal{R}_2^{-1}\rhd\mathrm{pr}(X)}
    \mathrm{pr}(Z))
\end{align*}
\end{allowdisplaybreaks}
for all $X,Y,Z\in\mathfrak{X}^1_t(\mathcal{A})$. Since
$$
\mathrm{Tor}^{\nabla^{\mathcal{A}/\mathcal{C}}}(\mathrm{pr}(X),\mathrm{pr}(Y))
=\mathrm{pr}_{\bf g}(\mathrm{Tor}^{\nabla^\mathcal{R}}(X,Y))
=\mathrm{pr}_{\bf g}(0)
=0
$$
$\nabla^{\mathcal{A}/\mathcal{C}}$ is the unique equivariant Levi-Civita covariant
derivative on $\mathcal{A}/\mathcal{C}$ with respect to ${\bf g}_{\mathcal{A}/\mathcal{C}}$.
This concludes the proof.
\end{proof}
One extends the projection $\mathrm{pr}_{\bf g}\colon
\mathfrak{X}^\bullet_\mathcal{R}(\mathcal{A})
\rightarrow\mathfrak{X}^\bullet_\mathcal{R}(\mathcal{A}/\mathcal{R})$
to braided multivector fields by defining it to coincide with $\mathrm{pr}$
on $\mathcal{A}$ and to be a homomorphism of the braided wedge product on
higher wedge powers. On braided differential forms we set
$\mathrm{pr}_{\bf g}=\mathrm{pr}$.
\begin{corollary}
Under the assumptions of Lemma~\ref{lemma15},
the equivariant covariant derivatives
$\nabla^{\mathcal{A}/\mathcal{C}}\colon
\mathfrak{X}_\mathcal{R}^1(\mathcal{A}/\mathcal{C})
\otimes\mathfrak{X}_\mathcal{R}^\bullet(\mathcal{A}/\mathcal{C})
\rightarrow\mathfrak{X}_\mathcal{R}^\bullet(\mathcal{A}/\mathcal{C})$ and
$\tilde{\nabla}^{\mathcal{A}/\mathcal{C}}
\colon\mathfrak{X}_\mathcal{R}^1(\mathcal{A}/\mathcal{C})
\otimes\Omega_\mathcal{R}^\bullet(\mathcal{A}/\mathcal{C})
\rightarrow\Omega_\mathcal{R}^\bullet(\mathcal{A}/\mathcal{C})$,
induced by the projected covariant derivative $\nabla^{\mathcal{A}/\mathcal{C}}$
on $\mathcal{A}/\mathcal{C}$ are
projected from the covariant derivatives induced by $\nabla^\mathcal{R}$. Namely, 
\begin{equation}\label{eq36}
    \nabla^{\mathcal{A}/\mathcal{C}}_{\mathrm{pr}(X)}\mathrm{pr}(Y)
    =\mathrm{pr}_{\bf g}(\nabla^\mathcal{R}_XY)
    \text{ and }
    \tilde{\nabla}^{\mathcal{A}/\mathcal{C}}_{\mathrm{pr}(X)}\mathrm{pr}(\omega)
    =\mathrm{pr}(\tilde{\nabla}^\mathcal{R}_X\omega)
\end{equation}
for
all $X\in\mathfrak{X}^1_t(\mathcal{A})$, $Y\in\mathfrak{X}^\bullet_t(\mathcal{A})$ and
$\omega\in\Omega_\mathcal{R}^\bullet(\mathcal{A})$.
\end{corollary}
\begin{proof}
Let $X,Y\in\mathfrak{X}^1_t(\mathcal{A})$ and
$\omega\in\Omega^\bullet_\mathcal{R}(\mathcal{A})$. Then
\begin{align*}
    \langle\tilde{\nabla}^{\mathcal{A}/\mathcal{C}}_{\mathrm{pr}(X)}
    \mathrm{pr}(\omega),\mathrm{pr}(Y)\rangle_\mathcal{R}
    =&\mathscr{L}^\mathcal{R}_{\mathrm{pr}(X)}\langle\mathrm{pr}(\omega),
    \mathrm{pr}(Y)\rangle_\mathcal{R}
    -\langle\mathcal{R}_1^{-1}\rhd\mathrm{pr}(\omega),
    \nabla^{\mathcal{A}/\mathcal{C}}_{\mathcal{R}_2^{-1}\rhd\mathrm{pr}(X)}
    \mathrm{pr}(Y)\rangle_\mathcal{R}\\
    =&\mathscr{L}^\mathcal{R}_{\mathrm{pr}(X)}\mathrm{pr}(\langle\omega,
    Y\rangle_\mathcal{R})
    -\langle\mathrm{pr}(\mathcal{R}_1^{-1}\rhd\omega),
    \mathrm{pr}_{\bf g}(\nabla^\mathcal{R}_{\mathcal{R}_2^{-1}\rhd X}
    Y)\rangle_\mathcal{R}\\
    =&\mathrm{pr}(\mathscr{L}^\mathcal{R}_X\langle\omega,Y\rangle_\mathcal{R}
    -\langle\mathcal{R}_1^{-1}\rhd\omega,
    \nabla^\mathcal{R}_{\mathcal{R}_2^{-1}\rhd X}Y\rangle_\mathcal{R})\\
    =&\mathrm{pr}(\langle\tilde{\nabla}^\mathcal{R}_X\omega,
    Y\rangle_\mathcal{R})\\
    =&\langle\mathrm{pr}(\tilde{\nabla}^\mathcal{R}_X\omega),
    \mathrm{pr}(Y)\rangle_\mathcal{R}
\end{align*}
implies $\tilde{\nabla}^{\mathcal{A}/\mathcal{C}}_{\mathrm{pr}(X)}
\mathrm{pr}(\omega)=\mathrm{pr}(\tilde{\nabla}^\mathcal{R}_X\omega)$ by
the non-degeneracy of $\langle\cdot,\cdot\rangle_\mathcal{R}$. Assume now
that we proved (\ref{eq36}) for all $X\in\mathfrak{X}^1_t(\mathcal{A})$,
$Y\in\mathfrak{X}^k_t(\mathcal{A})$ and $\omega\in\Omega^k_\mathcal{R}
(\mathcal{A})$ for a fixed $k>0$. Let $X,Y\in\mathfrak{X}^1_t(\mathcal{A})$,
$Z\in\mathfrak{X}^k_t(\mathcal{A})$, $\omega\in\Omega^1_\mathcal{R}(\mathcal{A})$
and $\eta\in\Omega^k_\mathcal{R}(\mathcal{A})$.
Then
\begin{align*}
    \nabla^{\mathcal{A}/\mathcal{C}}_{\mathrm{pr}(X)}
    (\mathrm{pr}(Y)\wedge_\mathcal{R}\mathrm{pr}(Z))
    =&(\nabla^{\mathcal{A}/\mathcal{C}}_{\mathrm{pr}(X)}\mathrm{pr}(Y))
    \wedge_\mathcal{R}\mathrm{pr}(Z)\\
    &+(\mathcal{R}_1^{-1}\rhd Y)\wedge_\mathcal{R}
    (\nabla^{\mathcal{A}/\mathcal{C}}_{\mathcal{R}_2^{-1}\rhd\mathrm{pr}(X)}
    \mathrm{pr}(Z))\\
    =&\mathrm{pr}_{\bf g}((\nabla^\mathcal{R}_XY)\wedge_\mathcal{R}Z
    +(\mathcal{R}_1^{-1}\rhd Y)\wedge_\mathcal{R}
    (\nabla^\mathcal{R}_{\mathcal{R}_2^{-1}\rhd X}Z))\\
    =&\mathrm{pr}_{\bf g}(\nabla^\mathcal{R}_X(Y\wedge_\mathcal{R}Z))
\end{align*}
and
\begin{align*}
    \tilde{\nabla}^{\mathcal{A}/\mathcal{C}}_{\mathrm{pr}(X)}
    (\mathrm{pr}(\omega)\wedge_\mathcal{R}\mathrm{pr}(\eta))
    =&(\tilde{\nabla}^{\mathcal{A}/\mathcal{C}}_{\mathrm{pr}(X)}\mathrm{pr}
    (\omega))
    \wedge_\mathcal{R}\mathrm{pr}(\eta)\\
    &+(\mathcal{R}_1^{-1}\rhd\omega)\wedge_\mathcal{R}
    (\tilde{\nabla}^{\mathcal{A}/\mathcal{C}}_{
    \mathcal{R}_2^{-1}\rhd\mathrm{pr}(X)}
    \mathrm{pr}(\eta))\\
    =&\mathrm{pr}_{\bf g}((\tilde{\nabla}^\mathcal{R}_X\omega)\wedge_\mathcal{R}\eta
    +(\mathcal{R}_1^{-1}\rhd\omega)\wedge_\mathcal{R}
    (\tilde{\nabla}^\mathcal{R}_{\mathcal{R}_2^{-1}\rhd X}\eta))\\
    =&\mathrm{pr}_{\bf g}(\tilde{\nabla}^\mathcal{R}_X(\omega\wedge_\mathcal{R}\eta))
\end{align*}
prove that (\ref{eq36}) even holds for all
$X\in\mathfrak{X}^1_t(\mathcal{A})$,
$Y\in\mathfrak{X}^{k+1}_t(\mathcal{A})$ and $\omega\in\Omega^k_\mathcal{R}
(\mathcal{A})$. By an inductive argument we conclude the proof of the corollary.
\end{proof}
As a special case, we recover the observation
(see e.g. \cite{KobayashiII}~Chap.~VII~Prop.~3.1)
that the Levi-Civita covariant derivative on a Riemannian manifold projects
to any Riemannian submanifold.

\section{Gauge Equivalences and Submanifold Algebras}\label{Sec5.3}

In the next theorem we prove that the gauge equivalence given by the Drinfel'd functor
is compatible with the notion of submanifold ideals. In other words, the projection
to submanifold algebras and twisting commutes. In the particular case of a
cocommutative Hopf algebra with trivial triangular structure this means that
twist quantization and projection to the submanifold algebra commute.
\begin{theorem}\label{thm07}
Let $\mathcal{C}$ be a submanifold ideal of $\mathcal{A}$.
Then, for any twist $\mathcal{F}$ on $H$,
the projection of the twisted Gerstenhaber algebra 
$(\mathfrak{X}_t^\bullet(\mathcal{A})_\mathcal{F},\wedge_\mathcal{F},
\llbracket\cdot,\cdot\rrbracket_\mathcal{F})$ 
of braided multivector fields on $\mathcal{A}$ which are tangent to 
$\mathcal{A}/\mathcal{C}$ coincides with the twisted Gerstenhaber algebra 
$(\mathfrak{X}^\bullet_\mathcal{R}(\mathcal{A}/\mathcal{C})_\mathcal{F},
\wedge_\mathcal{F},
\llbracket\cdot,\cdot\rrbracket_\mathcal{F})$ on $\mathcal{A}/\mathcal{C}$.
Moreover, the twisted Cartan calculus on $\mathcal{A}/\mathcal{C}$
is given by the projection of the twisted Cartan calculus on $\mathcal{A}$.
Namely, $\Omega^\bullet_\mathcal{R}(\mathcal{A}/\mathcal{C})_\mathcal{F}
=\mathrm{pr}(\Omega^\bullet_\mathcal{R}(\mathcal{A})_\mathcal{F})$,
\begin{align*}
    \mathscr{L}^\mathcal{F}_{\mathrm{pr}(X)}\mathrm{pr}(\omega)
    =\mathrm{pr}(\mathscr{L}^\mathcal{F}_X\omega),
    \mathrm{i}^\mathcal{F}_{\mathrm{pr}(X)}\mathrm{pr}(\omega)
    =\mathrm{pr}(\mathrm{i}^\mathcal{F}_X\omega)
    \text{ and }
    \mathrm{d}(\mathrm{pr}(\omega))
    =\mathrm{pr}(\mathrm{d}\omega)
\end{align*}
for all $X\in\mathfrak{X}^\bullet_t(\mathcal{A})$ and
$\omega\in\Omega^\bullet_\mathcal{R}(\mathcal{A})$.
\end{theorem}
\begin{proof}
Note that the twisted multivector fields are a braided Gerstenhaber algebra 
since the braided multivector fields which are tangent to $\mathcal{A}/\mathcal{C}$
are an $H$-submodule and a braided symmetric $\mathcal{A}$-sub-bimodule of
$\mathfrak{X}^\bullet_\mathcal{R}(\mathcal{A})$.
We already noticed that $\mathrm{pr}\colon\mathfrak{X}^\bullet_t(\mathcal{A})
\rightarrow\mathfrak{X}^\bullet(\mathcal{A}/\mathcal{C})$ is surjective.
Let $X,Y\in\mathfrak{X}^\bullet_t(\mathcal{A})$ and $a\in\mathcal{A}$. Then
$$
\mathrm{pr}(X)\wedge_\mathcal{F}\mathrm{pr}(Y)
=(\mathcal{F}_1^{-1}\rhd\mathrm{pr}(X))
\wedge_\mathcal{R}(\mathcal{F}_2^{-1}\rhd\mathrm{pr}(Y))
=\mathrm{pr}(X\wedge_\mathcal{F}Y),
$$
and similarly $\llbracket\mathrm{pr}(X),\mathrm{pr}(Y)\rrbracket_\mathcal{F}
=\mathrm{pr}(\llbracket X,Y\rrbracket_\mathcal{F})$
and $\mathrm{pr}(a)\cdot_\mathcal{F}\mathrm{pr}(X)
=\mathrm{pr}(a\cdot_\mathcal{F}X)$ follow.
Moreover,
$$
\mathscr{L}^\mathcal{F}_{\mathrm{pr}(X)}\mathrm{pr}(\omega)
=\mathscr{L}^\mathcal{R}_{\mathcal{F}_1^{-1}\rhd\mathrm{pr}(X)}
(\mathcal{F}_2^{-1}\rhd\mathrm{pr}(\omega))
=\mathrm{pr}(\mathscr{L}^\mathcal{F}_X\omega)
$$
and
$$
\mathrm{i}^\mathcal{F}_{\mathrm{pr}(X)}\mathrm{pr}(\omega)
=\mathrm{i}^\mathcal{R}_{\mathcal{F}_1^{-1}\rhd\mathrm{pr}(X)}
(\mathcal{F}_2^{-1}\rhd\mathrm{pr}(\omega))
=\mathrm{pr}(\mathrm{i}^\mathcal{F}_X\omega)
$$
for all $X\in\mathfrak{X}^\bullet_t(\mathcal{A})$ and 
$\omega\in\Omega^\bullet_\mathcal{R}(\mathcal{A})$ by Theorem~\ref{thm03}.
\end{proof}
In zero degree $\mathfrak{X}^0_t(\mathcal{A})=\mathcal{A}$. Thus,
Theorem~\ref{thm07} implies that the twisted product on $(\mathcal{A}/\mathcal{C})_\mathcal{F}$
\begin{equation}\label{eq53}
    \mathrm{pr}(a)\cdot_\mathcal{F}\mathrm{pr}(b)
    =\mathrm{pr}(a\cdot_\mathcal{F}b),
\end{equation}
where $a,b\in\mathcal{A}$, coincides with the projection of
the twisted product on $\mathcal{A}_\mathcal{F}$.
Furthermore, twisted equivariant covariant derivatives behave well under projection.
Fix a strongly non-degenerate equivariant metric
${\bf g}$ and a submanifold ideal $\mathcal{C}$ such that Axiom~1 and Axiom~2,
defined in the previous section, are satisfied.
Also fix an equivariant covariant derivative
$\nabla^\mathcal{R}$ on $\mathcal{A}$.
\begin{proposition}
For any twist $\mathcal{F}$ on $H$, the projection
of the twisted covariant derivative coincides with the twist deformation
of the projected equivariant covariant derivative, i.e.
$(\nabla^{\mathcal{A}/\mathcal{C}})^\mathcal{F}_{\mathrm{pr}(X)}\mathrm{pr}(Y)
=\mathrm{pr}_{\bf g}(\nabla^\mathcal{F}_XY)$
for all $X,Y\in\mathfrak{X}_t^1(\mathcal{A})$. Similar statements hold for the
induced (twisted) equivariant covariant derivatives on braided differential forms and
braided multivector fields.
\end{proposition}
\begin{proof}
For all $X,Y\in\mathfrak{X}^1_t(\mathcal{A})$ one obtains
\begin{align*}
    \mathrm{pr}_{\bf g}(\nabla^\mathcal{F}_XY)
    =&\mathrm{pr}_{\bf g}(\nabla^\mathcal{R}_{\mathcal{F}_1^{-1}\rhd X}
    (\mathcal{F}_2^{-1}\rhd Y))
    =\nabla^{\mathcal{A}/\mathcal{C}}_{\mathrm{pr}(\mathcal{F}_1^{-1}\rhd X)}
    (\mathrm{pr}(\mathcal{F}_2^{-1}\rhd Y))\\
    =&\nabla^{\mathcal{A}/\mathcal{C}}_{\mathcal{F}_1^{-1}\rhd\mathrm{pr}(X)}
    (\mathcal{F}_2^{-1}\rhd\mathrm{pr}(Y))
    =(\nabla^{\mathcal{A}/\mathcal{C}})^\mathcal{F}_{\mathrm{pr}(X)}\mathrm{pr}(Y)
\end{align*}
and similarly one proves the statements about the induced
equivariant covariant derivatives.
\end{proof}
There are explicit formulas for the curvature and torsion of the twisted
covariant derivative on a submanifold algebra in terms of the initial
curvature and torsion.
\begin{corollary}
For all $X,Y,Z\in\mathfrak{X}^1_t(\mathcal{A})$
\begin{allowdisplaybreaks}
\begin{align*}
    R^{(\nabla^{\mathcal{A}/\mathcal{C}})^\mathcal{F}}&
    (\mathrm{pr}(X),\mathrm{pr}(Y))(\mathrm{pr}(Z))\\
    =&R^{\nabla^{\mathcal{A}/\mathcal{C}}}\bigg(
    (\mathcal{F}_{1(1)}^{-1}\mathcal{F}_1^{'-1})
    \rhd\mathrm{pr}(X),
    (\mathcal{F}_{1(2)}^{-1}\mathcal{F}_2^{'-1})\rhd\mathrm{pr}(Y)
    \bigg)(\mathcal{F}_2^{-1}\rhd\mathrm{pr}(Z))\\
    =&\mathrm{pr}\bigg(R^{\nabla^\mathcal{R}}\bigg(
    (\mathcal{F}_{1(1)}^{-1}\mathcal{F}_1^{'-1})
    \rhd X,
    (\mathcal{F}_{1(2)}^{-1}\mathcal{F}_2^{'-1})\rhd Y
    \bigg)(\mathcal{F}_2^{-1}\rhd Z)
    \bigg)
\end{align*}
\end{allowdisplaybreaks}
and
\begin{align*}
    \mathrm{Tor}^{(\nabla^{{\mathcal{A}/\mathcal{C}}})^\mathcal{F}}
    (\mathrm{pr}(X),\mathrm{pr}(Y))
    =&\mathrm{Tor}^{\nabla^{\mathcal{A}/\mathcal{C}}}
    (\mathcal{F}_1^{-1}\rhd\mathrm{pr}(X),
    \mathcal{F}_2^{-1}\rhd\mathrm{pr}(Y))\\
    =&\mathrm{pr}\bigg(\mathrm{Tor}^{\nabla^\mathcal{R}}(\mathcal{F}_1^{-1}\rhd X,
    \mathcal{F}_2^{-1}\rhd Y)\bigg)
\end{align*}
hold.
\end{corollary}
\begin{proof}
Let $X,Y,Z\in\mathfrak{X}^1_t(\mathcal{A})$. Then
\begin{allowdisplaybreaks}
\begin{align*}
    R^{(\nabla^{\mathcal{A}/\mathcal{C}})^\mathcal{F}}&
    (\mathrm{pr}(X),\mathrm{pr}(Y))(\mathrm{pr}(Z))\\
    =&(\nabla^{\mathcal{A}/\mathcal{C}})^\mathcal{F}_{\mathrm{pr}(X)}
    (\nabla^{\mathcal{A}/\mathcal{C}})^\mathcal{F}_{\mathrm{pr}(Y)}
    \mathrm{pr}(Z)
    -(\nabla^{\mathcal{A}/\mathcal{C}})^\mathcal{F}_{
    \mathcal{R}_{\mathcal{F}1}^{-1}\rhd\mathrm{pr}(Y)}
    (\nabla^{\mathcal{A}/\mathcal{C}})^\mathcal{F}_{
    \mathcal{R}_{\mathcal{F}2}^{-1}\rhd\mathrm{pr}(X)}\mathrm{pr}(Z)\\
    &-(\nabla^{\mathcal{A}/\mathcal{C}})^\mathcal{F}_{[\mathrm{pr}(X),
    \mathrm{pr}(Y)]_{\mathcal{R}_\mathcal{F}}}\mathrm{pr}(Z)\\
    =&\nabla^{\mathcal{A}/\mathcal{C}}_{\mathcal{F}_1^{-1}\rhd\mathrm{pr}(X)}
    \nabla^{\mathcal{A}/\mathcal{C}}_{
    (\mathcal{F}_{2(1)}^{-1}\mathcal{F}_1^{'-1})\rhd\mathrm{pr}(Y)}
    ((\mathcal{F}_{2(2)}^{-1}\mathcal{F}_2^{'-1})\rhd\mathrm{pr}(Z))\\
    &-\nabla^{\mathcal{A}/\mathcal{C}}_{
    (\mathcal{F}_1^{-1}\mathcal{R}_{\mathcal{F}1}^{-1})\rhd\mathrm{pr}(Y)}
    \nabla^{\mathcal{A}/\mathcal{C}}_{
    (\mathcal{F}_{2(1)}^{-1}\mathcal{F}_1^{'-1}\mathcal{R}_{\mathcal{F}2}^{-1})
    \rhd\mathrm{pr}(X)}
    ((\mathcal{F}_{2(2)}^{-1}\mathcal{F}_2^{'-1})\rhd\mathrm{pr}(Z))\\
    &-\nabla^{\mathcal{A}/\mathcal{C}}_{[\mathcal{F}_{1\widehat{(1)}}^{-1}\rhd
    \mathrm{pr}(X),
    \mathcal{F}_{1\widehat{(2)}}^{-1}\rhd\mathrm{pr}(Y)]_{
    \mathcal{R}_\mathcal{F}}}
    (\mathcal{F}_2^{-1}\rhd\mathrm{pr}(Z))\\
    =&\nabla^{\mathcal{A}/\mathcal{C}}_{\mathcal{F}_1^{-1}\rhd\mathrm{pr}(X)}
    \nabla^{\mathcal{A}/\mathcal{C}}_{
    (\mathcal{F}_{2(1)}^{-1}\mathcal{F}_1^{'-1})\rhd\mathrm{pr}(Y)}
    ((\mathcal{F}_{2(2)}^{-1}\mathcal{F}_2^{'-1})\rhd\mathrm{pr}(Z))\\
    &-\nabla^{\mathcal{A}/\mathcal{C}}_{
    (\mathcal{F}_{1(1)}^{-1}\mathcal{F}_1^{'-1}\mathcal{R}_{\mathcal{F}1}^{-1})
    \rhd\mathrm{pr}(Y)}
    \nabla^{\mathcal{A}/\mathcal{C}}_{
    (\mathcal{F}_{1(2)}^{-1}\mathcal{F}_2^{'-1}\mathcal{R}_{\mathcal{F}2}^{-1})
    \rhd\mathrm{pr}(X)}
    (\mathcal{F}_{2}^{-1}\rhd\mathrm{pr}(Z))\\
    &-\nabla^{\mathcal{A}/\mathcal{C}}_{[
    (\mathcal{F}_{1(1)}^{-1}\mathcal{F}_1^{'-1})\rhd
    \mathrm{pr}(X),
    (\mathcal{F}_{1(2)}^{-1}\mathcal{F}_2^{'-1})\rhd\mathrm{pr}(Y)]_\mathcal{R}}
    (\mathcal{F}_2^{-1}\rhd\mathrm{pr}(Z))\\
    =&\nabla^{\mathcal{A}/\mathcal{C}}_{\mathcal{F}_1^{-1}\rhd\mathrm{pr}(X)}
    \nabla^{\mathcal{A}/\mathcal{C}}_{
    (\mathcal{F}_{2(1)}^{-1}\mathcal{F}_1^{'-1})\rhd\mathrm{pr}(Y)}
    ((\mathcal{F}_{2(2)}^{-1}\mathcal{F}_2^{'-1})\rhd\mathrm{pr}(Z))\\
    &-\nabla^{\mathcal{A}/\mathcal{C}}_{
    (\mathcal{R}_{1}^{-1}\mathcal{F}_{1(2)}^{-1}\mathcal{F}_2^{'-1})
    \rhd\mathrm{pr}(Y)}
    \nabla^{\mathcal{A}/\mathcal{C}}_{
    (\mathcal{R}_{\mathcal{F}1}^{-1}\mathcal{F}_{1(1)}^{-1}\mathcal{F}_2^{'-1})
    \rhd\mathrm{pr}(X)}
    (\mathcal{F}_{2}^{-1}\rhd\mathrm{pr}(Z))\\
    &-\nabla^{\mathcal{A}/\mathcal{C}}_{[
    (\mathcal{F}_{1(1)}^{-1}\mathcal{F}_1^{'-1})\rhd
    \mathrm{pr}(X),
    (\mathcal{F}_{1(2)}^{-1}\mathcal{F}_2^{'-1})\rhd\mathrm{pr}(Y)]_\mathcal{R}}
    (\mathcal{F}_2^{-1}\rhd\mathrm{pr}(Z))\\
    =&R^{\mathcal{A}/\mathcal{C}}(
    (\mathcal{F}_{1(1)}^{-1}\mathcal{F}_1^{'-1})\rhd\mathrm{pr}(X),
    (\mathcal{F}_{1(2)}^{-1}\mathcal{F}_2^{'-1})\rhd\mathrm{pr}(Y))
    (\mathcal{F}_2^{-1}\rhd\mathrm{pr}(Z))
\end{align*}
\end{allowdisplaybreaks}
and
\begin{allowdisplaybreaks}
\begin{align*}
    \mathrm{Tor}^{(\nabla^{\mathcal{A}/\mathcal{C}})^\mathcal{F}}
    (\mathrm{pr}(X),\mathrm{pr}(Y))
    =&(\nabla^{\mathcal{A}/\mathcal{C}})^\mathcal{F}_{\mathrm{pr}(X)}
    \mathrm{pr}(Y)
    -(\nabla^{\mathcal{A}/\mathcal{C}})^\mathcal{F}_{
    \mathcal{R}_{\mathcal{F}1}^{-1}\rhd\mathrm{pr}(Y)}
    (\mathcal{R}_{\mathcal{F}2}^{-1}\rhd\mathrm{pr}(X))\\
    &-[\mathrm{pr}(X),\mathrm{pr}(Y)]_{\mathcal{R}_\mathcal{F}}\\
    =&\nabla^{\mathcal{A}/\mathcal{C}}_{\mathcal{F}_1^{-1}\rhd\mathrm{pr}(X)}
    (\mathcal{F}_2^{-1}\rhd\mathrm{pr}(Y))\\
    &-\nabla^{\mathcal{A}/\mathcal{C}}_{
    (\mathcal{F}_1^{-1}\mathcal{R}_{\mathcal{F}1}^{-1})\rhd\mathrm{pr}(Y)}
    ((\mathcal{F}_2^{-1}\mathcal{R}_{\mathcal{F}2}^{-1})\rhd\mathrm{pr}(X))\\
    &-[\mathcal{F}_1^{-1}\rhd\mathrm{pr}(X),
    \mathcal{F}_2^{-1}\rhd\mathrm{pr}(Y)]_\mathcal{R}\\
    =&\nabla^{\mathcal{A}/\mathcal{C}}_{\mathcal{F}_1^{-1}\rhd\mathrm{pr}(X)}
    (\mathcal{F}_2^{-1}\rhd\mathrm{pr}(Y))\\
    &-\nabla^{\mathcal{A}/\mathcal{C}}_{
    (\mathcal{R}_{1}^{-1}\mathcal{F}_2^{-1})\rhd\mathrm{pr}(Y)}
    ((\mathcal{R}_{2}^{-1}\mathcal{F}_1^{-1})\rhd\mathrm{pr}(X))\\
    &-[\mathcal{F}_1^{-1}\rhd\mathrm{pr}(X),
    \mathcal{F}_2^{-1}\rhd\mathrm{pr}(Y)]_\mathcal{R}\\
    =&\mathrm{Tor}^{\nabla^{\mathcal{A}/\mathcal{C}}}
    (\mathcal{F}_1^{-1}\rhd\mathrm{pr}(X),
    \mathcal{F}_2^{-1}\rhd\mathrm{pr}(Y))
\end{align*}
\end{allowdisplaybreaks}
follow, where we viewed $(\nabla^{\mathcal{A}/\mathcal{C}})^\mathcal{F}$
as a covariant derivative with respect to $\mathcal{R}_\mathcal{F}$ 
(see Proposition~\ref{prop12}) and we identified $[\cdot,\cdot]_{
\mathcal{R}_\mathcal{F}}$ with the twisted commutator as in 
Proposition~\ref{prop13}.
\end{proof}
This completes the discussion of the commutative diagram
(\ref{DiagramSubmanifolds}). We conclude the chapter with the study
of an explicit example of twist deformation quantization on a smooth
submanifold.

\section{Twist Deformation of Quadric Surfaces}\label{Sec4.4}

The purpose of this section is to suggest an explicitly construction scheme for twisted 
Cartan calculi and exemplifying this by elaborating one example.
The strategy is to consider some classes of submanifolds of $\mathbb{R}^D$,
find suitable symmetries which inherit explicit Drinfel'd twists that also respect the
submanifolds and project the twisted Cartan calculi to the submanifolds.
This should illustrate the utility of the abstract machinery we developed in the
previous sections. The reason not to consider a deformation of the submanifolds
from the beginning but rather performing a detour, is that the Cartan calculus on
$\mathbb{R}^D$
is much easier to handle. Furthermore, the submanifolds are given in terms
of relations in coordinates of $\mathbb{R}^D$. From our previous results we know
that projection and twist deformation commute, so we are able to
first quantize $\mathbb{R}^D$ and pass to the submanifolds afterwards.
Let us quickly recall the notion of symmetries for $\mathbb{R}^D$.
Consider global coordinates
$x=(x^1,\ldots,x^D)$ of $\mathbb{R}^D$ together with the dual basis
$(\partial_1,\ldots,\partial_D)$
of vector fields with ${}^*$-involutions $(x^i)^*=x^i$ and $\partial_i^*=-\partial_i$.
Assume that there is a $\mathscr{U}\mathfrak{g}$-module ${}^*$-algebra action
$\rhd\colon\mathscr{U}\mathfrak{g}\otimes\mathscr{C}^\infty(\mathbb{R}^D)\rightarrow
\mathscr{C}^\infty(\mathbb{R}^D)$ for a complex Lie ${}^*$-algebra $\mathfrak{g}$.
The induced $\mathscr{U}\mathfrak{g}$-actions on
$X=
X^{i_1\cdots i_k}
\partial_{i_1}\wedge\ldots\wedge\partial_{i_k}
\in\mathfrak{X}^k(\mathbb{R}^D)$
and
$\omega=
\omega_{i_1\cdots i_k}
\mathrm{d}x^{i_1}\wedge\ldots\wedge\mathrm{d}x^{i_k}
\in\Omega^k(\mathbb{R}^D)$ are
$$
\xi\rhd X
=(\xi_{(1)}\rhd X^{i_1\cdots i_k})
(\xi_{(2)}\rhd\partial_{i_1})\wedge\ldots\wedge
(\xi_{(k+1)}\rhd\partial_{i_k})\in\mathfrak{X}^k(\mathbb{R}^D)
$$
and
$$
\xi\rhd\omega
=(\xi_{(1)}\rhd\omega_{i_1\cdots i_k})
(\mathrm{d}(\xi_{(2)}\rhd x^{i_1}))\wedge\ldots\wedge
(\mathrm{d}(\xi_{(2)}\rhd x^{i_k}))
\in\Omega^k(\mathbb{R}^D)
$$
for all $\xi\in\mathscr{U}\mathfrak{g}$, respectively, where
$\xi\rhd\partial_i\in\mathfrak{X}^1(\mathbb{R}^D)$ is defined by
$$
\mathscr{L}_{\xi\rhd\partial_i}f
=\xi_{(1)}\rhd\big(\mathscr{L}_{\partial_i}(S(\xi_{(2)})\rhd f)\big)
\in\mathscr{C}^\infty(\mathbb{R}^D)
$$
for $f\in\mathscr{C}^\infty(\mathbb{R}^D)$.
As symmetries $\mathfrak{g}$ of $\mathbb{R}^D$, we choose a finite-dimensional
Lie ${}^*$-subalgebra of $\mathfrak{X}^1(\mathbb{R}^D)$ with module algebra action
$\rhd\colon\mathscr{U}\mathfrak{g}\otimes\mathscr{C}^\infty(\mathbb{R}^D)
\rightarrow\mathscr{C}^\infty(\mathbb{R}^D)$, given by the Lie derivative
$\mathscr{L}$. In fact, for all $\xi,\eta\in\mathscr{U}\mathfrak{g}$ one has
$\mathscr{L}_{\xi\eta}f=\mathscr{L}_\xi(\mathscr{L}_\eta f)$,
$\mathscr{L}_1f=f$
and
$
\mathscr{L}_\xi(fg)=(\mathscr{L}_{\xi_{(1)}}f)(\mathscr{L}_{\xi_{(2)}}g),~
\mathscr{L}_\xi 1=\epsilon(\xi)1
$
for all $f,g\in\mathscr{C}^\infty(\mathbb{R}^D)$, which is easily verified
on primitive elements.
Now, let us turn to the submanifolds we are interested in.
Consider a smooth
function $F\colon\mathbb{R}^D\rightarrow\mathbb{R}$ having zero as a regular value.
According to the regular value theorem (see e.g. \cite{Lee2003}~Cor.~5.24),
the zero set $N=F^{-1}(\{0\})$ is a
closed embedded submanifold of dimension $D-1$. We denote the embedding by
$\iota\colon N\hookrightarrow \mathbb{R}^D$ and the corresponding vanishing ideal
of functions by $\mathcal{C}$. As an additional assumption we suppose that $N$
is closed under the ${}^*$-involution. The surjective projection
$\mathrm{pr}\colon\Omega^\bullet(\mathbb{R}^D)\rightarrow\Omega^\bullet(N)$ is
given by the pullback of $\iota$. On functions it reads
$\mathrm{pr}\colon\mathscr{C}^\infty(\mathbb{R}^D)\ni f\mapsto f+\mathcal{C}
\in\mathscr{C}^\infty(N)$. A vector field $X\in\mathfrak{X}^1(\mathbb{R}^D)$
on $\mathbb{R}^D$ is tangent to $N$ if and only if its action on functions respects
the vanishing ideal, i.e. if and only if $\mathscr{L}_X\mathcal{C}\subseteq\mathcal{C}$.
As usual we write $X\in\mathfrak{X}^1_t(\mathbb{R}^D)$
in this case. The Lie ${}^*$-algebra $\mathfrak{X}^1_t(\mathbb{R}^D)$ can be projected
to $\mathfrak{X}^1(N)$ by assigning to every tangent vector field on $X\in\mathfrak{X}^1_t(\mathbb{R}^D)$ 
the unique $\iota$-related vector field on $N$ (c.f. \cite{Lee2003}~Lem.~5.39).
Geometrically one might think of
this $\iota$-related vector field as the restriction of $X$ to $\iota(N)$.
On the other hand one might view this projection as assigning to $X$ its
equivalence class consisting of all vector fields on $\mathbb{R}^D$ that coincide
with $X$ up to vector fields $X_0\in\mathfrak{X}^1(\mathbb{R}^D)$ satisfying
$\mathscr{L}_{X_0}\mathscr{C}^\infty(\mathbb{R}^D)\subseteq\mathcal{C}$.
More generally, $\mathfrak{X}^\bullet_t(\mathbb{R}^D)$ is the
Gerstenhaber algebra of multivector fields on $\mathbb{R}^D$ which are tangent to $N$
and $\mathrm{pr}\colon\mathfrak{X}^\bullet_t(\mathbb{R}^D)
\rightarrow\mathfrak{X}^\bullet(N)$ denotes the surjective projection
with kernel $\mathfrak{X}^\bullet_0(\mathbb{R}^D)$.
Since we are interested in quantizing the submanifold $N$
we have to require $\mathfrak{g}\subseteq\mathfrak{X}^1_t(\mathbb{R}^D)
\subseteq\mathfrak{X}^1(\mathbb{R}^D)$. In other words, we have
to choose $\mathfrak{g}$ such that
$
\mathscr{L}_\mathfrak{g}\mathcal{C}\subseteq\mathcal{C}.
$
If this is achieved, the extension to the universal enveloping algebra automatically
satisfies
$
\mathscr{L}_{\mathscr{U}\mathfrak{g}}\mathcal{C}\subseteq\mathcal{C},
$
giving a well-defined $\mathscr{U}\mathfrak{g}$-module ${}^*$-algebra action on
$\mathscr{C}^\infty(\mathbb{R}^D)$ that projects to a $\mathscr{U}\mathfrak{g}$-module
${}^*$-algebra action on $\mathscr{C}^\infty(N)$. From now on $D=3$. We are going to
discuss a twist quantization of the $2$-sheet elliptic hyperboloid, which is
a quadric surface of $\mathbb{R}^3$. The cases of the $1$-sheet elliptic hyperboloid and the
elliptic cone are entirely similar and likewise all quadric surfaces of $\mathbb{R}^3$ 
admit a twist quantization (see \cite{GaetanoThomas19}). Furthermore, we limit our consideration to
the Cartan calculus and its twist deformation and only mention that functions,
vector fields and differential forms on the submanifold are determined by relations
that admit twist quantization, such that the latter controll the twisted objects.
This point of view is immersed in \cite{GaetanoThomas19}.

\subsection*{$2$-Sheet Elliptic Hyperboloid}

Let $a,c>0$ be two positive parameters. The zero set
$N=f^{-1}_{2EH}(\{0\})$, where
\begin{equation}
    f_{2EH}(x)
    =\frac{1}{2}x^1x^3+\frac{a}{2}(x^2)^2+c
\end{equation}
is said to be the $2$-sheet elliptic hyperboloid
in light-like coordinates. It is a closed embedded
smooth submanifold of $\mathbb{R}^3$ according to the regular value theorem.
It is obtained from the more common normal form
$f_{\mathrm{EC}}(y)=\frac{1}{2}((y^1)^2+a(y^2)^2-(y^2)^2)+c$ of the defining
equation via the coordinate transformation
\begin{equation}\label{eq54}
    x^1
    =y^1+y^3,~
    x^2
    =y^2,~
    x^3
    =y^1-y^3
\end{equation}
from Cartesian coordinates $(y^1,y^2,y^3)$. The three vector fields 
\begin{align*}
    H=&2x^1\partial_1-2x^3\partial_3,\\
    E=&\frac{1}{\sqrt{a}}x^1\partial_2-2\sqrt{a}x^2\partial_3,\\
    E'=&\frac{1}{\sqrt{a}}x^3\partial_2-2\sqrt{a}x^2\partial_1
\end{align*}
of $\mathbb{R}^3$ satisfy
\begin{align*}
    [H,E]=2E,~
    [H,E']=-2E'
    \text{ and }
    [E',E]=H.
\end{align*}
They span the Lie ${}^*$-algebra $\mathfrak{g}=\mathfrak{so}(2,1)$
and provide a basis of the tangent vector fields $\mathfrak{X}^1_t(\mathbb{R}^3)$,
since
$$
H(f_{\mathrm{2EH}})=E(f_{\mathrm{2EH}})=E'(f_{\mathrm{2EH}})=0.
$$
Then, according to Example~\ref{example05}~ii.)
\begin{equation}\label{eq51}
    \mathcal{F}=\exp\bigg(\frac{H}{2}\otimes\log(1+\mathrm{i}\hbar E)\bigg)
    \in(\mathscr{U}\mathfrak{g}\otimes\mathscr{U}\mathfrak{g})[[\hbar]]
\end{equation}
is a unitary Jordanian Drinfel'd twist and Theorem~\ref{thm07}, also in the form of
eq.(\ref{eq53}), implies the following proposition.
\begin{proposition}
The twist star product
$$
f\star_\mathcal{F}g
=(\mathcal{F}_1^{-1}\rhd f)(\mathcal{F}_2^{-1}\rhd g),
\text{ where }f,g\in\mathscr{C}^\infty(\mathbb{R}^3),
$$
induced by the unitary twist (\ref{eq51})
projects to a twist star product on the $2$-sheet elliptic hyperboloid $N$,
i.e. $f\star_\mathcal{F}g\in\mathscr{C}^\infty(N)[[\hbar]]$ for all
$f,g\in\mathscr{C}^\infty(N)$. Note that
$(\mathscr{C}^\infty(N)[[\hbar]],\star_\mathcal{F})$ is a ${}^*$-algebra with
${}^*$-involution
\begin{equation}
    f^{*_\mathcal{F}}
    =S(\beta)\rhd\overline{f}
\end{equation}
for all $f\in\mathscr{C}^\infty(N)$, where $\beta=\mathcal{F}_1S(\mathcal{F}_2)$.
Furthermore, we obtain twist deformations 
$(\mathfrak{X}^\bullet(N)_\mathcal{F},\wedge_\mathcal{F},
\llbracket\cdot,\cdot\rrbracket_\mathcal{F})$ and
$(\Omega^\bullet(N)_\mathcal{F},\wedge_\mathcal{F})$ as projections from $\mathbb{R}^3$.
\end{proposition}
There is an explicit description of the twist
deformed Hopf ${}^*$-algebra structure of $\mathscr{U}\mathfrak{g}_\mathcal{F}$
(compare also to \cite{GiZh98}). 
Using the commutation relations of $H,E,E'$ as well as the series expansions
\begin{equation}
\begin{split}
    \log(1+\mathrm{i}\hbar E)
    =&-\sum_{n=1}^\infty\frac{(-\mathrm{i}\hbar E)^n}{n},\\
    \frac{1}{(1+\mathrm{i}\hbar E)^2}
    =&\sum_{n=1}^\infty n(-\mathrm{i}\hbar E)^{n-1}
\end{split}
\hspace{1cm}
\begin{split}
    \frac{1}{1+\mathrm{i}\hbar E}
    =\sum_{n=0}^\infty(-\mathrm{i}\hbar E)^n,
\end{split}
\end{equation}
we prove some preliminary equations.
\begin{lemma}
For all $n\geq 0$
\begin{equation}
\begin{split}
    E^n
    =&(H-2n)E^n,\\
    \bigg(\frac{H}{2}\bigg)^nE
    =&E\bigg(\frac{H}{2}+1\bigg)^n,
\end{split}
\hspace{1cm}
\begin{split}
    E^{'n}
    =&(H+2n)E^{'n},\\
    \bigg(\frac{H}{2}\bigg)^nE'
    =&E'\bigg(\frac{H}{2}-1\bigg)^n
\end{split}
\end{equation}
hold. Furthermore we obtain
\begin{equation}
    \log(1+\mathrm{i}\hbar E)^nH
    =H\log(1+\mathrm{i}\hbar E)^n
    -2\mathrm{i}\hbar n\frac{E}{1+\mathrm{i}\hbar E}\log(1+\mathrm{i}\hbar E)^{n-1}
\end{equation}
and
\begin{allowdisplaybreaks}
\begin{align*}
    \log(1+\mathrm{i}\nu E)^nE'
    =&E'\log(1+\mathrm{i}\hbar E)^{n}\\
    &-\mathrm{i}\hbar nH\log(1+\mathrm{i}\hbar E)^{n-1}\frac{1}{1+\mathrm{i}\hbar E}\\
    &+n\hbar^2\frac{E}{(1+\mathrm{i}\hbar E)^2}\log(1+\mathrm{i}\hbar E)^{n-1}\\
    &+\hbar^2n(n-1)\frac{E}{(1+\mathrm{i}\hbar E)^2}\log(1+\mathrm{i}\hbar E)^{n-2}
\end{align*}
\end{allowdisplaybreaks}
for all $n\geq 0$.
\end{lemma}
\begin{proof}
First remark that
\begin{align*}
    E^nH
    =E^{n-1}(H-2)E
    =(H-2n)E^n,
\end{align*}
where $n\geq 0$, implies
\begin{align*}
    \log(1+\mathrm{i}\hbar E)H
    =&-\sum_{n=1}^\infty\frac{(-\mathrm{i}\hbar)^n}{n}E^nH\\
    =&-\sum_{n=1}^\infty\frac{(-\mathrm{i}\hbar)^n}{n}(H-2n)E^n\\
    =&H\log(1+\mathrm{i}\hbar E)
    -2\mathrm{i}\hbar E\sum_{n=1}^\infty(-\mathrm{i}\hbar E)^{n-1}\\
    =&H\log(1+\mathrm{i}\hbar E)
    -2\mathrm{i}\hbar\frac{E}{1+\mathrm{i}\hbar E}.
\end{align*}
Inductively, this leads to
\begin{align*}
    \log(1+\mathrm{i}\hbar E)^nH
    =&\log(1+\mathrm{i}\hbar E)^{n-1}\bigg(
    H\log(1+\mathrm{i}\hbar E)
    -2\mathrm{i}\hbar\frac{E}{1+\mathrm{i}\hbar E}
    \bigg)\\
    =&H\log(1+\mathrm{i}\hbar E)^n
    -2\mathrm{i}\hbar n\frac{E}{1+\mathrm{i}\hbar E}\log(1+\mathrm{i}\hbar E)^{n-1}
\end{align*}
for all $n\geq 0$. Furthermore
\begin{align*}
    E^nE'
    =&E^{n-1}(E'E-H)\\
    =&E^{n-1}E'E
    +(2(n-1)-H)E^{n-1}\\
    =&E'E^n
    +(2((n-1)+(n-2)+\cdots+1)-nH)E^{n-1}\\
    =&E'E^n+n(n-1)E^{n-1}-nHE^{n-1},
\end{align*}
where we employed the "little Gauß" $\sum_{n=1}^nn=\frac{n(n+1)}{2}$.
Then
\begin{allowdisplaybreaks}
\begin{align*}
    \log(1+\mathrm{i}\hbar E)E'
    =&-\sum_{n=1}^\infty\frac{(-\mathrm{i}\hbar)^n}{n}E^nE'\\
    =&-\sum_{n=1}^\infty\frac{(-\mathrm{i}\hbar)^n}{n}
    (E'E^n+n(n-1)E^{n-1}-nHE^{n-1})\\
    =&E'\log(1+\mathrm{i}\hbar E)
    -\mathrm{i}\hbar H\frac{1}{1+\mathrm{i}\hbar E}
    +\mathrm{i}\hbar\sum_{n=1}^\infty(-\mathrm{i}\hbar)^{n-1}(n-1)E^{n-1}\\
    =&E'\log(1+\mathrm{i}\hbar E)
    -\mathrm{i}\hbar H\frac{1}{1+\mathrm{i}\hbar E}
    +\hbar^2E\sum_{n=2}^\infty(-\mathrm{i}\hbar)^{n-2}(n-1)E^{n-2}\\
    =&E'\log(1+\mathrm{i}\hbar E)
    -\mathrm{i}\hbar H\frac{1}{1+\mathrm{i}\hbar E}
    +\hbar^2\frac{E}{(1+\mathrm{i}\hbar E)^2}
\end{align*}
\end{allowdisplaybreaks}
and inductively, for $n>1$ we obtain
\begin{allowdisplaybreaks}
\begin{align*}
    \log(1+\mathrm{i}\hbar E)^nE'
    =&\log(1+\mathrm{i}\hbar E)^{n-1}\bigg(
    E'\log(1+\mathrm{i}\hbar E)\\
    &-\mathrm{i}\hbar H\frac{1}{1+\mathrm{i}\hbar E}
    +\hbar^2\frac{E}{(1+\mathrm{i}\hbar E)^2}
    \bigg)\\
    =&\log(1+\mathrm{i}\hbar E)^{n-1}E'\log(1+\mathrm{i}\hbar E)\\
    &-\mathrm{i}\hbar\log(1+\mathrm{i}\hbar E)^{n-1}H\frac{1}{1+\mathrm{i}\hbar E}\\
    &+\hbar^2\frac{E}{(1+\mathrm{i}\hbar E)^2}\log(1+\mathrm{i}\hbar E)^{n-1}\\
    =&\log(1+\mathrm{i}\hbar E)^{n-1}E'\log(1+\mathrm{i}\hbar E)\\
    &-\mathrm{i}\hbar\bigg(
    H\log(1+\mathrm{i}\hbar E)^{n-1}\\
    &-2\mathrm{i}\hbar(n-1)\frac{E}{1+\mathrm{i}\hbar E}\log(1+\mathrm{i}\hbar E)^{n-2}
    \bigg)\frac{1}{1+\mathrm{i}\hbar E}\\
    &+\hbar^2\frac{E}{(1+\mathrm{i}\hbar E)^2}\log(1+\mathrm{i}\hbar E)^{n-1}\\
    =&\log(1+\mathrm{i}\hbar E)^{n-1}E'\log(1+\mathrm{i}\hbar E)\\
    &-\mathrm{i}\hbar H\log(1+\mathrm{i}\hbar E)^{n-1}\frac{1}{1+\mathrm{i}\hbar E}\\
    &+2\hbar^2(n-1)\frac{E}{(1+\mathrm{i}\hbar E)^2}\log(1+\mathrm{i}\hbar E)^{n-2}\\
    &+\hbar^2\frac{E}{(1+\mathrm{i}\hbar E)^2}\log(1+\mathrm{i}\hbar E)^{n-1}\\
    =&E'\log(1+\mathrm{i}\hbar E)^{n}\\
    &-\mathrm{i}\hbar nH\log(1+\mathrm{i}\hbar E)^{n-1}\frac{1}{1+\mathrm{i}\hbar E}\\
    &+2\hbar^2((n-1)+(n-2)+\cdots+1)
    \frac{E}{(1+\mathrm{i}\hbar E)^2}\log(1+\mathrm{i}\hbar E)^{n-2}\\
    &+n\hbar^2\frac{E}{(1+\mathrm{i}\hbar E)^2}\log(1+\mathrm{i}\hbar E)^{n-1}\\
    =&E'\log(1+\mathrm{i}\hbar E)^{n}\\
    &-\mathrm{i}\hbar nH\log(1+\mathrm{i}\hbar E)^{n-1}\frac{1}{1+\mathrm{i}\hbar E}\\
    &+\hbar^2n(n-1)\frac{E}{(1+\mathrm{i}\hbar E)^2}\log(1+\mathrm{i}\hbar E)^{n-2}\\
    &+n\hbar^2\frac{E}{(1+\mathrm{i}\hbar E)^2}\log(1+\mathrm{i}\hbar E)^{n-1}.
\end{align*}
\end{allowdisplaybreaks}
\end{proof}
\begin{lemma}
The twisted coproduct and antipode of $\mathscr{U}\mathfrak{g}_\mathcal{F}$
are given by
\begin{allowdisplaybreaks}
\begin{align*}
    \Delta_\mathcal{F}(H)
    =&\Delta(H)-\mathrm{i}\hbar\bigg(H\otimes\frac{E}{1+\mathrm{i}\hbar E}\bigg),\\
    \Delta_\mathcal{F}(E)
    =&\Delta(E)+\mathrm{i}\hbar E\otimes E,\\
    \Delta_\mathcal{F}(E')
    =&1\otimes E'
    +E'\otimes\frac{1}{1+\mathrm{i}\hbar E}
    -\frac{\mathrm{i}\hbar}{2}\bigg(H\otimes H\frac{1}{1+\mathrm{i}\hbar E}\bigg)\\
    &+\frac{\hbar^2}{2}\bigg(
    H\bigg(\frac{H}{2}+1\bigg)\otimes\frac{E}{(1+\mathrm{i}\hbar E)^2}
    \bigg)
\end{align*}
\end{allowdisplaybreaks}
and
\begin{allowdisplaybreaks}
\begin{align*}
    S_\mathcal{F}(H)
    =&S(H)(1+\mathrm{i}\hbar E),\\
    S_\mathcal{F}(E)
    =&\frac{S(E)}{1+\mathrm{i}\hbar E},\\
    S_\mathcal{F}(E')
    =&S(E')(1+\mathrm{i}\hbar E)
    -\frac{\mathrm{i}\hbar}{2}H^2
    +\frac{\hbar^2}{2}\bigg(\frac{H}{2}-1\bigg)HE
    +\frac{\mathrm{i}\hbar^3}{2}\bigg(1-\frac{H}{2}\bigg)HE^2,
\end{align*}
\end{allowdisplaybreaks}
respectively.
\end{lemma}
\begin{proof}
Note that $(H\otimes 1)$ commutes with $\mathcal{F}$. However
\begin{align*}
    \mathcal{F}(1\otimes H)
    =&\sum_{n=0}^\infty\frac{1}{n!}\bigg(\frac{H}{2}\bigg)^n\otimes
    \log(1+\mathrm{i}\hbar E)^nH\\
    =&\sum_{n=0}^\infty\frac{1}{n!}\bigg(\frac{H}{2}\bigg)^n\otimes
    \bigg(
    H\log(1+\mathrm{i}\hbar E)^n
    -2\mathrm{i}\hbar n\frac{E}{1+\mathrm{i}\hbar E}\log(1+\mathrm{i}\hbar E)^{n-1}
    \bigg)\\
    =&(1\otimes H)\mathcal{F}
    -\mathrm{i}\hbar\bigg(H\otimes\frac{E}{1+\mathrm{i}\hbar E}\bigg)\mathcal{F}
\end{align*}
only commutes up to the second term, proving
\begin{equation*}
    \Delta_\mathcal{F}(H)
    =\Delta(H)-\mathrm{i}\hbar\bigg(H\otimes\frac{E}{1+\mathrm{i}\hbar E}\bigg),
\end{equation*}
since $H$ is primitive, i.e. $\Delta(H)=H\otimes 1+1\otimes H$
and $\Delta_\mathcal{F}(H)=\mathcal{F}\Delta(H)\mathcal{F}^{-1}$.
On the other hand $1\otimes E$ commutes with $\mathcal{F}$, while
\begin{allowdisplaybreaks}
\begin{align*}
    \mathcal{F}(E\otimes 1)
    =&\sum_{n=0}^\infty\frac{1}{n!}\bigg(\frac{H}{2}\bigg)^nE
    \otimes\log(1+\mathrm{i}\hbar E)^n\\
    =&(E\otimes 1)\sum_{n=0}^\infty\frac{1}{n!}\bigg(\frac{H}{2}+1\bigg)^n
    \otimes\log(1+\mathrm{i}\hbar E)^n\\
    =&(E\otimes 1)\exp\bigg(\bigg(\frac{H}{2}+1\bigg)\otimes\log(1+\mathrm{i}\hbar E)\bigg).
\end{align*}
\end{allowdisplaybreaks}
Then
\begin{align*}
    \Delta_\mathcal{F}(E)
    =&1\otimes E\\
    &+(E\otimes 1)\exp\bigg(\bigg(\frac{H}{2}+1\bigg)\otimes\log(1+\mathrm{i}\hbar E)\bigg)
    \exp\bigg(-\frac{H}{2}\otimes\log(1+\mathrm{i}\hbar E)\bigg)\\
    =&1\otimes E
    +(E\otimes 1)\exp\bigg(\bigg(\frac{H}{2}+1\bigg)\otimes\log(1+\mathrm{i}\hbar E)
    -\frac{H}{2}\otimes\log(1+\mathrm{i}\hbar E)\bigg)\\
    =&1\otimes E
    +(E\otimes 1)\exp(1\otimes\log(1+\mathrm{i}\hbar E))\\
    =&\Delta(E)+\mathrm{i}\hbar E\otimes E,
\end{align*}
where we used in the second equation that the exponents commute,
leading to a trivial BCH series. Similarly to the last computation we obtain
\begin{align*}
    \mathcal{F}(E'\otimes 1)
    =(E'\otimes 1)
    \exp\bigg(\bigg(\frac{H}{2}-1\bigg)\otimes\log(1+\mathrm{i}\hbar E)\bigg),
\end{align*}
which implies
\begin{align*}
    \mathcal{F}(E'\otimes 1)\mathcal{F}^{-1}
    =(E'\otimes 1)\exp(-1\otimes\log(1+\mathrm{i}\hbar E))
    =E'\otimes\frac{1}{1+\mathrm{i}\hbar E}.
\end{align*}
On the other hand
\begin{allowdisplaybreaks}
\begin{align*}
    \mathcal{F}(1\otimes E')
    =&\sum_{n=0}^\infty\frac{1}{n!}\bigg(\frac{H}{2}\bigg)^n\otimes
    \log(1+\mathrm{i}\hbar E)^nE'\\
    =&\sum_{n=0}^\infty\frac{1}{n!}\bigg(\frac{H}{2}\bigg)^n\otimes
    \bigg(
    E'\log(1+\mathrm{i}\hbar E)^{n}
    -\mathrm{i}\hbar nH\log(1+\mathrm{i}\hbar E)^{n-1}\frac{1}{1+\mathrm{i}\hbar E}\\
    &+n\hbar^2\frac{E}{(1+\mathrm{i}\hbar E)^2}\log(1+\mathrm{i}\hbar E)^{n-1}\\
    &+\hbar^2n(n-1)\frac{E}{(1+\mathrm{i}\hbar E)^2}\log(1+\mathrm{i}\hbar E)^{n-2}
    \bigg)\\
    =&(1\otimes E')\mathcal{F}
    -\frac{\mathrm{i}\hbar}{2}\bigg(H\otimes H\frac{1}{1+\mathrm{i}\hbar E}\bigg)\mathcal{F}
    +\frac{\hbar^2}{2}\bigg(H\otimes\frac{E}{(1+\mathrm{i}\hbar E)^2}\bigg)\mathcal{F}\\
    &+\frac{\hbar^2}{4}\bigg(H^2\otimes\frac{E}{(1+\mathrm{i}\hbar E)^2}\bigg)\mathcal{F}\\
    =&(1\otimes E')\mathcal{F}
    -\frac{\mathrm{i}\hbar}{2}\bigg(H\otimes H\frac{1}{1+\mathrm{i}\hbar E}\bigg)\mathcal{F}\\
    &+\frac{\hbar^2}{2}\bigg(H\bigg(\frac{H}{2}+1\bigg)
    \otimes\frac{E}{(1+\mathrm{i}\hbar E)^2}\bigg)\mathcal{F},
\end{align*}
\end{allowdisplaybreaks}
leading to
\begin{align*}
    \mathcal{F}(1\otimes E')\mathcal{F}^{-1}
    =&(1\otimes E')
    -\frac{\mathrm{i}\hbar}{2}\bigg(H\otimes H\frac{1}{1+\mathrm{i}\hbar E}\bigg)\\
    &+\frac{\hbar^2}{2}\bigg(H\bigg(\frac{H}{2}+1\bigg)
    \otimes\frac{E}{(1+\mathrm{i}\hbar E)^2}\bigg).
\end{align*}
Combining this with $\mathcal{F}(E'\otimes 1)\mathcal{F}^{-1}$ we obtain the
formula for the twisted coproduct of $E'$.
For the twisted antipode $S_\mathcal{F}$ note that
$\xi_{\widehat{(1)}}S_\mathcal{F}(\xi_{\widehat{(2)}})
=\epsilon(\xi)1
=S_\mathcal{F}(\xi_{\widehat{(1)}})\xi_{\widehat{(2)}}$ for all
$\xi\in\mathscr{U}\mathfrak{g}$. Applying this to $\xi=H$ gives
\begin{align*}
    0
    =\epsilon(H)1
    =S_\mathcal{F}(H_{\widehat{(1)}})H_{\widehat{(2)}}
    =S_\mathcal{F}(H)+H-\mathrm{i}\hbar S_\mathcal{F}(H)\frac{E}{1+\mathrm{i}\hbar E},
\end{align*}
implying
\begin{align*}
    S(H)
    =S_\mathcal{F}(H)\bigg(1-\mathrm{i}\hbar\frac{E}{1+\mathrm{i}\hbar E}\bigg)
    =S_\mathcal{F}(H)\frac{1}{1+\mathrm{i}\hbar E}.
\end{align*}
Since $S(H)=-H$ and
$1+\mathrm{i}\hbar E$ is the inverse of $\frac{1}{1+\mathrm{i}\hbar E}$ this gives
$S_\mathcal{F}(H)=S(H)(1+\mathrm{i}\hbar E)$. Similarly,
$S_\mathcal{F}(E)=\frac{S(E)}{1+\mathrm{i}\hbar E}$ follows. Finally, for the
twisted antipode of $E'$ we obtain
\begin{align*}
    0
    =&\epsilon(E')1
    =S_\mathcal{F}(E'_{\widehat{(1)}})E'_{\widehat{(2)}}\\
    =&E'+S_\mathcal{F}(E')\frac{1}{1+\mathrm{i}\hbar E}
    -\frac{\mathrm{i}\hbar}{2}S_\mathcal{F}(H)H\frac{1}{1+\mathrm{i}\hbar E}
    +\frac{\hbar^2}{2}S_\mathcal{F}\bigg(H\bigg(\frac{H}{2}+1\bigg)\bigg)
    \frac{E}{(1+\mathrm{i}\hbar E)^2},
\end{align*}
which implies
\begin{allowdisplaybreaks}
\begin{align*}
    S_\mathcal{F}(E')
    =&S(E')(1+\mathrm{i}\hbar E)
    +\frac{\mathrm{i}\hbar}{2}S_\mathcal{F}(H)H
    -\frac{\hbar^2}{2}S_\mathcal{F}\bigg(H\bigg(\frac{H}{2}+1\bigg)\bigg)
    \frac{E}{1+\mathrm{i}\hbar E}\\
    =&S(E')(1+\mathrm{i}\hbar E)
    +\frac{\mathrm{i}\hbar}{2}S(H)(1+\mathrm{i}\hbar E)H\\
    &-\frac{\hbar^2}{2}
    \bigg(\frac{S(H)(1+\mathrm{i}\hbar E)}{2}+1\bigg)S(H)(1+\mathrm{i}\hbar E)
    \frac{E}{1+\mathrm{i}\hbar E}\\
    =&S(E')(1+\mathrm{i}\hbar E)
    -\frac{\mathrm{i}\hbar}{2}H^2(1+\mathrm{i}\hbar E)
    -\hbar^2HE\\
    &+\frac{\hbar^2}{2}
    \bigg(\frac{S(H)(1+\mathrm{i}\hbar E)}{2}+1\bigg)HE\\
    =&S(E')(1+\mathrm{i}\hbar E)
    -\frac{\mathrm{i}\hbar}{2}H^2
    +\frac{\hbar^2}{2}H^2E
    -\hbar^2HE
    +\frac{\hbar^2}{2}HE\\
    &+\frac{\hbar^2}{4}S(H)(1+\mathrm{i}\hbar E)HE\\
    =&S(E')(1+\mathrm{i}\hbar E)
    -\frac{\mathrm{i}\hbar}{2}H^2
    +\frac{\hbar^2}{2}H^2E
    -\hbar^2HE
    +\frac{\hbar^2}{2}HE\\
    &-\frac{\hbar^2}{4}H^2(1+\mathrm{i}\hbar E)E
    +\frac{\mathrm{i}\hbar^3}{2}HE^2\\
    =&S(E')(1+\mathrm{i}\hbar E)
    -\frac{\mathrm{i}\hbar}{2}H^2
    +\frac{\hbar^2}{2}H^2E
    -\hbar^2HE
    +\frac{\hbar^2}{2}HE\\
    &-\frac{\hbar^2}{4}H^2E
    -\frac{\mathrm{i}\hbar^3}{4}H^2E^2
    +\frac{\mathrm{i}\hbar^3}{2}HE^2\\
    =&S(E')(1+\mathrm{i}\hbar E)
    -\frac{\mathrm{i}\hbar}{2}H^2
    +\frac{\hbar^2}{2}\bigg(\frac{H}{2}-1\bigg)HE
    +\frac{\mathrm{i}\hbar^3}{2}\bigg(1-\frac{H}{2}\bigg)HE^2.
\end{align*}
\end{allowdisplaybreaks}
\end{proof}
Since $N$ is given in terms of relations of coordinate functions
$x^1,x^2,x^3$, it is worth to study their twist deformation in detail.
Note that $H\rhd x^i=\lambda_ix^i$, where we define $\lambda_1=2=-\lambda_3$ and
$\lambda_2=0$. Then
\begin{align*}
    (\mathcal{F}_1^{-1}\rhd x^i)\otimes
    (\mathcal{F}_2^{-1}\rhd x^j)
    =&\sum_{n=0}^\infty\frac{(-1)^n}{n!}
    \bigg(\bigg(\frac{H}{2}\bigg)^n\rhd x^i\bigg)
    \otimes(\log(1+\mathrm{i}\hbar E)^n\rhd x^j)\\
    =&\sum_{n=0}^\infty\frac{(-1)^n}{n!}
    \bigg(\frac{\lambda_i}{2}\bigg)^nx^i\otimes(\log(1+\mathrm{i}\hbar E)^n\rhd x^j)\\
    =&x^i\otimes((1+\mathrm{i}\hbar E)^{-\frac{\lambda_i}{2}}\rhd x^j)
\end{align*}
and
\begin{align*}
    (x^i)^{*_\mathcal{F}}
    =&S(\beta)\rhd(x^i)^*
    =(\mathcal{F}_2S(\mathcal{F}_1))\rhd x^i\\
    =&\sum_{n=0}^\infty\frac{(-1)^n}{n!}
    \bigg(\log(1+\mathrm{i}\hbar E)^n\bigg(\frac{H}{2}\bigg)^n\bigg)\rhd x^i\\
    =&(1+\mathrm{i}\hbar E)^{-\frac{\lambda_i}{2}}\rhd x^i
\end{align*}
follow. Since $E\rhd x^i=\delta^i_2\frac{1}{\sqrt{a}}x^1-2\delta^i_3\sqrt{a}x^2$,
$E^2\rhd x^i=\delta^i_3x^1$ and $E^n\rhd x^i=0$ for $n>2$ we conclude the
following result.
\begin{lemma}
The twisted star products of coordinate functions on $N$ are
\begin{align*}
    x^1\star_\mathcal{F}x^1
    =&(x^1)^2,\\
    x^1\star_\mathcal{F}x^2
    =&x^1x^2-\frac{\mathrm{i}\hbar}{\sqrt{a}}(x^1)^2,\\
    x^1\star_\mathcal{F} x^3
    =&x^1x^3+2\mathrm{i}\hbar\sqrt{a}x^1x^2-\hbar^2(x^1)^2,\\
    x^2\star_\mathcal{F} x^i
    =&x^2x^i,\text{ for all }1\leq i\leq 3\\
    x^3\star_\mathcal{F} x^1
    =&x^1x^3,\\
    x^3\star_\mathcal{F} x^2
    =&x^2x^3+\frac{\mathrm{i}\hbar}{\sqrt{a}}x^1x^3,\\
    x^3\star_\mathcal{F} x^3
    =&(x^3)^2-2\mathrm{i}\hbar\sqrt{a}x^2x^3.
\end{align*}
Furthermore, the twisted ${}^*$-involution on coordinate functions is given by
\begin{align*}
    (x^1)^{*_\mathcal{F}}=x^1,~
    (x^2)^{*_\mathcal{F}}=x^2
    \text{ and }
    (x^3)^{*_\mathcal{F}}=x^3-2\mathrm{i}\hbar\sqrt{a}x^2,
\end{align*}
The relation defining $N$ becomes
\begin{equation}
    \frac{1}{2}x^3\star_\mathcal{F}x^1+\frac{a}{2}x^2\star_\mathcal{F}x^2+c
    =0.
\end{equation}
\end{lemma}
Similarly one calculates the twisted wedge product of coordinate vector fields
and differential $1$-forms and determines the submanifold condition in terms
of the deformed generators (see \cite{GaetanoThomas19}). Instead we further examine the twisted
insertion and Lie derivative. Note that the de Rham differential is undeformed
since it commutes with the Hopf ${}^*$-algebra action.
\begin{lemma}
One coordinate vector fields the twisted insertion and Lie derivative read
\begin{align*}
    \mathrm{i}^\mathcal{F}_{\partial_i}\omega
    =&\mathrm{i}_{\partial_i}((1+\mathrm{i}\hbar E)^{\frac{\lambda_i}{2}}\rhd\omega),\\
    \mathscr{L}^\mathcal{F}_{\partial_i}\omega
    =&\mathscr{L}_{\partial_i}((1+\mathrm{i}\hbar E)^{\frac{\lambda_i}{2}}\rhd\omega),
\end{align*}
for all $\omega\in\Omega^\bullet(N)$. Using the left 
$\mathscr{C}^\infty(N)$-linearity of $\mathrm{i}^\mathcal{F}$ and
$\mathscr{L}^\mathcal{F}$ in the first argument as well as
$\mathrm{i}^\mathcal{F}_{X\wedge_\mathcal{F}Y}
=\mathrm{i}^\mathcal{F}_X\mathrm{i}^\mathcal{F}_Y$ and
$\mathscr{L}^\mathcal{F}_{X\wedge_\mathcal{F}Y}
=\mathrm{i}^\mathcal{F}_X\mathscr{L}^\mathcal{F}_Y
+(-1)^\ell\mathscr{L}^\mathcal{F}_X\mathrm{i}^\mathcal{F}_Y$
for all $X\in\mathfrak{X}^\bullet(N)$ and $Y\in\mathfrak{X}^\ell(N)$,
these formulas can be used to determine the action of higher multivector fields.
\end{lemma}
As a last observation we consider the Minkowski metric
\begin{equation}\label{eq52}
    {\bf g}=\frac{1}{2}(\mathrm{d}x^1\otimes\mathrm{d}x^3
    +\mathrm{d}x^3\otimes\mathrm{d}x^1)
    +\mathrm{d}x^2\otimes\mathrm{d}x^2
\end{equation}
in the case $a=b=1$, i.e. for the \textit{circular $2$-sheet elliptic hyperboloid}.
Note that (\ref{eq52}) is in fact the Minkowski metric, however not in Cartesian
coordinates but rather in the coordinates (\ref{eq54}).
\begin{proposition}
The Minkowski metric (\ref{eq52}) on the circular $2$-sheet elliptic hyperboloid $N$
is $\mathscr{U}\mathfrak{g}$-equivariant and admits a twist quantization 
\begin{equation}
    {\bf g}_\mathcal{F}
    =\frac{1}{2}(\mathrm{d}x^1\otimes_\mathcal{F}\mathrm{d}x^3
    +\mathrm{d}x^3\otimes_\mathcal{F}\mathrm{d}x^1)
    +\mathrm{d}x^2\otimes_\mathcal{F}\mathrm{d}x^2
    +2\mathrm{i}\hbar\sqrt{a}\mathrm{d}x^1\otimes_\mathcal{F}\mathrm{d}x^2
    +\hbar^2\mathrm{d}x^1\otimes_\mathcal{F}\mathrm{d}x^1
\end{equation}
The corresponding twisted Levi-Civita covariant derivative reads
\begin{equation}\label{eq56}
\begin{split}
    \nabla^\mathcal{F}_EH
    =&\nabla_EH+2\mathrm{i}\nu\nabla_EE,\\
    \nabla^\mathcal{F}_{E'}H
    =&\nabla_{E'}H-2\mathrm{i}\nu\nabla_{E'}E,\\
    \nabla^\mathcal{F}_EE'
    =&\nabla_EE'+\mathrm{i}\nu\nabla_EH-2\nu^2\nabla_EE,\\
    \nabla^\mathcal{F}_{E'}E'
    =&\nabla_{E'}E'
    -\mathrm{i}\nu\nabla_{E'}H
\end{split}
\end{equation}
on generators of $\mathfrak{g}$, while we are not listing the combinations that
keep undeformed. Using the left $\mathscr{C}^\infty(N)$-linearity in the first
argument and the braided Leibniz rule in the second argument these formulas
determine the twisted Levi-Civita covariant derivative on
higher order vector fields.
\end{proposition}
\begin{proof}
Let $X\in\mathfrak{X}^1_t(\mathbb{R}^3)$. Then
\begin{align*}
    (\mathcal{F}_1^{-1}\rhd E)\otimes(\mathcal{F}_2^{-1}\rhd X)
    =&\sum_{n=0}^\infty\frac{(-1)^n}{n!}\bigg(\bigg(\frac{H}{2}\bigg)^n\rhd E\bigg)
    \otimes(\log(1+\mathrm{i}\hbar E)^n\rhd X)\\
    =&\sum_{n=0}^\infty\frac{(-1)^n}{n!}
    E\otimes(\log(1+\mathrm{i}\hbar E)^n\rhd X)\\
    =&E\otimes((1+\mathrm{i}\hbar E)^{-1}\rhd X)
\end{align*}
and similarly
\begin{align*}
    (\mathcal{F}_1^{-1}\rhd E')\otimes(\mathcal{F}_2^{-1}\rhd X)
    =&E'\otimes((1+\mathrm{i}\hbar E)\rhd X),\\
    (\mathcal{F}_1^{-1}\rhd H)\otimes(\mathcal{F}_2^{-1}\rhd X)
    =&H\otimes X
\end{align*}
follow, implying (\ref{eq56}).
\end{proof}

%% file: chapters/appendixA.tex
This first appendix recalls some well-known concepts of category
theory which are used without reference throughout the thesis.
It is included for convenience of the reader and should be seen as
a reference section rather than a self-contained chapter. After
recalling the definitions of category, functor and natural transformation
we continue by discussing their (braided) monoidal versions. We close
this appendix by explaining rigidity and some fundamental
properties of rigid monoidal categories.
As sources we refer to \cite{Ka95}~Part~Three, \cite{Ma95}~Chap.~9
and \cite{ChPr94}~Chap.~5. For a general introduction to
category theory not only focusing on quantum groups see \cite{Aw10}.

A \textit{category} $\mathcal{C}$ consists of a class $\mathrm{ob}(\mathcal{C})$
of \textit{objects} and a class $\mathrm{hom}(\mathcal{C})$ of \textit{morphisms}.
For every morphisms $f$ of $\mathcal{C}$ there is a \textit{source} object $A$ and a
\textit{target} object $B$ and we say that $f$ is a morphism from $A$ to $B$,
writing $f\colon A\rightarrow B$. The class
of morphisms from $A$ to $B$ is denoted by $\mathrm{hom}(A,B)$. For three objects
$A$, $B$ and $C$ of $\mathcal{C}$ we require the morphisms $\mathrm{hom}(A,B)$ and
$\mathrm{hom}(B,C)$ to be \textit{composable}, i.e. for two morphisms
$f\colon A\rightarrow B$ and $g\colon B\rightarrow C$ there is a morphisms in
$\mathrm{hom}(A,C)$, denoted by $g\circ f$. The composition is defined to be
\textit{associative}, which means that for three composable morphisms $f,g,h$
one postulates $(h\circ g)\circ f=h\circ(g\circ f)$. The last axiom says that
for each object $A$ of $\mathcal{C}$ there is an \textit{identity morphism}
$\mathrm{id}_A\colon A\rightarrow A$, satisfying $f\circ\mathrm{id}_A=f$ and
$\mathrm{id}_A\circ g=g$ for any morphism $f\colon A\rightarrow B$ and
$g\colon B\rightarrow A$. One calls a morphism $f\colon A\rightarrow B$
an \textit{isomorphism} if there exists a morphism $f^{-1}\colon B\rightarrow A$
such that $f^{-1}\circ f=\mathrm{id}_A$ and $f\circ f^{-1}=\mathrm{id}_B$.
The category $\mathcal{C}$ is said to be \textit{small}
if $\mathrm{ob}(\mathcal{C})$ and $\mathrm{hom}(\mathcal{C})$ are sets. If for
a category $\mathcal{C}$ the morphisms $\mathrm{hom}(A,B)$ form a set for every
pair of objects $A,B$ it is said to be \textit{locally small}.
Prominent examples are the categories $\mathrm{Set}$ with sets as objects and
functions between sets as morphisms, ${}_\mathbb{K}\mathrm{Vec}$ with objects
being vector spaces over a fixed ground field $\mathbb{K}$ of characteristic zero
and linear maps as morphisms and ${}_\Bbbk\mathcal{A}$ with $\Bbbk$-algebras
as objects and algebra homomorphisms as morphisms.
The appropriate
notion of morphism between categories $\mathcal{C}$ and $\mathcal{D}$ is given by
a (covariant) \textit{functor} $F\colon\mathcal{C}\rightarrow\mathcal{D}$.
It assigns to every object $A$ of $\mathcal{C}$ an object $F(A)$ in $\mathcal{D}$
and to any morphism $f\colon A\rightarrow B$ in $\mathcal{C}$ a morphism
$F(f)\colon F(A)\rightarrow F(B)$ such that $F(g\circ f)=F(g)\circ F(f)$ and
$F(\mathrm{id}_A)=\mathrm{id}_{F(A)}$ for all morphisms $f\colon A\rightarrow B$ and
$g\colon B\rightarrow C$ in $\mathcal{C}$. A contravariant functor
$F\colon\mathcal{C}\rightarrow\mathcal{D}$ reverses the order in the sense that
$F(f)\colon F(B)\rightarrow F(A)$ for $f\colon A\rightarrow B$. In this case
$F(g\circ f)=F(f)\circ F(g)$ has to be adapted in the definition.
Functors can be composed in an obvious way and for every category there exists
the identity functor. This leads to the category of
categories with categories as objects and functors as morphisms.
Two categories $\mathcal{C}$ and $\mathcal{D}$ are said to be \textit{isomorphic}
if they are isomorphic in the category of categories. Plunging down deeper into
the abyss of category theory we define a \textit{natural transformation}
$\Theta\colon F\rightarrow G$ of two (contravariant) functors $F,G\colon\mathcal{C}
\rightarrow\mathcal{D}$ to be a collection of morphisms in $\mathcal{D}$ of the
form $\Theta_A\colon F(A)\rightarrow G(A)$ for every object $A$ of $\mathcal{C}$,
such that for every morphism $f\colon A\rightarrow B$ in $\mathcal{C}$ one has
$\Theta_B\circ F(f)=G(f)\circ\Theta_A$. If all $\Theta_A$ are isomorphisms
the functors $F$ and $G$ are said to be \textit{naturally equivalent}. Furthermore,
we say that two categories $\mathcal{C}$ and $\mathcal{D}$ are \textit{naturally
equivalent} if there are two functors $F\colon\mathcal{C}\rightarrow\mathcal{D}$
and $G\colon\mathcal{D}\rightarrow\mathcal{C}$ such that $G\circ F$ and $F\circ G$
are naturally equivalent to the identity functors.
To digest this bunch of definitions and absorb some of its nutrition we
invite the reader to consider Proposition~\ref{prop08} in 
Section~\ref{SectionHopfAlgebraModules}.

The notion of categories is extremely useful to organize concepts.
However, many of the examples we want to encounter inherit more structure.
For this reason we review the definition of a \textit{monoidal category}.
In a nutshell one mimics the properties of the tensor product of vector spaces
on a categorical level. Starting with a general category $\mathcal{C}$
we assume the existence of a functor
$
\otimes\colon\mathcal{C}\times\mathcal{C}\rightarrow\mathcal{C}
$
obeying an \textit{associativity constraint}. Namely we postulate the
existence of an isomorphism $\alpha_{U,V,W}\colon(U\otimes V)\otimes W
\rightarrow U\otimes(V\otimes W)$ for any triple $U,V,W$ of objects in $\mathcal{C}$,
such that
\begin{equation*}
\begin{tikzcd}
(U\otimes V)\otimes W \arrow{r}{\alpha_{U,V,W}}
\arrow{d}[swap]{(f\otimes g)\otimes h}
& U\otimes(V\otimes W) \arrow{d}{f\otimes(g\otimes h)} \\
(U'\otimes V')\otimes W' \arrow{r}{\alpha_{U',V',W'}}
& U'\otimes(V'\otimes W')
\end{tikzcd}
\end{equation*}
commutes for all morphisms $f\colon U\rightarrow U'$, $g\colon V\rightarrow V'$ 
and $h\colon W\rightarrow W'$ in $\mathcal{C}$. In other words, there is
a natural equivalence 
$\alpha\colon\otimes\circ(\otimes\times\mathrm{id})\rightarrow
\otimes\circ(\mathrm{id}\times\otimes)$
of functors $\mathcal{C}^{\times 3}\rightarrow\mathcal{C}$. For a $\mathbb{K}$-vector
space $V$ we know that $\mathbb{K}\otimes V\cong V\cong V\otimes\mathbb{K}$.
This result is generalized by the left and right \textit{unit constraint} of
an object $I$ in $\mathcal{C}$, which are natural equivalences
$\ell\colon\otimes\circ(I\times\mathrm{id})\rightarrow\mathrm{id}$ and
$r\colon\otimes\circ(\mathrm{id}\times I)\rightarrow\mathrm{id}$
of functors $\mathcal{C}\rightarrow\mathcal{C}$, respectively.
\begin{definitionApp}[Monoidal Category]
A category $\mathcal{C}$ together with a functor
$\otimes\colon\mathcal{C}\times\mathcal{C}\rightarrow\mathcal{C}$ satisfying
an associativity constraint with respect to $\alpha$ and an object $I$ of
$\mathcal{C}$ satisfying a left and right unit constraint with respect to
$\ell$ and $r$ is said to be a monoidal category, if in addition the
\textit{pentagon relation}
\begin{equation*}
\begin{tikzcd}
((U\otimes V)\otimes W)\otimes X
\arrow{r}{\alpha_{U\otimes V,W,X}}
\arrow{d}[swap]{\alpha_{U,V,W}\otimes\mathrm{id}_X}
& (U\otimes V)\otimes (W\otimes X)
\arrow{r}{\alpha_{U,V,W\otimes X}}
& U\otimes(V\otimes(W\otimes X)) \\
(U\otimes(V\otimes W))\otimes X
\arrow{rr}{\alpha_{U,V\otimes W,X}}
& & U\otimes((V\otimes W)\otimes X)
\arrow{u}[swap]{\mathrm{id}_U\otimes\alpha_{V,W,X}}
\end{tikzcd}
\end{equation*}
and the \textit{triangle relation}
\begin{equation*}
\begin{tikzcd}
(V\otimes I)\otimes W
\arrow{rr}{\alpha_{V,I,W}}
\arrow{rd}[swap]{r_V\otimes\mathrm{id}_W}
& & V\otimes(I\otimes W)
\arrow{ld}{\mathrm{id}_V\otimes\ell_W} \\
& V\otimes W &
\end{tikzcd}
\end{equation*}
hold, for all objects $U,V,W,X$ in $\mathcal{C}$. If in addition
$\alpha,\ell$ and $r$ are identities the category $\mathcal{C}$ is called
\textit{strict monoidal}.
\end{definitionApp}
In this thesis we assume all monoidal categories to be strict monoidal, which can 
be done without loss of generality according to \cite{Ka95}~Sec.~XI.5.
As already mentioned as a motivating example, the category ${}_\mathbb{K}\mathrm{Vec}$
is monoidal with functor given by the usual tensor product of vector spaces and
unit object $\mathbb{K}$.
Also the categories $\mathrm{Set}$ and ${}_\Bbbk\mathcal{A}$ of sets and
associative unital algebras over a commutative unital ring $\Bbbk$ are monoidal with
functor given by the Cartesian product and the tensor product of algebras,
respectively. The latter coincides with the tensor product of
$\Bbbk$-modules but endows the tensor product of two algebras with the
tensor product multiplication, i.e. $(a\otimes x)\cdot(b\otimes y)
=(ab)\otimes(xy)$ for all $a,b\in\mathcal{A}$, $x,y\in\mathcal{X}$ defines an
associative unital product on the tensor product of two algebras $\mathcal{A}$
and $\mathcal{X}$. However, the category ${}_\mathcal{A}\mathcal{M}$ of
left $\mathcal{A}$-modules for an algebra $\mathcal{A}$ is not monoidal in
general. In fact, it is monoidal with respect to the usual associativity and
unit constraints of $\Bbbk$-modules if and only if
$\mathcal{A}$ is a bialgebra (see Proposition~\ref{prop04}). Since categories are
always considered together with their morphisms, the next definition is relevant in
this context.
\begin{definitionApp}[Monoidal Functor]
A functor $F\colon\mathcal{C}\rightarrow\mathcal{D}$ between two monoidal
categories
$(\mathcal{C},\otimes_\mathcal{C},I_\mathcal{C},
\alpha^\mathcal{C},\ell^\mathcal{C},r^\mathcal{C})$
and
$(\mathcal{D},\otimes_\mathcal{D},I_\mathcal{D},
\alpha^\mathcal{D},\ell^\mathcal{D},r^\mathcal{D})$
is said to be a \textit{monoidal functor} if there is a natural transformation
$$
\Xi_{A,B}\colon F(A)\otimes_\mathcal{D}F(B)
\rightarrow F(A\otimes_\mathcal{C}B)
$$
for any pair $(A,B)$ of objects in $\mathcal{C}$ and a morphism
$\phi\colon I_\mathcal{D}\rightarrow F(I_\mathcal{C})$ such that
\begin{equation*}
\begin{tikzcd}
(F(A)\otimes_\mathcal{D}F(B))\otimes_\mathcal{D}F(C)
\arrow{rr}{\alpha^\mathcal{D}_{F(A),F(B),F(C)}}
\arrow{d}[swap]{\Xi_{A,B}\otimes_\mathcal{D}\mathrm{id}_{F(C)}}
& &  F(A)\otimes_\mathcal{D}(F(B)\otimes_\mathcal{D}F(C))
\arrow{d}{\mathrm{id}_{F(A)}\otimes_\mathcal{D}\Xi_{B,C}} \\
F(A\otimes_\mathcal{C}B)\otimes_\mathcal{D}F(C)
\arrow{d}[swap]{\Xi_{A\otimes_\mathcal{C}B,C}}
& & F(A)\otimes_\mathcal{D}F(B\otimes_\mathcal{C}C)
\arrow{d}{\Xi_{A,B\otimes_\mathcal{C}C}} \\
F((A\otimes_\mathcal{C}B)\otimes_\mathcal{C}C)
\arrow{rr}{F(\alpha^\mathcal{C}_{A,B,C})}
& & F(A\otimes_\mathcal{C}(B\otimes_\mathcal{C}C))
\end{tikzcd},
\end{equation*}
\begin{equation*}
\begin{tikzcd}
F(A)\otimes_\mathcal{D}I_\mathcal{D}
\arrow{r}{\mathrm{id}\otimes_\mathcal{D}\phi}
\arrow{d}[swap]{r_\mathcal{D}}
& F(A)\otimes_\mathcal{D}F(I_\mathcal{C})
\arrow{d}{\Xi_{A,I_\mathcal{C}}} \\
F(A)
& F(A\otimes_\mathcal{C}I_\mathcal{C})
\arrow{l}[swap]{F(r_\mathcal{C})}
\end{tikzcd}
\text{ and }
\begin{tikzcd}
I_\mathcal{D}\otimes_\mathcal{D}F(A)
\arrow{r}{\phi\otimes_\mathcal{D}\mathrm{id}}
\arrow{d}[swap]{\ell_\mathcal{D}}
& F(I_\mathcal{C})\otimes_\mathcal{D}F(A)
\arrow{d}{\Xi_{I_\mathcal{C},A}} \\
F(A)
& F(I_\mathcal{C}\otimes_\mathcal{C}A)
\arrow{l}[swap]{F(\ell_\mathcal{C})}
\end{tikzcd}
\end{equation*}
commute as diagrams in $\mathcal{D}$. If $\Xi$ and $\phi$ are isomorphisms
in $\mathcal{D}$, the monoidal functor $F$ is said to be \textit{strong monoidal}.
If $\Xi$ and $\phi$ are even identities in $\mathcal{D}$, $F$ is said to be
\textit{strict monoidal}.
\end{definitionApp}
It is clear that natural transformations
of monoidal functors should respect the underlying monoidal structure in addition.
\begin{definitionApp}[Monoidal Equivalence]
A natural transformation $\Theta\colon F\rightarrow F'$
between monoidal functors $(F,\Xi,\phi)$ and $(F',\Xi',\phi')$ between
monoidal categories $\mathcal{C}$ and $\mathcal{D}$ is said to be a
\textit{monoidal natural transformation} if
\begin{equation*}
\begin{tikzcd}
 & I_\mathcal{D} 
 \arrow{ld}[swap]{\phi}
 \arrow{rd}{\phi'}
 & \\
 F(I_\mathcal{C})
 \arrow{rr}{\Theta_{I_\mathcal{C}}}
 & & F'(I_\mathcal{C})
\end{tikzcd}
\text{ and }
\begin{tikzcd}
 F(A)\otimes_\mathcal{D}F(B)
 \arrow{r}{\Xi_{A,B}}
 \arrow{d}[swap]{\Theta_A\otimes_\mathcal{D}\Theta_B}
 & F(A\otimes_\mathcal{C}B)
 \arrow{d}{\Theta_{A\otimes_\mathcal{C}B}} \\
 F'(A)\otimes_\mathcal{D}F'(B)
 \arrow{r}{\Xi'_{A,B}}
 & F'(A\otimes_\mathcal{C}B)
\end{tikzcd}
\end{equation*}
commute for all pairs $(A,B)$ of objects in $\mathcal{C}$ as diagrams in
$\mathcal{D}$. If $\Theta$ is a natural isomorphism
in addition, it is said to be a \textit{monoidal natural isomorphism}.
Finally, two monoidal categories $\mathcal{C}$ and $\mathcal{D}$ are
called \textit{monoidally equivalent} if there exist two monoidal functors
$F\colon\mathcal{C}\rightarrow\mathcal{D}$ and
$F'\colon\mathcal{D}\rightarrow\mathcal{C}$ together with monoidal natural
isomorphisms $F'\circ F\rightarrow\mathrm{id}_\mathcal{C}$ and
$F\circ F'\rightarrow\mathrm{id}_\mathcal{D}$.
\end{definitionApp}
\textit{Commutativity constraints} of a
monoidal category $(\mathcal{C},\otimes,I,\alpha,\ell,r)$ are natural
isomorphisms $\beta\colon\otimes\rightarrow\otimes\circ\tau$ of
functors $\mathcal{C}\times\mathcal{C}\rightarrow\mathcal{C}$, where
$\tau\colon\mathcal{C}\times\mathcal{C}\rightarrow
\mathcal{C}\times\mathcal{C}$ denotes the \textit{flip functor}
$\tau(A,B)=(B,A)$ for any pair $(A,B)$ of objects in $\mathcal{C}$.
Demanding compatibility with the associativity constraint we end up
with the definition of a braiding.
\begin{definitionApp}[Braided Monoidal Category]
A monoidal category $(\mathcal{C},\otimes,I,\alpha,\ell,r)$ is said to
be \textit{braided monoidal} if there is a commutativity constraint $\beta$
satisfying the \textit{hexagon relations}
\begin{equation*}
\begin{tikzcd}
 (A\otimes B)\otimes C
 \arrow{r}{\alpha_{A,B,C}}
 \arrow{d}[swap]{\beta_{A,B}\otimes\mathrm{id}_C}
 & A\otimes(B\otimes C)
 \arrow{r}{\beta_{A,B\otimes C}}
 & (B\otimes C)\otimes A
 \arrow{d}{\alpha_{B,C,A}} \\
 (B\otimes A)\otimes C
 \arrow{r}{\alpha_{B,A,C}}
 & B\otimes(A\otimes C)
 \arrow{r}{\mathrm{id}_B\otimes\beta_{A,C}}
 & B\otimes(C\otimes A)
\end{tikzcd}
\end{equation*}
and
\begin{equation*}
\begin{tikzcd}
  A\otimes(B\otimes C)
 \arrow{r}{\alpha^{-1}_{A,B,C}}
 \arrow{d}[swap]{\mathrm{id}_A\otimes\beta_{B,C}}
 & (A\otimes B)\otimes C
 \arrow{r}{\beta_{A\otimes B,C}}
 & C\otimes(A\otimes B)
 \arrow{d}{\alpha^{-1}_{C,A,B}} \\
 A\otimes(C\otimes B)
 \arrow{r}{\alpha^{-1}_{A,C,B}}
 & (A\otimes C)\otimes B
 \arrow{r}{\beta_{A,C}\otimes\mathrm{id}_B}
 & (C\otimes A)\otimes B
\end{tikzcd},
\end{equation*}
which are commutative diagrams in $\mathcal{C}$. If 
$\beta_{B,A}\circ\beta_{A,B}=\mathrm{id}_{A\otimes B}$ the braided monoidal
category is said to be \textit{symmetric}. A monoidal functor
$F\colon\mathcal{C}\rightarrow\mathcal{D}$ between braided monoidal
categories is called \textit{braided monoidal functor} if
\begin{equation*}
\begin{tikzcd}
F(A)\otimes_\mathcal{D}F(B)
\arrow{r}{\Xi_{A,B}}
\arrow{d}[swap]{\beta^\mathcal{D}_{F(A),F(B)}}
& F(A\otimes_\mathcal{C}B)
\arrow{d}{F(\beta^\mathcal{C}_{A,B})} \\
F(B)\otimes_\mathcal{D}F(A)
\arrow{r}{\Xi_{B,A}}
& F(B\otimes_\mathcal{C}A)
\end{tikzcd}
\end{equation*}
commutes in $\mathcal{D}$ for any pair $(A,B)$ of objects in $\mathcal{C}$.
\end{definitionApp}
The commutativity constraint $\beta$ of a braided monoidal category is also
called \textit{braiding}. It follows that the braiding respects the unit
object, i.e. that
\begin{equation*}
\begin{tikzcd}
A\otimes I
\arrow{rr}{\beta_{A,I}}
\arrow{rd}[swap]{r_A}
& & I\otimes A
\arrow{ld}{\ell_A} \\
& A
&
\end{tikzcd}
\end{equation*}
commutes for every object $A$ in the braided monoidal category.
Another interesting class of monoidal categories is given by rigid ones.
Their key feature is that they admit dual objects in a way which generalizes the
following example: a finite-dimensional $\mathbb{K}$-vector spaces $V$
possesses a basis $e_1,\ldots,e_n\in V$ and a corresponding
\textit{dual basis} $e^1,\ldots,e^n\in V^*
=\mathrm{Hom}_\mathbb{K}(V,\mathbb{K})$, such that
$e^i(e_j)=\delta^i_j$ and every element $v\in V$ can be represented as
$v=\sum_{i=1}^ne^i(v)e_i$.
\begin{definitionApp}[Duality]
For a strict monoidal category $(\mathcal{C},\otimes,I)$ we say
that an object $A^*$ of $\mathcal{C}$ is an \textit{left dual object}
of an object $A$ of $\mathcal{C}$ if there are morphisms
$\mathrm{ev}_A\colon A^*\otimes A\rightarrow I$ and
$\pi_A\colon I\rightarrow A\otimes A^*$ in $\mathcal{C}$ such that
$$
(\mathrm{id}_A\otimes\mathrm{ev}_A)\circ(\pi_A\otimes\mathrm{id}_A)
=\mathrm{id}_A
$$
and
$$
(\mathrm{ev}_A\otimes\mathrm{id}_{A^*})\circ(\mathrm{id}_{A^*}\otimes\pi_A)
=\mathrm{id}_{A^*}
$$
hold. If there are left duals for two objects $A$ and $B$ we can define
the \textit{left transpose} $f^*\colon B^*\rightarrow A^*$
of a morphism $f\colon A\rightarrow B$ by
$$
f^*
=(\mathrm{ev}_B\otimes\mathrm{id}_{A^*})
\circ(\mathrm{id}_{B^*}\otimes f\otimes\mathrm{id}_{A^*})
\circ(\mathrm{id}_{B^*}\otimes\pi_A).
$$
An object ${}^*A$ is said to be a \textit{right dual} of an object
$A$ if there are morphisms
$
\mathrm{ev}'_A\colon A\otimes {}^*A\rightarrow I
\text{ and }
\pi'_A\colon I\rightarrow {}^*A\otimes A
$
such that
$$
(\mathrm{ev}'_A\otimes\mathrm{id}_A)\circ(\mathrm{id}_A\otimes\pi'_A)
=\mathrm{id}_A
$$
and
$$
(\mathrm{id}_{{}^*A}\otimes\mathrm{ev}'_A)
\circ(\pi'_A\otimes\mathrm{id}_{{}^*A})
=\mathrm{id}_{{}^*A}
$$
hold. If there are right duals ${}^*A$ and ${}^*B$ of two objects $A$ and $B$,
the \textit{right transpose} ${}^*f\colon{}^*B\rightarrow {}^*A$
of a morphism $f\colon A\rightarrow B$ is defined by
$$
{}^*f
=(\mathrm{id}_{{}^*B}\otimes\mathrm{ev}'_B)
\circ(\mathrm{id}_{{}^*B}\otimes f\otimes\mathrm{id}_{{}^*A})
\circ(\pi'_A\otimes\mathrm{id}_{{}^*A}).
$$
If there is a left and a right dual for every object in $\mathcal{C}$
we call the strict monoidal category $\mathcal{C}$ rigid.
\end{definitionApp}
In fact, rigid strict monoidal categories behave very much like
the category ${}_\mathbb{K}\mathrm{Vec}_f$ of finite-dimensional
vector spaces. This is underlined in the next proposition, taken from
\cite{Ka95}~Prop.~XIV.2.2.
\begin{propositionApp}
Let $(\mathcal{C},\otimes,I)$ be a rigid strict monoidal category.
\begin{enumerate}
\item[i.)] For all morphisms $f\colon A\rightarrow B$ and
$g\colon B\rightarrow C$ in $\mathcal{C}$ one has
$$
(g\circ f)^*
=f^*\circ g^*
\text{ and }
\mathrm{id}^*_A=\mathrm{id}_{A^*}.
$$

\item[ii.)] There are natural bijections
$$
\mathrm{Hom}(A\otimes B,C)
\cong\mathrm{Hom}(A,C\otimes B^*)
$$
and
$$
\mathrm{Hom}(A^*\otimes B,C)
\cong\mathrm{Hom}(B,A\otimes C),
$$
where $A,B,C$ are objects in $\mathcal{C}$.

\item[iii.)] There is an isomorphism 
$
(A\otimes B)^*
\cong B^*\otimes A^*
$
for any two objects $A,B$ in $\mathcal{C}$.
\end{enumerate}
Similar statements hold for the right transpose. Moreover, there are
isomorphisms
$$
{}^*(A^*)\cong A\cong({}^*A)^*.
$$
\end{propositionApp}
Note that in general $(A^*)^*\ncong A$ and ${}^*({}^*A)\ncong A$.
%

%% file: chapters/appendixB.tex
Let $\Bbbk$ be a commutative ring with unit $1$. A \textit{graded module}
over $\Bbbk$ is a direct sum $V^\bullet=\bigoplus_{k\in\mathbb{Z}}V^k$
of $\Bbbk$-modules. Note that $V^\bullet$ itself is a $\Bbbk$-module
with respect to the component-wise action. Remark that the notion of
graded modules is usually utilized in the context of graded rings
(see \cite{Bourbaki1989}~Chap.~II~Sec.~11). For our
purpose it is sufficient to restrict our consideration to usual rings
which can be seen as graded rings concentrated in degree zero.
A map $\Phi\colon V^\bullet\rightarrow W^\bullet$ between graded modules
$V^\bullet=\bigoplus_{k\in\mathbb{Z}}V^k$ and 
$W^\bullet=\bigoplus_{k\in\mathbb{Z}}W^k$ is said to
be \textit{homogeneous of degree $k\in\mathbb{Z}$} if
$\Phi(V^\ell)\subseteq W^{k+\ell}$. We often write
$\Phi\colon V^\bullet\rightarrow W^{\bullet+k}$ in this case.
As an example, consider the Graßmann algebra $(\Omega^\bullet(M),\wedge)$
of differential forms for a smooth manifold $M$. It is a graded 
$\mathbb{R}$-module as well as a graded $\mathscr{C}^\infty(M)$-module and
the de Rham differential $\mathrm{d}\colon\Omega^\bullet(M)
\rightarrow\Omega^{\bullet+1}(M)$ is a homogeneous map of degree $1$.
The \textit{graded commutator} of two homogeneous maps
$\Phi,\Psi\colon V^\bullet\rightarrow V^\bullet$ of degree $k$ and $\ell$
is defined by 
$$
[\Phi,\Psi]=\Phi\circ\Psi-(-1)^{k\ell}\Psi\circ\Phi.
$$
In Section~\ref{Sec3.3} braided differential forms
$(\Omega^\bullet_\mathcal{R}(\mathcal{A}),\wedge_\mathcal{R})$ of
a braided commutative algebra $\mathcal{A}$ for a triangular Hopf algebra
$(H,\mathcal{R})$ are introduced as a generalization to differential forms
on a smooth manifold. In the following lines we give the general framework
for this space to fit in (see also \cite{BZ2008}).
It is the generalization of Graßmann algebra in the category of
$H$-equivariant braided symmetric $\mathcal{A}$-bimodules.
Remark that braided Lie algebras and their quantum analogues can be
formulated in a more general categorical setting (see e.g.
\cite{GoMa2003,Ma94}). 

Fix a triangular Hopf algebra $(H,\mathcal{R})$ and a braided commutative
algebra $\mathcal{A}$. For any $H$-equivariant braided symmetric
$\mathcal{A}$-bimodule we are able to define the \textit{tensor algebra}
$$
\mathrm{T}^\bullet\mathcal{M}
=\bigoplus_{k\in\mathbb{Z}}\mathcal{M}^k
=\mathcal{A}\oplus\mathcal{M}\oplus(\mathcal{M}\otimes_\mathcal{A}\mathcal{M})
\oplus\cdots
$$
with respect to the tensor product $\otimes_\mathcal{A}$ over $\mathcal{A}$,
where by definition $\mathcal{M}^k=\{0\}$ if $k<0$ and
$\mathcal{M}^0=\mathcal{A}$. The tensor algebra $\mathrm{T}^\bullet\mathcal{M}$
is an associative unital algebra with respect to the product given by
the tensor product $\otimes_\mathcal{A}$ and the unit $1\in\mathcal{A}$.
\begin{lemmaApp}\label{lemma13}
The tensor algebra $\mathrm{T}^\bullet\mathcal{M}$ of an $H$-equivariant
braided symmetric $\mathcal{A}$-bimodule is an $H$-equivariant
braided symmetric $\mathcal{A}$-bimodule with respect to the following
module actions, given on factorizing elements $m_1\otimes_\mathcal{A}\cdots
\otimes_\mathcal{A}m_k$ by
\begin{align*}
    \xi\rhd(m_1\otimes_\mathcal{A}\cdots\otimes_\mathcal{A}m_k)
    =&(\xi_{(1)}\rhd m_1)\otimes_\mathcal{A}\cdots\otimes_\mathcal{A}
    (\xi_{(k)}\rhd m_k),\\
    a\cdot(m_1\otimes_\mathcal{A}\cdots\otimes_\mathcal{A}m_k)
    =&(a\cdot m_1)\otimes_\mathcal{A}\cdots\otimes_\mathcal{A}m_k,\\
    (m_1\otimes_\mathcal{A}\cdots\otimes_\mathcal{A}m_k)\cdot a
    =&m_1\otimes_\mathcal{A}\cdots\otimes_\mathcal{A}(m_k\cdot a)
\end{align*}
for all $\xi\in H$ and $a\in\mathcal{A}$, where $k\geq 0$.
\end{lemmaApp}
It is an easy exercise to verify this lemma. Furthermore, 
there is an ideal $I$ in $(\mathrm{T}^\bullet\mathcal{M},
\otimes_\mathcal{A})$, generated by elements
$
m_1\otimes_\mathcal{A}\cdots\otimes_\mathcal{A}m_k
\in\mathrm{T}^k\mathcal{M}
$
which equal
\begin{align*}
    m_1\otimes_\mathcal{A}&\cdots\otimes_\mathcal{A}m_{i-1}
    \otimes_\mathcal{A}\bigg(\mathcal{R}_1^{'-1}\rhd\bigg(
    (\mathcal{R}_1^{-1}\rhd m_j)\otimes_\mathcal{A}
    (\mathcal{R}_2^{-1}\rhd(m_{i+1}\otimes_\mathcal{A}
    \cdots\otimes_\mathcal{A}m_{j-1}))\bigg)\bigg)\\
    &\otimes_\mathcal{A}(\mathcal{R}_2^{'-1}\rhd m_i)
    \otimes_\mathcal{A}m_{j+1}\otimes_\mathcal{A}\cdots
    \otimes_\mathcal{A}m_k
\end{align*}
for a pair $(i,j)$ such that $1\leq i<j\leq k$.
\begin{lemmaApp}\label{lemma14}
The left $H$-action and the left and right $\mathcal{A}$-actions
from Lemma~\ref{lemma13} respect the ideal $I$.
\end{lemmaApp}
\begin{proof}
Let $\xi\in H$, $a\in\mathcal{A}$ and $m_1\otimes_\mathcal{A}\cdots
\otimes_\mathcal{A}m_k\in I$ for a $k>1$. Then
\begin{allowdisplaybreaks}
\begin{align*}
    (\xi_{(1)}\rhd& m_1)\otimes_\mathcal{A}\cdots\otimes_\mathcal{A}
    (\xi_{(k)}\rhd m_k)
    =\xi\rhd(m_1\otimes_\mathcal{A}\cdots\otimes_\mathcal{A}m_k)\\
    =&\xi\rhd\bigg(m_1\otimes_\mathcal{A}\cdots\otimes_\mathcal{A}m_{i-1}
    \otimes_\mathcal{A}\bigg(\mathcal{R}_1^{'-1}\rhd\bigg(\\
    &(\mathcal{R}_1^{-1}\rhd m_j)
    \otimes_\mathcal{A}
    (\mathcal{R}_2^{-1}\rhd(m_{i+1}\otimes_\mathcal{A}
    \cdots\otimes_\mathcal{A}m_{j-1}))\bigg)\bigg)\\
    &\otimes_\mathcal{A}(\mathcal{R}_2^{'-1}\rhd m_i)
    \otimes_\mathcal{A}m_{j+1}\otimes_\mathcal{A}\cdots
    \otimes_\mathcal{A}m_k\bigg)\\
    =&\xi_{(1)}\rhd\bigg(m_1\otimes_\mathcal{A}\cdots
    \otimes_\mathcal{A}m_{i-1}\bigg)
    \otimes_\mathcal{A}\bigg((\xi_{(2)}\mathcal{R}_1^{'-1})\rhd\bigg(\\
    &(\mathcal{R}_1^{-1}\rhd m_j)
    \otimes_\mathcal{A}
    (\mathcal{R}_2^{-1}\rhd(m_{i+1}\otimes_\mathcal{A}
    \cdots\otimes_\mathcal{A}m_{j-1}))\bigg)\bigg)\\
    &\otimes_\mathcal{A}((\xi_{(3)}\mathcal{R}_2^{'-1})\rhd m_i)
    \otimes_\mathcal{A}\xi_{(4)}\rhd\bigg(m_{j+1}\otimes_\mathcal{A}\cdots
    \otimes_\mathcal{A}m_k\bigg)\\
    =&\xi_{(1)}\rhd\bigg(m_1\otimes_\mathcal{A}\cdots
    \otimes_\mathcal{A}m_{i-1}\bigg)
    \otimes_\mathcal{A}\bigg(\mathcal{R}_1^{'-1}\rhd\bigg(\\
    &((\xi_{(3)}\mathcal{R}_1^{-1})\rhd m_j)
    \otimes_\mathcal{A}
    ((\xi_{(4)}\mathcal{R}_2^{-1})\rhd(m_{i+1}\otimes_\mathcal{A}
    \cdots\otimes_\mathcal{A}m_{j-1}))\bigg)\bigg)\\
    &\otimes_\mathcal{A}((\mathcal{R}_2^{'-1}\xi_{(2)})\rhd m_i)
    \otimes_\mathcal{A}\xi_{(5)}\rhd\bigg(m_{j+1}\otimes_\mathcal{A}\cdots
    \otimes_\mathcal{A}m_k\bigg)\\
    =&(\xi_{(1)}\rhd m_1)\otimes_\mathcal{A}\cdots
    \otimes_\mathcal{A}(\xi_{(i-1)}\rhd m_{i-1})
    \otimes_\mathcal{A}\bigg(\mathcal{R}_1^{'-1}\rhd\bigg(\\
    &((\mathcal{R}_1^{-1}\xi_{(j)})\rhd m_j)
    \otimes_\mathcal{A}
    (\mathcal{R}_2^{-1}\rhd((\xi_{(i+1)}\rhd m_{i+1})\otimes_\mathcal{A}
    \cdots\otimes_\mathcal{A}(\xi_{(j-1)}\rhd m_{j-1})))\bigg)\bigg)\\
    &\otimes_\mathcal{A}((\mathcal{R}_2^{'-1}\xi_{(i)})\rhd m_i)
    \otimes_\mathcal{A}(\xi_{(j+1)}\rhd m_{j+1})\otimes_\mathcal{A}\cdots
    \otimes_\mathcal{A}(\xi_{(k)}\rhd m_k)
\end{align*}
\end{allowdisplaybreaks}
for a pair $(i,j)$ such that $1\leq i<j\leq k$, which implies
$\xi\rhd I\subseteq I$. The left $\mathcal{A}$-module action is only affected
if $i=1$. In this case
\begin{allowdisplaybreaks}
\begin{align*}
    (a\cdot m_1)&\otimes_\mathcal{A}m_2\otimes_\mathcal{A}\cdots
    \otimes_\mathcal{A}m_k
    =a\cdot(m_1\otimes_\mathcal{A}\cdots\otimes_\mathcal{A}m_k)\\
    =&a\cdot\bigg(\bigg(\mathcal{R}_1^{'-1}\rhd\bigg(
    (\mathcal{R}_1^{-1}\rhd m_j)\otimes_\mathcal{A}
    (\mathcal{R}_2^{-1}\rhd(m_{2}\otimes_\mathcal{A}
    \cdots\otimes_\mathcal{A}m_{j-1}))\bigg)\bigg)\\
    &\otimes_\mathcal{A}(\mathcal{R}_2^{'-1}\rhd m_1)
    \otimes_\mathcal{A}m_{j+1}\otimes_\mathcal{A}\cdots
    \otimes_\mathcal{A}m_k\bigg)\\
    =&\bigg(
    a\cdot((\mathcal{R}_{1(1)}^{'-1}\mathcal{R}_1^{-1})\rhd m_j)
    \otimes_\mathcal{A}
    ((\mathcal{R}_{1(2)}^{'-1}\mathcal{R}_2^{-1})\rhd(m_{2}
    \otimes_\mathcal{A}
    \cdots\otimes_\mathcal{A}m_{j-1}))\bigg)\\
    &\otimes_\mathcal{A}(\mathcal{R}_2^{'-1}\rhd m_1)
    \otimes_\mathcal{A}m_{j+1}\otimes_\mathcal{A}\cdots
    \otimes_\mathcal{A}m_k\\
    =&\bigg(
    ((\mathcal{R}_{1(1)}^{''-1}\mathcal{R}_{1(1)}^{'-1}\mathcal{R}_1^{-1})
    \rhd m_j)
    \otimes_\mathcal{A}
    ((\mathcal{R}_{1(2)}^{''-1}\mathcal{R}_{1(2)}^{'-1}\mathcal{R}_2^{-1})
    \rhd(m_{2}
    \otimes_\mathcal{A}
    \cdots\otimes_\mathcal{A}m_{j-1}))\bigg)\\
    &\otimes_\mathcal{A}((\mathcal{R}_2^{''-1}\rhd a)
    \cdot(\mathcal{R}_2^{'-1}\rhd m_1))
    \otimes_\mathcal{A}m_{j+1}\otimes_\mathcal{A}\cdots
    \otimes_\mathcal{A}m_k\\
    =&\bigg(
    ((\mathcal{R}_{1(1)}^{'-1}\mathcal{R}_1^{-1})
    \rhd m_j)
    \otimes_\mathcal{A}
    ((\mathcal{R}_{1(2)}^{'-1}\mathcal{R}_2^{-1})
    \rhd(m_{2}
    \otimes_\mathcal{A}
    \cdots\otimes_\mathcal{A}m_{j-1}))\bigg)\\
    &\otimes_\mathcal{A}((\mathcal{R}_{2(1)}^{'-1}\rhd a)
    \cdot(\mathcal{R}_{2(2)}^{'-1}\rhd m_1))
    \otimes_\mathcal{A}m_{j+1}\otimes_\mathcal{A}\cdots
    \otimes_\mathcal{A}m_k\\
    =&\bigg(\mathcal{R}_1^{'-1}\rhd\bigg(
    (\mathcal{R}_1^{-1}\rhd m_j)\otimes_\mathcal{A}
    (\mathcal{R}_2^{-1}\rhd(m_{2}\otimes_\mathcal{A}
    \cdots\otimes_\mathcal{A}m_{j-1}))\bigg)\bigg)\\
    &\otimes_\mathcal{A}(\mathcal{R}_2^{'-1}\rhd(a\cdot m_1))
    \otimes_\mathcal{A}m_{j+1}\otimes_\mathcal{A}\cdots
    \otimes_\mathcal{A}m_k.
\end{align*}
\end{allowdisplaybreaks}
On the other hand, if $1\leq i<j=k$
\begin{allowdisplaybreaks}
\begin{align*}
    m_1\otimes_\mathcal{A}&\cdots\otimes_\mathcal{A}m_{k-1}
    \otimes_\mathcal{A}(m_k\cdot a)
    =(m_1\otimes_\mathcal{A}\cdots\otimes_\mathcal{A}m_k)\cdot a\\
    =&m_1\otimes_\mathcal{A}\cdots\otimes_\mathcal{A}m_{i-1}
    \otimes_\mathcal{A}\bigg(\mathcal{R}_1^{'-1}\rhd\bigg(
    (\mathcal{R}_1^{-1}\rhd m_k)\\
    &\otimes_\mathcal{A}
    (\mathcal{R}_2^{-1}\rhd(m_{i+1}\otimes_\mathcal{A}
    \cdots\otimes_\mathcal{A}m_{k-1}))\bigg)\bigg)
    \otimes_\mathcal{A}((\mathcal{R}_2^{'-1}\rhd m_i)\cdot a)\\
    =&m_1\otimes_\mathcal{A}\cdots\otimes_\mathcal{A}m_{i-1}
    \otimes_\mathcal{A}\bigg(
    ((\mathcal{R}_{1(1)}^{'-1}\mathcal{R}_1^{-1})\rhd m_k)
    \cdot(\mathcal{R}_1^{''-1}\rhd a)\\
    &\otimes_\mathcal{A}
    ((\mathcal{R}_{2(1)}^{''-1}\mathcal{R}_{1(2)}^{'-1}\mathcal{R}_2^{-1})
    \rhd(m_{i+1}
    \otimes_\mathcal{A}
    \cdots\otimes_\mathcal{A}m_{k-1}))\bigg)
    \otimes_\mathcal{A}((\mathcal{R}_{2(2)}^{''-1}\mathcal{R}_2^{'-1})
    \rhd m_i)\\
    =&m_1\otimes_\mathcal{A}\cdots\otimes_\mathcal{A}m_{i-1}
    \otimes_\mathcal{A}\bigg(
    ((\mathcal{R}_{1}^{'-1}\mathcal{R}_1^{-1})\rhd m_k)
    \cdot(\mathcal{R}_1^{''-1}\rhd a)\\
    &\otimes_\mathcal{A}
    ((\mathcal{R}_{2(1)}^{''-1}\mathcal{R}_{1}^{'''-1}\mathcal{R}_2^{-1})
    \rhd(m_{i+1}
    \otimes_\mathcal{A}
    \cdots\otimes_\mathcal{A}m_{k-1}))\bigg)\\
    &\otimes_\mathcal{A}((\mathcal{R}_{2(2)}^{''-1}
    \mathcal{R}_2^{'''-1}\mathcal{R}_2^{'-1})\rhd m_i)\\
    =&m_1\otimes_\mathcal{A}\cdots\otimes_\mathcal{A}m_{i-1}
    \otimes_\mathcal{A}\bigg(
    (\mathcal{R}_1^{-1}\rhd m_k)
    \cdot(\mathcal{R}_1^{''-1}\rhd a)\\
    &\otimes_\mathcal{A}
    ((\mathcal{R}_{1}^{'''-1}\mathcal{R}_{2(2)}^{''-1}
    \mathcal{R}_{2(2)}^{-1})
    \rhd(m_{i+1}
    \otimes_\mathcal{A}
    \cdots\otimes_\mathcal{A}m_{k-1}))\bigg)\\
    &\otimes_\mathcal{A}((\mathcal{R}_2^{'''-1}\mathcal{R}_{2(1)}^{''-1}
    \mathcal{R}_{2(1)}^{-1})\rhd m_i)\\
    =&m_1\otimes_\mathcal{A}\cdots\otimes_\mathcal{A}m_{i-1}
    \otimes_\mathcal{A}\bigg((\mathcal{R}_1^{-1}\rhd(m_k\cdot a))\\
    &\otimes_\mathcal{A}
    ((\mathcal{R}_{1}^{'''-1}\mathcal{R}_{2(2)}^{-1})
    \rhd(m_{i+1}
    \otimes_\mathcal{A}
    \cdots\otimes_\mathcal{A}m_{k-1}))\bigg)
    \otimes_\mathcal{A}((\mathcal{R}_2^{'''-1}
    \mathcal{R}_{2(1)}^{-1})\rhd m_i)\\
    =&m_1\otimes_\mathcal{A}\cdots\otimes_\mathcal{A}m_{i-1}
    \otimes_\mathcal{A}\bigg(((\mathcal{R}_1^{'-1}\mathcal{R}_1^{-1})
    \rhd(m_k\cdot a))\\
    &\otimes_\mathcal{A}
    ((\mathcal{R}_{2}^{'-1}\mathcal{R}_{1}^{'''-1})
    \rhd(m_{i+1}
    \otimes_\mathcal{A}
    \cdots\otimes_\mathcal{A}m_{k-1}))\bigg)
    \otimes_\mathcal{A}((\mathcal{R}_{2}^{-1}\mathcal{R}_2^{'''-1}
    )\rhd m_i)\\
    =&m_1\otimes_\mathcal{A}\cdots\otimes_\mathcal{A}m_{i-1}
    \otimes_\mathcal{A}\bigg(((\mathcal{R}_{1(1)}^{-1}\mathcal{R}_1^{'-1})
    \rhd(m_k\cdot a))\\
    &\otimes_\mathcal{A}
    ((\mathcal{R}_{1(2)}^{-1}\mathcal{R}_{2}^{'-1})
    \rhd(m_{i+1}
    \otimes_\mathcal{A}
    \cdots\otimes_\mathcal{A}m_{k-1}))\bigg)
    \otimes_\mathcal{A}(\mathcal{R}_{2}^{-1}\rhd m_i)\\
    =&m_1\otimes_\mathcal{A}\cdots\otimes_\mathcal{A}m_{i-1}
    \otimes_\mathcal{A}\bigg(\mathcal{R}_1^{'-1}\rhd\bigg(
    (\mathcal{R}_1^{-1}\rhd(m_k\cdot a))\\
    &\otimes_\mathcal{A}
    (\mathcal{R}_2^{-1}\rhd(m_{i+1}\otimes_\mathcal{A}
    \cdots\otimes_\mathcal{A}m_{k-1}))\bigg)\bigg)
    \otimes_\mathcal{A}(\mathcal{R}_2^{'-1}\rhd m_i)
\end{align*}
\end{allowdisplaybreaks}
shows that the right $\mathcal{A}$-action also respects the ideal.
This concludes the proof of the lemma.
\end{proof}
By Lemma~\ref{lemma13} and Lemma~\ref{lemma14} we conclude the
following statement.
\begin{propositionApp}\label{prop14}
The quotient
$\mathrm{T}^\bullet\mathcal{M}/I$ is a well-defined associative unital
graded algebra and an $H$-equivariant braided symmetric
$\mathcal{A}$-bimodule.
\end{propositionApp}
We denote the quotient by $\Lambda^\bullet\mathcal{M}$ and
the induced product by $\wedge_\mathcal{R}$.
On factorizing elements $m_1\wedge_\mathcal{R}\cdots\wedge_\mathcal{R}m_k
\in\Lambda^k\mathcal{M}$ the induced module actions read
\begin{align*}
    \xi\rhd(m_1\wedge_\mathcal{R}\cdots\wedge_\mathcal{R}m_k)
    =&(\xi_{(1)}\rhd m_1)\wedge_\mathcal{R}\cdots\wedge_\mathcal{R}
    (\xi_{(k)}\rhd m_k),\\
    a\cdot(m_1\wedge_\mathcal{R}\cdots\wedge_\mathcal{R}m_k)
    =&(a\cdot m_1)\wedge_\mathcal{R}\cdots\wedge_\mathcal{R}m_k,\\
    (m_1\wedge_\mathcal{R}\cdots\wedge_\mathcal{R}m_k)\cdot a
    =&m_1\wedge_\mathcal{R}\cdots\wedge_\mathcal{R}(m_k\cdot a),
\end{align*}
where $\xi\in H$ and $a\in\mathcal{A}$.
Note that the module actions respect the degree by definition.
\begin{definitionApp}
The associative unital graded algebra and $H$-equivariant braided symmetric
$\mathcal{A}$-bimodule $(\Lambda^\bullet\mathcal{M},\wedge_\mathcal{R})$
is said to be the braided Graßmann algebra or braided exterior algebra
of the $H$-equivariant braided symmetric $\mathcal{A}$-bimodule $\mathcal{M}$.
\end{definitionApp}
We prove that the product of a braided Graßmann algebra inherits the
braided symmetry from $\mathcal{M}$.
\begin{lemmaApp}\label{lemma09}
Let $\mathcal{M}$ be an $H$-equivariant braided symmetric
$\mathcal{A}$-bimodule.
The braided wedge product $\wedge_\mathcal{R}$ is graded braided
commutative, i.e.
$$
Y\wedge_\mathcal{R}X
=(-1)^{k\ell}(\mathcal{R}_1^{-1}\rhd X)
\wedge_\mathcal{R}(\mathcal{R}_2^{-1}\rhd Y)
$$
for all $X\in\Lambda^k\mathcal{M}$ and $Y\in\Lambda^\ell\mathcal{M}$.
\end{lemmaApp}
\begin{proof}
Consider two factorizing elements
$X=X_1\wedge_\mathcal{R}\cdots\wedge_\mathcal{R}X_k
\in\Lambda^k\mathcal{M}$ and
$Y=Y_1\wedge_\mathcal{R}\cdots\wedge_\mathcal{R}Y_\ell
\in\Lambda^\ell\mathcal{M}$. First remark that
$$
X
=-X_1\wedge_\mathcal{R}\cdots\wedge_\mathcal{R}
X_{i-1}\wedge_\mathcal{R}
(\mathcal{R}_1^{-1}\rhd X_{i+1})\wedge_\mathcal{R}
(\mathcal{R}_2^{-1}\rhd X_i)\wedge_\mathcal{R}
X_{i+2}\wedge_\mathcal{R}\cdots\wedge_\mathcal{R}X_k,
$$
for any $1\leq i<k$ since
$$
X_1\otimes_\mathcal{A}\cdots\otimes_\mathcal{A}X_k
+X_1\otimes_\mathcal{A}\cdots\otimes_\mathcal{A}
X_{i-1}\otimes_\mathcal{A}
(\mathcal{R}_1^{-1}\rhd X_{i+1})\otimes_\mathcal{A}
(\mathcal{R}_2^{-1}\rhd X_i)\otimes_\mathcal{A}
X_{i+2}\otimes_\mathcal{A}\cdots\otimes_\mathcal{A}X_k
$$
is an element of $I$.
Then
\begin{align*}
    X\wedge_\mathcal{R}Y
    =&X_1\wedge_\mathcal{R}\cdots\wedge_\mathcal{R}X_k
    \wedge_\mathcal{R}Y_1\wedge_\mathcal{R}\cdots\wedge_\mathcal{R}Y_\ell\\
    =&(-1)^\ell X_1\wedge_\mathcal{R}\cdots\wedge_\mathcal{R}X_{k-1}\\
    &\wedge_\mathcal{R}(\mathcal{R}_{1(1)}^{-1}\rhd Y_1)
    \wedge_\mathcal{R}\cdots
    \wedge_\mathcal{R}(\mathcal{R}_{1(\ell)}^{-1}\rhd Y_\ell)
    \wedge_\mathcal{R}(\mathcal{R}_2^{-1}\rhd X_k)\\
    =&(-1)^{2\cdot\ell}
    X_1\wedge_\mathcal{R}\cdots\wedge_\mathcal{R}X_{k-2}
    \wedge_\mathcal{R}((\mathcal{R}_{1}^{'-1}\mathcal{R}_{1}^{-1})_{(1)}
    \rhd Y_1)
    \wedge_\mathcal{R}\cdots\\
    &\wedge_\mathcal{R}((\mathcal{R}_{1}^{'-1}
    \mathcal{R}_{1}^{-1})_{(\ell)}
    \rhd Y_\ell)
    \wedge_\mathcal{R}(\mathcal{R}_2^{'-1}\rhd X_{k-1})
    \wedge_\mathcal{R}(\mathcal{R}_2^{-1}\rhd X_k)\\
    =&(-1)^{2\cdot\ell}
    X_1\wedge_\mathcal{R}\cdots\wedge_\mathcal{R}X_{k-2}
    \wedge_\mathcal{R}(\mathcal{R}_{1(1)}^{-1}\rhd Y_1)
    \wedge_\mathcal{R}\cdots\\
    &\wedge_\mathcal{R}(\mathcal{R}_{1(\ell)}^{-1}\rhd Y_\ell)
    \wedge_\mathcal{R}(\mathcal{R}_2^{-1}\rhd(
    X_{k-1}\wedge_\mathcal{R}X_k))\\
    =&\cdots\\
    =&(-1)^{k\cdot\ell}(\mathcal{R}_1^{-1}\rhd Y)
    \wedge_\mathcal{R}(\mathcal{R}_2^{-1}\rhd X)
\end{align*}
follows.
\end{proof}
It remains to generalize the concept of Gerstenhaber algebra to our setting.
\begin{definitionApp}
An associative unital graded algebra and $H$-equivariant braided symmetric
$\mathcal{A}$-bimodule $(\mathfrak{G}^\bullet,\wedge_\mathcal{R})$
is said to be a braided Gerstenhaber algebra if the module actions respect
the degree and if there is an $H$-equivariant
graded (with degree shifted by $-1$) braided Lie bracket
$\llbracket\cdot,\cdot\rrbracket_\mathcal{R}
\colon\mathfrak{G}^k\times\mathfrak{G}^\ell\rightarrow\mathfrak{G}^{k+\ell-1}$,
i.e. 
$$
\llbracket X,Y\rrbracket_\mathcal{R}
=-(-1)^{(k-1)(\ell-1)}\llbracket(\mathcal{R}_1^{-1}\rhd Y),
(\mathcal{R}_2^{-1}\rhd X)\rrbracket_\mathcal{R}
$$
and
$$
\llbracket X,\llbracket Y,Z\rrbracket_\mathcal{R}\rrbracket_\mathcal{R}
=\llbracket\llbracket X,Y\rrbracket_\mathcal{R},Z\rrbracket_\mathcal{R}
+(-1)^{(k-1)(\ell-1)}\llbracket(\mathcal{R}_1^{-1}\rhd Y),
\llbracket(\mathcal{R}_2^{-1}\rhd X),Z\rrbracket_\mathcal{R}\rrbracket_\mathcal{R}
$$
satisfying a graded braided Leibniz rule
$$
\llbracket X,Y\wedge_\mathcal{R} Z\rrbracket_\mathcal{R}
=\llbracket X,Y\rrbracket_\mathcal{R}\wedge_\mathcal{R}Z
+(-1)^{(k-1)\ell}(\mathcal{R}_1^{-1}\rhd Y)\wedge_\mathcal{R}
\llbracket(\mathcal{R}_2^{-1}\rhd X),Z\rrbracket_\mathcal{R}
$$
with respect to $\wedge_\mathcal{R}$ in addition.
Above $X\in\mathfrak{G}^k$, $Y\in\mathfrak{G}^\ell$ and $Z\in\mathfrak{G}^\bullet$.
\end{definitionApp}
Let $\mathfrak{G}^\bullet$ be a braided Gerstenhaber algebra.
It follows that $\mathfrak{G}^0$ is an associative, braided commutative
$H$-module algebra and $\mathfrak{G}^1$ is a braided Lie algebra.
Moreover, $\mathfrak{G}^1$ is an $H$-equivariant,
braided symmetric $\mathfrak{G}^0$-bimodule and
$\mathfrak{G}^k$ is an $H$-equivariant, braided symmetric
$\mathfrak{G}^1$-bimodule.
This means that for any $X\in\mathfrak{G}^1$ we can define the 
\textit{braided Lie derivative}
$\mathscr{L}_X^\mathcal{R}=\llbracket X,\cdot\rrbracket_\mathcal{R}
\colon\mathfrak{G}^k\rightarrow\mathfrak{G}^k$
which is a braided derivation, i.e.
$$
\mathscr{L}_X^\mathcal{R}(Y\wedge_\mathcal{R}Z)
=\mathscr{L}_X^\mathcal{R}Y\wedge_\mathcal{R}Z
+(\mathcal{R}_1^{-1}\rhd Y)\wedge_\mathcal{R}
(\mathcal{R}_2^{-1}\rhd\mathscr{L}_X^\mathcal{R})Z
$$
for all $X\in\mathfrak{G}^1$ and $Y,Z\in\mathfrak{G}^\bullet$.
It furthermore satisfies
$\mathscr{L}^\mathcal{R}_{[X,Y]_\mathcal{R}}
=\mathscr{L}^\mathcal{R}_X\mathscr{L}^\mathcal{R}_Y
-\mathscr{L}^\mathcal{R}_{\mathcal{R}_1^{-1}\rhd Y}
\mathscr{L}^\mathcal{R}_{\mathcal{R}_2^{-1}\rhd X}$ for all
$X,Y\in\mathfrak{G}^1$.
